\documentclass[11p,a4paper]{article}
\usepackage{amsfonts}
\usepackage{amssymb}
\usepackage{amsthm}
\usepackage{amsmath}
\usepackage{graphicx}
\usepackage{mathrsfs}
\usepackage{graphics}
\usepackage{abstract}
\usepackage{color}
\usepackage{cite}
\usepackage[colorlinks=true,citecolor=red]{hyperref}
\hypersetup{linkcolor=blue}
\graphicspath{ {./result/} }
\usepackage{titling}
\usepackage{authblk}
\usepackage{geometry}
\usepackage{fancyhdr}
\usepackage{lineno}
\usepackage{titlesec}
\titleformat{\section}[block]{\centering\Large\bfseries}{\thesection}{1em}{}
\titleformat{\subsection}[block]{\centering\large\bfseries}{\thesubsection}{1em}{}
\titleformat{\subsubsection}[block]{\centering\normalsize\bfseries}{\thesubsubsection}{1em}{}

\newtheorem{theorem}{\textbf{Theorem}}[section]
\newtheorem{lemma}{\textbf{Lemma}}[section]
\newtheorem{proposition}{\textbf{Proposition}}[section]
\newtheorem{corollary}{\textbf{Corollary}}[section]
\newtheorem{remark}{\textbf{Remark}}[section]
\newtheorem{definition}{\textbf{Definition}}[section]

\providecommand{\keywords}[1]{\textbf{\textit{Keywords---}} #1}
\allowdisplaybreaks[4]

\def\be{\begin{equation}}
\def\ee{\end{equation}}
\def\bea{\begin{eqnarray}}
\def\eea{\end{eqnarray}}
\def\bt{\begin{theorem}}
\def\et{\end{theorem}}
\def\bl{\begin{lemma}}
\def\el{\end{lemma}}
\def\br{\begin{remark}}
\def\er{\end{remark}}
\def\bp{\begin{proposition}}
\def\ep{\end{proposition}}
\def\bc{\begin{corollary}}
\def\ec{\end{corollary}}
\def\bd{\begin{definition}}
\def\ed{\end{definition}}

\def\Pi{\mathbf{\psi}}

\newcommand{\abs}[1]{\left\lvert #1 \right\rvert}
\newcommand{\norm}[1]{\left\| #1 \right\|}

\usepackage{tikz}
\usetikzlibrary{calc,arrows,intersections}
\title{\bf{On the dynamical Rayleigh-Taylor instability of non-homogeneous fluid in annular region with Naiver-slip boundary }}

\author{
Liang Li$^{a}$
\thanks{llbohou@gzhu.edu.cn}
\quad\quad
Quan Wang$^{b}$
\thanks{Corresponding author:xihujunzi@scu.edu.cn }
\\ \footnotesize $^a$ School of Mathematics and Information Science,
\footnotesize Guangzhou University,
\\\footnotesize Guangzhou,
 Guangdong, 510000, China
  \\ \footnotesize $^{b}$ College of Mathematics, Sichuan University,
  \footnotesize
 Chengdu, Sichuan, 610065,  China
}
\medskip

\begin{document}
\maketitle
\begin{abstract}
This paper investigates the well-posedness and Rayleigh-Taylor (R-T) instability for a system of two-dimensional nonhomogeneous incompressible fluid, subject to the non-slip and Naiver-slip boundary conditions at the outer and inner boundaries, respectively, in an annular region. In order to effectively utilize the domain shape, we analyze this system in polar coordinates. First, for the well-posedness to this system, based on  the spectral properties of Stokes operator, Sobolev embedding inequalities and Stokes' estimate in the context of the specified boundary conditions, etc, we obtain the local existence of weak and strong solutions using semi-Galerkin method and prior estimates. Second, for the density profile with the property that it is increasing along radial radius in certain region, we demonstrate that it is linear instability (R-T instability)  through Fourier series and the settlement of a family of modified variational problems. Furthermore, based on the existence of the linear solutions exhibiting exponential growth over time, we confirm the nonlinear instability of this steady state in both  Lipschitz and Hadamard senses by nonlinear energy estimates.

\end{abstract}

\newpage
\tableofcontents

\newpage
\section{Introduction}
The Rayleigh-Taylor (R-T) instability holds great significance in the realm of fluid dynamics. Its investigation traces back to Rayleigh's pioneering observations regarding the formation of cirrus clouds \cite{Jevons1857}, which laid the groundwork for subsequent in-depth research into this fascinating phenomenon. In \cite{Rayleigh1882,Rayleigh1900}, Rayleigh proposed that when a heavy fluid was situated above a lighter fluid, the upper fluid, influenced by gravity, would flow downward, resulting in an instability phenomenon; On the other hand, Taylor\cite{Taylor1950} considered the case where a lighter fluid was positioned above a heavier fluid and experienced a downward acceleration, leading to similar fluid movement. Taylor argued that these two types of instability phenomena were consistent with each other.

The R-T instability phenomena is observed widely in various contexts, ranging from small-scale quantum plasmas\cite{Bychkov2008,Modestov2009,Momeni2013} to large-scale interstellar medium and galaxy clusters\cite{Zweibel1991,Robinson2004,Jones2005,Ruszkowski2007,Scannapieco2008,Novak2011,Caproni2015}. It has extensive applications across fields, such as geophysics, astrophysics, nuclear physics, industry and art. For instance, in geophysics, R-T instability is closely related to the formation of volcanic islands and salt domes\cite{Wilcock1991,Mazariegos1996}, convective thinning of the lithosphere\cite{Chen2015}, weather inversions\cite{David2006}, and coast upwelling near the eastern boundary of oceans\cite{Cui2004}; Additionally, R-T instability plays a crucial role  in industrial process, including combustion\cite{Veynante2002}, underwater explosions\cite{Geers2002} and inertial confinement fusion\cite{Srinivasan2012}; In oil painting, R-T instability contributes to the creation of more aesthetically pleasing lines\cite{Delacalleja2014,Zetina2014}. For applications of R-T instability in astrophysics and nuclear physics, please refer to the relevant literature\cite{Gull1975,Haan2011}. For a comprehensive overview of RT instability applications, we recommend consulting additional resources\cite{Zhou2021}.

In this paper, we consider the R-T instability of a two-dimensional (2D) incompressible, nonhomogeneous fluid, which is described by the following Naiver-Stokes equations:
   \begin{align}\label{main-equation}
\begin{cases}
\rho\frac{\partial \mathbf{u}}{\partial t}+ \rho(\mathbf{u} \cdot \nabla )\mathbf{u} =\mu\Delta \mathbf{u} -
 \nabla p+\rho \bf{g}
 \\ \frac{\partial \rho}{\partial t}+(\mathbf{u}\cdot  \nabla)\rho
 =0
\\ \nabla \cdot  \mathbf{u}=0,
\end{cases}
\end{align}
where the unknowns \(\mathbf{u}=\left(u_{1},u_{2}\right)\), \(\rho\) and \(p\) denote the velocity, the density and the pressure of the fluid, respectively. \(\mu>0\) represents the viscosity constant of the fluid and \(\mathbf{g}\) is the gravitational acceleration. Generally, researchers apply the approximate system (Boussinesq system) of this equation \eqref{main-equation} to study atmospheric and oceanographic dynamics, as well as other astrophysical dynamics where rotation and stratification play significant roles\cite{Pedlosky1987,Majda2003,William2019}. For instance, In \cite{Pedlosky1987}, the Boussinesq system is employed to model the large-scale geophysical flow such as oceanic or atmospheric currents. The Boussinesq system, with its rich physical background, has 
become a fertile ground for mathematical exploration. Relevant works can be found in \cite{Dreyfuss2021,Dong2022,Junha2022,Juhi2023,Yao2024} and reference therein.
However,  it is worth mentioning that the Boussinesq system is only applicable in situations where the density variation is small relative to other factors. Physically, it is more accurate to use the system \eqref{main-equation} rather than the Boussinesq system to investigate the dynamic behavior of fluids.

Given that the nonhomogeneous incompressible and viscous Naiver-Stokes equations (NIVNSE) play a crucial role in fluid dynamic, it has attracted the interest of many researchers since it was established. For example, Kim\cite{kim_weak_1987} established the local existence of weak solutions to the NIVNSE contained in an open bounded subset of \(\mathbf{R^{3}}\) with a smooth boundary, subject to the non-slip boundary condition; Choe\cite{Choe2003} studied strong solutions of the NIVNSE in a domain, where the domain is either the whole space \(\mathbf{R^{3}}\) (where the fluid velocity approaches a constant vector at infinity) or a bounded open subset of \(\mathbf{R}^{3}\) with smooth boundary (where the fluid remains stationary at the boundary); Danchin\cite{Danchin2014} considered the global existence and uniqueness of solutions to the NIVNSE in the half-space \(\mathbf{R}_{+}^{d}\) (\(d\geq 2\)), also subject to the non-slip boundary condition. For more information about well-posedness of NIVNSE, one can refer to\cite{Simon1978,Simon1990,Lions1996,Danchin2009}. Additionally, there are lots of research on the R-T instability of the NIVNSE. For instance, Jiang\cite{Jiang2013} investigated the nonlinear instability in Lipschitz sense of a smooth steady density profile solution to the NIVNSE in \(\mathbf{R}^{3}\) and in the presence of a uniform gravitational filed; And another study by Jiang\cite{Jiang2014} examined the nonlinear instability in Hadamard sense of some steady states of a three-dimensional NIVNSE driven by gravity in a bounded domain of class \(C^{2}\), where the boundary condition is Dirichlet type. For more research on R-T instability, one can refer to \cite{Livescu2021,Sonf2021,Cheung2023,Hwang2003,Guo2010,Jiang2019,Jiang20191,Peng2022,Mao2024,Xing2024}.

It is worth noting that the aforementioned mathematical investigations on well-posedness to the NIVNSE mainly focus on the non-slip boundary condition. Furthermore, there is a lack of studies addressing the R-T instability of nonhomogeneous fluids in annular regions. The R-T instability of NIVNSE in annular regions has significant applications in astrophysics and atmospheric science. For instance, planetary rings represent typical annular fluid systems, and the study of R-T instability contributes to our understanding the dynamics of flow with these rings. In the field of atmospheric science, the global atmosphere near equator is also an annular fluid system, and investigating R-T instability in this context can help illuminate the complex behavior of the global atmosphere in the equatorial region.

In summary, regarding the system \eqref{main-equation}, the main contributions of this article are twofold: one is the well-posedness issue, and the other is the R-T instability.   
Here, the region where the fluid flows is an annular domain \(\Omega=\{(x,y)\in \mathbf{R}^2|R_{1}^2\leq x^2+y^2\leq R_{2}^{2}\}\) and the gravitational acceleration \(\mathbf{g}\) is directed toward the center rather than vertically downward, as illustrated in \autoref{zhongli1119}. In addition,  we apply the Naiver-slip and Dirichlet boundary conditions for the inner and outer boundary conditions to the system \eqref{main-equation}, respectively, as follows, 
\begin{align}\label{bijie1120}
\begin{aligned}
&~~~~~~~~~~\mathbf{u}\cdot \mathbf{n}|_{x^2+y^2=R_{1}^2}=0,~\mathbf{u}|_{x^2+y^2=R_{2}^2}=\mathbf{0},
\\
&\left[
-p\mathbf{I}+\mu\left(\nabla\mathbf{u}+\left(\nabla\mathbf{u}\right)^{\text{T}_{r}}\right)\cdot\mathbf{n}
\right]\cdot\tau=\alpha\mathbf{u}\cdot\tau,~\text{as~}x^2+y^2=R_{1}^2,
\end{aligned}
\end{align}
where \(\mathbf{n}\) is the outward normal vector
field to \(\partial \Omega\), \(\tau\) is the corresponding tangential vector, \(\text{T}_r\) means matrix transposition, \(\mathbf{I}\) is the \(2\times 2\)
identity matrix, \(\alpha\), known as Naiver coefficient, is a parameter that can be a constant, a function, or even a matrix.  Naiver-slip boundary condition was introduced by the renowned mathematician and physicist C. Naiver in 1827\cite{Naiver1827}, taking into account fluid slip on the boundary and nonzero surface shear force of the fluid along the boundary. Note that in the \(\alpha\rightarrow \infty\) limit, the Naiver-slip boundary condition resembles the non-slip case, as illustrated in \cite{Doering2006,Amrouche2021}. On the other hand, \(\alpha=0\) yields the free-slip boundary condition, implying that Naiver-slip condition boundary conditions interpolate between the two extreme cases. 
For more information regarding the Naiver-slip boundary condition, one can refer to \cite{Ding2017,Ding2020,Hu2018,Wang2021}.
\begin{figure}[tbh]
\centering
        \begin{tikzpicture}[>=stealth',xscale=1,yscale=1,every node/.style={scale=1}]
\draw [thick,->] (0,-3) -- (0,3) ;
\draw [thick,->] (-3,0) -- (3,0) ;
\draw [thick,->] (1.5,1.5) -- (1.0,1.0) ;
\node [below right] at (3.1,0) {$x$};
\node [below left] at (1.2,1.5) {$\mathbf{g}$};
\node [left] at (0,3.1) {$y$};
\draw[domain = -2:360][samples = 200] plot({cos(\x)}, {sin(\x)});
\draw[domain = -2:360][samples = 200] plot({2*cos(\x)}, {2*sin(\x)});
\end{tikzpicture}
\caption{$\mathbf{g}=-g\left(\frac{x}{r}, \frac{y}{r}\right)^T$.}
\label{zhongli1119}
\end{figure}
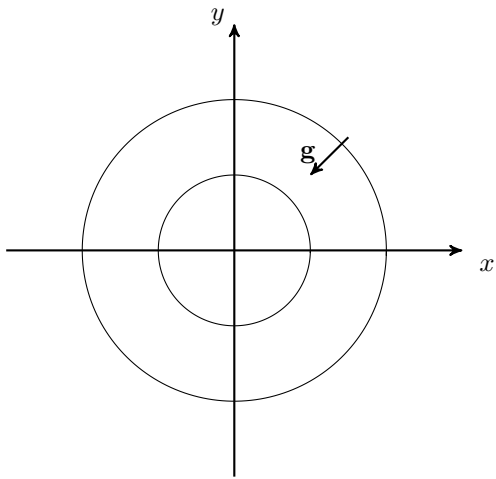

 To investigate  well-posedness of the system \eqref{main-equation}-\eqref{bijie1120}, it is more convenient to consider \eqref{main-equation} in polar coordinates since the region is annular. Let $x=r\cos{\theta}$ and $y=r\sin{\theta}$. Then one can obtain the transformation formula for the velocity as follows, 
\begin{align}\label{bianhuan0603}
  u_{1}=\frac{dx}{dt}=v_{r}\cos{\theta}-v_{\theta}\sin{\theta}, 
  ~u_{2}=\frac{dy}{dt}=v_{r}\sin{\theta}+v_{\theta}\cos{\theta},
\end{align}
which implies that 
\begin{align}\label{bianhuan0913}
v_{r}=u_{1}\cos{\theta}+u_{2}\sin{\theta},~v_{\theta}=u_{2}\cos{\theta}-u_{1}\sin{\theta}.
\end{align}
Consequently, the model \eqref{main-equation} can be rewritten in forms
  \begin{align}\label{model0329}
  \begin{cases}
    \rho\frac{\partial v_{r}}{\partial t}+\rho(\mathbf{u}\cdot\nabla_{r})v_{r}-\frac{\rho v_{\theta}^2}{r}
    =\mu(\Delta_{r} v_{r}-\frac{v_{r}}{r^2}-\frac{2}{r^2}\frac{\partial v_{\theta}}{\partial \theta})-
    \frac{\partial p}{\partial r}-\rho g, 
    \\
    \rho\frac{\partial v_{\theta}}{\partial t}+\rho(\mathbf{u}\cdot\nabla_{r})v_{\theta}+\frac{\rho v_{r} v_{\theta}}{r}
    =\mu(\Delta_{r} v_{\theta}-\frac{v_{\theta}}{r^2}+\frac{2}{r^2}\frac{\partial v_{r}}{\partial \theta})-
    \frac{1}{r}\frac{\partial p}{\partial \theta},
    \\
    \frac{\partial \rho}{\partial t}+v_{r}\frac{\partial \rho}{\partial r}
    +\frac{v_{\theta}}{r}\frac{\partial \rho}{\partial \theta}=0,
    \\
    \frac{\partial (rv_{r})}{\partial r}+\frac{\partial v_{\theta}}{\partial \theta}=0,
  \end{cases}
\end{align}
where 
\begin{align*}
  \mathbf{u}\cdot\nabla_{r}=v_{r}\frac{\partial}{\partial r}+\frac{v_{\theta}}{r}\frac{\partial}{\partial \theta},
  ~\Delta_{r}=\frac{\partial^2}{\partial r^2}+\frac{1}{r}\frac{\partial}{\partial r}+\frac{1}{r^2}
  \frac{\partial^2}{\partial \theta^2}.
\end{align*}
Meanwhile, the boundary conditions \eqref{bijie1120} and the domain \(\Omega\) can be transformed as follows, 
\begin{align}\label{bianjie0329}
  \begin{aligned}
  v_{r}|_{r=R_{1},R_{2}}&=v_{\theta}|_{r=R_{2}}=0,
  ~\frac{\partial v_{\theta}}{\partial r}
  =\left(\frac{1}{r}-\frac{\alpha}{\mu}\right) v_{\theta},~\text{as}~r=R_{1},
  \\
  &\widetilde{\Omega}=\left\{\left(r,\theta\right)|R_{1}\leq r\leq R_{2}, 0\leq\theta\leq 2\pi\right\}.
  \end{aligned}
\end{align}
Throughout this article, we require that \(1-\frac{\alpha R_{1}}{\mu}\geq 0\). In addition, the initial value is denoted by 
\begin{align}\label{chuzhi0617}
  \left(\mathbf{V},\rho\right)\left(r,\theta\right)|_{t=0}=
  \left(\left(v_{r},v_{\theta}\right),\rho\right)\left(r,\theta\right)|_{t=0}=\left(\mathbf{V}_{0},\rho_{0}\right)\left(r,\theta\right).
\end{align}

The investigation of well-posedness to the system \eqref{model0329}-\eqref{chuzhi0617} presents several challenges. First, 
analyzing the spectrum of the linear operator associated with this system \eqref{model0329} proves to be difficult for the existence of 
\(\eqref{model0329}_{3}\). This complication prevents us from employing directly the Galerkin
approximation method to problem \eqref{model0329}. Drawing inspiration from \cite{kim_weak_1987}, we utilize a semi-Galerkin method to address this problem \eqref{model0329}. Specifically, the Galerkin approximation method and the characteristic method are applied to 
\(\eqref{model0329}_{1,2}\) and \(\eqref{model0329}_{3}\), respectively. Second, it is well established that Stokes' estimate is an essential tool for improving the regularity of weak solutions and is a classical topic that has attracted considerable research activity. For instance, under Naiver boundary conditions, it has been shown in \cite{Clopeau1998} that Stokes' estimate holds for \(q\in (1,2]\) and in \cite{li_global_2021}, the estimate was proven for \(q=2\). However, this paper addresses a more complex situation as the outer and inner boundary conditions are distinct (see \eqref{bijie1120}). In addition, for the existence of \(\eqref{model0329}_{3}\), we also have to consider the estimate for \(q> 2\). To tackle these problems, we combined finite difference method for the outer boundary with the approach used for the inner boundary as described in \cite{Clopeau1998,li_global_2021} to derive the Stokes' estimate for \(q\in (1,+\infty)\), see Lemmas \ref{stokesguji}-\ref{stokesguji0521}.

After the application of semi-Galerkin method, we obtain a continuous and compact map (see Lemma \ref{lianxuan0618}), which enable us to verify the existence of approximate solution (see Remark \ref{bijie0618}) for this system \eqref{model0329}-\eqref{chuzhi0617} using the
Leray-Schauder fixed point theorem\cite{daoxingxia2011}. Afterwards, the local weak solution can be obtained by prior estimate of approximate solution, as described in Theorem \ref{jubucunzai0605}. By employing a smoother initial data, we then apply a standard method to obtain the local strong solution, as detailed in Theorems \ref{qiangjie0613} and \ref{qiangjiele0614}. Finally, we demonstrate the uniqueness of the strong solution by energy estimate, as shown in Theorem \ref{weiyixing0615}. 

Subsequently, we state the main conclusions about well-posedness.
\begin{theorem}\label{jubucunzai0605}Suppose \(\mathbf{V}_{0}\in\mathbf{G}^{1,2}\), \(\rho_{0}\in L^{\infty}
  \left(\widetilde{\Omega}\right)\) and \(\rho_{0}\left(r,\theta\right)>\delta>0\) (\(\delta\) is a constant fixed) 
  in \(\widetilde{\Omega}\). Then, 
  there exist a number \(T^{*}>0\) and a weak solution \(\left(\mathbf{V},\rho\right)\)
  to the system \eqref{model0329}-\eqref{chuzhi0617} satisfying 
  \(\mathbf{V}\in L^{2}\left(0,T;\mathbf{V}^{1,2}\right)
  \cap L^{\infty}\left(0,T;\mathbf{G}^{1,2}\right)\) and \(\rho\in L^{\infty}\left(\left(0,T\right)
  \times \widetilde{\Omega}\right)\) for any \(T<T^{*}\). And we have the following estimates 
  \begin{align}\label{yaobaoqian0618}
    \begin{aligned}
      &\left\|r^{\frac{1}{q}}\rho\right\|_{L^{q}\left(\widetilde{\Omega}\right)}=
      \left\|r^{\frac{1}{q}}\rho_{0}\right\|_{L^{q}\left(\widetilde{\Omega}\right)}~\text{for~any~}q>1,
      ~\left\|\widetilde{\nabla}\mathbf{V}\right\|_{L^2\left(\widetilde{\Omega}\right)}^2\leq 
      C,
      \\
      &\int_{0}^{t}\left\|\widetilde{\Delta}\mathbf{V}\right\|_{L^2\left(\widetilde{\Omega}\right)}^2ds
\leq C,~
\int_{0}^{t}\left\|\sqrt{\rho}\frac{\partial \mathbf{V}}{\partial s}
\right\|_{L^2\left(\widetilde{\Omega}\right)}^2ds\leq C
    \end{aligned}
  \end{align}
  where \(C=C\left(\mu,R_{1},R_{2},\rho_{0},\mathbf{V}_{0}\right)>0\).
\end{theorem}
The definition of a weak solution can be found in section \ref{zhuyaozhengming1212}, while the definitions of \(\mathbf{G}^{1,2}\), \(\mathbf{V}^{1,2}\), \(\widetilde{\Delta}\) and \(\widetilde{\nabla}\) are provided in section \ref{chubuzhishi1212}. Note that the existence of \(p\) follows immediately from by a classical method,
see Proposition \ref{pdecunzaixing0528}. 
\begin{theorem}\label{qiangjie0613}If \(\mathbf{V}_{0}\in\mathbf{V}^{1,2}\) in Theorem \ref{jubucunzai0605}, 
  then we have 
  \begin{align}\label{xiekaidian0614}
    \begin{aligned}
      &\sqrt{\rho}\partial_{t}\mathbf{V}\in L^{\infty}\left(0,T;\left[L^2
      \left(\widetilde{\Omega}\right)\right]^2\right);~
      \partial_{t}\mathbf{V}\in L^{2}\left(0,T;\left[H^{1}\left(\widetilde{\Omega}\right)\right]^2\right);
      \\
      &\mathbf{V}\in L^{\infty}\left(0,T;\mathbf{V}^{1,2}\right)\cap 
      L^{2}\left(0,T;\left[W^{2,4}\left(\widetilde{\Omega}\right)\right]^2\right);
      \\
      &\nabla p\in L^{\infty}\left(0,T;\left[L^{2}\left(\widetilde{\Omega}\right)\right]^2\right)
      \cap L^{2}\left(0,T;\left[L^{4}\left(\widetilde{\Omega}\right)\right]^2\right).
    \end{aligned}
  \end{align}
\end{theorem} 
\begin{theorem}\label{qiangjiele0614}Let \(\rho_{0}\in H^{1}\left(\widetilde{\Omega}\right)\) in 
  Theorem \ref{qiangjie0613}. Then we have 
  \begin{align}\label{midu0614}
    \rho\in L^{\infty}\left(0,T;H^1\left(\widetilde{\Omega}\right)\right),~
    \partial_{t}\rho\in L^{\infty}\left(0,T;L^{2}\left(\widetilde{\Omega}\right)\right).
  \end{align}
\end{theorem}
\begin{theorem}\label{weiyixing0615}The strong solution 
   \(\left(\mathbf{V},\rho\right)\) to \eqref{model0329}-\eqref{bianjie0329} is unique. 
\end{theorem}
\begin{remark}\(p\) is not unique, but \(\widetilde{\nabla} p\) is unique. 
\end{remark}

With the local well-posedness in hand, the second goal of this paper is to study the R-T instability of a smooth steady profile solution \(\left(\overline{\mathbf{V}},\overline{\rho},\overline{p}\right)\) to the system \eqref{model0329}-\eqref{chuzhi0617}, where 
\begin{align}\label{steadystate0403}
\overline{v}_{r}=0, \overline{v}_{\theta}=0,~
0<\overline{\rho}=\overline{\rho}(r)\in C^{\infty}([R_{1},R_{2}]),
~\overline{p}=\overline{p}(r),
\end{align}
and 
\begin{align}\label{zizhu0410}
D\overline{\rho}(r_{s}):=\frac{d\overline{\rho}(r_{s})}{dr}>0,~D\overline{p}:=
\frac{d\overline{p}}{dr}=-\overline{\rho}g, ~\text{for~some~}r_{s}\in (R_{1},R_{2}).
\end{align}

It is not difficult to find that the condition \eqref{zizhu0410} indicates the occurrence of R-T instability, as there exists at least one region where the density is monotonically increasing. Before stating our main conclusions about R-T instability, we introduce the following perturbation, 
\[v_{r}=\widetilde{v}_{r}, \rho=\widetilde{\rho}+\overline{\rho}(r),~
v_{\theta}=\widetilde{v}_{\theta},~
p=\widetilde{p}+\overline{p}(r).
\] 
Put the above equalities into \eqref{model0329}, we then obtain
\begin{align}\label{raodongfeixianxing0403}
  \begin{cases}
    \begin{aligned}
      \widetilde{\rho}\partial_{t}\widetilde{v}_{r}+\overline{\rho}
      \partial_{t} \widetilde{v}_{r}&+
      \widetilde{\rho}\left(\mathbf{\widetilde{u}}\cdot\nabla_{r}\right)\widetilde{v}_{r}
      +\overline{\rho}\left(\mathbf{\widetilde{u}}\cdot \nabla_{r}\right)
      \widetilde{v}_{r}-\frac{\widetilde{\rho}\widetilde{v}_{\theta}^2}{r}
      -\frac{\overline{\rho}\widetilde{v}_{\theta}^2}{r}
      \\
      &=\mu\left(\Delta_{r}\widetilde{v}_{r}-\frac{\widetilde{v}_{r}}{r^2}
      -\frac{2}{r^2}
      \partial_{\theta}\widetilde{v}_{\theta}\right)
      -\partial_{r}\widetilde{p}-\widetilde{\rho}g,
    \end{aligned}
    \\
    \begin{aligned}
      \widetilde{\rho}\partial_{t}\widetilde{v}_{\theta}+
    \overline{\rho}\partial_{t}\widetilde{v}_{\theta}&+
    \widetilde{\rho}\left(\mathbf{\widetilde{u}}\cdot\nabla_{r}\right)\widetilde{v}_{\theta}+
    \overline{\rho}\left(\mathbf{\widetilde{u}}\cdot\nabla_{r}\right)\widetilde{v}_{\theta}
    +\frac{\widetilde{\rho}\widetilde{v}_{r}\widetilde{v}_{\theta}}{r}+
    \frac{\overline{\rho}\widetilde{v}_{r}\widetilde{v}_{\theta}}{r}
    \\
    &=\mu\left(\Delta_{r}\widetilde{v}_{\theta}-\frac{\widetilde{v}_{\theta}}{r^2}+
    \frac{2}{r^2}\partial_{\theta}\widetilde{v}_{r}\right)-
    \frac{1}{r}\partial_{\theta}\widetilde{p},
    \end{aligned}
    \\
\partial_{t}\widetilde{\rho}+\widetilde{v}_{r}\partial_{r}\widetilde{\rho}
    +\frac{\widetilde{v}_{\theta}}{r}\partial_{\theta}\widetilde{\rho}+
\widetilde{v}_{r}D\overline{\rho}=0,~ \partial_{r} \left(r\widetilde{v}_{r}\right)
    +\partial_{\theta} \widetilde{v}_{\theta}=0,
  \end{cases}
\end{align} 
where \(\mathbf{\widetilde{u}}\cdot\nabla_{r}=\widetilde{v}_{r}\frac{\partial}{\partial r}+\frac{\widetilde{v}_{\theta}}{r}\frac{\partial}{\partial \theta}\), the notations \(\partial_{t}\) , \(\partial_{r}\) and \(\partial_{\theta}\) can be found in \eqref{zuihao1212}. The linearized equations read as follows, 
\begin{align}\label{xiaxing0403}
  \begin{cases}
    \overline{\rho}\partial_{t} \widetilde{v}_{r}=
    \mu\left(\Delta_{r}\widetilde{v}_{r}-\frac{\widetilde{v}_{r}}{r^2}
      -\frac{2}{r^2}
\partial_{\theta}\widetilde{v}_{\theta}\right)
      -\partial_{r}\widetilde{p}-\widetilde{\rho}g,
      \\
      \overline{\rho}\partial_{t} \widetilde{v}_{\theta}=
      \mu\left(\Delta_{r}\widetilde{v}_{\theta}-\frac{\widetilde{v}_{\theta}}{r^2}+
      \frac{2}{r^2}\partial_{\theta}\widetilde{v}_{r}\right)-
      \frac{1}{r}\partial_{\theta}\widetilde{p},
      \\
      \partial_{t}\widetilde{\rho}=-\widetilde{v}_{r}D\overline{\rho},~
      \partial_{r} \left(r\widetilde{v}_{r}\right)
    +\partial_{\theta} \widetilde{v}_{\theta}=0,
  \end{cases}
\end{align}
and the corresponding boundary conditions are in forms 
\begin{align}\label{wuming1211}
\widetilde{v}_{r}|_{r=R_{1},R_{2}}=\widetilde{v}_{\theta}|_{r=R_{2}}=0,
  ~\partial_{r} \widetilde{v}_{\theta}
  =\left(\frac{1}{r}-\frac{\alpha}{\mu}\right) \widetilde{v}_{\theta},~\text{as}~r=R_{1}.
\end{align}

The conclusions about instability are summarized in the following three theorems. 

\begin{theorem}(\textbf{Linear instability})\label{xianxing1211}Assume that \(1-\frac{\alpha R_{1}}{\mu}\geq 0\) and the steady density profile \(\overline{\rho}\) satisfies \eqref{zizhu0410}. Then, the steady state \(\left(\overline{\mathbf{V}},\overline{\rho}\right)\) is linearly unstable. Namely, there exists initial value \(\left(\widetilde{\mathbf{V}}_{0},\widetilde{\rho}_{0}\right)=\left(\left(\widetilde{v}_{0r},\widetilde{v}_{0\theta}\right),\widetilde{\rho}_{0}\right)\) and constant \(\lambda>0\) such that 
\[\left(\widetilde{\mathbf{V}},\widetilde{\rho}\right)=e^{\lambda t}\left(\widetilde{\mathbf{V}}_{0},\widetilde{\rho}_{0}\right)\]
is the solution to the system \eqref{xiaxing0403}-\eqref{wuming1211} with initial value \(\left(\widetilde{\mathbf{V}}_{0},\widetilde{\rho}_{0}\right)\).
\end{theorem}
 To construct linearly unstable solutions to the perturbed linear equations \eqref{xiaxing0403}-\eqref{wuming1211}, we seek solutions with an exponential growth factor of the form \(e^{\lambda(k)t}\) where \(k\in \mathbf{Z}/{0}\) represents the  \(\theta-\)frequency. As described in subsection \ref{buwending1211}, we utilize separation of spatiotemporal variables and Fourier series for the \(\theta-\)variable. This approach reduces the linear equations to an ordinary differential equation(ODE) (see \eqref{kanbian0405}) which is solved by the modified variational method\cite{Guo2010}, leading to the existence of solution to ODE and the exponential growth rate \(\lambda_{0}\left(k\right)>0\) for each \(k\). Furthermore, by treating  \(k\) as continuous variable, we establish that \(\lambda_{0}\left(k\right)\) attains maximum value, denoted by \(\widetilde{\Lambda}\)(see Remark \ref{lisanzuidazhi1101}) which can be achieved at a finite frequency \(k\). Consequently, we can derive the linearly unstable solution to the linear equations \eqref{xiaxing0403}-\eqref{wuming1211} for each fixed \(k\in \mathbf{Z}/{0}\), as detailed in subsection \ref{gouzao1211}. Finally, applying the principle of superposition, we construct the linearly unstable solution to the linear equations \eqref{xiaxing0403}-\eqref{wuming1211} with the specified exponential growth rate, as discussed in subsection \ref{butongzengzhang1211}.
 
\begin{theorem}(\textbf{Nonlinear instability in Lipschitz sense})
\label{lipuxizhi1104} Assume that \(1-\frac{\alpha R_{1}}{\mu}\geq 0\), \(\inf\limits_{\left(r,\theta\right)\in\widetilde{\Omega}}\) \(\left\{\widetilde{\rho}_{0}+\overline{\rho}\right\}\)\(>\delta\), where \(\delta>0\) is a constant given and 
the steady state \(\left(\mathbf{0},\overline{\rho}\right)\) of \eqref{model0329}-\eqref{chuzhi0617} satisfies \eqref{steadystate0403}-\eqref{zizhu0410}, then, the steady state \(\left(\mathbf{0},\overline{\rho}\right)\) is unstable in Lipschitz sense. That is, for any \(K>0\), \(\delta_{0}>0\) small enough and Lispchitz continuous function \(\overline{F}\) satisfying 
\[
\overline{F}\left(y\right)\leq Ky,~\text{for~any~}y\in [0,\infty),
\]
there exists an initial value 
\[\left(\widetilde{\mathbf{V}}_{0},\widetilde{\rho}_{0}\right)\in\left(H^{2}\left(\widetilde{\Omega}\right)\right)^{2}\times \left[H^{1}\left(\widetilde{\Omega}\right)\cap L^{\infty}\left(\widetilde{\Omega}\right)\right],\]
where \(\left\|\left(\widetilde{\mathbf{V}}_{0},\widetilde{\rho}_{0}\right)\right\|_{H^2\left(\widetilde{\Omega}\right)}:=\sqrt{\left\|\widetilde{\mathbf{V}}_{0}\right\|_{H^2\left(\widetilde{\Omega}\right)}^2+\left\|\widetilde{\rho}_{0}\right\|_{H^{1}\left(\widetilde{\Omega}\right)}^2}\leq \delta_{0}\), the unique strong solution \(\left(\widetilde{\mathbf{\mathbf{V}}},\widetilde{\rho}\right)\) to the nonlinear problem \eqref{raodongfeixianxing0403} with the initial value \(\left(\widetilde{\mathbf{V}}_{0},\widetilde{\rho}_{0}\right)\) satisfies 
\[
\left\|\widetilde{v}_{r}\left(t_{K}\right)\right\|_{L^2\left(\widetilde{\Omega}\right)}>\overline{F}\left(\left\|\left(\widetilde{\mathbf{V}}_{0},\widetilde{\rho}_{0}\right)\right\|_{H^2\left(\widetilde{\Omega}\right)}\right),~t_{K}=\frac{2}{\widetilde{\Lambda}}\ln{\frac{2K}{a}}\in (0,T_{max}),
\]
where \(a:=\frac{\left\|\widetilde{v}_{0r}\right\|_{L^2\left(\widetilde{\Omega}\right)}}{\delta_{0}}\) and the constant \(\widetilde{\Lambda}\) is defined in Remark \ref{lisanzuidazhi1101} and \(T_{max}\) denotes the maximal existence time of the solution \(\left(\widetilde{\mathbf{V}},\widetilde{\rho}\right)\).
\end{theorem}
Our approach to the proof of Theorem \ref{lipuxizhi1104} can be outlined as follows,

$(1)$ Verify the uniqueness of the strong solution to linearized perturbed problem \eqref{xiaxing0403} and \eqref{huajuan0403}, as detailed in Lemma \ref{manmanlai0627}.

$(2)$ Establish the estimate of the strong solution to  nonlinear problem \eqref{raodongfeixianxing0403}, as presented in Proposition \ref{mingti0627}.

$(3)$ Utilize above estimate to construct a family of solutions for
nonlinear problem \eqref{raodongfeixianxing0403}. Verify that the limit of this family of solutions corresponds to the solution
of linearized problem \eqref{xiaxing0403} and \eqref{huajuan0403}, leveraging the  Lemma \ref{manmanlai0627}.

$(4)$ Exploit the exponential growth rate of solution to
linearized problem  and the estimate (see Proposition \ref{mingti0627}) to derive contradiction, as showed in
Lemma \ref{zuixuyaodeyinli0701}.
\begin{theorem}\label{hadamardyiyixia1217}(\textbf{Nonlinear instability in Hadamard sense})\label{hadamard1105}
Under the conditions of Theorem \ref{lipuxizhi1104} and \(3\widetilde{\Lambda}>2\widetilde{\widetilde{\Lambda}}\), then \(\left(\mathbf{0},\overline{\rho}\right)\) is unstable in Hadamard sense. That is, there are constants  \(\epsilon\), \(\delta_{0}\) and functions \(\left(\widetilde{\mathbf{V}}_{0},\widetilde{\rho}_{0}\right)\)\(\in H^{2}\left(\widetilde{\Omega}\right)\times \left[H^{1}\left(\widetilde{\Omega}\right)\cap L^{\infty}\left(\widetilde{\Omega}\right)\right]\), such that for any \(\delta^{*}\in\left(0,\delta_{0}\right)\) and the initial value \(\left(\widetilde{\mathbf{V}}_{0}^{\delta^{*}},\widetilde{\rho}_{0}^{\delta^{*}}\right):=\delta^{*}\left(\widetilde{\mathbf{V}}_{0},\widetilde{\rho}_{0}\right)\), the strong solution \(\left(\widetilde{\mathbf{V}}^{\delta^{*}},\widetilde{\rho}^{\delta^{*}}\right)\in C\left(0,T_{\text{max}},\left[H^{1}\left(\widetilde{\Omega}\right)\right]^2\times L^{2}\left(\widetilde{\Omega}\right)\right)\) to the nonlinear problem \eqref{raodongfeixianxing0403} with the initial value \(\left(\widetilde{\mathbf{V}}_{0}^{\delta^{*}},\widetilde{\rho}_{0}^{\delta^{*}}\right)\) satisfies 
\[
\left\|\widetilde{\rho}^{\delta^{*}}\left(T^{\delta^{*}}\right)\right\|_{L^{1}\left(\widetilde{\Omega}\right)}\geq \epsilon,~
\left\|\widetilde{\mathbf{V}}^{\delta^{*}}\left(T^{\delta^{*}}\right)\right\|_{L^1\left(\widetilde{\Omega}\right)}\geq \epsilon,
\]
for some escape time \(0<T^{\delta^{*}}<T_{max}\), where \(T_{max}\) is the maximal existence time of \(\left(\widetilde{\mathbf{V}}^{\delta^{*}},\widetilde{\rho}^{\delta^{*}}\right)\), the constants \(\widetilde{\Lambda}\) and \(\widetilde{\widetilde{\Lambda}}\) are defined in Remark \ref{lisanzuidazhi1101} and subsection \ref{zuidazhi1212}, respectively.
\end{theorem}
\begin{remark}
    The condition \(3\widetilde{\Lambda}>2\widetilde{\widetilde{\Lambda}}\) can be satisfied when \(D\overline{\rho}>0\). In that case, it can be verified \(\widetilde{\Lambda}=\widetilde{\widetilde{\Lambda}}\) since the equations \eqref{yalouwuyan0403} possesses a variation structure. Then, the radial velocity of the solution to the system \eqref{yalouwuyan0403} at \(\lambda=\widetilde{\widetilde{\Lambda}}\) must also solve the equation \eqref{kanbian0405} at \(\Lambda=\widetilde{\widetilde{\Lambda}}\), this implies \(\widetilde{\widetilde{\Lambda}}\leq \widetilde{\Lambda}\) by Fourier series. Therefore from \(\widetilde{\Lambda}\leq \widetilde{\widetilde{\Lambda}}\) (see subsection \ref{zuidazhi1212}), we conclude that \(\widetilde{\widetilde{\Lambda}}=\widetilde{\Lambda}\).
\end{remark}
The outline of the poof to the Theorem \ref{hadamard1105} is as follows, 

$(1)$ Consider the linear instability from an alternative perspective to derive the maximum linear growth rate \(\widetilde{\widetilde{\Lambda}}\), as described in subsection \ref{zuidazhi1212}; 

$(2)$ conduct a nonlinear energy estimate for strong solution, as detailed in Proposition \ref{gujishizi0828};

$(3)$ Utilize the solution of the linearized equation corresponding to the growth rate \(\widetilde{\Lambda}\) (not \(\widetilde{\widetilde{\Lambda}}\)) to verify the existence of the nonlinear unstable solution, as shown in subsection \ref{zhengming1227}.

From Theorems \ref{lipuxizhi1104} and \ref{hadamardyiyixia1217}, one can easily identify the differences between the two types of nonlinear instabilities. Specifically, the instability in Lipschitz sense is characterized by  Lipshcitz functions while the instability in Hadamard sense is expressed directly. However, these two types of instability are related. On one hand, proving the two types  instability requires the existence of linear growth solution; on the other hand,
the proof processes for the two types of instabilities reveal that the result (\(\widetilde{\Lambda}>0\)) obtained from the linear instability analysis for the first nonlinear instability is useful for determining the maximum linear growth rate \(\widetilde{\widetilde{\Lambda}}\). This examination of the maximum linear growth rate aims to address the second type of instability.

Although there is considerable research on R-T instability, most studies focus on a single aspect of nonlinear instability. 
This paper examines both Lipschitz and Hadamard senses of nonlinear instability. In doing so, we can relax the requirement for the steady density profile \(\overline{\rho}\) for considering the nonlinear instability in Hadamard sense. Specifically, Jiang\cite{Jiang2014} investigated the nonlinear instability in sense of Hadamard for a three-dimensional nonhomogeneous incompressible
viscous flow driven by gravity in a bounded domain, where the derivative of the steady density profile \(\overline{\rho}\) is positive in every point. In contrast, here we only require \(D\overline{\rho}>0\) at some point, along with an additional assumption. In addition, while most works in the Hadamard sense require the strong solution \(\left(\widetilde{\mathbf{V}},\widetilde{\rho}\right)\in C\left(0,T; \left[H^{2}\left(\widetilde{\Omega}\right)\right]^2\times L^2\left(\widetilde{\Omega}\right)\right)\), we only require 
\(\left(\widetilde{\mathbf{V}},\widetilde{\rho}\right)\in C\left(0,T; \left[H^{1}\left(\widetilde{\Omega}\right)\right]^2\times L^2\left(\widetilde{\Omega}\right)\right)\), as discussed in Proposition \ref{gujishizi0828}. Ultimately, we establish the nonlinear instability of Hadamard sense in \(L^{1}-\)norm, which is a stronger conclusion than the instability in \(L^2-\)norm since \(\widetilde{\Omega}\) is bounded.

The remainder of this paper is organized as follows: Section \ref{chubuzhishi1212} introduces some essential notations, function spaces and lemmas including Stokes' estimate and the spectrum of Stokes operator; Section \ref{zhuyaozhengming1212} provides the proofs for the key conclusions regarding the well-posedness; Section \ref{xianxingbuwending1212} examines the linear instability to this system \eqref{raodongfeixianxing0403}-\eqref{wuming1211}; The sections \ref{lipschitz1211} and \ref{Hadamard1211} address the nonlinear instability in Lipschitz and Hadamard senses, respectively.
The section \ref{zijixiugai0113} presents the proofs for the lemmas in section \ref{chubuzhishi1212}.

\section{Preliminaries}\label{chubuzhishi1212}
In this section, we introduce function spaces, some notations and basic inequalities. 
For convenience, we define the following spaces for \(q\in\left(1,+\infty\right)\),
\begin{align}\label{kongjian0517}
  \begin{aligned}
    &\mathbf{G}^{1,q}=\left\{\mathbf{V}=
      \left(v_{r},v_{\theta}\right)\in \left[W^{1,q}\left(\widetilde{\Omega}\right)\right]^2\bigg{|}
      \frac{\partial \left(r v_{r}\right)}{\partial r}+\frac{\partial v_{\theta}}{\partial\theta}=0,~
      v_{r}|_{r=R_{1},R_{2}}=v_{\theta}|_{r=R_{2}}=0
    \right\},
    \\
    &\mathbf{V}^{1,q}=\left\{\mathbf{V}=
      \left(v_{r},v_{\theta}\right)\in \left[W^{2,q}\left(\widetilde{\Omega}\right)\right]^2\cap 
      \mathbf{G}^{1,q}\bigg{|}
      \frac{\partial v_{\theta}}{\partial r}=\left(\frac{1}{r}-\frac{\alpha}{\mu}\right)v_{\theta},~\text{as~}
      r=R_{1}
    \right\}.
  \end{aligned}
\end{align}
Note that 
all the functions in above spaces are periodic in \(\theta\). We also define that 
\begin{align}\label{zuihao1212}
\frac{\partial}{\partial t}:=\partial_{t},
~\frac{\partial}{\partial r}:=\partial_{r},~ \frac{\partial}{\partial \theta}:=\partial_{\theta},~
\frac{\partial^2}{\partial r^2}:=\partial_{rr},~
\frac{\partial^2}{\partial \theta^2}:=\partial_{\theta\theta},~
\frac{\partial^2}{\partial r\partial\theta}:=\partial_{r\theta},
\end{align}
and the following operators, 
\begin{align}\label{suanzi0525}
\widetilde{\nabla}=\left(\partial_{r},\partial_{\theta}\right),~
\widetilde{\Delta}=\partial_{rr}+\partial_{\theta\theta}.
\end{align}

The following presents some lemmas we need. The proofs of the lemmas will be provided in Appendix \ref{zijixiugai0113}, while the references for the lemmas without proofs will also be indicated.
\begin{lemma}\label{poincarebudengshi0517}
  Let \(\mathbf{V}=\left(v_{r},v_{\theta}\right)\in \mathbf{G}^{1,2}\). Then for each \(k\in \mathbf{N}^{*}\), we have 
  \begin{align}\label{mingtiannihao0518}
    \begin{aligned}
    \left\|\mathbf{V}\right\|_{L^{2k}\left(\widetilde{\Omega}\right)}\leq C 
    \left\|\widetilde{\nabla}\mathbf{V}\right\|_{L^2\left(\widetilde{\Omega}\right)},
  \end{aligned}
\end{align}
  where \(C=C\left(k,R_{1},R_{2}\right)\) is a positive constant depending on \(k,R_{1},R_{2}\).
\end{lemma}

\begin{lemma}\label{yuanlaihaishao0619} 
  Let \(\mathbf{V}=\left(v_{r},v_{\theta}\right)\in \mathbf{G}^{1,2}\). Then, there exists a positive constant 
  \(C=C\left(R_{1},R_{2}\right)\) satisfying 
  \[
  \left\|\mathbf{V}\right\|_{L^4\left(\widetilde{\Omega}\right)}^2\leq C\left 
  \|\widetilde{\nabla}\mathbf{V}\right\|_{L^2\left(\widetilde{\Omega}\right)}
  \left\|\mathbf{V}\right\|_{L^2\left(\widetilde{\Omega}\right)}.   
  \]
\end{lemma}

\begin{lemma}\label{tiduL2guji0518}Let \(\mathbf{V}=\left(v_{r},v_{\theta}\right)\in \mathbf{V}^{1,2}\).
  In addition, 
  \(1-\frac{\alpha R_{1}}{\mu}\geq 0\). Then 
  \begin{align}\label{pingban0518}
    \begin{aligned}
      &\left\|\partial_{r}v_{r}\right\|_{L^2\left(\widetilde{\Omega}\right)}^{2}\leq
      \left\|v_{r}\right\|_{L^2\left(\widetilde{\Omega}\right)}
      \left\|\partial_{rr}v_{r}\right\|_{L^2\left(\widetilde{\Omega}\right)},
      ~\left\|\partial_{\theta}v_{r}\right\|_{L^2\left(\widetilde{\Omega}\right)}^{2}
      \leq\left\|v_{r}\right\|_{L^2\left(\widetilde{\Omega}\right)}
      \left\|\partial_{\theta\theta}v_{r}\right\|_{L^2\left(\widetilde{\Omega}\right)},
      \\
      &\left\|\partial_{r}v_{\theta}\right\|_{L^2\left(\widetilde{\Omega}\right)}^{2}\leq
      \left\|v_{\theta}\right\|_{L^2\left(\widetilde{\Omega}\right)}
      \left\|\partial_{rr}v_{\theta}\right\|_{L^2\left(\widetilde{\Omega}\right)},
      ~\left\|\partial_{\theta}v_{\theta}\right\|_{L^2\left(\widetilde{\Omega}\right)}^{2}\leq
      \left\|v_{\theta}\right\|_{L^2\left(\widetilde{\Omega}\right)}
      \left\|\partial_{\theta\theta}v_{\theta}\right\|_{L^2\left(\widetilde{\Omega}\right)}.
    \end{aligned}
  \end{align}
\end{lemma}

\begin{lemma}\label{tiduL4guji0619} 
  Let \(\mathbf{V}=\left(v_{r},v_{\theta}\right)\in \mathbf{V}^{1,2}\) and \(1-\frac{\alpha R_{1}}{\mu}\geq 0\). Then,
  \begin{align*}
    \left\|\widetilde{\nabla}\mathbf{V}\right\|_{L^4\left(\widetilde{\Omega}\right)}^2
    \leq C\left(R_{1},R_{2}\right)\left\|\widetilde{\nabla}\mathbf{V}\right\|_{L^2\left(\widetilde{\Omega}\right)}
    \left\|\widetilde{\nabla}^2\mathbf{V}\right\|_{L^2\left(\widetilde{\Omega}\right)}.
  \end{align*}
\end{lemma}

\begin{lemma}\label{zuidazhi0517} 
  Let \(\mathbf{V}=\left(v_{r},v_{\theta}\right)\in \mathbf{V}^{1,2}\) and \(1-\frac{\alpha R_{1}}{\mu}\geq 0\). Then 
  \begin{align}\label{zuidazhi05172014}
    \begin{aligned}
      &\|v_{r}\|_{L^{\infty}\left(\widetilde{\Omega}\right)}^2\leq 
      2\left\|v_{r}\right\|_{L^2\left(\widetilde{\Omega}\right)}\left(
      \left\|\partial_{rr}v_{r}\right\|_{L^2\left(\widetilde{\Omega}\right)}^{\frac{1}{2}}
      \left\|\partial_{\theta\theta}v_{r}\right\|_{L^2\left(\widetilde{\Omega}\right)}^{\frac{1}{2}}+
        \left\|\partial_{r\theta}v_{r}\right\|_{L^2\left(\widetilde{\Omega}\right)}\right),
      \\
      &\|v_{\theta}\|_{L^{\infty}\left(\widetilde{\Omega}\right)}^2\leq 
      C\left\|v_{\theta}\right\|_{L^2\left(\widetilde{\Omega}\right)}
      \left(\left\|\partial_{rr}v_{\theta}\right\|_{L^2\left(\widetilde{\Omega}\right)}
      +\left\|\partial_{\theta\theta}v_{\theta}\right\|_{L^2\left(\widetilde{\Omega}\right)}
      +\|\partial_{r\theta}v_{\theta}\|_{L^2\left(\widetilde{\Omega}\right)}\right),
    \end{aligned}
  \end{align}
  where \(C>0\) is a constant depending on \(R_{1}\) and \(R_{2}\).
\end{lemma}

\begin{lemma}\label{dengjia0523}Let \(\mathbf{V}=\left(v_{r},v_{\theta}\right)\in \mathbf{G}^{1,2}\). Then, 
  \begin{align*}
      F_{1}\left(\mathbf{V}\right):=\int_{\widetilde{\Omega}}\left\{r\left(\left|\partial_{r} v_{r}\right|^2+\left|
     \partial_{r} v_{\theta}\right|^2\right)
     +\frac{1}{r}\left[\left(\partial_{\theta} v_{r}-v_{\theta}\right)^2 +
     \left(\partial_{\theta} v_{\theta}+v_{r}\right)^2\right]
     \right\}drd\theta,
  \end{align*}
  is equivalent to \(\left\|\widetilde{\nabla}\mathbf{V}\right\|_{L^2\left(\widetilde{\Omega}\right)}^2\). That is,
  there exist two positive constants \(C_{1}\) and \(C_{2}\) depending on \(R_{1}\) and \(R_{2}\) such that
  \[
  F_{1}\left(\mathbf{V}\right)\in \left[
   C_{1}\left\|\widetilde{\nabla}\mathbf{V}\right\|_{L^2\left(\widetilde{\Omega}\right)}^2, 
   C_{2}\left\|\widetilde{\nabla}\mathbf{V}\right\|_{L^2\left(\widetilde{\Omega}\right)}^2
  \right].
  \] 
\end{lemma}

\begin{lemma}\label{zhengqi0523}Let \(\mathbf{V}=\left(v_{r},v_{\theta}\right)\in \mathbf{V}^{1,2}\). Then 
  there exist two positive constants \(C_{1}\) and \(C_{2}\) depending on \(R_{1}\) and \(R_{2}\) such that 
  \[
  C_{1}\left\|\widetilde{\Delta}\mathbf{V}\right\|_{L^2\left(\widetilde{\Omega}\right)}^2\leq\left\|
  \begin{pmatrix}
    \Delta_{r}v_{r}-\frac{v_{r}}{r^2}-\frac{2}{r^2}\partial_{\theta} v_{\theta}
    \\
    \Delta_{r}v_{\theta}-\frac{v_{\theta}}{r^2}+\frac{2}{r^2}\partial_{\theta} v_{r}
  \end{pmatrix}
  \right\|_{L^2\left(\widetilde{\Omega}\right)}^2
  \leq C_{2}\left\|\widetilde{\Delta}\mathbf{V}\right\|_{L^2\left(\widetilde{\Omega}\right)}^2. 
  \]
\end{lemma}

Next, we introduce the Stokes' estimate which plays a crucial role in improving the 
regularity of weak solution. 
First, from the relationships \eqref{bianhuan0603} and \eqref{bianhuan0913} with the boundary condition \(\mathbf{u}\cdot \mathbf{n}|_{\partial \Omega}=0\), a direct computation gives that if \(\mathbf{u}=\left(u_{1},u_{2}\right)\in \left[W^{2,q}\left(\Omega\right)\right]^2\), then
\begin{align*}
\left(\partial_{1}u_{2}-\partial_{2}u_{1}\right)|_{x_{1}^2+x_{2}^2=R_{1}^2}=\left(\partial_{r} v_{\theta}+\frac{v_{\theta}}{r}\right)|_{r=R_{1}}.
\end{align*}
Then, combining \eqref{bianjie0329} with the above equality, one can obtain 
\begin{align}\label{huajian0126}
\left(\partial_{1}u_{2}-\partial_{2}u_{1}\right)|_{x_{1}^2+x_{2}^2=R_{1}^2}=\left(\frac{2}{r}-\frac{\alpha}{\mu}\right)v_{\theta}|_{r=R_{1}}.
\end{align}
We 
consider the following Stokes' equation, 
\begin{align}\label{stokeguji1020}
\begin{cases}
-\mu\Delta\mathbf{u}+\nabla p=\mathbf{f},
\\
\nabla\cdot\mathbf{u}=0,
\\
\mathbf{u}\cdot\mathbf{n}|_{x_{1}^2+x_{2}^2=R_{1}^2}=0,~\mathbf{u}|_{x_{1}^2+x_{2}^2=R_{2}^2}=\mathbf{0},
\\
\left[\left(-p \mathbf{I}+\mu\left(\nabla\mathbf{u}+\left(\nabla\mathbf{u}\right)^{\text{Tr}}\right)\right)\cdot\mathbf{n}\right]\cdot\tau=
  \alpha \mathbf{u}\cdot\tau,~\text{as~}x_{1}^2+x_{2}^2=R_{1}^2,
\end{cases}
\end{align}
where \(\mathbf{f}=\left(f_{1},f_{2}\right)\in \left[L^{q}\left(\Omega\right)\right]^2\) and \(q\in \left(0,+\infty\right)\).
Next, we give the definition of \(q-\)generalized solution to \eqref{stokeguji1020} as follows, 
\begin{definition}\(\left(\mathbf{u},p\right)\) is the \(q-\)generalized solution to \eqref{stokeguji1020}, if 
\begin{itemize}
    \item \(\mathbf{u}\in \widetilde{\mathbf{G}}^{1,q}\left(\Omega\right)=\left\{
    \mathbf{u}\in \left[W^{1,q}\left(\Omega\right)\right]^2|\nabla\cdot\mathbf{u}=0,~\mathbf{u}|_{r=R_{2}}=\mathbf{0}~\text{and~}\mathbf{u}\cdot\mathbf{n}|_{r=R_{1}}=0
    \right\}\) and \(p\in L^{q}\left(\Omega\right)\). 
    \item for any \(\mathbf{\phi}\in \widetilde{\mathbf{E}}^{1,q'}\left(\Omega\right)=\left\{
    \mathbf{\phi}=\left(\phi_{1},\phi_{2}\right)\in \left[W^{1,q'}\left(\Omega\right)\right]^2\bigg{|}\mathbf{\phi}|_{r=R_{2}}=\mathbf{0}~\text{and~}\mathbf{\phi}\cdot\mathbf{n}|_{r=R_{1}}=0
    \right\}\), we have 
    \begin{align}\label{youzhege1020}
\mu\int_{\Omega}\nabla\mathbf{u}\cdot\nabla\phi dx-\mu\left\langle \frac{\partial \mathbf{u}}{\partial\mathbf{n}},\phi\right\rangle
    -\int_{\Omega}p\nabla\cdot\phi dx=
    \int_{\Omega}\mathbf{f}\cdot\phi dx,
    \end{align}
    where \(\mu\left\langle \frac{\partial \mathbf{u}}{\partial\mathbf{n}},\phi\right\rangle:=-\int_{0}^{2\pi}\left(\mu-R_{1}\alpha\right)v_{\theta}\phi_{\theta}|_{r=R_{1}}d\theta\), \(\phi_{\theta}=\phi_{2}\cos{\theta}-\phi_{1}\sin{\theta}\), \(v_{\theta}=u_{2}\cos{\theta}-u_{1}\sin{\theta}\) and \(q'=\frac{q}{q-1}\).
\end{itemize}
\end{definition}
Using the knowledge of functional analysis, we give an equivalent definition of the above generalized 
  solution, and thus can prove the existence of the pressure \(p\), please refer to the Lemma IV.1.1 in \cite{Galdi2011} for the details. 
\begin{proposition}\label{dengjia1020}
If \(\mathbf{u}\in \widetilde{\mathbf{G}}^{1,q}\left(\Omega\right)\), and for any \(\mathbf{\Phi}=\left(\Phi_{1},\Phi_{2}\right)\in \widetilde{\mathbf{G}}^{1,q'}\left(\Omega\right)\), we have 
\begin{align*}
\mu\int_{\Omega}\nabla\mathbf{u}\cdot 
\nabla\mathbf{\Phi}dx-\mu 
\left\langle \frac{\partial \mathbf{u}}{\partial\mathbf{n}},\mathbf{\Phi}\right\rangle
=\int_{\Omega}\mathbf{f}\cdot \mathbf{\Phi}dx,
\end{align*}
then there exists a unique 
\(p\in \widetilde{L}^{q}\left(\Omega\right)=\left\{
p\in L^{q}\left(\Omega\right)|
\int_{\Omega}p dx=0
\right\}\) such that \(\left(\mathbf{u},p\right)\) is the \(q-\)generalized solution to \eqref{stokeguji1020} and 
\begin{align}\label{youshihou1020}
\left\|p\right\|_{L^q\left(\Omega\right)}\leq C\left(\left\|\mathbf{f}\right\|_{L^q\left(\Omega\right)}+\left\|\mathbf{u}\right\|_{W^{1,q}\left(\Omega\right)}\right),
\end{align}
where \(C=C\left(\Omega\right)>0\).
\end{proposition}
Now, we can verify the following estimate, 
\begin{lemma}\label{stokesguji}
    If \(\left(\mathbf{u},p\right)\) is the \(q-\)generalized solution to \eqref{stokeguji1020}, then
    \(\mathbf{u}\) satisfies the boundary condition \(\eqref{stokeguji1020}_{4}\) and 
    \begin{align}\label{manzu1020}
    \left\|\nabla^2\mathbf{u}\right\|_{L^{q}\left(\Omega\right)}
    +\left\|\nabla p\right\|_{L^{q}\left(\Omega\right)}\leq C\left(\Omega,\alpha,\mu\right)
    \left(\left\|\mathbf{f}\right\|_{L^{q}\left(\Omega\right)}+\left\|\mathbf{u}\right\|_{W^{1,q\left(\Omega\right)}}\right).
    \end{align}
\end{lemma}

It is easy for readers to verify the equivalence of norms in \(\Omega\) and \(\widetilde{\Omega}\). For instance, 
\(\left\|\mathbf{u}\right\|_{L^{q}\left(\Omega\right)}\) is equivalent to 
\(\left\|\mathbf{V}\right\|_{L^{q}\left(\widetilde{\Omega}\right)}\), where \(\mathbf{V}=\left(v_{r},v_{\theta}\right)\), \(\mathbf{u}\) and \(\mathbf{V}\) satisfy the relationships \eqref{bianhuan0603},\eqref{bianhuan0913}. In the polar coordinate, the Stokes' problem \eqref{stokeguji1020} can be transformed into the following equation, 
\begin{align}\label{baochianjing0521}
  \begin{cases}
    \mu\left(\Delta_{r}v_{r}-\frac{v_{r}}{r^2}-\frac{2}{r^2}
    \frac{\partial v_{\theta}}{\partial \theta}\right)-\frac{\partial p}{\partial r}=f_{1},
    \\
    \mu\left(\Delta_{r}v_{\theta}-\frac{v_{\theta}}{r^2}+\frac{2}{r^2}
    \frac{\partial v_{r}}{\partial \theta}\right)-\frac{1}{r}\frac{\partial p}{\partial \theta}=f_{2},
   \\
   \mathbf{V}=\left(v_{r},v_{\theta}\right)\in \mathbf{V}^{1,2},
  \end{cases}
\end{align}
where we denote the \(\mathbf{f}\) in \eqref{stokeguji1020} by \(\mathbf{f}\) for convenience. 

Parallelly, for the equation \eqref{baochianjing0521}, we have the definition of \(q-\)generalized solution and Stokes' estimates.
  \begin{definition}\(\left(\mathbf{V},p\right)\) is the \(q-\)generalized solution to \eqref{baochianjing0521}, if 
    \begin{enumerate}
    \item[\rm{(1)}] \(\mathbf{V}\in \mathbf{G}^{1,q}\) and \(p\in L^{q}\left(\widetilde{\Omega}\right)\);
    \item[\rm{(2)}] for any \(\widetilde{\mathbf{V}}=\left(\widetilde{v}_{r},\widetilde{v}_{\theta}\right)\in 
    \mathbf{E}^{1,q'}=\left\{\mathbf{V}=
      \left(v_{r},v_{\theta}\right)\in \left[W^{1,q'}\left(\widetilde{\Omega}\right)\right]^2\bigg{|}
      v_{r}|_{r=R_{1},R_{2}}=v_{\theta}|_{r=R_{2}}=0
    \right\}\), where \(q'=\frac{q}{q-1}\), we have 
    \begin{align*}
      \begin{aligned}
        &\mu \int_{\widetilde{\Omega}}\left[
        r\left(\partial_{r} v_{r} \partial_{r} \widetilde{v}_{r}
        +\partial_{r} v_{\theta}\partial_{r} \widetilde{v}_{\theta}\right)
        +\frac{1}{r}\left(\partial_{\theta} v_{r}
        \partial_{\theta}\widetilde{v}_{r}
        +\partial_{\theta} v_{\theta}
        \partial_{\theta}\widetilde{v}_{\theta}\right)
        \right]drd\theta
        \\
        &+\mu\int_{\widetilde{\Omega}}\frac{1}{r}\left(
         v_{r}\widetilde{v}_{r}+v_{\theta}\widetilde{v}_{\theta}+2\partial_{\theta} v_{\theta}
         \widetilde{v}_{r}-2\partial_{\theta} v_{r}\widetilde{v}_{\theta}
        \right)drd\theta+\int_{0}^{2\pi}\left(\mu-r\alpha\right)
        v_{\theta}\widetilde{v}_{\theta}|_{r=R_{1}}drd\theta
        \\
        &-\int_{\widetilde{\Omega}}p\left[\partial_{r} \left(r\widetilde{v}_{r}\right)
        +\partial_{\theta}\widetilde{v}_{\theta}\right]drd\theta=-\int_{\widetilde{\Omega}}
        r\left(f_{1}\widetilde{v}_{r}+f_{2}\widetilde{v}_{\theta}\right)drd\theta.
      \end{aligned}
    \end{align*}
    \end{enumerate}
    In particular, the \(2-\)generalized solution \(\left(\mathbf{V},p\right)\) 
    is called  weak solution.
  \end{definition}
  Similarly, we have the equivalent definition.
\begin{proposition}\label{pdecunzaixing0528}Let \(\mathbf{V}\in \mathbf{G}^{1,q}\). And for 
    any \(\widetilde{\mathbf{V}}\in \mathbf{G}^{1,q'}\), we have 
    \begin{align*}
      \begin{aligned}
        &\mu \int_{\widetilde{\Omega}}\left[
        r\left(\partial_{r} v_{r} \partial_{r} \widetilde{v}_{r}
        +\partial_{r} v_{\theta}\partial_{r} \widetilde{v}_{\theta}\right)
        +\frac{1}{r}\left(\partial_{\theta} v_{r}
        \partial_{\theta}\widetilde{v}_{r}
        +\partial_{\theta} v_{\theta}
        \partial_{\theta}\widetilde{v}_{\theta}\right)
        \right]drd\theta
        \\
        &+\mu\int_{\widetilde{\Omega}}\frac{1}{r}\left(
          v_{r}\widetilde{v}_{r}+v_{\theta}\widetilde{v}_{\theta}+2\partial_{\theta} v_{\theta}
          \widetilde{v}_{r}-2\partial_{\theta} v_{r}\widetilde{v}_{\theta}
         \right)drd\theta+\int_{0}^{2\pi}\left(\mu-r\alpha\right)
         v_{\theta}\widetilde{v}_{\theta}|_{r=R_{1}}drd\theta
        \\
        &=-\int_{\widetilde{\Omega}}
        r\left(f_{1}\widetilde{v}_{r}+f_{2}\widetilde{v}_{\theta}\right)drd\theta.
      \end{aligned}
    \end{align*}
    Then, there exists a unique \(p\in \widehat{L}^{q}\left(\widetilde{\Omega}\right)=\left\{ 
    p\in L^{q}\left(\widetilde{\Omega}\right)|\int_{\widetilde{\Omega}}rpdrd\theta=0\right\}\) such that 
    \(\left(\mathbf{V},p\right)\) is the \(q-\)generalized solution to \eqref{baochianjing0521}.
  \end{proposition}
  By virtue of the Lemma \ref{stokesguji}, we have the following Stokes' estimate. 
\begin{lemma}[\text{Stokes' estimate}]\label{stokesguji0521} 
  If \(\left(\mathbf{V},p\right)\) is \(q-\)generalized solution to \eqref{baochianjing0521}, then 
  \(\mathbf{V}\in\mathbf{V}^{1,q}\) and \(p\in W^{1,q}\left(\widetilde{\Omega}\right)\). Furthermore, 
  we have 
  \begin{align*}
    \left\|\widetilde{\nabla}^2\mathbf{V}\right\|_{L^q\left(\widetilde{\Omega}\right)}
    +\left\|\widetilde{\nabla}p\right\|_{L^q\left(\widetilde{\Omega}\right)}\leq C
    \left(\left\|\left(f_{1},f_{2}\right)\right\|_{L^q\left(\widetilde{\Omega}\right)}+ 
    \|\mathbf{V}\|_{W^{1,q}\left(\widetilde{\Omega}\right)}\right),
  \end{align*}
  where \(C>0\) is a constant. 
\end{lemma}
Since \(\mathbf{G}^{1,2}\) has an inner product, then it is not difficult to verify 
that \eqref{baochianjing0521} has a weak solution as \(q=2\) by Lax-Milgram theorem\cite{noauthor_partial_nodate}.
\begin{lemma}\label{nanrenyaoyoubaqi0528} The Stokes' equation \eqref{baochianjing0521} 
  has a weak solution \(\left(\mathbf{V},p\right)\in 
  \mathbf{G}^{1,2}\times \widehat{L}^{2}\left(\widetilde{\Omega}\right)\)
   when \(q=2\). Thus, from the Theorem \ref{stokesguji0521}, we can conclude that 
  \(\mathbf{V}=\left(v_{r},v_{\theta}\right)\in \mathbf{V}^{1,2}\) and 
  \(p\in H^{1}\left(\widetilde{\Omega}\right)\). Furthermore, we have
  \begin{align*}
    \left\|\widetilde{\nabla}^2\mathbf{V}\right\|_{L^2\left(\widetilde{\Omega}\right)}
    +\left\|\widetilde{\nabla}p\right\|_{L^2\left(\widetilde{\Omega}\right)}\leq C
    \left\|\left(f_{1},f_{2}\right)\right\|_{L^2\left(\widetilde{\Omega}\right)},
  \end{align*}
  where \(C>0\) is a constant. 
\end{lemma}

Subsequently, we give the conclusion about the Leray projection as \(q=2\) which is important in 
proving the existence of weak solution.  
\begin{lemma}\label{leraytouying0514}
  For any \(\mathbf{f}=\left(f_{1},f_{2}\right)\in \left[L^2\left(\widetilde{\Omega}\right)\right]^2\), we have 
  the following decomposition 
  \begin{align}\label{fenjiechengli0514}
    \mathbf{f}=
      \begin{pmatrix}
      v_{r}
      \\
      v_{\theta}
      \end{pmatrix}
    +
      \begin{pmatrix}
      \partial_{r} p
      \\
      \frac{1}{r}\partial_{\theta} p
      \end{pmatrix},
  \end{align}
  where \(v_{r}\) and \(v_{\theta}\) belong to \(L^2\left(\widetilde{\Omega}\right)\) satisfying 
  \(\partial_{r}\left(rv_{r}\right)+\partial_{\theta} v_{\theta}=0\) and 
  \(v_{r}\left(R_{1},\theta\right)=v_{r}\left(R_{2},\theta\right)=0\) in the 
  distribution sense, \(p\in H^{1}\left(\widetilde{\Omega}\right)\). Furthermore, the above decomposition is 
  unique. 
\end{lemma}
\begin{remark}\label{touying20240514}
  Now, we can define a bounded and linear symmetric operator \(\mathbf{P}:
  \left[L^2\left(\widetilde{\Omega}\right)\right]^2\rightarrow \mathbf{H}_{div}\), where 
  \(\mathbf{H}_{div}=\left\{\left(v_{r},v_{\theta}\right)\in\left[L^2\left(\widetilde{\Omega}\right)\right]^2
  \bigg{|}\partial_{r}\left(rv_{r}\right)+\partial_{\theta} v_{\theta}=0~
  \text{and}~ 
  v_{r}|_{r=R_{1},R_{2}}=0 \text{~in~distribution~sense}
  \right\}\). That is, for any \(\mathbf{f}\in \left[L^2\left(\widetilde{\Omega}\right)\right]^2\), 
  \eqref{fenjiechengli0514} holds, and 
  \[
  \mathbf{P}\left(\mathbf{f}\right)=\begin{pmatrix}
    v_{r}
    \\
    v_{\theta}
  \end{pmatrix}.  
  \]
 Moreover, it is easy to verify that 
  \(\|\mathbf{P}\|= 1\).
\end{remark}
\begin{remark}\label{touyinggujijiehe0521}From Lemma \ref{nanrenyaoyoubaqi0528} and Lemma \ref{leraytouying0514}, 
  we can conclude that 
  \[
    \left\|\begin{pmatrix}
      \Delta_{r}v_{r}-\frac{v_{r}}{r^2}-\frac{2}{r^2}
  \partial_{\theta} v_{\theta}
  \\
  \Delta_{r}v_{\theta}-\frac{v_{\theta}}{r^2}+\frac{2}{r^2}
  \partial_{\theta} v_{r}
    \end{pmatrix}\right\|_{L^2\left(\widetilde{\Omega}\right)}
  \leq C\left\|\mathbf{P}
      \begin{pmatrix}
        \Delta_{r}v_{r}-\frac{v_{r}}{r^2}-\frac{2}{r^2}
    \partial_{\theta} v_{\theta}
    \\
    \Delta_{r}v_{\theta}-\frac{v_{\theta}}{r^2}+\frac{2}{r^2}
    \partial_{\theta} v_{r}
      \end{pmatrix}\right\|_{L^2\left(\widetilde{\Omega}\right)}.
  \]
  Thus, from Lemma \ref{zhengqi0523}, we can conclude that 
  \[
   \left\|\widetilde{\Delta}\mathbf{V}\right\|_{L^{2}\left(\widetilde{\Omega}\right)}
  \leq C\left\|\mathbf{P}
     \left(\widetilde{\Delta}\mathbf{V}\right)\right\|_{L^2\left(\widetilde{\Omega}\right)},
  \]
  where \(\mathbf{V}=\left(v_{r},v_{\theta}\right)\) satisfies 
  \begin{align*}
    \begin{cases}
      \mu\left(\Delta_{r}v_{r}-\frac{v_{r}}{r^2}-\frac{2}{r^2}
      \partial_{\theta} v_{\theta}\right)-\partial_{r} p=f_{1},
      \\
      \mu\left(\Delta_{r}v_{\theta}-\frac{v_{\theta}}{r^2}+\frac{2}{r^2}
      \partial_{\theta} v_{r}\right)-\frac{1}{r}\partial_{\theta} p=f_{2},
     \\
     \mathbf{V}=\left(v_{r},v_{\theta}\right)\in \mathbf{V}^{1,2},
    \end{cases}
  \end{align*}
  where \(\left(f_{1},f_{2}\right)\in \mathbf{H}_{div}\).
\end{remark}

To employ the Galerkin method, we have to consider the following eigenvalue problem,
\begin{align}\label{lingyizhong04018}
  \begin{cases}
    \mu\left(\Delta_{r}v_{r}-\frac{v_{r}}{r^2}-\frac{2}{r^2}
    \partial_{\theta} v_{\theta}\right)-\partial_{r} p=-\beta v_{r},
    \\
    \mu\left(\Delta_{r}v_{\theta}-\frac{v_{\theta}}{r^2}+\frac{2}{r^2}
    \partial_{\theta} v_{r}\right)-\frac{1}{r}\partial_{\theta} p=-\beta v_{\theta},
    \\
    \mathbf{V}=\left(v_{r},v_{\theta}\right)\in \mathbf{V}^{1,2}.
  \end{cases}
\end{align}
\begin{lemma}\label{tezhengzhi04018}Let
  \(1-\frac{\alpha R_{1}}{\mu}\geq 0\). Then, 
  about system \eqref{lingyizhong04018}, we have the following conclusions, 
  \begin{enumerate}
    \item[\rm{(1)}] there exist three sequences \(\left\{\mathbf{e}_{n}\right\}\subset\) 
    \(\left[C^{\infty}\left(\widetilde{\Omega}\right)\right]^{2}\), 
    \(\left\{p_{n}\right\}\) and 
    \(\{\beta_{n}\}\subset \mathbf{R}^{+}\), where \(\mathbf{e}_{n}=\left(e_{r,n},e_{\theta,n}\right)\), such that 
    \begin{align*}
      \begin{cases}
        \mu\left(\Delta_{r}e_{r,n}-\frac{e_{r,n}}{r^2}-\frac{2}{r^2}
        \partial_{\theta} e_{\theta,n}\right)-\partial_{r} p_{n}=-\beta_{n} e_{r,n},
        \\
        \mu\left(\Delta_{r}e_{\theta,n}-\frac{e_{\theta,n}}{r^2}+\frac{2}{r^2}
        \partial_{\theta} e_{r,n}\right)-\frac{1}{r}\partial_{\theta} p_{n}
        =-\beta_{n} e_{\theta,n},
        \\
       \mathbf{e}_{n}=\left(e_{r,n},e_{\theta,n}\right)\in \mathbf{V}^{1,2}.
      \end{cases}
    \end{align*}
    where 
    \(\int_{\widetilde{\Omega}}r\left[\left|e_{r,n}\right|^2+\left|e_{\theta,n}\right|^2\right]drd\theta=1\).
    \item[\rm{(2)}] \(0<\beta_{1}\leq \beta_{2}\leq\cdots\leq \beta_{n}\rightarrow+\infty\) as \(n\rightarrow+\infty\). 
    \item[\rm{(3)}] For any \(\left(v_{r},v_{\theta}\right)\in \left[H^{1}\left(\widetilde{\Omega}\right)
    \right]^2\), we have 
    \begin{align*}
      \begin{aligned}
        \left\|v_{r}-\sum\limits_{k=1}^{m}c_{k}e_{r,k}\right\|_{H^{1}\left(\widetilde{\Omega}\right)}
        +\left\|v_{\theta}-\sum\limits_{k=1}^{m}c_{k}e_{\theta,k}\right\|_{H^{1}\left(\widetilde{\Omega}\right)}
        \rightarrow 0~,m\rightarrow+\infty,
      \end{aligned}
    \end{align*}
    where \(c_{k}=\int_{\widetilde{\Omega}}\left(rv_{r}e_{r,k}+rv_{\theta}e_{\theta,k}\right)drd\theta\).
  \end{enumerate}
\end{lemma}
  \begin{lemma}\label{budengshidezhengming0530}
    If the following equality holds, 
    \begin{align*}
      \left\|\widetilde{\nabla}\mathbf{V}\right\|_{L^2\left(\widetilde{\Omega}\right)}^2 
      \leq C\int_{0}^{t}\left\|\widetilde{\nabla}\mathbf{V}\right\|_{L^2\left(\widetilde{\Omega}\right)}^6ds
      +C_{1}(t+1),
    \end{align*}
    where \(C\) and \(C_{1}\) are positive constants and \(C_{1}\) is small enough, then there exist a positive constant \(C^*\) and \(T^{*}>0\) such that when 
    \(t\leq T^{*}\),
    \begin{align*}
      \left\|\widetilde{\nabla}\mathbf{V}\left(t\right)\right\|_{L^2\left(\widetilde{\Omega}\right)}^2
      \leq C^{*}.
    \end{align*}
  \end{lemma}

\section{The proofs of main results on well-posedness}\label{zhuyaozhengming1212}
We have described the conclusions regarding well-posedness to the system \eqref{model0329}-\eqref{chuzhi0617} in the Introduction, as illustrated in Theorems \ref{jubucunzai0605}-\ref{weiyixing0615}. This section provides the detailed proofs for the well-posedness. Now, we give the definition of weak solution to \eqref{model0329}-\eqref{chuzhi0617} as follows.
\begin{definition}\label{weaksolution0601}
  A pair of functions \(\mathbf{V}\left(t,r,\theta\right)=\left(v_{r}\left(t,r,\theta\right),
  v_{\theta}\left(t,r,\theta\right)\right)\) 
  and \(\rho\left(t,r,\theta\right)\) is a weak solution of \eqref{model0329}-\eqref{chuzhi0617} if 
  \(\mathbf{V}\left(t,r,\theta\right)\in L^2\left(0,T;\mathbf{G}^{1,2}\right)\), 
  \(\rho\left(t,r,\theta\right)\in L^{\infty}\left(0,T;\widetilde{\Omega}\right)\) and 
  \begin{align}\label{ruojiedecunzai0601}
    \begin{aligned}
      &\int_{0}^{T}\int_{\widetilde{\Omega}}r\rho \mathbf{V}\cdot\partial_{t} \mathbf{\Phi}
      drd\theta dt+\int_{0}^{T}\int_{\widetilde{\Omega}}\rho 
      \left(rv_{r}\partial_{r}\mathbf{\Phi}+v_{\theta}\partial_{\theta}\mathbf{\Phi}\right)\cdot 
      \mathbf{V}drd\theta dt 
      \\
      &+\int_{0}^{T}\int_{\widetilde{\Omega}}\rho\bigg{(}v_{\theta}^2\Phi_{r}-v_{r}v_{\theta}\Phi_{\theta}\bigg{)}
      drd\theta dt
      -\mu 
      \int_{0}^{T}\int_{\widetilde{\Omega}}
      \bigg{[}r\partial_{r}\mathbf{V}\cdot\partial_{r}\mathbf{\Phi}
      \\
      &+ 
      \frac{1}{r}\bigg{(}\partial_{\theta}\mathbf{V}\cdot\partial_{\theta}\mathbf{\Phi}
      +\mathbf{V}\cdot\mathbf{\Phi}-2v_{\theta}\partial_{\theta}\Phi_{r}+2v_{r}\partial_{\theta}\Phi_{\theta}
      \bigg{)}\bigg{]}drd\theta dt
      \\
      &-\int_{0}^{T}\int_{0}^{2\pi}r\left(\frac{\mu}{r}-\alpha\right)v_{\theta}\Phi_{\theta}|_{r=R_{1}}d\theta dt
      =-\int_{\widetilde{\Omega}}r\rho_{0}\mathbf{V}_{0}\mathbf{\Phi}\left(0,r,\theta\right)drd\theta +
      \int_{0}^{T}\int_{\widetilde{\Omega}}r\rho g\Phi_{r}drd\theta dt,
      \\
      &\int_{0}^{T}\int_{\widetilde{\Omega}}r\rho\partial_{t}\Psi drd\theta dt 
      +\int_{0}^{T}\int_{\widetilde{\Omega}}\rho\left(rv_{r}\partial_{r}\Psi
      +v_{\theta}\partial_{\theta}\Psi \right)drd\theta dt =- 
      \int_{\widetilde{\Omega}}r\rho_{0}\Psi\left(0,r,\theta\right)drd\theta,
    \end{aligned}
  \end{align}
  hold when \(\mathbf{\Phi}\left(t,r,\theta\right)=\left(\Phi_{r}\left(t,r,\theta\right), 
  \Phi_{\theta}\left(t,r,\theta\right)\right)\in C^{1}\left([0,T];\mathbf{G}^{1,2}\right)\) and  
  \(\Psi\in C^{1}\left([0,T];H^{1}\left(\widetilde{\Omega}\right)\right)\) satisfying 
  \(\mathbf{\Phi}\left(T,r,\theta\right)=\mathbf{0}\) and \(\Psi\left(T,r,\theta\right)=0\), where 
  \(\mathbf{V}_{0}\in \mathbf{G}^{1,2}\) and \(\rho_{0}\in L^{\infty}\left(\widetilde{\Omega}\right)\) are given.   
\end{definition}


\subsection{Proof of Theorem \ref{jubucunzai0605}}
The proof of Theorem \ref{jubucunzai0605} can be divided into four main parts: \textbf{Part 1}, constructing a 
mapping \(\mathbf{F}^{m}\) \(\left(\text{for~each~}m\in\mathbf{N^{*}}\text{~fixed}\right)\) 
defined on a Banach space \(\mathbf{Y}^{m}\), see Lemma \ref{gudingu04016} and Remark \ref{zhengming0618}; 
\textbf{Part 2}, deriving several estimates, independent of
\(m\), see \eqref{xianyanguji0618}, \eqref{yangge0603} and \eqref{yongdedao0603}; 
\textbf{Part 3}, using the estimates obtained to prove that \(\mathbf{F}^{m}\) is 
completely continuous and can be defined 
on a convex and closed subset \(B_{\widetilde{R}\mathbf{Y}^{m}}\) of \(\mathbf{Y}^{m}\), 
see Lemma \ref{lianxuan0618},
thus verifying the existence of 
approximate solution \(\left\{\mathbf{V}^{m},\rho^{m}\right\}\), see Remark \ref{bijie0618},
via Schauder fixed point theorem\cite{daoxingxia2011}; \textbf{Part 4}, passing \(m\rightarrow+\infty\) to obtain the 
weak solution.

\textbf{Part 1:} We briefly introduce how to construct \(\mathbf{F}^{m}:\mathbf{Y}^{m}\rightarrow 
\mathbf{Y}^{m}\) for each \(m\in\mathbf{N}^{*}\) fixed, where 
\begin{align*}
  \mathbf{Y}^{m}=\left(C\left[0,T\right]\right)^m,
  \end{align*}
  the element \(\mathbf{C}_m\left(t\right)=\left(C_{m1}\left(t\right),\dots,C_{mm}\left(t\right)\right)
  \in \mathbf{Y}^{m}\) and \(T>0\) is a constant less than certain number. Furthermore, we define the norm as 
  \(
    \left\|\mathbf{C}_{m}\right\|_{\mathbf{Y}^{m}}=\left[\sum\limits_{i=1}^{m}
    \max\limits_{t\in[0,T]}\left|C_{mi}\left(t\right)\right|^2\right]^{\frac{1}{2}}.
  \) It is not difficult to check that \(\mathbf{Y}^{m}\) is a Banach space. Besides, we define 
  the convex and closed subset of \(\mathbf{Y}^{m}\) as follows 
  \[
  B_{\widetilde{R}\mathbf{Y}^{m}}=\left\{
    \mathbf{C}_{m}\in\mathbf{Y}^{m}|\left\|\mathbf{C}_{m}\right\|_{\mathbf{Y}^{m}}\leq \widetilde{R}
  \right\},  
  \]
  where \(\widetilde{R}\) is a constant determined via prior estimates. 

  Now, we choose \(\widetilde{\mathbf{C}}_m\left(t\right)=
  \left(\widetilde{C}_{m1}\left(t\right),\dots,\widetilde{C}_{mm}\left(t\right)\right)
  \in \mathbf{Y}^{m}\) and define \(\widetilde{\mathbf{V}}^{m}=\sum\limits_{k=1}^{m}
  \widetilde{C}_{mk}\left(t\right)\mathbf{e}_{k}\), where \(\mathbf{e}_{k}=\left(e_{r,k},e_{\theta,k}\right)\) in 
  Lemma \ref{tezhengzhi04018}. Subsequently, from Lemma \ref{gudingu04016} where \(\mathbf{V}=\widetilde{\mathbf{V}}^{m}\), we obtain \(\widetilde{\rho}^{m}\). 
  Let \(\rho=\widetilde{\rho}^{m}\) in equations \(\eqref{model0329}_{1}\) and \(\eqref{model0329}_{2}\).
  Then, using Galerkin approximation method for equations \(\eqref{model0329}_{1}\) and \(\eqref{model0329}_{2}\), 
  we can obtain \(\mathbf{C}_{m}\left(t\right)\in \mathbf{Y}^{m}\), see Remark \ref{zhengming0618}. At present, we can define \(\mathbf{F}^{m}\left(\widetilde{\mathbf{C}}_{m}\left(t\right)\right)=\mathbf{C}_{m}\left(t\right)\).
  Therefore, we complete the explanation 
  of construction for \(\mathbf{F}^{m}\).

  Next, we give the details for how to construct \(\mathbf{F}^{m}\). Specifically,
we first consider the following equation when \(\mathbf{V}=\left(v_{r},v_{\theta}\right)\) is given, 
\begin{align}\label{zhengbao04016}
  \begin{cases}
    \partial_{t} \rho+v_{r}\partial_{r} \rho
    +\frac{v_{\theta}}{r}\partial_{\theta} \rho=0,
    \\
    \rho\left(0,r,\theta\right)=\overline{\rho}_{0},~\left(r,\theta\right)\in\widetilde{\Omega},
  \end{cases}
\end{align}
where \(\overline{\rho}_{0}\) is the mollification of \(\rho_{0}\).
\begin{lemma}\label{gudingu04016}
  Assume \(\mathbf{V}=\left(v_{r},v_{\theta}\right)\in 
  C\left(0,T;\left(C^{1}\left(\overline{\widetilde{\Omega}}\right)\right)^{2}\right)\), 
  \(\partial_{r}\left(rv_{r}\right)+\partial_{\theta}v_{\theta}=0\) and satisfies \eqref{bianjie0329}. Then, the equation \eqref{zhengbao04016} 
  has a unique solution \(\rho\in C^{1}\left(\left[0,T\right]\times \overline{\widetilde{\Omega}}\right)\).
\end{lemma}
\begin{proof}
  Let \(E\) be a ball in \(\mathbf{R^{2}}\) and \(\overline{\widetilde{\Omega}}\subset E\). Then we extend \(\mathbf{V}\) from 
  \(C\left(0,T;\left(C^{1}\left(\overline{\widetilde{\Omega}}\right)\right)^{2}\right)\) to 
  \(C\left(0,T;\left(C^{1}\left(\overline{E}\right)\right)^{2}\right)\) and denote it by \(\mathbf{w}=
  \left(w_{r},w_{\theta}\right)\), i.e., 
  for \(\left(t,r,\theta\right)\in \left[0,T\right]\times\overline{\widetilde{\Omega}}\), \(\mathbf{V}\equiv \mathbf{w}\). 
  Now, we employ the method of characteristics to solve the equation \eqref{zhengbao04016}. Therefore, 
  we consider the following equation, 
  \begin{align}\label{haofan04016}
    \begin{cases}
      \frac{d\widetilde{r}}{dt}=w_{r}\left(t,\widetilde{r},\widetilde{\theta}\right),
      \\
      \frac{d\widetilde{\theta}}{dt}=\frac{w_{\theta}\left(t,\widetilde{r},\widetilde{\theta}\right)}{\widetilde{r}},
      \\
      \left(\widetilde{r}\left(0\right),\widetilde{\theta}\left(0\right)\right)
      =\left(r_{0},\theta_{0}\right)\in \overline{\widetilde{\Omega}}.
    \end{cases}
  \end{align} 
From the Cauchy-Lipschitz theorem (see A.3 in \cite{bedrossian_mathematical_2022}), we can conclude that 
there exists a unique solution \(\left(\widetilde{r},\widetilde{\theta}\right)\left(t,r_{0},\theta_{0}\right)
\in C^{1}\left(0,\widetilde{T};
\left(C^{1}\left(\overline{\widetilde{\Omega}}\right)\right)^2\right)\), where \(0<\widetilde{T}\leq T\). In fact, we can 
let \(\widetilde{T}=T\). 

On one hand, if \(\left(r_{0},\theta_{0}\right)\), where \(r_{0}=R_{1},R_{2}\), 
one can obtain \(\widetilde{r}=R_{1},~R_{2}\) from \(w_{r}\left(t,R_{1},\widetilde{\theta}\right)
=w_{r}\left(t,R_{2},\widetilde{\theta}\right)=0\). Since \(\mathbf{w}\) is periodic in \(\theta\), then 
when \(\left(r_{0},\theta_{0}\right)\)\(=\left(r_{0},0\right)\) or \(\left(r_{0},2\pi\right)\), we can conclude that 
\(\left(\widetilde{r},\widetilde{\theta}\right)\in \overline{\widetilde{\Omega}}\). 
On the other hand, if \(\left(r_{0},\theta_{0}\right)\in \widetilde{\Omega}\), 
then we can conclude that \(\left(\widetilde{r},\widetilde{\theta}\right)\left(t\right)\in 
\overline{\widetilde{\Omega}}\) for all \(t\in\left[0,\widetilde{T}\right]\) because of the uniqueness. In brief, 
the flow generated by \eqref{haofan04016} remains in \(\overline{\widetilde{\Omega}}\) when initial point is in 
\(\overline{\widetilde{\Omega}}\). 
Hence, we can take \(\widetilde{T}=T\) and \(\mathbf{w}=\mathbf{V}\).

Furthermore, from the incompressible condition, we can conclude that 
\(\left(\widetilde{r},\widetilde{\theta}\right)\left(t,\cdot\right):\overline{\widetilde{\Omega}}
\rightarrow\overline{\widetilde{\Omega}}\) is \(C^{1}-\)diffeomorphism.
Therefore, we can define an inverse map \(\mathbf{A}\left(t,\cdot\right)=\left(A_{1},A_{2}\right)\) satisfying 
\(\mathbf{A}\left(0,r,\theta\right)=\left(r,\theta\right)\), \(\left(r,\theta\right)=\mathbf{A}
\left(t,\left(\widetilde{r},\widetilde{\theta}\right)\left(t,r,\theta\right)\right)\) 
and \(\left(r,\theta\right)=\left(\widetilde{r},\widetilde{\theta}\right)\left(t,\mathbf{A}
\left(t,r,\theta\right)\right)\) where \(\left(r,\theta\right)\in\overline{\widetilde{\Omega}}\). 

Next, we prove that \(\rho\left(t,r,\theta\right)=\overline{\rho}_{0}\left(\mathbf{A}\left(t,r,\theta\right)\right)\) 
is a solution to the equation \eqref{zhengbao04016}. For convenience, let \(\left(\widetilde{r},\widetilde{\theta}\right)
=\mathbf{X}=\left(X_{1},X_{2}\right)\). Then \(\frac{dX_{1}}{dt}=v_{r}\) and \(\frac{dX_{2}}{dt}=
\frac{v_{\theta}}{r}\).

From \(\mathbf{y}=\left(r,\theta\right)=\left(y_{1},y_{2}\right)=\mathbf{A}\left(t,\mathbf{X}
\left(t,\mathbf{y}\right)\right)\) and \(\mathbf{x}=\left(x_{1},x_{2}\right)
=\mathbf{X}
\left(t,\mathbf{A}\left(t,\mathbf{x}\right)\right)\), we have 
\begin{align*}
  \delta_{ij}=\frac{\partial A_{i}}{\partial y_{j}}=\sum\limits_{k=1}^{2}\frac{\partial A_{i}}{\partial y_{k}}
  \frac{\partial X_{k}}{\partial y_{j}},~
  0=\partial_{t}X_{k}+\sum\limits_{j=1}^{2}\frac{\partial X_{k}}{\partial y_{i}}\frac{\partial A_{i}}{\partial t},
\end{align*}
which implies that 
\[
\frac{\partial A_{i}}{\partial t}+v_{r}\frac{\partial A_{i}}{\partial y_{1}}+\frac{v_{\theta}}{r}
\frac{\partial A_{i}}{\partial y_{2}}=0.   
\]
A direct computation gives that 
\[
\partial_{t}\rho=\sum\limits_{i=1}^{2}\frac{\partial\overline{\rho}_{0}}{\partial y_{i}}
\frac{\partial A_{i}}{\partial t},~\partial_{y_{k}}\rho=\sum\limits_{i=1}^{2}
\frac{\partial \overline{\rho}_{0}}{\partial y_{i}}\frac{\partial A_{i}}{\partial y_{k}}, 
\]
which yields
\[
\partial_{t}\rho+v_{r}\partial_{r}\rho+\frac{v_{\theta}}{r}\partial_{\theta}\rho=
\frac{\partial\overline{\rho}_{0}}{\partial y_{1}}\left(\frac{\partial A_{1}}{\partial t}
+v_{r}\frac{\partial A_{1}}{\partial y_{1}}+
\frac{v_{\theta}}{r}\frac{\partial A_{1}}{\partial y_{2}}\right)+ 
\frac{\partial\overline{\rho}_{0}}{\partial y_{2}}\left(\frac{\partial A_{2}}{\partial t}
+v_{r}\frac{\partial A_{2}}{\partial y_{1}}+
\frac{v_{\theta}}{r}\frac{\partial A_{2}}{\partial y_{2}}\right)
=0.  
\]
One can easily verify the uniqueness of the equation 
\eqref{zhengbao04016} by taking \(\overline{\rho}_{0}=0\) and energy estimate.
\end{proof}
\begin{remark}\label{buweiling04016}
  From the expression \(\rho\left(t,r,\theta\right)=\overline{\rho}_{0}\left(\mathbf{A}\left(t,r,\theta\right)\right)\), 
  one can conclude that 
  if \(\overline{\rho}_{0}>\delta\)  everywhere, then \(\rho\left(t,r,\theta\right)>\delta\) everywhere for any \(t\geq 0\).
\end{remark}
  Now, \(0<\delta<\rho\in C^{1}\left(0,T;C^{1}\left(\widetilde{\Omega}\right)\right) \) is given, we consider the equations 
  \(\eqref{model0329}_{1}\) and \(\eqref{model0329}_{2}\). We multiply \(\eqref{model0329}_{1}\) and \(\eqref{model0329}_{2}\)
  by \(re_{r,k}\) and \(rv_{\theta,k}\), respectively,  add the results and integrate over \(\widetilde{\Omega}\),   
  then, we obtain the following equation
  \begin{align}\label{rhoquedinghou0618}
    \begin{aligned}
        &\int_{\widetilde{\Omega}}\rho r \partial_{t}\mathbf{V}^{m}\cdot \mathbf{e}_{k} 
        drd\theta 
        +\int_{\widetilde{\Omega}}\rho\left(v_{r}^{m}\frac{\partial v_{r}^{m}}{\partial r}
        +\frac{v_{\theta}^{m}}{r}\frac{\partial v_{r}^{m}}{\partial \theta}\right)re_{r,k}drd\theta 
        -\int_{\widetilde{\Omega}}\rho \left(v_{\theta}^{m}\right)^2e_{r,k}
        drd\theta 
        \\
        &+\int_{\widetilde{\Omega}}\rho\left(v_{r}^{m}\frac{\partial v_{\theta}^{m}}{\partial r}
        +\frac{v_{\theta}^{m}}{r}\frac{\partial v_{\theta}^{m}}{\partial \theta}\right)re_{\theta,k}drd\theta
        +\int_{\widetilde{\Omega}}\rho 
        v_{r}^{m}v_{\theta}^{m}e_{\theta,k}drd\theta= 
        -\int_{\widetilde{\Omega}}\rho g r e_{r,k}drd\theta
        \\
        &+\mu\int_{\widetilde{\Omega}}\left(\Delta_{r}v_{r}^{m}-\frac{v_{r}^{m}}{r^2}
        -\frac{2}{r^2}\frac{\partial v_{\theta}^{m}}{\partial\theta}\right)r e_{r,k}drd\theta 
        +\mu\int_{\widetilde{\Omega}}\left(\Delta_{r}v_{\theta}^{m}
        -\frac{v_{\theta}^{m}}{r^2}+\frac{2}{r^2}\frac{\partial v_{r}^{m}}{\partial \theta}\right)
        re_{\theta,k}drd\theta,
    \end{aligned}
  \end{align} 
  where \(\mathbf{V}^{m}=\left(v_{r}^{m},v_{\theta}^{m}\right)=\sum\limits_{k=1}^{m}C_{mk}\left(t\right)
  \mathbf{e}_{k}\), 
  \(\mathbf{V}^{m}_{0}=\sum\limits_{k=1}^{m}\left(\int_{\widetilde{\Omega}}r\mathbf{V}_{0}\cdot\mathbf{e}_{k}
  drd\theta\right)\mathbf{e}_{k}\) 
  and \(k\in\left\{1,2,\cdots,m\right\}\). It is easy to check that for any \(m\in\mathbf{N}^{*}\) fixed,
  \eqref{rhoquedinghou0618} is equivalent to the following ODEs,
  \begin{align}\label{nulitishengziji0502}
    \begin{cases}
      \sum\limits_{j=1}^{m}A_{j,k}^{m}\frac{dC_{mj}\left(t\right)}{dt}
      +\sum\limits_{j,s=1}^{m}B_{jsk}^{m}\left(t\right)C_{mj}\left(t\right)C_{ms}\left(t\right)
      +\beta_{k}C_{mk}\left(t\right)+\int_{\widetilde{\Omega}}\rho r g e_{r,k}drd\theta=0,
      \\
      C_{mk}\left(0\right)=\int_{\widetilde{\Omega}}r\mathbf{V}_{0}\cdot\mathbf{e}_{k}
      drd\theta,
    \end{cases}
  \end{align}
where \(k\in\left\{1,2,\cdots,m\right\}\) and 
\begin{align*}
  \begin{aligned}
    &A_{j,k}^{m}\left(t\right)=\int_{\widetilde{\Omega}}
    \rho r\mathbf{e}_{j}\cdot\mathbf{e}_{k}drd\theta, 
    \\
    &\begin{aligned}
      B_{jsk}^{m}\left(t\right)=&\int_{\widetilde{\Omega}} 
      \rho r\left[\left(e_{r,j}\partial_{r} e_{r,s}+ 
      \frac{e_{\theta,j}}{r}\partial_{\theta} e_{r,s}\right)e_{r,k}
      +\left(e_{r,j}\partial_{r} e_{\theta,s}+
      \frac{e_{\theta,j}}{r}\partial_{\theta} e_{\theta,s}\right)e_{\theta,k}\right]drd\theta 
      \\
      &-\int_{\widetilde{\Omega}}\rho\left(e_{\theta,j}e_{\theta,s}e_{r,k}-e_{r,s}e_{\theta,j}e_{\theta,k}\right)
      drd\theta. 
    \end{aligned}
  \end{aligned}
\end{align*}
In order to show the existence of solution for \eqref{nulitishengziji0502}, we have to 
verify that the matrix \(\mathbf{A}^{m}\left(t\right)=\left(A_{j,k}^{m}\left(t\right)\right)_{m\times m}\) 
is invertible. Therefore, we have the following proposition.  
\begin{proposition}\label{luying0502}
  For any \(m\in\mathbf{N}^{*}\) fixed and \(\forall t\in [0,T]\), 
  \(\mathbf{A}^{m}\left(t\right)\) is invertible and the elements in 
  the inverse \(\left(\mathbf{A}^{m}\left(t\right)\right)^{-1}\) belong to \(C^{1}\left(0,T\right)\) 
  when \(0<\rho\in C^{1}\left[0,T\right]\).  
\end{proposition}
\begin{proof}
  Suppose that \(\mathbf{A}^{m}\left(t\right)\) is not invertible, 
  then there exists a \(t_{0}\in [0,T]\) such that 
  \(\mathbf{A}^{m}\left(t_{0}\right)\) is not invertible. As a result, there is one row in
  \(\mathbf{A}^{m}\left(t_{0}\right)\) can be expressed by the linear combination of other 
  rows in \(\mathbf{A}^{m}\left(t_{0}\right)\). Without loss of generality, 
  we assume that the first row \(\mathbf{A}^{m}_{1}\left(t_{0}\right)\) is the linear combination of the rest rows 
  \(\mathbf{A}^{m}_{j}\left(t_{0}\right)\) \(j\in\left\{2,3,\cdots,m\right\}\). Therefore, there exist constants
  \(k_{2},k_{3},\cdots,k_{m}\) such that 
  \[
  \mathbf{A}_{1}^{m}\left(t_{0}\right)=\sum\limits_{i=2}^{m}k_{i}\mathbf{A}_{i}^{m}\left(t_{0}\right),  
  \]
  where not all of \(k_{i}\) (\(i\in\left\{2,3,\cdots,m\right\}\)) are zero. 

  Define \(\left(\eta_{r}^{m},\eta_{\theta}^{m}\right)= 
  \left(e_{r,1}-\sum\limits_{i=2}^{m}k_{i}e_{r,i},e_{\theta,1}-\sum\limits_{i=2}^{m}k_{i}e_{\theta,i}\right)\), 
  then 
  \[
  \int_{\widetilde{\Omega}}\rho r\left(\eta_{r}^m e_{r,k}+\eta_{\theta}^{m}e_{\theta,k}\right)drd\theta
  =0,~\forall k\in\left\{1,2,\cdots,m\right\},  
  \]
  which implies that 
  \[
    \int_{\widetilde{\Omega}}\rho r\left(\left|\eta_{r}^m\right|^2 +\left|\eta_{\theta}^{m}\right|^2\right)drd\theta
  =0.
  \]
  Thus, \(e_{r,1}=\sum\limits_{i=2}^{m}k_{i}e_{r,i}\) 
  and \(e_{\theta,1}=\sum\limits_{i=2}^{m}k_{i}e_{\theta,i}\). However, 
  from the orthogonality of \(\left\{\mathbf{e}_{k}\right\}\) in Lemma \ref{tezhengzhi04018}, we can 
  easily obtain that all \(k_{i}\) \(i\in\left\{i=2,3,\cdots,m\right\}\) are zeros which is a contradiction. 
  Thus, \(\mathbf{A}^{m}\left(t\right)\) is invertible for any \(m\) and \(t\).  
  
  Furthermore, by observing \(A_{j,k}^{m}\left(t\right)\) and the condition of \(\rho\), we can conclude that 
  all elements in the inverse \(\left(\mathbf{A}^{m}\left(t\right)\right)^{-1}\) are differential 
  with respect to \(t\). 
\end{proof}
\begin{remark}\label{zhengming0618}
  From the Proposition \ref{luying0502} and the theory of 
ODEs, we can conclude that \eqref{nulitishengziji0502} has a solution 
\(\mathbf{C}_{m}\left(t\right)=\left(C_{m1}\left(t\right),C_{m2}\left(t\right),\cdots, 
C_{mm}\left(t\right)\right)\) \(\in \left[C^{1}\left[0,T^{m}\right)\right]^{m}\) since 
\(B_{jsk}^{m}\left(t\right)\) and \(\int_{\widetilde{\Omega}}\rho r e_{r,k}drd\theta\) are differential 
with respect to \(t\).  Later, we will show that there exists a positive constant \(T^{*}\) independent of \(m\) and \(T^{m}\leq T^{*}\) by a prior estimate, see 
\eqref{yangge0603} and \eqref{yongdedao0603}.
\end{remark}
At present, we have constructed the mapping \(\mathbf{F}^{m}:\mathbf{Y}^{m}\rightarrow\mathbf{Y}^{m}\), see 
Lemma \ref{gudingu04016} and Remark \ref{zhengming0618}. More specifically, give a \(\widetilde{\mathbf{C}}_{m}
\left(t\right)\in\mathbf{Y}^{m}\), then obtain 
\(\rho^{m}\in C^{1}\left(\left[0,T\right]\times\widetilde{\Omega}\right)\) from Lemma \ref{gudingu04016}, finally, 
the existence of solution to \eqref{nulitishengziji0502} 
gives a \(\mathbf{C}_{m}\left(t\right)\in \mathbf{Y}^{m}\). Thus, we define \(\mathbf{F}^{m}\left(
  \widetilde{\mathbf{C}}_{m}\left(t\right)
\right)=\mathbf{C}_{m}\left(t\right)\).

\textbf{Part 2.}
Now, we derive estimations of \(\left(\mathbf{V}^{m},\rho\right)\), where 
\(\mathbf{V}^{m}\) satisfies \eqref{rhoquedinghou0618} and \(\rho\) satisfies 
\eqref{zhengbao04016}.  

Multiplying \eqref{zhengbao04016} by \(r\rho^{q-1}\) for any \(q>1\), 
then integrating the result over \(\widetilde{\Omega}\) and utilizing the incompressible condition give 
the following estimate for \(\rho\),
\begin{align*}
  \frac{d}{dt}\left\|r^{\frac{1}{q}}\rho\right\|_{L^{q}\left(\widetilde{\Omega}\right)}^{q}=0,
\end{align*} 
indicating that 
\begin{align}\label{xianyanguji0618}
  \left\|r^{\frac{1}{q}}\rho\right\|_{L^{q}\left(\widetilde{\Omega}\right)}=
  \left\|r^{\frac{1}{q}}\overline{\rho}_{0}\right\|_{L^{q}\left(\widetilde{\Omega}\right)}.
\end{align}

Subsequently, we multiply \eqref{rhoquedinghou0618} 
by \(\frac{dC_{mk}\left(t\right)}{dt}\) and sum over \(k=1,2,\cdots,m\), then after integrating by parts 
we obtain
\begin{align}\label{jianjuebudangtiangou0515}
  \begin{aligned}
    &\int_{\widetilde{\Omega}}\rho r \left|\partial_{t} \mathbf{V}^{m}\right|^2
     drd\theta+ 
    \frac{1}{2}\frac{d}{dt}\left[\int_{0}^{2\pi}\left(\mu-r\alpha\right)\left|v_{\theta}^{m}\right|^2
    |_{r=R_{1}}d\theta\right]
    \\
    &+\frac{\mu}{2}\frac{d}{dt}\left\{\int_{\widetilde{\Omega}}
    \left[r\left|\partial_{r} \mathbf{V}^{m}\right|^2
    +\frac{1}{r}\left(\left|\partial_{\theta} v_{r}^{m}-v_{\theta}^{m}\right|^2
    +\left|\partial_{\theta} v_{\theta}^{m}+v_{r}^{m}\right|^2\right)\right]
    drd\theta\right\}
    \\
    &+\int_{\widetilde{\Omega}}\rho r \left[
      \left(v_{r}^{m}\partial_{r} \mathbf{V}^{m}
      +\frac{v_{\theta}^{m}}{r}\partial_{\theta} \mathbf{V}^{m}\right)\cdot\partial_{t} \mathbf{V}^{m}
    \right]drd\theta
    \\
    &-\int_{\widetilde{\Omega}}\rho\left[\left(v_{\theta}^{m}\right)^2\partial_{t} v_{r}^{m}
    -v_{r}^{m}v_{\theta}^{m}\partial_{t} v_{\theta}^{m}\right]drd\theta
    +\int_{\widetilde{\Omega}}\rho r g\partial_{t} v_{r}^{m}drd\theta=0. 
  \end{aligned}
\end{align}
In addition, we multiply \eqref{rhoquedinghou0618} 
by \(-\beta_{k}C_{k}^{m}\left(t\right)\) and sum over \(k=1,2,\cdots,m\), then from Lemma \ref{tezhengzhi04018}
we have
\begin{align}\label{nansuanyedesuan0515}
  \begin{aligned}
    &\int_{\widetilde{\Omega}}\rho r\partial_{t}\mathbf{V}^{m}\cdot \mathbf{U}^{m}drd\theta
    -\mu\int_{\widetilde{\Omega}}r\left|\mathbf{U}^{m}\right|^2drd\theta
    +\int_{\widetilde{\Omega}}\rho r 
    \left[\left(v_{r}^{m}\partial_{r} \mathbf{V}^{m}
    +\frac{v_{\theta}^{m}}{r}\partial_{\theta} \mathbf{V}^{m}\right)\cdot \mathbf{U}^{m}\right]drd\theta
    \\
    &-\int_{\widetilde{\Omega}}\rho\left[\left(\left(v_{\theta}^{m}\right)^2,-v_{r}^{m}v_{\theta}^{m}\right)
    \cdot\mathbf{U}^{m}\right]drd\theta+\int_{\widetilde{\Omega}}
    \left(\rho rg,0\right)\cdot\mathbf{U}^{m}drd\theta=0,
\end{aligned}
\end{align}
where 
\[
\mathbf{U}^{m}=\mathbf{P}
\begin{pmatrix}
  \Delta_{r}v_{r}^{m}-\frac{v_{r}^{m}}{r^2}-\frac{2}{r^2}\frac{\partial v_{\theta}^{m}}{\partial \theta}
  \\
  \Delta_{r}v_{\theta}^{m}-\frac{v_{\theta}^{m}}{r^2}+\frac{2}{r^2}\frac{\partial v_{r}^{m}}{\partial \theta}
\end{pmatrix},~\mathbf{P}~\text{is~Leray~projection}.  
\]
Via the Cauchy inequality and \eqref{xianyanguji0618}, we have 
\begin{align}\label{diyige0516}
  \begin{aligned}
    &\left|\int_{\widetilde{\Omega}}
      \rho r \left[
      \left(v_{r}^{m}\partial_{r} \mathbf{V}^{m}
      +\frac{v_{\theta}^{m}}{r}\partial_{\theta} \mathbf{V}^{m}\right)\cdot\partial_{t} \mathbf{V}^{m}
    \right]drd\theta\right|
    \\
    &\leq \frac{1}{8}\int_{\widetilde{\Omega}}\rho r \left|\partial_{t} \mathbf{V}^{m}\right|^2
  drd\theta
    +C
    \int_{\widetilde{\Omega}}\left|\mathbf{V}^{m}\right|^2
    \left|\widetilde{\nabla}\mathbf{V}^{m}\right|^2
    drd\theta, 
  \end{aligned}
\end{align}
\begin{align}\label{dierge0516}
  \begin{aligned}
    \bigg{|}\int_{\widetilde{\Omega}}\rho\left[\left(v_{\theta}^{m}\right)^2\partial_{t} v_{r}^{m}
    -v_{r}^{m}v_{\theta}^{m}\partial_{t} v_{\theta}^{m}\right]drd\theta\bigg{|}
    \leq \frac{1}{8}\int_{\widetilde{\Omega}}\rho r \left|\partial_{t} \mathbf{V}^{m}\right|^2
    drd\theta+C
      \int_{\widetilde{\Omega}}\left|\mathbf{V}^{m}\right|^4
      drd\theta,
  \end{aligned}
\end{align}
\begin{align}\label{disange0516}
  \begin{aligned}
    \left|\int_{\widetilde{\Omega}}\rho r\partial_{t} v_{r}^{m}dr d\theta
    \right|
    \leq \frac{1}{8}\int_{\widetilde{\Omega}}\rho r \left|\partial_{t} \mathbf{V}^{m}\right|^2drd\theta
    +C,
  \end{aligned}
\end{align}
\begin{align}\label{disige0516}
  \begin{aligned}
\left|
\int_{\widetilde{\Omega}}\rho r\partial_{t}\mathbf{V}^{m}\cdot \mathbf{U}^{m}drd\theta
\right|\leq \frac{\mu}{8}\int_{\widetilde{\Omega}}r\left|\mathbf{U}^{m}\right|^2
drd\theta+C
\int_{\widetilde{\Omega}}\rho r\left|\partial_{t}\mathbf{V}^{m}\right|^2drd\theta,
  \end{aligned}
\end{align}
\begin{align}\label{diwuge0516}
  \begin{aligned}
    &\left|\int_{\widetilde{\Omega}}\rho r 
    \left[\left(v_{r}^{m}\partial_{r} \mathbf{V}^{m}
    +\frac{v_{\theta}^{m}}{r}\partial_{\theta} \mathbf{V}^{m}\right)\cdot \mathbf{U}^{m}\right]drd\theta\right|
    \\
    &\leq \frac{\mu}{8}\int_{\widetilde{\Omega}}r\left|\mathbf{U}^{m}\right|^2
    drd\theta+C
    \int_{\widetilde{\Omega}}\left|\mathbf{V}^{m}\right|^2
    \left|\widetilde{\nabla}\mathbf{V}^{m}\right|^2
    drd\theta,
  \end{aligned}
\end{align}
\begin{align}\label{diliuge0516}
  \begin{aligned}
    \left|
    \int_{\widetilde{\Omega}}\rho\left[\left(\left(v_{\theta}^{m}\right)^2,-v_{r}^{m}v_{\theta}^{m}\right)
    \cdot\mathbf{U}^{m}\right]drd\theta
    \right|\leq \frac{\mu}{8}\int_{\widetilde{\Omega}}r\left|\mathbf{U}^{m}\right|^2
    drd\theta+C\int_{\widetilde{\Omega}}
    \left|\mathbf{V}^{m}\right|^4drd\theta,
  \end{aligned}
\end{align}
\begin{align}\label{diqige0516}
  \begin{aligned}
    \left|
    \int_{\widetilde{\Omega}}
    \left(\rho rg,0\right)\cdot\mathbf{U}^{m}drd\theta
    \right|\leq \frac{\mu}{8}\int_{\widetilde{\Omega}}r\left|\mathbf{U}^{m}\right|^2
    drd\theta+C.
  \end{aligned}
\end{align}
The constants \(C\) are dependent of \(\rho_{0}\) or \(R_{i}\)\(\left(i=1,2\right)\). 
Then 
put \eqref{diyige0516}-\eqref{disange0516} into \eqref{jianjuebudangtiangou0515} and 
\eqref{disige0516}-\eqref{diqige0516} into \eqref{nansuanyedesuan0515}, we obtain
\begin{align}\label{yonglaihuajian0516}
  \begin{aligned}
    &\frac{5}{8}\int_{\widetilde{\Omega}}\rho r \left|\partial_{t}\mathbf{V}^{m}\right|^2 drd\theta+ 
    \frac{1}{2}\frac{d}{dt}\left[\int_{0}^{2\pi}\left(\mu-r\alpha\right)\left|v_{\theta}^{m}\right|^2
    |_{r=R_{1}}d\theta\right]
    \\
    &+\frac{\mu}{2}\frac{d}{dt}\left\{\int_{0}^{2\pi}
    \left[r\left|\partial_{r} \mathbf{V}^{m}\right|^2
    +\frac{1}{r}\left(\left|\partial_{\theta} v_{r}^{m}-v_{\theta}^{m}\right|^2
    +\left|\partial_{\theta} v_{\theta}^{m}+v_{r}^{m}\right|^2\right)\right]
    drd\theta\right\}
    \\
    &\leq C\left(\rho_{0},R_{1},R_{2}\right)\left\{
      \int_{\widetilde{\Omega}}\left[\left|\mathbf{V}^{m}\right|^2
      \left|\widetilde{\nabla}\mathbf{V}^{m}\right|^2
       + \left|\mathbf{V}^{m}\right|^4\right]drd\theta+1
    \right\}
  \end{aligned}
\end{align}
and 
\begin{align}\label{laoxiaoa0516}
  \begin{aligned}
  \int_{\widetilde{\Omega}}r|\mathbf{U}^{m}|^2drd\theta
    \leq C\left(\mu,\rho_{0},R_{1},R_{2}\right)\int_{\widetilde{\Omega}}
    \bigg{[}
 \rho r\left|\partial_{t}\mathbf{V}^{m}\right|^2+1
 +\left|\mathbf{V}^{m}\right|^2
 \left|\widetilde{\nabla}\mathbf{V}^{m}\right|^2
  + \left|\mathbf{V}^{m}\right|^4
    \bigg{]} drd\theta.
  \end{aligned}
\end{align}
Multiplying \eqref{laoxiaoa0516} by a sufficiently small constant \(\varepsilon\) and putting this result into 
\eqref{yonglaihuajian0516} give that 
\begin{align}\label{yibuyibu0516}
  \begin{aligned}
    &\frac{1}{2}\int_{\widetilde{\Omega}}\rho r \left|\partial_{t} \mathbf{V}^{m}\right|^2
    drd\theta+
    \frac{1}{2}\frac{d}{dt}\left[\int_{0}^{2\pi}\left(\mu-r\alpha\right)\left|v_{\theta}^{m}\right|^2
    \bigg{|}_{r=R_{1}}d\theta\right]
    \\
    &+\frac{\mu}{2}\frac{d}{dt}\left\{\int_{0}^{2\pi}
    \left[r\left|\partial_{r} \mathbf{V}^{m}\right|^2
    +\frac{1}{r}\left(\left|\partial_{\theta} v_{r}^{m}-v_{\theta}^{m}\right|^2
    +\left|\partial_{\theta} v_{\theta}^{m}+v_{r}^{m}\right|^2\right)\right]
    drd\theta\right\}
    \\
    &\leq C\left(\mu,\rho_{0},R_{1},R_{2}\right)\left\{
      \int_{\widetilde{\Omega}}\left[\left|\mathbf{V}^{m}\right|^2
      \left|\widetilde{\nabla}\mathbf{V}^{m}\right|^2
       + \left|\mathbf{V}^{m}\right|^4\right]drd\theta+1
    \right\}.
  \end{aligned}
\end{align}

From Lemmas \ref{poincarebudengshi0517}, \ref{zuidazhi0517} and Cauchy inequality, we have 
\begin{align}\label{zhaozhaokankan0530}
  \begin{aligned}
    \int_{\widetilde{\Omega}}\left[\left|\mathbf{V}^{m}\right|^2
      \left|\widetilde{\nabla}\mathbf{V}^{m}\right|^2\right]drd\theta 
      \leq \widetilde{\varepsilon}
      \left\|\widetilde{\nabla}^2\mathbf{V}^{m}\right\|_{L^2\left(\widetilde{\Omega}\right)}^2 
      +C\left(R_{1},R_{2}\right)\left\|\widetilde{\nabla}\mathbf{V}^{m}
      \right\|_{L^2\left(\widetilde{\Omega}\right)}^{6}.
  \end{aligned}
\end{align}
Similarly, we can obtain that 
\begin{align}\label{zhaozhao0530}
  \int_{\widetilde{\Omega}}\left|\mathbf{V}^{m}\right|^4drd\theta
  \leq \widetilde{\varepsilon}
      \left\|\widetilde{\nabla}^2\mathbf{V}^{m}\right\|_{L^2\left(\widetilde{\Omega}\right)}^2 
      +C\left(R_{1},R_{2}\right)\left\|\widetilde{\nabla}\mathbf{V}^{m}
      \right\|_{L^2\left(\widetilde{\Omega}\right)}^{6}.
\end{align}
From Remark \ref{touyinggujijiehe0521}, we conclude that 
\(\int_{\widetilde{\Omega}}r|\mathbf{U}^{m}|^2drd\theta\) is equivalent to 
\(\left\|\widetilde{\nabla}^2\mathbf{V}^{m}\right\|_{L^2\left(\widetilde{\Omega}\right)}^2\). Thus, take 
\(\widetilde{\varepsilon}\) small sufficiently, we from \eqref{yibuyibu0516}-\eqref{zhaozhao0530} obtain 
\begin{align*}
  \begin{aligned}
    &\frac{\mu}{2}\frac{d}{dt}\left\{\int_{0}^{2\pi}
    \left[r\left|\partial_{r} \mathbf{V}^{m}\right|^2
    +\frac{1}{r}\left(\left|\partial_{\theta} v_{r}^{m}-v_{\theta}^{m}\right|^2
    +\left|\partial_{\theta} v_{\theta}^{m}+v_{r}^{m}\right|^2\right)\right]
    drd\theta\right\}
    \\
    &+\frac{1}{2}\frac{d}{dt}\left[\int_{0}^{2\pi}\left(\mu-r\alpha\right)\left|v_{\theta}^{m}\right|^2
    \bigg{|}_{r=R_{1}}d\theta\right]
    +\frac{\varepsilon}{2}
    \left\|\widetilde{\nabla}^2\mathbf{V}^{m}\right\|_{L^2\left(\widetilde{\Omega}\right)}^2
    \\
    &\leq C\left(\mu,\rho_{0},R_{1},R_{2}\right)\left\|\widetilde{\nabla}\mathbf{V}^{m}
    \right\|_{L^2\left(\widetilde{\Omega}\right)}^{6}+C\left(\mu,\rho_{0},R_{1},R_{2}\right).
  \end{aligned}
\end{align*}
Since \(\mu-R_{1}\alpha\geq 0\), then from Lemma \ref{dengjia0523}, we have 
\begin{align*}
  \begin{aligned}
    \left\|\widetilde{\nabla}\mathbf{V}^{m}\right\|_{L^2\left(\widetilde{\Omega}\right)}^2
    \leq C\left(\mu,\rho_{0},R_{1},R_{2}\right)\int_{0}^{t}\left\|\widetilde{\nabla}\mathbf{V}^{m}\left(s\right)
    \right\|_{L^2\left(\widetilde{\Omega}\right)}^{6}ds +C\left(\mu,\rho_{0},R_{1},R_{2},\mathbf{V}_{0}\right)(t+1). 
  \end{aligned}
\end{align*}
Then, from Lemma \ref{budengshidezhengming0530}, we have 
\begin{align}\label{yangge0603}
  \left\|\widetilde{\nabla}\mathbf{V}^{m}\right\|_{L^2\left(\widetilde{\Omega}\right)}^2\leq 
  C^{*}\left(\mu,\rho_{0},R_{1},R_{2},\mathbf{V}_{0},T\right),~\text{for~any~}T\leq T^{*}.
\end{align}
Furthermore, from \eqref{yibuyibu0516} we can conclude that for any \(t\leq T\leq T^{*}\), we have
\begin{align}\label{yongdedao0603}
\int_{0}^{t}\left\|\widetilde{\nabla}^2\mathbf{V}^{m}\right\|_{L^2\left(\widetilde{\Omega}\right)}^2ds
\leq C^{*},~
\int_{0}^{t}\left\|\sqrt{\rho}\partial_{s} \mathbf{V}^{m}
\right\|_{L^2\left(\widetilde{\Omega}\right)}^2ds\leq C^{*},  
\end{align}
where \(C^{*}=C^{*}\left(\mu,\rho_{0},R_{1},R_{2},\mathbf{V}_{0},T\right)>0\).
\begin{remark}\label{jinlaingxiao1031}
When the initial value \(\left(\mathbf{V}_{0},\rho_{0}\right)\) is small, the existence time \(T^{*}\) can be very large from the proof of the Lemma \ref{budengshidezhengming0530}.
\end{remark}
\textbf{Part 3.}Utilizing the above estimate, we can show that \(\mathbf{F}^{m}\) is continuous and compact. 
Furthermore, there exists an \(\widetilde{R}>0\) such that 
\(\mathbf{F}^{m}\)\(:B_{\widetilde{R}\mathbf{Y}^{m}}\rightarrow 
B_{\widetilde{R}\mathbf{Y}^{m}}\).

To explain the continuity of \(\mathbf{F}^{m}\), we need the following lemma. 
\begin{lemma}\label{piaoliang04016} Let \(\mathbf{V}^{k}\left(t,x\right)=
  \left(v_{r}^{k},v_{\theta}^{k}\right)\in C\left(0,T; 
  \left(C^{1}\left(\overline{\widetilde{\Omega}}\right)\right)^2\right)\), \(\partial_{r}\left(rv_{r}^{k}\right)
  +\partial_{\theta}v_{\theta}^{k}=0\) and 
  satisfies \eqref{bianjie0329} for \(k=1,2,\cdots\). Furthermore, 
  \(\mathbf{V}^{k}\rightarrow \mathbf{V}=\left(v_{r},v_{\theta}\right)\) in \(C\left(0,T;\left(C^{1}
  \left(\overline{\widetilde{\Omega}}\right)\right)^2\right)\) as \(k\rightarrow+\infty\). 
  If \(\rho^{k}\left(t,r,\theta\right)\) and \(\rho\left(t,r,\theta\right)\) satisfy the following two 
  equations, respectively, 
  \begin{align*}
    \begin{aligned}
    \begin{cases}
      \partial_{t} \rho^{k}+v_{r}^{k}\partial_{r} \rho^{k}
      +\frac{v_{\theta}^{k}}{r}\partial_{\theta} \rho^{k} =0,
      \\
      \rho^{k}\left(0,r,\theta\right)=\overline{\rho}_{0},~\left(r,\theta\right)\in\widetilde{\Omega},
    \end{cases}~
      \begin{cases}
        \partial \rho +v_{r}\partial_{r} \rho
        +\frac{v_{\theta}}{r}\partial_{\theta} \rho =0,
        \\
        \rho\left(0,r,\theta\right)=\overline{\rho}_{0},~\left(r,\theta\right)\in\widetilde{\Omega}.
      \end{cases}
  \end{aligned}
  \end{align*}
  Then, \(\rho^{k}\left(t,r,\theta\right)\rightarrow\rho\left(t,r,\theta\right)\) 
  in \(C\left([0,T]\times\overline{\widetilde{\Omega}}\right)\)
  as \(k\rightarrow+\infty\).
\end{lemma}
\begin{proof}
  Let \(\mathbf{X}^{k}\left(t,r,\theta\right)=\left(X_{1}^{k},X_{2}^{k}\right)\) 
  and \(\mathbf{X}\left(t,r,\theta\right)=\left(X_{1},X_{2}\right)\) solve the following two equations, 
  respectively, 
  \begin{align*}
    \begin{aligned}
      \begin{cases}
        \frac{d X_{1}^{k}}{dt}=v_{r}^{k}\left(t,\mathbf{X}^{k}\right),
        \\
        \frac{d X_{2}^{k}}{dt}=\frac{v_{\theta}^{k}}{X_{1}^{k}}\left(t,\mathbf{X}^{k}\right),
        \\
        \mathbf{X}^{k}\left(0,r,\theta\right)=\left(r,\theta\right)\in\overline{\widetilde{\Omega}},
      \end{cases}
      ~\begin{cases}
        \frac{d X_{1}}{dt}=v_{r}\left(t,\mathbf{X}\right),
        \\
        \frac{d X_{2}}{dt}=\frac{v_{\theta}}{X_{1}}\left(t,\mathbf{X}\right),
        \\
        \mathbf{X}\left(0,r,\theta\right)=\left(r,\theta\right)\in\overline{\widetilde{\Omega}}.
      \end{cases}
    \end{aligned}
  \end{align*}
We also denote \(\mathbf{A}^{k}\left(t,r,\theta\right)\) and \(\mathbf{A}\left(t,r,\theta\right)\) as the inverse 
maps of \(\mathbf{X^{k}}\) and \(\mathbf{X}\), respectively. That is, 
\(\left(r,\theta\right)=\mathbf{A}^{k}\left(t,\mathbf{X}^{k}\left(t,r,\theta\right)\right)=
\mathbf{A}\left(t,\mathbf{X}\left(t,r,\theta\right)\right)\). 

We first consider the convergence of \(\left\{\mathbf{X}^{k}\right\}\) by investigating the following 
equation,
\begin{align*}
  \begin{cases}
  \frac{d\left(\mathbf{X}^{k}-\mathbf{X}\right)}{dt}=\begin{pmatrix}
    v_{r}^{k}\left(t,\mathbf{X}^{k}\right)-v_{r}\left(t,\mathbf{X}\right)
    \\
    \frac{v_{\theta}^{k}}{X_{1}^{k}}\left(t,\mathbf{X}^{k}\right)-
    \frac{v_{\theta}}{X_{1}}\left(t,\mathbf{X}\right)
  \end{pmatrix}
  ,
  \\
  \left(\mathbf{X}^{k}-\mathbf{X}\right)\left(0,r,\theta\right)=0.
  \end{cases}
\end{align*}
Integrating the above equality from \(0\) to \(t\) with the boundedness of \(\mathbf{X}\)
and \(\mathbf{X}^{k}\) gives that 
\begin{align*}
  \left|\mathbf{X}^{k}-\mathbf{X}\right|
  \leq C\int_{0}^{t}\left|\mathbf{V}^{k}\left(s,\mathbf{X}^{k}\right)
  -\mathbf{V}\left(s,\mathbf{X}^{k}\right)\right|ds 
  +C\int_{0}^{t}\left|\mathbf{V}\left(s,\mathbf{X}^{k}\right)-\mathbf{V}\left(s,\mathbf{X}\right)\right|ds.
\end{align*}
Furthermore, by virtue of convergence of \(\left\{\mathbf{V}^{k}\right\}\) and the differentiability of 
\(\mathbf{V}\), one can conclude that 
there exist \(C>0\) and \(\varepsilon_{k}>0\) such that 
\begin{align*}
  \left|\mathbf{X}^{k}-\mathbf{X}\right|\leq \varepsilon_{k}T+C\int_{0}^{t}
  \left|\mathbf{X}^{k}-\mathbf{X}\right|ds.
\end{align*}
From Gronwall's inequality, one has 
\[
\int_{0}^{t}\left|\mathbf{X}^{k}-\mathbf{X}\right|ds\leq \frac{\varepsilon_{k}T}{C}
\left(e^{CT}-1\right)\rightarrow 0,~\text{as~}k\rightarrow+\infty,
\]
which implies that \(\mathbf{X}^{k}\rightarrow\mathbf{X}\) in \(C\left([0,T]\times\overline{\widetilde{\Omega}}\right)\)
as \(k\rightarrow+\infty\). As a result, 
\(\mathbf{A}^{k}\left(t,r,\theta\right)\rightarrow\mathbf{A}\left(t,r,\theta\right)\) in 
\(C\left([0,T]\times\overline{\widetilde{\Omega}}\right)\)
as \(k\rightarrow+\infty\).

From the proof of Lemma \ref{gudingu04016}, we can conclude that 
\(\rho^{k}\left(t,r,\theta\right)=\overline{\rho}_{0}\left(\mathbf{A}^{k}\left(t,r,\theta\right)\right)
\rightarrow \overline{\rho}_{0}\left(\mathbf{A}\left(t,r,\theta\right)\right)=\rho\left(t,r,\theta\right)\) in
\(C\left([0,T]\times\overline{\widetilde{\Omega}}\right)\)
as \(k\rightarrow+\infty\).
\end{proof}
\begin{lemma}\label{lianxuan0618}
  \(\mathbf{F}^{m}:\mathbf{Y}^{m}\rightarrow\mathbf{Y}^{m}\) is completely continuous. And 
  there exists an \(\widetilde{R}>0\) such that \(\mathbf{F}^{m}\)\(:B_{\widetilde{R}\mathbf{Y}^{m}}\rightarrow 
  B_{\widetilde{R}\mathbf{Y}^{m}}\).
\end{lemma}
\begin{proof}
  The first step is to verify that \(\mathbf{F}^{m}\) is continuous. 
To this end, we choose a sequence \(\left\{\widetilde{\mathbf{C}}_{m}^{n}\left(t\right)\right\}_{n=1}^{\infty}\) 
which converges to \(\widetilde{\mathbf{C}}_{m}\left(t\right)\) in \(\mathbf{Y}^{m}\). 
Let \(\widetilde{\mathbf{V}}^{mn}=\sum\limits_{i=1}^{m}\widetilde{C}_{mi}^{n}\left(t\right)\mathbf{e}_{i}\)
and 
\(\widetilde{\mathbf{V}}^{m}=\sum\limits_{i=1}^{m}\widetilde{C}_{mi}\left(t\right)\mathbf{e}_{i}\). 
Then, \(\widetilde{\mathbf{V}}^{mn}\) and \(\widetilde{\mathbf{V}}^{m}\in 
C\left(0,T;C^{1}\left(\overline{\widetilde{\Omega}}\right)\right)\) satisfying 
\(\widetilde{\mathbf{V}}^{mn}\rightarrow \widetilde{\mathbf{V}}^{m}\) in \(C\left(0,T;\left(C^{1}
  \left(\overline{\widetilde{\Omega}}\right)\right)^2\right)\) as \(n\rightarrow+\infty\).
Thus, from the Lemma \ref{piaoliang04016}, we can conclude that 
\(\rho^{mn}\left(t,x\right)\rightarrow\rho^{m}\left(t,x\right)\) in 
\(C\left([0,T]\times\overline{\widetilde{\Omega}}\right)\)
  as \(n\rightarrow+\infty\). From the definition of \(\mathbf{F}^{m}\), there exist 
  \(\mathbf{C}_{m}^{n}\left(t\right)\) and \(\mathbf{C}_{m}\left(t\right)\in \mathbf{Y}_{m}\).  
 Let 
\[
\mathbf{V}^{mn}\left(t\right)=\left(v_{r}^{mn},v_{\theta}^{mn}\right)
=\sum\limits_{i=1}^{m}C_{mi}^{n}\left(t\right)\mathbf{e}_{i},~
\mathbf{V}^{m}\left(t\right)=\left(v_{r}^{m},v_{\theta}^{m}\right)
=\sum\limits_{i=1}^{m}C_{mi}\left(t\right)\mathbf{e}_{i}.
\]
Then, from \eqref{rhoquedinghou0618}, we have
\begin{align}\label{bijinjie06031}
  \begin{aligned}
    &\int_{\widetilde{\Omega}}\rho^{m} r \partial_{t} \mathbf{V}^{m}
    \cdot \mathbf{e}_{k}
      drd\theta 
        -\int_{\widetilde{\Omega}}\rho^{m} \left(\left(v_{\theta}^{m}\right)^2,-v_{r}^{m}v_{\theta}^{m}
        \right)\cdot \mathbf{e}_{k}
        drd\theta 
        \\
        & +\int_{\widetilde{\Omega}}\rho^{m}r\left(v_{r}^{m}\partial_{r} \mathbf{V}^{m}
        +\frac{v_{\theta}^{m}}{r}\partial_{\theta} \mathbf{V}^{m} 
        \right)
        \cdot\mathbf{e}_{k}drd\theta = 
        -\int_{\widetilde{\Omega}}\rho^{m} g r e_{r,k}drd\theta
        \\
        &+\mu\int_{\widetilde{\Omega}}r\left(\Delta_{r}v_{r}^{m}-\frac{v_{r}^{m}}{r^2}
        -\frac{2}{r^2}\partial_{\theta} v_{\theta}^{m}
        ,\Delta_{r}v_{\theta}^{m}
        -\frac{v_{\theta}^{m}}{r^2}+\frac{2}{r^2}\partial_{\theta} v_{r}^{m}\right)
        \cdot \mathbf{e}_{k}drd\theta 
  \end{aligned}
\end{align} 
and 
\begin{align}\label{bijinjie06032}
  \begin{aligned}
    &\int_{\widetilde{\Omega}}\rho^{mn} r \partial_{t} \mathbf{V}^{mn}
    \cdot \mathbf{e}_{k}
      drd\theta 
        -\int_{\widetilde{\Omega}}\rho^{mn} \left(\left(v_{\theta}^{mn}\right)^2,-v_{r}^{mn}v_{\theta}^{mn}
        \right)\cdot \mathbf{e}_{k}
        drd\theta 
        \\
        & +\int_{\widetilde{\Omega}}\rho^{mn}r\bigg{(}v_{r}^{mn}\partial_{r} \mathbf{V}^{mn}
        +\frac{v_{\theta}^{mn}}{r}\partial_{\theta} \mathbf{V}^{mn} 
        \bigg{)}
        \cdot\mathbf{e}_{k}drd\theta = -\int_{\widetilde{\Omega}}\rho^{mn} g r e_{r,k}drd\theta
        \\
        &+\mu\int_{\widetilde{\Omega}}r\left(\Delta_{r}v_{r}^{mn}-\frac{v_{r}^{mn}}{r^2}
        -\frac{2}{r^2}\partial_{\theta} v_{\theta}^{mn}
        ,\Delta_{r}v_{\theta}^{mn}
        -\frac{v_{\theta}^{mn}}{r^2}+\frac{2}{r^2}\partial_{\theta} v_{r}^{mn}\right)
        \cdot \mathbf{e}_{k}drd\theta.
  \end{aligned}
\end{align}
In addition, \(\mathbf{V}^{mn}\left(0\right)=\mathbf{V}^{m}\left(0\right)\). Let 
\eqref{bijinjie06031} minus \eqref{bijinjie06032} and replace \(\mathbf{e}_{k}\) by 
\(\partial_{t} \left(\mathbf{V}^{m}-\mathbf{V}^{mn}\right)\). Then, we obtain 
\begin{align}\label{xuyaodebudengshi0603}
  \begin{aligned}
    &\frac{2}{\mu}\int_{\widetilde{\Omega}}\rho^{mn}r \left|
    \partial_{t} \left(\mathbf{V}^{m}-\mathbf{V}^{mn}\right)\right|^2 dr d\theta 
    +\frac{d}{dt}
    \left\{
      \int_{\widetilde{\Omega}}
      r\left|\partial_{r} \left(\mathbf{V}^{m}-\mathbf{V}^{mn}\right)\right|^2
      drd\theta
    \right\}
    \\
    &+\frac{d}{dt}
    \left\{
      \int_{\widetilde{\Omega}}\frac{1}{r}
      \left[
        \left(\partial_{\theta} \left(v_{r}^{m}-v_{r}^{mn}\right)
        -\left(v_{\theta}^{m}-v_{\theta}^{mn}\right)\right)^2 +
        \left(\partial_{\theta}\left(v_{\theta}^{m}-v_{\theta}^{mn}\right)
        +\left(v_{r}^{m}-v_{r}^{mn}\right)\right)^2
      \right]drd\theta
    \right\}
    \\
    &+\frac{1}{\mu}\frac{d}{dt}\left[ 
      \int_{0}^{2\pi}\left(\mu-r\alpha\right)\left|v_{\theta}^{m}-v_{\theta}^{mn}\right|^2|_{r=R_{1}}
      d\theta
          \right]=\frac{2}{\mu}\sum\limits_{i=1}^{6}I_{i}:=\frac{2}{\mu}I,
  \end{aligned}
\end{align}
where 
\begin{align*}
  \begin{aligned}
&I_{1}=-\int_{\widetilde{\Omega}}\left(\rho^{m}-\rho^{mn}\right)r 
\partial_{t} \mathbf{V}^{m}\cdot 
\partial_{t} \left(\mathbf{V}^{m}-\mathbf{V}^{mn}\right)drd\theta,
\end{aligned}
\end{align*}
\begin{align*}
  \begin{aligned}
&I_{2}=-\int_{\widetilde{\Omega}}r\left(\rho^{m}-\rho^{mn}\right) 
\left(v_{r}^{m}\partial_{r} \mathbf{V}^{m}+ 
\frac{v_{\theta}^{m}}{r}\partial_{\theta} \mathbf{V}^{m}\right)
\cdot \partial_{t} \left(\mathbf{V}^{m}-\mathbf{V}^{mn}\right)drd\theta,
\\
&I_{3}=-\int_{\widetilde{\Omega}}
r\rho^{mn}\left(
  v_{r}^{m}\partial_{r} \mathbf{V}^{m}+\frac{v_{\theta}^{m}}{r}
  \partial_{\theta} \mathbf{V}^{m}
  -v_{r}^{mn}\partial_{r} \mathbf{V}^{mn}-\frac{v_{\theta}^{mn}}{r}
  \partial_{\theta} \mathbf{V}^{mn}
\right)\cdot\partial_{t}\left(\mathbf{V}^{m}-\mathbf{V}^{mn}\right)
drd\theta,
\\
&I_{4}=-\int_{\widetilde{\Omega}}
\left(\rho^{m}-\rho^{mn}\right)\left(-\left(v_{\theta}^{m}\right)^2,v_{r}^{m}v_{\theta}^{m}\right)
\cdot\partial_{t} \left(\mathbf{V}^{m}-\mathbf{V}^{mn}\right)
drd\theta,
\\
&I_{5}=-\int_{\widetilde{\Omega}} 
\rho^{mn}\left(\left(v_{\theta}^{mn}\right)^2-\left(v_{\theta}^{m}\right)^2, 
v_{r}^{m}v_{\theta}^{m}-v_{r}^{mn}v_{\theta}^{mn}\right)\cdot\partial_{t} \left(
  \mathbf{V}^{m}-\mathbf{V}^{mn}
\right)
drd\theta,
\\
&I_{6}=
\int_{\widetilde{\Omega}}gr\left(\rho^{mn}-\rho^{m},0\right)\cdot 
\partial_{t}\left(\mathbf{V}^{m}-\mathbf{V}^{mn}\right)
drd\theta.
  \end{aligned}
\end{align*}
From Cauchy inequality, we have 
\begin{align*}
  \begin{aligned}
    I
    \leq& \varepsilon\left\|\sqrt{\rho^{mn}r}
    \partial_{t} \left(\mathbf{V}^{m}-\mathbf{V}^{mn}\right)\right\|
    _{L^2\left(\widetilde{\Omega}\right)}^2
    +C\left\|\rho^{m}-\rho^{mn}\right\|_{L^\infty\left(\widetilde{\Omega}\right)}^2 
    \bigg{[}
\left\|\partial_{t}\mathbf{V}^{m}\right\|_{L^2\left(\widetilde{\Omega}\right)}^2 
+\left\|\mathbf{V}^{m}\cdot\widetilde{\nabla}\mathbf{V}^{m}\right\|_{L^2\left(\widetilde{\Omega}\right)}^2
\\
&+\left\|\widetilde{\nabla}\left(\mathbf{V}^{m}-\mathbf{V}^{mn}\right)\right\|_{L^2\left(\widetilde{\Omega}\right)}^2
+\left\|\left(\mathbf{V}^{m}\right)^2\right\|_{L^2\left(\widetilde{\Omega}\right)}^2 +1
    \bigg{]}.
  \end{aligned}
\end{align*}
Then, from \eqref{xuyaodebudengshi0603} and \(\varepsilon\) is sufficiently small, we have 
\begin{align}\label{haiyaosuan0603}
  \begin{aligned}
&\int_{\widetilde{\Omega}}\rho^{mn}r \left|
    \partial_{t} \left(\mathbf{V}^{m}-\mathbf{V}^{mn}\right)\right|^2 dr d\theta 
    +\frac{d}{dt}\left[ 
\int_{0}^{2\pi}\left(\mu-r\alpha\right)\left|v_{\theta}^{m}-v_{\theta}^{mn}\right|^2|_{r=R_{1}}
d\theta
    \right]
    \\
    &+\frac{d}{dt}
    \left\{
      \int_{\widetilde{\Omega}}\left[
      r\left(\left|\partial_{r} \left(\mathbf{V}^{m}-\mathbf{V}^{mn}\right)\right|^2
      \right)
      \right]drd\theta
    \right\}
    \\
    &+\frac{d}{dt}
    \left\{
      \int_{\widetilde{\Omega}}\frac{1}{r}
      \left[
        \left(\partial_{\theta} \left(v_{r}^{m}-v_{r}^{mn}\right)
        -\left(v_{\theta}^{m}-v_{\theta}^{mn}\right)\right)^2 +
        \left(\partial_{\theta}\left(v_{\theta}^{m}-v_{\theta}^{mn}\right)
        +\left(v_{r}^{m}-v_{r}^{mn}\right)\right)^2
      \right]drd\theta
    \right\}
    \\
    &\leq C\left\|\rho^{m}-\rho^{mn}\right\|_{L^\infty\left(\widetilde{\Omega}\right)}^2 
    \bigg{[}
\left\|\partial_{t}\mathbf{V}^{m}\right\|_{L^2\left(\widetilde{\Omega}\right)}^2 
+\left\|\mathbf{V}^{m}\cdot\widetilde{\nabla}\mathbf{V}^{m}\right\|_{L^2\left(\widetilde{\Omega}\right)}^2
\\
&~~~~~~+\left\|\widetilde{\nabla}\left(\mathbf{V}^{m}-\mathbf{V}^{mn}\right)\right\|_{L^2\left(\widetilde{\Omega}\right)}^2
+\left\|\left(\mathbf{V}^{m}\right)^2\right\|_{L^2\left(\widetilde{\Omega}\right)}^2 +1
    \bigg{]},
  \end{aligned}
\end{align}
Since \(\rho^{m}\), \(\rho^{mn}\), \(\mathbf{V}^{m}\) 
and \(\mathbf{V}^{mn}\) are smooth, from estimates \eqref{yangge0603} and \eqref{yongdedao0603}, 
we conclude that the following function is integrable,
\[
 M\left(t\right)= \left\|\partial_{t}\mathbf{V}^{m}\right\|_{L^2\left(\widetilde{\Omega}\right)}^2 
+\left\|\mathbf{V}^{m}\cdot\widetilde{\nabla}\mathbf{V}^{m}\right\|_{L^2\left(\widetilde{\Omega}\right)}^2
+\left\|\widetilde{\nabla}\left(\mathbf{V}^{m}-\mathbf{V}^{mn}\right)\right\|_{L^2\left(\widetilde{\Omega}\right)}^2
+\left\|\left(\mathbf{V}^{m}\right)^2\right\|_{L^2\left(\widetilde{\Omega}\right)}^2 +1.
\]
Thus, from \eqref{haiyaosuan0603} and Lemma \ref{dengjia0523}, we have 
\[
\left\|\widetilde{\nabla}\left(\mathbf{V}^{m}-\mathbf{V}^{mn}\right)\right\|_{L^2\left(\widetilde{\Omega}\right)}^2 
\leq C\int_{0}^{t}M\left(t\right)\left\|\rho^{m}-\rho^{mn}\right\|_{L^\infty\left(\widetilde{\Omega}\right)}^2
ds 
,
\]
which implies that 
\[
  \left\|\widetilde{\nabla}\left(\mathbf{V}^{m}-\mathbf{V}^{mn}\right)\right\|_{L^2\left(\widetilde{\Omega}\right)}^2 
  \rightarrow 0,
\]
since \(\max\limits_{t\in[0,T]}
  \left\|\rho^{m}-\rho^{mn}\right\|_{L^\infty\left(\widetilde{\Omega}\right)}^2
\rightarrow 0\) as \(n\rightarrow+\infty\). Thus, we can conclude that as 
\(n\rightarrow+\infty\), 
\[
\mathbf{C}_{m}^{n}\left(t\right)\rightarrow \mathbf{C}_{m}\left(t\right)~\text{in~}\mathbf{Y}_{m}.  
\]
This verifies the continuity of \(\mathbf{F}^{m}\). 

Furthermore, since
\(\sum\limits_{k=1}^{m}\left|C_{mk}\left(t\right)\right|^2=\int_{\widetilde{\Omega}}
r\left|\mathbf{V}^{m}\right|^2drd\theta\) and 
the estimate \eqref{yangge0603} is independent of
\(\widetilde{\mathbf{C}}_{m}\left(t\right)\), where \(\mathbf{F}^{m}\left(
  \widetilde{\mathbf{C}}_{m}\left(t\right)\right)=\mathbf{C}_{m}\left(t\right)\), 
  then we can choose \(\widetilde{R}\) large sufficiently and 
\(\rho_{0}\) small enough such that \(\sum\limits_{k=1}^{m}\left|C_{mk}\left(t\right)\right|^2\leq \widetilde{R}^2\). Therefore,
we can conclude that \(\mathbf{F}^{m}:B_{\widetilde{R}\mathbf{Y}^{m}}\rightarrow 
B_{\widetilde{R}\mathbf{Y}^{m}}\). 

Finally, since \(\mathbf{C}_{m}\left(t\right)\in \left[C^{1}
\left[0,T\right]\right]^{m}\) is bounded, see \eqref{nulitishengziji0502},
 then from Ascoli-Arzela theorem, we have the compactness of 
 \(\mathbf{F}^{m}\).
\end{proof}

\begin{remark}\label{bijie0618}
 Since \(\mathbf{F}^{m}:B_{\widetilde{R}\mathbf{Y}_{m}}\rightarrow 
 B_{\widetilde{R}\mathbf{Y}_{m}}\) is continuous and compact, then from Schauder fixed point theorem, there 
 exists \(\mathbf{C}^{m}\left(t\right)=\left(C_{m1}\left(t\right),\dots,C_{mm}\left(t\right)\right)\in 
 B_{\widetilde{R}\mathbf{Y}_{m}}\) such that \(\mathbf{F}^m\left(\mathbf{C}^{m}\left(t\right)\right)=\mathbf{C}^m\left(t\right)\)
 that is for any 
 \(~k\in\left\{1,\cdots,m\right\}\), we have 
\begin{align}\label{bijinjie0605}
\begin{aligned}
&\begin{aligned}
&\int_{\widetilde{\Omega}}\rho^m r \partial_{t}\mathbf{V}^{m}\cdot \mathbf{e}_{k} 
      drd\theta 
      +\int_{\widetilde{\Omega}}r\rho^m\left(v_{r}^{m}\partial_{r} \mathbf{V}^{m}
      +\frac{v_{\theta}^{m}}{r}\partial_{\theta} \mathbf{V}^{m}\right)\cdot \mathbf{e}_{k}drd\theta 
      \\
      &-\int_{\widetilde{\Omega}}\rho^m \left(v_{\theta}^{m}\right)^2e_{r,k}
      drd\theta 
      +\int_{\widetilde{\Omega}}\rho^m 
      v_{r}^{m}v_{\theta}^{m}e_{\theta,k}drd\theta= 
      -\int_{\widetilde{\Omega}}\rho^m g r e_{r,k}drd\theta
      \\
      &+\mu\int_{\widetilde{\Omega}}\left(\Delta_{r}v_{r}^{m}-\frac{v_{r}^{m}}{r^2}
      -\frac{2}{r^2}\partial_{\theta} v_{\theta}^{m}\right)r e_{r,k}drd\theta 
      +\mu\int_{\widetilde{\Omega}}\left(\Delta_{r}v_{\theta}^{m}
      -\frac{v_{\theta}^{m}}{r^2}+\frac{2}{r^2}\partial_{\theta} v_{r}^{m}\right)
      re_{\theta,k}drd\theta,
\end{aligned}
\\
&\begin{aligned}
 \partial_{t}\rho^m +v_{r}^{m}\partial_{r}\rho^m+\frac{v_{\theta}^{m}}{r}\partial_{\theta}\rho^{m}=0, 
      ~\rho^{m}\left(0\right)=\overline{\rho}_{0},
\end{aligned}
\end{aligned}
\end{align}
where \(\mathbf{V}^{m}=\sum\limits_{k=1}^{m}C_{mk}\left(t\right)\mathbf{e}_{k}\). 
 \(\left(\mathbf{V}^m,\rho^m\right)\) is defined as the approximate solution.
\end{remark}

\textbf{Part 4.} We complete the proof of the Theorem \ref{jubucunzai0605} by passing \(m\rightarrow+\infty\).

From Lemma \ref{tezhengzhi04018}, we know that \(\left\{\mathbf{e}_{k}\right\}\) is dense in 
\(\mathbf{G}^{1,2}\). Then, for any \(T\leq T^{*}\) and 
\(\mathbf{\Phi}\in C^{1}\left(0,T;\mathbf{G}^{1,2}\right)\)
satisfying \(\mathbf{\Phi}\left(T\right)=0\), \(t\in \left[0,T\right]\)
and 
\(\varepsilon>0\), there 
exists an \(N\left(t,\varepsilon\right)\in\mathbf{N}^{*}\) such that when \(m>N\), 
\begin{align*}
  \left\|\mathbf{\Phi}-\sum\limits_{k=1}^{m}C_{k}\left(t\right)\mathbf{e}_{k}
  \right\|_{\mathbf{G}^{1,2}}<\varepsilon,~C_{k}\left(t\right)=\int_{\widetilde{\Omega}}r\mathbf{\Phi}\cdot 
  \mathbf{e}_{k}drd\theta.
\end{align*} 
Furthermore, from Dini theorem, we can conclude that there exists an \(N\left(\varepsilon\right)\) independent of \(t\), such that 
\begin{align*}
  \left\|\mathbf{\Phi}-\sum\limits_{k=1}^{n}C_{k}\left(t\right)
  \mathbf{e}_{k}\right\|_{\mathbf{G}^{1,2}}<\varepsilon,
  ~\left\|\partial_{t}\mathbf{\Phi}-\sum\limits_{k=1}^{n}\frac{d}{dt}C_{k}\left(t\right)
  \mathbf{e}_{k}\right\|_{\mathbf{G}^{1,2}}<\varepsilon,
\end{align*}  
where \(n>N\left(\varepsilon\right)\) and \(C_{k}\left(T\right)=0\) for all \(k\in\left\{1,2,\dots,n\right\}\).
Replace \(\mathbf{e}_{k}\) by \(C_{k}\left(t\right)\mathbf{e}_{k}\) in \(\eqref{bijinjie0605}_{1}\), 
then add up the results for \(k=1,2,\cdots,n\) and integrate the result over \([0,T]\) by parts, for any \(\Psi
\in C^{1}\left([0,T];H^{1}\left(\widetilde{\Omega}\right)\right)\) and \(m\geq n\), we have
\begin{align}\label{zhiguanzhongyao0606}
\begin{aligned}
  &\int_{0}^{T}\int_{\widetilde{\Omega}}r\rho^m \mathbf{V}^m\partial_{t} \sum\limits_{k=1}^{n}C_{k}\mathbf{e}_{k}
  drd\theta dt+\int_{0}^{T}\int_{\widetilde{\Omega}}\rho^m 
  \left(rv_{r}^m\partial_{r}\sum\limits_{k=1}^{n}C_{k}\mathbf{e}_{k}
  +v_{\theta}\partial_{\theta}\sum\limits_{k=1}^{n}C_{k}\mathbf{e}_{k}\right)\cdot 
  \mathbf{V}^mdrd\theta dt 
  \\
  &+\int_{0}^{T}\int_{\widetilde{\Omega}}\rho^{m}
  \bigg{(}\left(v_{\theta}^{m}\right)^2\sum\limits_{k=1}^{n}C_{k}e_{r,k}-v_{r}^{m}v_{\theta}^{m}
  \sum\limits_{k=1}^{n}C_{k}e_{\theta,k}\bigg{)}
  drd\theta dt
  -\mu 
  \int_{0}^{T}\int_{\widetilde{\Omega}}
  \bigg{[}r\partial_{r}\mathbf{V}^{m}\cdot\partial_{r}\sum\limits_{k=1}^{n}C_{k}\mathbf{e}_{k}
  \\
  &+ 
  \frac{1}{r}\bigg{(}\partial_{\theta}\mathbf{V}^{m}\cdot\partial_{\theta}\sum\limits_{k=1}^{n}C_{k}\mathbf{e}_{k}
  +\mathbf{V}^{m}\cdot\sum\limits_{k=1}^{n}C_{k}\mathbf{e}_{k}
  -2v_{\theta}^{m}\partial_{\theta}\sum\limits_{k=1}^{m}C_{k}e_{r,k}+2v_{r}^{m}\partial_{\theta}
  \sum\limits_{k=1}^{n}C_{k}e_{\theta,k}
  \bigg{)}\bigg{]}drd\theta dt
  \\
  &-\int_{0}^{T}\int_{0}^{2\pi}r\left(\frac{\mu}{r}-\alpha\right)v_{\theta}^{m}
  \sum\limits_{k=1}^{n}C_{k}e_{\theta,k}|_{r=R_{1}}d\theta dt
  =-\int_{\widetilde{\Omega}}r\rho_{0}\mathbf{V}_{0}^{m}\cdot
  \sum\limits_{k=1}^{n}C_{k}\left(0\right)\mathbf{e}_{k}drd\theta 
  \\
  &+
  \int_{0}^{T}\int_{\widetilde{\Omega}}r\rho^{m} g\sum\limits_{k=1}^{n}C_{k}e_{r,k}drd\theta dt,
  ~~\mathbf{V}^{m}_{0}=\sum\limits_{k=1}^{m}\left(\int_{\widetilde{\Omega}}r\mathbf{V}_{0}
  \cdot\mathbf{e}_{k}drd\theta\right)\mathbf{e}_{k},
  \\
  &\int_{0}^{T}\int_{\widetilde{\Omega}}r\rho^{m}\partial_{t}\Psi drd\theta dt 
  +\int_{0}^{T}\int_{\widetilde{\Omega}}\rho^{m}\left(rv_{r}^{m}\partial_{r}\Psi
  +v_{\theta}^{m}\partial_{\theta}\Psi \right)drd\theta dt =- 
  \int_{\widetilde{\Omega}}r\rho_{0}\Psi\left(0,r,\theta\right)drd\theta,
\end{aligned}
\end{align}

From the estimates \eqref{xianyanguji0618}, \eqref{yangge0603} and \eqref{yongdedao0603}, there are 
subsequences of \(\left\{\mathbf{V}^{m}\right\}\) and \(\left\{\rho^{m}\right\}\) (we still denote 
the subsequences as \(\left\{\mathbf{V}^{m}\right\}\) and \(\left\{\rho^{m}\right\}\)) such that as 
\(m\rightarrow+\infty\), there exist \(\mathbf{V}\) and \(\rho\) satisfying   
\begin{align}\label{shoulianxing0606}
  \begin{aligned}
    &\mathbf{V}^{m}\rightarrow \mathbf{V}~\text{weakly~ star~in~}L^{\infty}\left(0,T;\mathbf{G}^{1,2}\right),
    \\
    &\mathbf{V}^{m}\rightarrow \mathbf{V}~\text{weakly~in~}L^{2}\left(0,T;\mathbf{V}^{1,2}\right),
    \\
    &\rho^{m}\rightarrow \rho~\text{weakly~ star~in~}L^{\infty}\left(
      \left(0,T\right)\times\widetilde{\Omega}\right),
      \\
      &\partial_{t} \mathbf{V}^{m}~\text{weakly~convergent~in~}L^{2}\left(0,T;\left[L^2\left(\widetilde{\Omega}\right)\right]^2\right).
  \end{aligned}
\end{align}
From the Lemma 1.3 in \cite{kim_weak_1987} (where we take \(E_{0}=\mathbf{V}^{1,2}\), \(E=\mathbf{G}^{1,2}\) and 
\(E_{1}=\left[L^2\left(\widetilde{\Omega}\right)\right]^2\)), we can conclude that 
\begin{align}\label{shoulianxing06062}
  \begin{aligned}
    \mathbf{V}^{m}\rightarrow \mathbf{V}~\text{strongly~in~}L^{2}\left(0,T;\mathbf{G}^{1,2}\right).
  \end{aligned}
\end{align}
The convergences \eqref{shoulianxing0606} and \eqref{shoulianxing06062} allow us to take the limits 
\(m\rightarrow+\infty\) in \eqref{zhiguanzhongyao0606}. Thus, we have 
\begin{align*}
  \begin{aligned}
    &\int_{0}^{T}\int_{\widetilde{\Omega}}r\rho \mathbf{V}\partial_{t} \sum\limits_{k=1}^{n}C_{k}\mathbf{e}_{k}
    drd\theta dt+\int_{0}^{T}\int_{\widetilde{\Omega}}\rho 
    \left(rv_{r}\partial_{r}\sum\limits_{k=1}^{n}C_{k}\mathbf{e}_{k}
    +v_{\theta}\partial_{\theta}\sum\limits_{k=1}^{n}C_{k}\mathbf{e}_{k}\right)\cdot 
    \mathbf{V}drd\theta dt 
    \\
    &+\int_{0}^{T}\int_{\widetilde{\Omega}}\rho
    \bigg{(}\left(v_{\theta}\right)^2\sum\limits_{k=1}^{n}C_{k}e_{r,k}-v_{r}v_{\theta}
    \sum\limits_{k=1}^{n}C_{k}e_{\theta,k}\bigg{)}
    drd\theta dt
    -\mu 
    \int_{0}^{T}\int_{\widetilde{\Omega}}
    \bigg{[}r\partial_{r}\mathbf{V}\cdot\partial_{r}\sum\limits_{k=1}^{n}C_{k}\mathbf{e}_{k}
    \\
    &+ 
    \frac{1}{r}\bigg{(}\partial_{\theta}\mathbf{V}\cdot\partial_{\theta}\sum\limits_{k=1}^{n}C_{k}\mathbf{e}_{k}
    +\mathbf{V}\cdot\sum\limits_{k=1}^{n}C_{k}\mathbf{e}_{k}
    -2v_{\theta}\partial_{\theta}\sum\limits_{k=1}^{m}C_{k}e_{r,k}+2v_{r}\partial_{\theta}
    \sum\limits_{k=1}^{n}C_{k}e_{\theta,k}
    \bigg{)}\bigg{]}drd\theta dt
    \\
    &-\int_{0}^{T}\int_{0}^{2\pi}r\left(\frac{\mu}{r}-\alpha\right)v_{\theta}
    \sum\limits_{k=1}^{n}C_{k}e_{\theta,k}|_{r=R_{1}}d\theta dt
    =-\int_{\widetilde{\Omega}}r\rho_{0}\mathbf{V}_{0}\cdot
    \sum\limits_{k=1}^{n}C_{k}\left(0\right)\mathbf{e}_{k}drd\theta 
    \\
    &+
    \int_{0}^{T}\int_{\widetilde{\Omega}}r\rho g\sum\limits_{k=1}^{n}C_{k}e_{r,k}drd\theta dt,
    \\
    &\int_{0}^{T}\int_{\widetilde{\Omega}}r\rho\partial_{t}\Psi drd\theta dt 
    +\int_{0}^{T}\int_{\widetilde{\Omega}}\rho\left(rv_{r}\partial_{r}\Psi
    +v_{\theta}\partial_{\theta}\Psi \right)drd\theta dt =- 
    \int_{\widetilde{\Omega}}r\rho_{0}\Psi\left(0,r,\theta\right)drd\theta,
  \end{aligned}
  \end{align*}
  Finally, due to the density, we can conclude that 
  \(\mathbf{V}\) and \(\rho\) satisfy \eqref{ruojiedecunzai0601}. Hence, the proof is completed.

\subsection{Proof of Theorem \ref{qiangjie0613}}
Subsequently, we show the existence of strong solution when \(\mathbf{V}_{0}\left(r,\theta\right)
=\left(v_{0r},v_{0\theta }\right)\in\mathbf{V}^{1,2}\). 

Replace \(\mathbf{e}_{k}\) by \(\partial_{t}\mathbf{V}^{m}\) in \(\eqref{bijinjie0605}_{1}\), we have 
\begin{align}\label{hexibudengshi0618}
  \begin{aligned}
    &\int_{\widetilde{\Omega}}r\rho^{m}\left|\partial_{t}\mathbf{V}^{m}\right|^2drd\theta 
    =-\int_{\widetilde{\Omega}}r\rho^{m}\left(v_{r}^{m}\partial_{r}\mathbf{V}^{m}+
    \frac{v_{\theta}^{m}}{r}\partial_{\theta}\mathbf{V}^{m}\right)\cdot\partial_{t}\mathbf{V}^{m}drd\theta 
    \\
&+\int_{\widetilde{\Omega}}\rho^{m}\left(\left(v_{\theta}^{m}\right)^2,-v_{r}^{m}v_{\theta}^{m}\right)
    \cdot\partial_{t}\mathbf{V}^{m}drd\theta+\mu\int_{\widetilde{\Omega}}r\left(\Delta_{r}v_{r}^{m}-\frac{v_{r}^{m}}{r^2}
    -\frac{2}{r^2}\partial_{\theta}v_{\theta}^{m}\right)\partial_{t}v_{r}^{m}drd\theta
    \\
    &-\int_{\widetilde{\Omega}}\rho^{m}gr\partial_{t}v_{r}^{m}drd\theta 
+\mu\int_{\widetilde{\Omega}}r\left(\Delta_{r}v_{\theta}^{m}-\frac{v_{\theta}^{m}}{r^2}
    +\frac{2}{r^2}\partial_{\theta}v_{r}^{m}\right)\partial_{t}v_{\theta}^{m}drd\theta.
  \end{aligned}
\end{align}
Utilizing the Cauchy inequality to right side of \eqref{hexibudengshi0618} and letting \(t\rightarrow 0^{+}\) give 
that 
\begin{align}\label{jiushizheyang0618}
  \lim\limits_{t\rightarrow 0^{+}}\int_{\widetilde{\Omega}}r\rho\left|\partial_{t}\mathbf{V}^{m}\left(t,r,\theta\right)\right|^2drd\theta 
  \leq C\left(\mu,R_{1},R_{2},\left\|\rho_{0}\right\|_{L^\infty\left(\widetilde{\Omega}\right)},
  \left\|\mathbf{V}_{0}\right\|_{H^2}\right).
\end{align}

Due to the differentiability of the coefficients in
\eqref{nulitishengziji0502}, we can conclude that 
\(\mathbf{C}_{m}\left(t\right)\in \left[C^{2}\left[0,T\right]\right]^{m}\). Thus,
 we are allowed to differentiate \eqref{bijinjie0605}
with respect to \(t\) and have
\begin{align*}
  \begin{aligned}
    &\int_{\widetilde{\Omega}}r\partial_{t}\rho^{m}\partial_{t}\mathbf{V}^{m}\cdot 
    \mathbf{e}_{k}drd\theta+\int_{\widetilde{\Omega}}r\rho^{m}\partial_{tt}\mathbf{V}^{m}\cdot \mathbf{e}_{k}
    drd\theta+\int_{\widetilde{\Omega}}r\partial_{t}\rho^{m}
    \left(v_{r}^{m}\partial_{r}\mathbf{V}^{m}+\frac{v_{\theta}^{m}}{r}\partial_{\theta}\mathbf{V}^{m}\right)
    \cdot\mathbf{e}_{k}drd\theta 
    \\
    &+\int_{\widetilde{\Omega}}r\rho^{m}
    \left(\partial_{t}v_{r}^{m}\partial_{r}\mathbf{V}^{m}
    +\frac{\partial_{t}v_{\theta}^{m}}{r}\partial_{\theta}\mathbf{V}^{m}\right)
    \cdot\mathbf{e}_{k}drd\theta+\int_{\widetilde{\Omega}}r\rho^{m}
    \left(v_{r}^{m}\partial_{tr}\mathbf{V}^{m}+\frac{v_{\theta}^{m}}{r}\partial_{t\theta}\mathbf{V}^{m}\right)
    \cdot\mathbf{e}_{k}drd\theta  
    \\
    &-\int_{\widetilde{\Omega}}\partial_{t}\rho^{m}
    \left(\left(v_{\theta}^{m}\right)^2,-v_{r}^{m}v_{\theta}^{m}\right)\cdot\mathbf{e}_{k}drd\theta 
    -\int_{\widetilde{\Omega}}\rho^{m}
    \left(2v_{\theta}^{m}\partial_{t}v_{\theta}^{m},-\partial_{t}v_{r}^{m}v_{\theta}^{m}-v_{r}^{m}
    \partial_{t}v_{\theta}^{m}\right)\cdot\mathbf{e}_{k}drd\theta
    \\
    &=-\int_{\widetilde{\Omega}}gr\left(\partial_{t}\rho^{m},0\right)\cdot\mathbf{e}_{k}drd\theta 
    +\mu\int_{\widetilde{\Omega}}r\left(\Delta_{r}\partial_{tr}\mathbf{V}^{m}- 
    \frac{\partial_{t}\mathbf{V}^{m}}{r^2}\right)\cdot\mathbf{e}_{k}drd\theta 
    \\
    &-\mu\int_{\widetilde{\Omega}}\frac{2}{r}\left(\partial_{t\theta}v_{\theta}^{m},-\partial_{t\theta}
    v_{r}^{m}\right)\cdot\mathbf{e}_{k}drd\theta.
  \end{aligned}
\end{align*}
Then, by the virtue of \(\eqref{bijinjie0605}_{2}\), replace \(\mathbf{e}_{k}\) by
\(\partial_{t}\mathbf{V}^{m}\) in above equality, and integrate over \(\widetilde{\Omega}\) by parts, we can obtain 
\begin{align}\label{xiexie0611}
  \begin{aligned}
    \frac{1}{2}&\frac{d}{dt}\int_{\widetilde{\Omega}}r\rho^{m}\left(\partial_{t}\mathbf{V}^{m}\right)^2drd\theta 
    +\mu\int_{\widetilde{\Omega}}
    \bigg{\{}r\left|\partial_{tr}\mathbf{V}^{m}\right|^2
    \\&+
    \frac{1}{r}\left[
      \left(\partial_{t\theta}v_{\theta}^{m}+\partial_{t}v_{r}^{m}\right)^2
      +\left(\partial_{t\theta}v_{r}^{m}-\partial_{t}v_{\theta}^{m}\right)^2
    \right]\bigg{\}}drd\theta
    \leq\sum\limits_{i=1}^{13}J_{i},
  \end{aligned}
\end{align}
where 
\begin{align*}
  \begin{aligned}
    &J_{1}=-\int_{\widetilde{\Omega}}\rho^{m}
    \left(rv_{r}^{m}\partial_{tr}\mathbf{V}^{m}+v_{\theta}^{m}\partial_{t\theta}\mathbf{V}^{m}\right)
    \cdot\partial_{t}\mathbf{V}^{m}drd\theta, 
    \\
    &J_{2}=-\int_{\widetilde{\Omega}}rv_{r}^{m}\rho^{m}
    \left[\partial_{r}v_{r}^{m}\partial_{r}\mathbf{V}^{m}+v_{r}^{m}\partial_{rr}\mathbf{V}^{m}
    +\partial_{r}\left(\frac{v_{\theta}^{m}}{r}\right)\partial_{\theta}\mathbf{V}^{m}+ 
    \frac{v_{\theta}^{m}}{r}\partial_{r\theta}v_{\theta}^{m}\right]\cdot \partial_{t}\mathbf{V}^{m}drd\theta,
    \\
    &J_{3}=-\int_{\widetilde{\Omega}}\rho^{m}v_{\theta}^{m}
    \left[\partial_{\theta}v_{r}^{m}\partial_{r}\mathbf{V}^{m}+v_{r}^{m}\partial_{r\theta}\mathbf{V}^{m}
    +\partial_{\theta}\left(\frac{v_{\theta}^{m}}{r}\right)\partial_{\theta}\mathbf{V}^{m}+ 
    \frac{v_{\theta}^{m}}{r}\partial_{\theta\theta}v_{\theta}^{m}\right]\cdot \partial_{t}\mathbf{V}^{m}drd\theta,
    \\
    &J_{4}=-\int_{\widetilde{\Omega}}rv_{r}^{m}\rho^{m}
    \left(v_{r}^{m}\partial_{r}\mathbf{V}^{m}+\frac{v_{\theta}^{m}}{r}\partial_{\theta}\mathbf{V}^{m}\right)
    \cdot\partial_{tr}\mathbf{V}^{m}drd\theta, 
    \\
    &J_{5}=-\int_{\widetilde{\Omega}}v_{\theta}^{m}\rho^{m}
    \left(v_{r}^{m}\partial_{r}\mathbf{V}^{m}+\frac{v_{\theta}^{m}}{r}\partial_{\theta}\mathbf{V}^{m}\right)
    \cdot\partial_{t\theta}\mathbf{V}^{m}drd\theta, 
    \\
    &J_{6}=-\int_{\widetilde{\Omega}}r\rho^{m}
    \left(\partial_{t}v_{r}^{m}\partial_{r}\mathbf{V}^{m}
    +\frac{\partial_{t}v_{\theta}^{m}}{r}\partial_{\theta}\mathbf{V}^{m}\right)
    \cdot\partial_{t}\mathbf{V}^{m}drd\theta,
    \\
    &J_{7}=-\int_{\widetilde{\Omega}}r\rho^{m}
    \left(v_{r}^{m}\partial_{tr}\mathbf{V}^{m}
    +\frac{v_{\theta}^{m}}{r}\partial_{t\theta}\mathbf{V}^{m}\right)
    \cdot\partial_{t}\mathbf{V}^{m}drd\theta,
    \end{aligned}
    \end{align*}
    \begin{align*}
\begin{aligned}
&J_{8}=\int_{\widetilde{\Omega}}r\rho^{m}v_{r}^{m}\left(\partial_{r}\frac{\left(v_{\theta}^{m}\right)^2}{r},
    -\partial_{r}\frac{v_{r}^{m}v_{\theta}^{m}}{r}\right)\cdot 
    \partial_{t}\mathbf{V}^{m}drd\theta,
    \\
&J_{9}=\int_{\widetilde{\Omega}}\rho^{m}v_{\theta}^{m}\left(\partial_{\theta}
    \frac{\left(v_{\theta}^{m}\right)^2}{r},
    -\partial_{\theta}\frac{v_{r}^{m}v_{\theta}^{m}}{r}\right)\cdot 
    \partial_{t}\mathbf{V}^{m}drd\theta,
    \\
    &J_{10}=\int_{\widetilde{\Omega}}\rho^{m}v_{r}^{m}\left(\left(v_{\theta}^{m}\right)^2,
    -v_{r}^{m}v_{\theta}^{m}\right)\cdot\partial_{tr}\mathbf{V}^{m}drd\theta,
    \\
    &J_{11}=\int_{\widetilde{\Omega}}\frac{\rho^{m}v_{\theta}^{m}}{r}\left(\left(v_{\theta}^{m}\right)^2,
    -v_{r}^{m}v_{\theta}^{m}\right)\cdot\partial_{t\theta}\mathbf{V}^{m}drd\theta,
    \\
    &J_{12}=\int_{\widetilde{\Omega}}\rho^{m}\left(2v_{\theta}^{m}\partial_{t}v_{\theta}^{m}, 
    -\partial_{t}v_{r}^{m}v_{\theta}^{m}-v_{r}^{m}\partial_{t}v_{\theta}^{m}\right)
    \cdot\partial_{t}\mathbf{V}^{m}drd\theta, 
    \\
    &J_{13}=-\int_{\widetilde{\Omega}}\left[
      \rho^{m}\left(grv_{r}^{m},0\right)\cdot\partial_{tr}\mathbf{V}^{m}+
      \rho^{m}\left(gv_{\theta}^{m},0\right)\cdot\partial_{t\theta}\mathbf{V}^{m}
    \right]drd\theta. 
  \end{aligned}
\end{align*}
Then, from Lemmas \ref{poincarebudengshi0517}-\ref{dengjia0523}, Cauchy inequality,
 estimates \eqref{yangge0603}, we have 
\begin{align*}
  \begin{aligned}
    &J_{1},J_{6},J_{7},J_{10},J_{11},J_{13}\leq \varepsilon\left\|\partial_{t}\widetilde{\nabla}\mathbf{V}^{m}\right\|
    _{L^2\left(\widetilde{\Omega}\right)}^2+
    C\left\|\widetilde{\nabla}^2\mathbf{V}^{m}\right\|_{L^2\left(\widetilde{\Omega}\right)}^2
    \int_{\widetilde{\Omega}}r\rho^{m}\left|
    \partial_{t}\mathbf{V}^{m}\right|^2drd\theta+C,
    \\
    &J_{2},J_{3}\leq \varepsilon\left\|\partial_{t}\widetilde{\nabla}\mathbf{V}^{m}\right\|
    _{L^2\left(\widetilde{\Omega}\right)}^2+C\left\|\widetilde{\nabla}^2\mathbf{V}^{m}\right\|
    _{L^2\left(\widetilde{\Omega}\right)}^2+C\left\|\widetilde{\nabla}^2\mathbf{V}^{m}\right\|
    _{L^2\left(\widetilde{\Omega}\right)}^2\int_{\widetilde{\Omega}}r\rho^{m}\left|
    \partial_{t}\mathbf{V}^{m}\right|^2drd\theta,
    \\
    &J_{4},J_{5}\leq \varepsilon\left\|\partial_{t}\widetilde{\nabla}\mathbf{V}^{m}\right\|
    _{L^2\left(\widetilde{\Omega}\right)}^2+C\left\|\rho_{0}\right\|_{L^\infty\left(\widetilde{\Omega}\right)}^2
    \left\|\widetilde{\nabla}^2\mathbf{V}^{m}\right\|_{L^2\left(\widetilde{\Omega}\right)}^2,
    \\
    &J_{8},J_{9},J_{12}\leq C(t)\int_{\widetilde{\Omega}}\rho^{m}r
    \left|\partial_{t}\mathbf{V}^{m}\right|^2drd\theta+C\left\|\widetilde{\nabla}^2\mathbf{V}^{m}\right\|
    _{L^2\left(\widetilde{\Omega}\right)}^2+C,
  \end{aligned}
\end{align*}
where \(C\left(t\right)\) is an integrable function. Thus, from  \eqref{xiexie0611}, we have 
\begin{align*}
  \frac{d}{dt}\left\|\sqrt{r\rho^{m}}\partial_{t}\mathbf{V}^{m}\right\|_{L^2\left(\widetilde{\Omega}\right)}^2
  +\left\|\partial_{t}\widetilde{\nabla}\mathbf{V}^{m}\right\|_{L^2\left(\widetilde{\Omega}\right)}^2
  \leq C\left(t\right)\left\|\sqrt{r\rho^{m}}
  \partial_{t}\mathbf{V}^{m}\right\|_{L^2\left(\widetilde{\Omega}\right)}^2+C(t),
\end{align*}
which together with \eqref{jiushizheyang0618} and Gronwall's inequality yields that
\begin{align}\label{buyongdasao0611}
  \begin{aligned}
    \left\|\sqrt{r\rho^m}\partial_{t}\mathbf{V}^{m}\right\|_{L^2\left(\widetilde{\Omega}\right)}
    \leq C,~\int_{0}^{T}\left\|
    \partial_{t}\widetilde{\nabla}\mathbf{V}^{m}\right\|_{L^2\left(\widetilde{\Omega}\right)}^2dt 
    \leq C.
  \end{aligned}
\end{align}
In addition, from Sobolev embedding inequality and H\"{o}lder inequality, we can deduce that
\[
\begin{pmatrix}
  \rho\left(v_{r}\partial_{r}v_{r}+\frac{v_{\theta}}{r}\partial_{\theta}v_{r}
   \right)-\frac{\rho v_{\theta}^2}{r}
  \\
  \rho\left(v_{r}\partial_{r}v_{\theta}+\frac{v_{\theta}}{r}\partial_{\theta}v_{\theta}
   \right)+\frac{\rho v_{r} v_{\theta}}{r}
\end{pmatrix}
\in \left[L^{\frac{3}{2}}\left(\widetilde{\Omega}\right)\right]^2,
\]
where \(\mathbf{V},\rho\) is weak solution in Theorem \ref{jubucunzai0605}. From Proposition \ref{pdecunzaixing0528} and 
\eqref{buyongdasao0611}, we can conclude that there exists a 
\(p\in W^{1,\frac{3}{2}}\left(\widetilde{\Omega}\right)\) such that 
\begin{align*}
  \left\|\widetilde{\nabla}^2 \mathbf{V}\right\|_{L^{\frac{3}{2}}\left(\widetilde{\Omega}\right)}
  +\left\|\widetilde{\nabla}p\right\|_{L^\frac{3}{2}\left(\widetilde{\Omega}\right)}\leq C,
\end{align*}
which indicates that 
\[
\begin{pmatrix}
  \rho\left(v_{r}\partial_{r}v_{r}+\frac{v_{\theta}}{r}\partial_{\theta}v_{r}
   \right)-\frac{\rho v_{\theta}^2}{r}
  \\
  \rho\left(v_{r}\partial_{r}v_{\theta}+\frac{v_{\theta}}{r}\partial_{\theta}v_{\theta}
   \right)+\frac{\rho v_{r} v_{\theta}}{r}
\end{pmatrix}
\in \left[L^{2}\left(\widetilde{\Omega}\right)\right]^2.
\]
Again applying Lemma \ref{stokesguji0521}, we obtain that \(p\in H^{1}\left(\widetilde{\Omega}\right)\) and
\begin{align}\label{dierciyongsotkeguji0620}
  \left\|\widetilde{\nabla}^2 \mathbf{V}\right\|_{L^2\left(\widetilde{\Omega}\right)}
  +\left\|\widetilde{\nabla}p\right\|_{L^2\left(\widetilde{\Omega}\right)}\leq C.
\end{align}
Thus, 
\begin{align*}
  \mathbf{V}\in L^{\infty}\left(0,T;\mathbf{V}^{1,2}\right),~
  \partial_{t}\mathbf{V}\in L^{\infty}\left(0,T;\left[L^2\left(\widetilde{\Omega}\right)\right]^2\right),~
  \partial_{t}\widetilde{\nabla}\mathbf{V}\in L^{2}\left(0,T;\left[L^2\left(\widetilde{\Omega}
  \right)\right]^2\right).
\end{align*}
Furthermore, from \(\eqref{buyongdasao0611}_{2}\) and \eqref{dierciyongsotkeguji0620}, we can show that 
\begin{align*}
 \partial_{t}\mathbf{V}\in L^2\left(0,T;\left[L^{4}\left(\widetilde{\Omega}\right)\right]^2\right),~
 \begin{pmatrix}
  \rho\left(v_{r}\partial_{r}v_{r}+\frac{v_{\theta}}{r}\partial_{\theta}v_{r}
   \right)-\frac{\rho v_{\theta}^2}{r}
  \\
  \rho\left(v_{r}\partial_{r}v_{\theta}+\frac{v_{\theta}}{r}\partial_{\theta}v_{\theta}
   \right)+\frac{\rho v_{r} v_{\theta}}{r}
\end{pmatrix}
\in L^2\left(0,T,\left[L^{4}\left(\widetilde{\Omega}\right)\right]^2\right),
\end{align*}
which together with Lemma \ref{stokesguji0521}, we can deduce that 
\begin{align}\label{qiajing0611}
  \mathbf{V}\in L^{2}\left(0,T;\left[W^{2,4}\left(\widetilde{\Omega}\right)\right]^2\right),~
  \widetilde{\nabla}p\in L^{2}\left(0,T;\left[L^{4}\left(\widetilde{\Omega}\right)\right]^2\right).
\end{align}
Therefore, we complete this proof. 

\subsection{Proof of Theorem \ref{qiangjiele0614}}
To begin with, from \(\eqref{bijinjie0605}_{2}\) we observe  
that \(\partial_{r}\rho^{m}\) and \(\partial_{\theta}\rho^{m}\) satisfy
\begin{align}\label{haishilimao0611}
  \begin{cases}
    \partial_{tr}\rho^{m}+\partial_{r}v_{r}^{m}\partial_{r}\rho^m 
    -\frac{v_{\theta}^{m}}{r^2}\partial_{\theta}\rho^m+\frac{\partial_{r}v_{\theta}^{m}}{r}
    \partial_{\theta}\rho^m+\frac{v_{\theta}^m}{r}\partial_{r\theta}\rho^{m}+v_{r}^{m}\partial_{rr}\rho^m=0,
    \\
    \partial_{t\theta}\rho^{m}+\partial_{\theta}v_{r}^{m}\partial_{r}\rho^{m}
    +\frac{\partial_{\theta}v_{\theta}^{m}}{r}\partial_{\theta}\rho^{m}
    +\frac{v_{\theta}^{m}}{r}\partial_{\theta\theta}\rho^{m}+v_{r}^{m}\partial_{r\theta}\rho^{m}=0.
  \end{cases}
\end{align}
Then multiplying \(\eqref{haishilimao0611}_{1}\) and \(\eqref{haishilimao0611}_{2}\) by \(r\partial_{r}\rho^m\) and \(r\partial_{\theta}\rho^m\), respectively,
integrating over  and summing up the results give that
\begin{align*}
  \begin{aligned}
  \frac{1}{2}\frac{d}{dt}\int_{\widetilde{\Omega}}&r\left|\widetilde{\nabla}\rho^m\right|^2drd\theta
  =-\int_{\widetilde{\Omega}}r\partial_{r}v_{r}^{m}\left|
  \partial_{r}\rho^{m}\right|^2drd\theta
  +\int_{\widetilde{\Omega}}\frac{v_{\theta}^{m}}{r}\partial_{\theta}\rho^{m}\partial_{r}\rho^m drd\theta 
  \\
  &-\int_{\widetilde{\Omega}}\partial_{r}v_{\theta}^{m}\partial_{\theta}\rho^m \partial_{r}\rho^m drd\theta 
  -\int_{\widetilde{\Omega}}r\partial_{\theta}v_{r}^{m}\partial_{r}\rho^{m}\partial_{\theta}\rho^{m}drd\theta 
  -\int_{\widetilde{\Omega}}\partial_{\theta}v_{\theta}^{m}\left|
  \partial_{\theta}\rho^{m}\right|^2drd\theta,
  \end{aligned}
\end{align*}
which gives that 
\begin{align*}
  \frac{d}{dt}\int_{\widetilde{\Omega}}r\left|\widetilde{\nabla}\rho^m\right|^2drd\theta
  \leq C\left(t\right)\left\|\widetilde{\nabla}\rho^m\right\|_{L^2\left(\widetilde{\Omega}\right)}^2,
\end{align*}
where \(C\left(t\right)=C\left\|\widetilde{\nabla}\mathbf{V}^{m}\right\|_{L^\infty\left(\widetilde{\Omega}\right)}\) is 
integrable from \eqref{qiajing0611}. 
Then, 
\begin{align*}
  \left\|\widetilde{\nabla}\rho^{m}\right\|_{L^2\left(\widetilde{\Omega}\right)}^2
  \leq C\left\|\widetilde{\nabla}\rho_{0}\right\|_{L^2\left(\widetilde{\Omega}\right)}^2
  +\int_{0}^{t}C\left(s\right)\left\|\widetilde{\nabla}
  \rho^{m}\right\|_{L^2\left(\widetilde{\Omega}\right)}^2ds,
\end{align*}
which together with Gronwall inequality deduces that  
\begin{align}\label{jiushi0611}
  \left\|\widetilde{\nabla}\rho^m\right\|_{L^2\left(\widetilde{\Omega}\right)}
  \leq C. 
\end{align}
In addition, we consider \(\partial_{t}\rho^{m}\). 
From \(\partial_{t}\rho^{m}=-v_{r}^{m}\partial_{r}\rho^m-\frac{v_{\theta}^{m}}{r}
\partial_{\theta}\rho^{m}\), we have 
\begin{align}\label{tianxiacangsheng0611}
  \left\|\partial_{t}\rho^m\right\|_{L^2\left(\widetilde{\Omega}\right)}\leq C.
\end{align}
Therefore, \eqref{jiushi0611} and \eqref{tianxiacangsheng0611} show that 
\[
\rho\in L^{\infty}\left(0,T;H^{1}\left(\widetilde{\Omega}\right)\right),~
\partial_{t}\rho\in L^{\infty}\left(0,T;L^{2}\left(\widetilde{\Omega}\right)\right).  
\]
\subsection{Proof of Theorem \ref{weiyixing0615}}
 We assume that there exist two strong solutions, 
  \(\left(\mathbf{V},\rho,p\right)\) and \(\left(\mathbf{\widetilde{V}},\widetilde{\rho},
  \widetilde{p}\right)\), to the system \eqref{model0329}-\eqref{bianjie0329} under the same initial 
  conditions. That is, 
  \begin{align}\label{yaozuocha0616}
    \begin{aligned}
      &\rho \partial_{t} v_{r}+\rho\left(v_{r} \partial_{r} v_{r}
      +\frac{v_{\theta}}{r}\partial_{\theta} v_{r}\right)-\frac{\rho v_{\theta}^2}{r}
    =\mu(\Delta_{r} v_{r}-\frac{v_{r}}{r^2}-\frac{2}{r^2}\partial_{\theta} v_{\theta})-
    \partial_{r} p-\rho g, 
    \\
    &\rho\partial_{t} v_{\theta}
    +\rho\left(v_{r} \partial_{r} v_{\theta}
    +\frac{v_{\theta}}{r}\partial_{\theta} v_{\theta}\right)+\frac{\rho v_{r} v_{\theta}}{r}
    =\mu(\Delta_{r} v_{\theta}-\frac{v_{\theta}}{r^2}+\frac{2}{r^2}\partial_{\theta} v_{r})-
    \frac{1}{r} \partial_{\theta} p,
    \\
    &\partial_{t} \rho+v_{r}\partial_{r} \rho
    +\frac{v_{\theta}}{r}\partial_{\theta} \rho=0,
    \end{aligned}
  \end{align}
  and 
  \begin{align}\label{yaozuocha06161}
    \begin{aligned}
      &\widetilde{\rho} \partial_{t} \widetilde{v}_{r}+\widetilde{\rho}
      \left(\widetilde{v}_{r} \partial_{r} \widetilde{v}_{r}
      +\frac{\widetilde{v}_{\theta}}{r}\partial_{\theta} \widetilde{v}_{r}\right)
      -\frac{\widetilde{\rho} \widetilde{v}_{\theta}^2}{r}
    =\mu(\Delta_{r} \widetilde{v}_{r}
    -\frac{\widetilde{v}_{r}}{r^2}-\frac{2}{r^2}\partial_{\theta} \widetilde{v}_{\theta})-
    \partial_{r} \widetilde{p}-\widetilde{\rho} g, 
    \\
    &\widetilde{\rho}\partial_{t} \widetilde{v}_{\theta}
    +\widetilde{\rho}\left(\widetilde{v}_{r} \partial_{r} \widetilde{v}_{\theta}
    +\frac{\widetilde{v}_{\theta}}{r}\partial_{\theta} \widetilde{v}_{\theta}\right)
    +\frac{\widetilde{\rho} \widetilde{v}_{r} \widetilde{v}_{\theta}}{r}
    =\mu(\Delta_{r} \widetilde{v}_{\theta}
    -\frac{\widetilde{v}_{\theta}}{r^2}+\frac{2}{r^2}\partial_{\theta} \widetilde{v}_{r})-
    \frac{1}{r} \partial_{\theta} \widetilde{p},
    \\
    &\partial_{t} \widetilde{\rho}+\widetilde{v}_{r}\partial_{r} \widetilde{\rho}
    +\frac{\widetilde{v}_{\theta}}{r}\partial_{\theta} \widetilde{\rho}=0.
    \end{aligned}
  \end{align}
Thus, by a series of algebraic and inner product operations, we can obtain 
\begin{align}\label{jiangwenming0621}
  \begin{aligned}
  &\frac{1}{2}\frac{d}{dt}\int_{\widetilde{\Omega}}\rho r\left|\mathbf{V}-\widetilde{\mathbf{V}}\right|^2drd\theta 
  +\mu\int_{\widetilde{\Omega}}\bigg{[}r\left|\partial_{r}
  \left(\mathbf{V}-\widetilde{\mathbf{V}}\right)\right|^2
  +\frac{1}{r}\bigg{(}\left|\partial_{\theta}\left(v_{\theta}-\widetilde{v}_{\theta}\right)
  +\left(v_{r}-\widetilde{v}_{r}\right)\right|^2
  \\
  &~~+\left|\partial_{\theta}\left(v_{r}-\widetilde{v}_{r}\right)-
  \left(v_{\theta}-\widetilde{v}_{\theta}\right)\right|^2\bigg{)}
  \bigg{]}drd\theta\leq -\sum\limits_{i=1}^{7}D_{i},
  \end{aligned}
\end{align}
where 
\begin{align*}
  \begin{aligned}
    &D_{1}=-\int_{\widetilde{\Omega}}\rho\left[rv_{r}\left(\mathbf{V}-\widetilde{\mathbf{V}}\right)
    \cdot\partial_{r}\left(\mathbf{V}-\widetilde{\mathbf{V}}\right)+v_{\theta}
    \left(\mathbf{V}-\widetilde{\mathbf{V}}\right)\cdot\partial_{\theta} 
    \left(\mathbf{V}-\widetilde{\mathbf{V}}\right)\right]drd\theta, 
    \\
    &D_{2}=\int_{\widetilde{\Omega}}
    r\left(\rho-\widetilde{\rho}\right)\partial_{t}\widetilde{\mathbf{V}}\cdot \left(
      \mathbf{V}-\widetilde{\mathbf{V}}
    \right)drd\theta,~D_{4}= \int_{\widetilde{\Omega}}
    r\left(\rho-\widetilde{\rho}\right)\left(\widetilde{v}_{r}\partial_{r}\widetilde{\mathbf{V}}
    +\frac{\widetilde{v}_{\theta}}{r}\partial_{\theta}\widetilde{\mathbf{V}}\right)\cdot \left(
      \mathbf{V}-\widetilde{\mathbf{V}}
    \right)drd\theta,
  \\
&D_{3}=\int_{\widetilde{\Omega}}
    \rho r\left[
      \left(v_{r}-\widetilde{v}_{r}\right)\partial_{r}\mathbf{V}+ 
      \widetilde{v}_{r}\partial_{r}\left(\mathbf{V}-\widetilde{\mathbf{V}}\right)
      +\frac{v_{\theta}-\widetilde{v}_{\theta}}{r}\partial_{\theta}\mathbf{V}+ 
      \frac{\widetilde{v}_{\theta}}{r}\partial_{\theta}\left(\mathbf{V}-\widetilde{\mathbf{V}}\right)
    \right]\cdot\left(\mathbf{V}-\widetilde{\mathbf{V}}\right)drd\theta,
    \\
    &D_{5}=-\int_{\widetilde{\Omega}}\rho 
    \left(\left(v_{\theta}-\widetilde{v}_{\theta}\right)\left(v_{\theta}+\widetilde{v}_{\theta}\right), 
    \left(\widetilde{v}_{r}-v_{r}\right)\widetilde{v}_{\theta}+v_{r}\left(\widetilde{v}_{\theta}
    -v_{\theta}\right)\right)\cdot\left(\mathbf{V}-\widetilde{\mathbf{V}}\right)
    drd\theta,
    \\
    &D_{6}=-\int_{\widetilde{\Omega}}\left(\rho-\widetilde{\rho}\right)
    \left(\widetilde{v}_{\theta}^2,\widetilde{v}_{r}\widetilde{v}_{\theta}\right)\cdot
    \left(\mathbf{V}-\widetilde{\mathbf{V}}\right)drd\theta,~
    D_{7}=\int_{\widetilde{\Omega}}gr\left(\rho-\widetilde{\rho}\right)\left(v_{r}-\widetilde{v}_{r}\right)
    drd\theta.
  \end{aligned}
\end{align*}
From Cauchy inequality, H\"{o}lder inequality and Sobolev embedding inequality, we have 
\begin{align*}
  \begin{aligned}
    &|D_{1}|\leq C\left\|\rho\right\|_{L^{\infty}\left(\widetilde{\Omega}\right)}
    \left\|\mathbf{V}\right\|_{L^\infty\left(\widetilde{\Omega}\right)}\int_{\widetilde{\Omega}}
    \rho r\left|\mathbf{V}-\widetilde{\mathbf{V}}\right|^2drd\theta+\varepsilon
    \left\|\widetilde{\nabla}\left(\mathbf{V}-\widetilde{\mathbf{V}}\right)\right\|
    _{L^2\left(\widetilde{\Omega}\right)}^2,
    \\
    &|D_{2}|\leq C\left\|\partial_{t}\widetilde{\mathbf{V}}\right\|_{L^4\left(\widetilde{\Omega}\right)}^2 
    \left\|\rho-\widetilde{\rho}\right\|_{L^{\frac{3}{2}}\left(\widetilde{\Omega}\right)}^2+ 
    \varepsilon \left\|\widetilde{\nabla}\left(\mathbf{V}-\widetilde{\mathbf{V}}\right)\right\|
    _{L^2\left(\widetilde{\Omega}\right)}^2,
    \\
    &|D_{3}|\leq C\left(\left\|\widetilde{\nabla}\mathbf{V}
    \right\|_{L^{\infty}\left(\widetilde{\Omega}\right)}+\left\|\widetilde{\mathbf{V}}\right\|_{L^{\infty}
    \left(\widetilde{\Omega}\right)}\right)\int_{\widetilde{\Omega}}\rho r\left|
    \mathbf{V}-\widetilde{\mathbf{V}}\right|^2drd\theta+ 
    \varepsilon\left\|\widetilde{\nabla}\left(\mathbf{V}-\widetilde{\mathbf{V}}\right)\right\|
    _{L^2\left(\widetilde{\Omega}\right)}^2,
    \\
    &|D_{4}|\leq C\left\|
    \widetilde{\nabla}\widetilde{\mathbf{V}}\right\|_{L^{\infty}\left(\widetilde{\Omega}\right)}^2
    \left\|\widetilde{\nabla}\widetilde{\mathbf{V}}\right\|_{L^2\left(\widetilde{\Omega}\right)}^2
    \left\|\rho-\widetilde{\rho}\right\|_{L^{\frac{3}{2}}\left(\widetilde{\Omega}\right)}^2 
    +\varepsilon \left\|\widetilde{\nabla}\left(\mathbf{V}-\widetilde{\mathbf{V}}\right)\right\|
    _{L^2\left(\widetilde{\Omega}\right)}^2,
    \\
    &|D_{5}|\leq C\left(\left\|\mathbf{V}
    \right\|_{L^{\infty}\left(\widetilde{\Omega}\right)}+\left\|\widetilde{\mathbf{V}}\right\|_{L^{\infty}
    \left(\widetilde{\Omega}\right)}\right)\int_{\widetilde{\Omega}}\rho r\left|
    \mathbf{V}-\widetilde{\mathbf{V}}\right|^2drd\theta,
    \\
    &|D_{6}|\leq C\left\|\widetilde{\mathbf{V}}\right\|_{L^{\infty}\left(\widetilde{\Omega}\right)}^4
    \left\|\rho-\widetilde{\rho}\right\|_{L^{\frac{3}{2}}\left(\widetilde{\Omega}\right)}^2+
    \varepsilon \left\|\widetilde{\nabla}\left(\mathbf{V}-\widetilde{\mathbf{V}}\right)\right\|
    _{L^2\left(\widetilde{\Omega}\right)}^2,
    \\
    &|D_{7}|\leq \varepsilon\left\|\widetilde{\nabla}\left(\mathbf{V}-\widetilde{\mathbf{V}}\right)
    \right\|_{L^2\left(\widetilde{\Omega}\right)}^2+C \left\|\rho-\widetilde{\rho}\right\|_{L^{\frac{3}{2}}
    \left(\widetilde{\Omega}\right)}^2.
  \end{aligned}
\end{align*}
Then, from \eqref{jiangwenming0621}, we can obtain 
\begin{align}\label{shixionghaoniu0621}
  \begin{aligned}
    \frac{d}{dt}\int_{\widetilde{\Omega}}\rho r\left|\mathbf{V}-\widetilde{\mathbf{V}}\right|^2drd\theta 
    +\left\|\widetilde{\nabla}\left(\mathbf{V}-\widetilde{\mathbf{V}}\right)
    \right\|_{L^2\left(\widetilde{\Omega}\right)}^2\leq 
    C\left(t\right)A\left(t\right),
  \end{aligned}
\end{align}
where \(A\left(t\right)=\left[\left\|\sqrt{r\rho}\left(\mathbf{V}-\widetilde{\mathbf{V}}\right)
\right\|_{L^2\left(\widetilde{\Omega}\right)}^2
  +\left\|r^{\frac{2}{3}}\left(\rho-\widetilde{\rho}\right)
\right\|_{L^{\frac{3}{2}}\left(\widetilde{\Omega}\right)}^2
\right]\) and  
 \(C\left(t\right)\) is a positive integrable function.
In addition, from \(\eqref{yaozuocha0616}_{3}\), \(\eqref{yaozuocha06161}_{3}\), incompressible condition
 and H\"{o}lder inequality, 
 we have
 \begin{align*}
  \begin{aligned}
    \int_{\widetilde{\Omega}}r\partial_{t}\left(\rho-\widetilde{\rho}\right) 
    \left(\rho-\widetilde{\rho}\right)^{\frac{1}{2}}drd\theta&=- 
    \int_{\widetilde{\Omega}}\left[r\left(v_{r}-\widetilde{v}_{r}\right)\partial_{r}\widetilde{\rho}
    +\left(v_{\theta}-\widetilde{v}_{\theta}\right)\partial_{\theta}\widetilde{\rho}\right]
    \left(\rho-\widetilde{\rho}\right)^{\frac{1}{2}}drd\theta
    \\
    &\leq C\left\|\mathbf{V}-\widetilde{\mathbf{V}}\right\|_{L^6\left(\widetilde{\Omega}\right)}
    \left\|\widetilde{\nabla}\widetilde{\rho}\right\|_{L^2\left(\widetilde{\Omega}\right)}
    \left\|r^{\frac{2}{3}}\left(\rho-\widetilde{\rho}\right)
    \right\|_{L^{\frac{3}{2}}\left(\widetilde{\Omega}\right)}^{\frac{1}{2}}.
  \end{aligned}
 \end{align*} 
Then, from H\"{o}lder inequality, we can obtain
\begin{align*}
  \begin{aligned}
    \frac{d}{dt}\left\|
    r^{\frac{2}{3}}\left(\rho-\widetilde{\rho}\right)\right\|_{L^{\frac{3}{2}}\left(\widetilde{\Omega}\right)}^2
    &=\frac{4}{3}\left\|r^{\frac{2}{3}}\left(\rho-\widetilde{\rho}\right)
    \right\|_{L^{\frac{3}{2}}\left(\widetilde{\Omega}\right)}^{\frac{1}{2}}
    \frac{d}{dt}\left\|r^{\frac{2}{3}}\left(\rho-\widetilde{\rho}\right)
    \right\|_{L^{\frac{3}{2}}\left(\widetilde{\Omega}\right)}^{\frac{3}{2}}
    \\
    &\leq C\left\|\mathbf{V}-\widetilde{\mathbf{V}}\right\|_{L^6\left(\widetilde{\Omega}\right)}
    \left\|\widetilde{\nabla}\widetilde{\rho}\right\|_{L^2\left(\widetilde{\Omega}\right)}
    \left\|r^{\frac{2}{3}}\left(\rho-\widetilde{\rho}\right)
    \right\|_{L^{\frac{3}{2}}\left(\widetilde{\Omega}\right)}
    \\
    &\leq \varepsilon \left\|\widetilde{\nabla}\left(\mathbf{V}-\widetilde{\mathbf{V}}\right)
    \right\|_{L^2\left(\widetilde{\Omega}\right)}^2+
    C\left\|\widetilde{\nabla}\rho\right\|_{L^2\left(\widetilde{\Omega}\right)}^2
    \left\|r^{\frac{2}{3}}\left(\rho-\widetilde{\rho}\right)
    \right\|_{L^{\frac{3}{2}}\left(\widetilde{\Omega}\right)}^{2},
  \end{aligned}
\end{align*}
 which together with \eqref{shixionghaoniu0621} yields that 
 \[
 \frac{d}{dt}A\left(t\right)\leq C\left(t\right)A\left(t\right). 
 \]
 In addition, we have \(A\left(0\right)=0
  \). Thus, from Gronwall's inequality we can deduce  that 
  \[
  \mathbf{V}=\widetilde{\mathbf{V}},~\rho=\widetilde{\rho},~\text{a.e.~in}~\widetilde{\Omega},
  \]
which verifies the uniqueness of strong solution. 

\section{Linear instability}\label{xianxingbuwending1212}
In this section, we investigate the linear instability of the steady state \(\left(\mathbf{0},\overline{\rho}\right)\) satisfying \eqref{steadystate0403} and \eqref{zizhu0410}. That is, we look for an exponential growth solution to the linearized system \eqref{xiaxing0403}-\eqref{wuming1211}. 
\subsection{linear unstable solution}\label{buwending1211}
First, we employ the method of separation of variables to consider the system \eqref{xiaxing0403}-\eqref{wuming1211}.
Let \((\widetilde{v}_{r},\widetilde{v}_{\theta},\widetilde{p},\widetilde{\rho})\)=
\(e^{\lambda t}(v_{1}(r,\theta),v_{2}(r,\theta),\Pi(r,\theta),h(r,\theta))\). Then the system \eqref{xiaxing0403}-\eqref{wuming1211} becomes in forms
\begin{align}\label{yalouwuyan0403}
  \begin{cases}
    \lambda\overline{\rho}v_{1}=\mu\left(\Delta_{r}v_{1}-\frac{v_{1}}{r^2}-\frac{2}{r^2}
    \partial_{\theta} v_{2}\right)-\partial_{r}\Pi-h g,
    \\
    \lambda\overline{\rho}v_{2}=\mu\left(\Delta_{r}v_{2}-\frac{v_{2}}{r^2}
    +\frac{2}{r^2}\partial_{\theta} v_{1}\right)-\frac{1}{r}\partial_{\theta} \Pi,
    \\
    \lambda h=-v_{1}D\overline{\rho},~\partial_{r}(rv_{1})+\partial_{\theta} v_{2}=0,
  \end{cases}
\end{align}
subject to  
\begin{align}\label{huajuan0403}
  v_{1}(R_{1},\theta)=v_{1}(R_{2},\theta)=v_{2}\left(R_{2},\theta\right)=0,~
  \partial_{r} v_{2}
  =\left(\frac{1}{r}-\frac{\alpha}{\mu}\right) v_{2}~\text{as}~r=R_{1}.
\end{align}
For the periodic condition on \(\theta\), we take \((
  v_{1},v_{2},\Pi,h
)\)=\( 
  (w_{1}(r),-iw_{2}(r),h_{1}(r),h_{2}(r)
)  e^{ik\theta} (k\in\mathbf{Z}/\{0\})\) into \eqref{yalouwuyan0403}. Then we obtain 
\begin{align}\label{bianlianfenli0403}
  \begin{cases}
    \mu\left(D^2+\frac{D}{r}-\frac{k^2+1}{r^2}-\frac{\lambda\overline{\rho}}{\mu}\right)w_{1}
    -\frac{2\mu k w_{2}}{r^2}-Dh_{1}-h_{2}g=0,
    \\
    \mu\left(D^2+\frac{D}{r}-\frac{k^2+1}{r^2}-\frac{\lambda\overline{\rho}}{\mu}\right)w_{2}
    -\frac{2\mu k w_{1}}{r^2}+\frac{h_{1}k}{r}=0,
    \\
    \lambda h_{2}=-w_{1}D\overline{\rho},~kw_{2}=-D\left(rw_{1}\right),
    \\
    w_{1}(R_{1})=w_{1}(R_{2})=w_{2}\left(R_{2}\right)=0,~
    Dw_{2}=\left(\frac{1}{r}-\frac{\alpha}{\mu}\right)w_{2}~\text{as}~r=R_{1}.
  \end{cases}
\end{align}
A series of calculation to eliminate \(w_{2},f_{1},f_{2}\) and taking $w_{1}=W$ in the 
result give that 
\begin{align}\label{kanbian0405}
  \begin{cases}
    \begin{aligned}
    \lambda\mu&\bigg{\{}
      \left(D^2+\frac{D}{r}-\frac{k^2+1}{r^2}\right)k^2 W+ \frac{4k^2 W}{r^2}
      -D\left[r\left(D^2+\frac{D}{r}-\frac{k^2+1}{r^2}\right)D\left(rW\right)\right]
      \bigg{\}}
      \\
      &-\lambda^2\left[\overline{\rho}k^2 W-D\left(r\overline{\rho}D\left(rW\right)\right)\right]
      +k^2 g W D\overline{\rho}=0,
    \end{aligned}
    \\
    W\left(R_{1}\right)=W\left(R_{2}\right)=DW\left(R_{2}\right)=0,~D^2\left(rW\right)=
    \left(\frac{1}{r}-\frac{\alpha}{\mu}\right)D\left(rW\right),~r=R_{1}.
  \end{cases}
\end{align}
Obviously, if there is a solution to \eqref{kanbian0405} with \(\lambda>0\) for some \(k\in Z/\{0\}\), then we can solve the system \eqref{bianlianfenli0403} which indicates that there is a linearly unstable solution to the system \eqref{xiaxing0403}-\eqref{wuming1211}.
In order to solve the above equation, we employ the modified 
variational method, that is, we consider a family of ODEs as follows,  
\begin{align}\label{kanbian0409}
  \begin{cases}
    \begin{aligned}
    s\mu&\bigg{\{}
      \left(D^2+\frac{D}{r}-\frac{\xi^2+1}{r^2}\right)\xi^2 W+ \frac{4\xi^2 W}{r^2}
      -D\left[r\left(D^2+\frac{D}{r}-\frac{\xi^2+1}{r^2}\right)D\left(rW\right)\right]
      \bigg{\}}
      \\
      &-\lambda^2\left[\overline{\rho}\xi^2 W-D\left(r\overline{\rho}D\left(rW\right)\right)\right]
      +\xi^2 g W D\overline{\rho}=0,
    \end{aligned}
    \\
    W\left(R_{1}\right)=W\left(R_{2}\right)=DW\left(R_{2}\right)=0,~D^2\left(rW\right)=
    \left(\frac{1}{r}-\frac{\alpha}{\mu}\right)D\left(rW\right),~r=R_{1},
  \end{cases}
\end{align}
where \(s>0\) and \(\xi\neq 0\). Multiplying \(\eqref{kanbian0409}_{1}\) by \(rW\) and integrating over 
\([R_{1},R_{2}]\) give the following equality
\begin{align*}
  -\lambda^2 =\frac{\mu s E_{2}\left(W\right)
  +E_{3}\left(W\right)}{E_{1}\left(W\right)}:=\frac{E(W)}{E_{1}(W)}, 
\end{align*}
where 
\begin{align*}
  \begin{aligned}
    &E_{1}(W)=E_{1}(\xi,W)=\int_{I}\left[\xi^2 r \overline{\rho}W^2 
    +r\overline{\rho}\left|D\left(rW\right)\right|^2\right]dr,
    \\
    &E_{3}(W)=E_{3}\left(\xi,W\right)=-\xi^2\int_{I}grW^2D\overline{\rho}dr,~I=[R_{1},R_{2}],
\\
      &E_{2}(W)=E_{2}(\xi,W)=
      \int_{I}\left[
        r(2\xi^2+1)\left|DW\right|^2+\frac{(\xi^2-1)^2}{r}W^2
      +r\left|D^2\left(rW\right)\right|^2\right]dr 
      \\
      &~~~~~~~~~~~~~~~~~~~~~~~~~~~~
      +\left(1-\frac{\alpha}{\mu} R_{1}\right)\left|D\left(rW\right)\right|^2\big{|}_{r=R_{1}}.
    \end{aligned}
\end{align*}
To consider the existence of \(W\), we investigate the following extreme value problem
\begin{align}\label{xinjiang0405}
  \Phi(s):=\Phi\left(s,\xi\right)=\inf\limits_{W\in H^2(I)\cap H_{0}^{1}(I)}
  \frac{E\left(W\right)}{E_{1}(W)}.
\end{align}
Note that all the functions defined are function of \(s\) and \(\xi\). If there is no ambiguity, sometimes 
we will just write down the variable \(s\) for convenience.  
We will solve \eqref{xinjiang0405} in Proposition \ref{kequdao0405}. By the homogeneity, 
\eqref{xinjiang0405} is equivalent to 
\begin{align}\label{xinjiang0409}
  \Phi(s)=\inf\limits_{W\in \mathcal{A}_{\xi}}
  E\left(W\right),
\end{align}
where \(\mathcal{A}_{\xi}=\left\{
  W\in H^2(I)\cap H_{0}^{1}(I)\big{|}E_{1}\left(\xi,W\right)=1,~DW\left(R_{2}\right)=0
\right\}\). 

We can achieve the minimum of \(E\left(W\right)\) on \(\mathcal{A}_{\xi}\), as shown in Proposition \ref{kequdao0405}. Furthermore the corresponding extreme value point \(\left(W_{0},\lambda_{0}\right)\) satisfies the equation \eqref{kanbian0409} with \(\lambda=s=\lambda_{0}\), as detailed in Proposition \ref{weaksolution0405}; To utilize the intermediate value theorem for continuous function to verify the existence of \(\lambda_{0}\), we consider the properties of \(\Phi\left(s\right)\). Specifically, \(\Phi\left(s\right)\) is bounded, as shown in Proposition \ref{junbuyoujie0405}, local Lipschitz continuous and increasing with respect to \(s\), as demonstrated in Proposition \ref{junbulibuxizhi0410}. We can then establish the existence of \(\lambda_{0}\), see Propositions \ref{jidazhi0410}-\ref{shaonian0410}. Lastly, we also examine the continuity and boundedness of \(\lambda_{0}\left(\xi\right)\) with respect to \(\xi\).    
\begin{proposition}\label{kequdao0405}
Assume that \(1-\frac{\alpha R_{1}}{\mu}\geq 0\), then
  for any \(\xi\neq 0\) and $s>0$, \(E(W)\) achieves its minimum on \(\mathcal{A}_{\xi}\).
\end{proposition}
\begin{proof}
  A direct computation shows that 
  \begin{align*}
    \begin{aligned}
      E(W)\geq -\xi^2g\|D\overline{\rho}\|_{L^{\infty}(I)}
      \int_{I}rW^2dr.
    \end{aligned}
  \end{align*}
 By virtue of the condition \(W\in\mathcal{A}_{\xi}\), there 
appears the following inequality 
\begin{align*}
  E(W)\geq -g\left\|\frac{D\overline{\rho}}{\overline{\rho}}\right\|_{L^{\infty}\left(I\right)}\geq -\infty.
\end{align*}
  It follows that for any $\xi\neq 0$ and $s>0$, \(E(W)\) has lower bound. Consequently, 
  there exists a sequence $\left\{W^{n}\right\}\subset\mathcal{A}_{\xi}$ such that 
  \[
    \lim\limits_{n\rightarrow+\infty}E(W^{n})=\inf\limits_{W\in\mathcal{A}_{\xi}}E(W).
    \]
    Thus, one can assume that 
    \[
      E(W^{n})\leq \inf\limits_{W\in\mathcal{A}_{\xi}}E(W)+1<+\infty,
    \]
    which indicates that $\{rW^{n}\}$ is bounded in $H^2(I)\cap H_{0}^1(I)$. Furthermore, from the reflexivity of 
    $H^2(I)\cap H_{0}^{1}(I)$, there exists $W_{0}\in H^2(I)\cap H_{0}^{1}(I)$ satisfying 
    $W^{n}\rightarrow W_{0}$ weakly in $H^2(I)\cap H_{0}^{1}(I)$ and $W^{n}\rightarrow W_{0}$ strongly in 
    $H_{0}^{1}(I)$. Besides, it is not difficult to verify that $E(W)$ is weakly lower semi-continuous. Hence 
    \[
      E(W_{0})\leq \lim\limits_{n\rightarrow+\infty}E(W^{n})=
       \inf\limits_{W\in\mathcal{A}_{\xi}}E(W).
      \]
Also, one can show that \(W_{0}\in \mathcal{A}_{\xi}\) according to the strong convergence. Therefore, the proof 
is completed. 
\end{proof}
\begin{proposition}\label{weaksolution0405}Assume $-\lambda_{0}^2=\Phi(\lambda_{0})=E(W_{0})$, i.e, \(W_{0}\) is the 
  minimizer, then \(W_{0}\) is smooth and $(W_{0},\lambda_{0})$ satisfies the equation \eqref{kanbian0409}. 
\end{proposition}
\begin{proof}
    $\forall~q\in\mathbf{N}^{*}$, \(W\in H^2(I)\cap H_{0}^{1}(I)\), $z$ and $t\in\mathbf{R}$, we define the 
    functional 
    \[
      J(t,z)=E_{1}(W_{0}+tW+zW_{0}).
      \]
      Then, a direct computation gives that 
      \begin{align*}
        \begin{aligned}
        &J(0,0)=1,
        \\
        &\partial_{t}J(0,0)=2\int_{I}\left[\xi^2 r\overline{\rho}W_{0}W 
        +r\overline{\rho}D\left(rW_{0}\right)D\left(rW\right)\right]dr, 
        \\
        &\partial_{z}J(0,0)=2\int_{I}\left[\xi^2 r\overline{\rho}W_{0}^2 
        +r\overline{\rho}\big{|}D\left(rW_{0}\right)\big{|}^2\right]dr=2\neq 0.
        \end{aligned}
      \end{align*} 
      Then from the implicit function theorem, there exists a smooth function \(z=z(t)\) defined near $0$, such 
      that \(J(t,z(t))=1\), \(z(0)=0\) and \(z'(0)=-\frac{\partial_{t}J(0,0)}{\partial_{z}J(0,0)}=
      -\int_{I}\left[\xi^2 r\overline{\rho}W_{0}W+r\overline{\rho}D\left(
        rW_{0}
      \right)D\left(rW\right)\right]dr\).
      
      Next, we consider the following functional 
      \[
        \widetilde{J}(t)=E\left(W_{0}+tW+z(t)W_{0}\right).   
      \]
      Since \(W_{0}\) is the extreme point, then 
      \begin{align}\label{haishiyaobijiao0622}
        \begin{aligned}
          0=&\frac{d\widetilde{J}(0)}{dt}=
          2\mu \lambda_{0}\int_{I}\bigg{[}r(2\xi^2+1)DW_{0}DW+\frac{(\xi^2-1)^2}{r}W_{0}W 
          +rD^2(rW_{0})D^2(rW)\bigg{]}dr 
          \\
          &
          +2\lambda_{0}(\mu-\alpha R_{1})D(rW_{0})D(rW)\big{|}_{r=R_{1}}
          -2\xi^2\int_{I}grW_{0}WD\overline{\rho}dr
          +2 z'(0)E(W_{0})
          \\
          =& 2\mu \lambda_{0}\int_{I}\bigg{[}r(2\xi^2+1)DW_{0}DW+\frac{(\xi^2-1)^2}{r}W_{0}W 
          +rD^2(rW_{0})D^2(rW)\bigg{]}dr 
          \\
          &
          +2\lambda_{0}(\mu-\alpha R_{1})D(rW_{0})D(rW)\big{|}_{r=R_{1}}
          -2\xi^2\int_{I}grW_{0}WD\overline{\rho}dr
          \\
          &+2 \lambda_{0}^2\int_{I}\left[\xi^2 r\overline{\rho}W_{0}W+r\overline{\rho}D\left(
            rW_{0}
          \right)D\left(rW\right)\right]dr.
        \end{aligned}
      \end{align}
      Let $W\in C_{0}^{\infty}\left(I\right)$ in the above equality. Then, one can show that 
      $W_{0}$ solves the equation in weak sense. By bootstrap method, one can verify that 
      \(W_{0}\in H^{q}(I)\) for any positive integer \(q\).
      Now, we have 
      \begin{align*}
        \begin{aligned}
          \mu \lambda_{0}&\bigg{\{}
            \left(D^2+\frac{D}{r}-\frac{\xi^2+1}{r^2}\right)\xi^2 W_{0}+ \frac{4 \xi^2 
            W_{0}}{r^2}
            - D\left[r\left(D^2+\frac{D}{r}-\frac{\xi^2+1}{r^2}\right)D\left(rW_{0}\right)\right]
            \bigg{\}}
            \\
            &-\lambda^2\left[\overline{\rho}\xi^2 W-D\left(r\overline{\rho}D\left(rW_{0}\right)\right)\right]
            +\xi^2 g W_{0} D\overline{\rho}=0.
          \end{aligned}
      \end{align*}
Multiplying the above equation by \(rW\in \mathcal{A}_{\xi}\) and integrating over \([R_{1},R_{2}]\) give that 
\begin{align*}
  \begin{aligned}
    &\lambda_{0}\mu\left\{
     \int_{I}\left[r(2\xi^2+1)DW_{0}DW+\frac{(\xi^2-1)^2}{r}W_{0}W+rD^2(rW_{0})D^2(rW)\right]dr+ 
    rD^2(rW_{0})D(rW)\big{|}_{r=R_{1}}
   \right\}
   \\
   & +\lambda_{0}^2\int_{I}\left[\overline{\rho}r\xi^2W_{0}W+r\overline{\rho} 
    D(rW_{0})D(rW)\right]dr-\xi^2\int_{I}grW_{0}WD\overline{\rho}dr=0.
  \end{aligned}
\end{align*}
Since \(W\) is arbitrary, comparing the above equality with \eqref{haishiyaobijiao0622}, one can obtain 
\begin{align*}
  rD^2(rW_{0})\big{|}_{r=R_{1}}&=(1-\frac{\alpha}{\mu}R_{1})\left(D\left(rW_{0}\right)\right)\big{|}_{r=R_{1}},
\end{align*}
which means that \(W_{0}\) satisfies the boundary conditions. Therefore, the proof is completed.
\end{proof}
 \begin{proposition}\label{junbuyoujie0405}
  \(\Phi(s)\) is bounded on \(I_{1}\times I_{2}\), where \(I_{1}=[a,b]\subset(0,+\infty)\) and 
  \(I_{2}=[c,d]\subset(0,+\infty)\) are two bounded sets. That is, there exists a positive constant 
  \(M(I_{1},I_{2})\) dependent of \(I_{1}\) and \(I_{2}\) such that 
  \[
  \left|\Phi(s,\xi)\right|\leq M(I_{1},I_{2}).  
  \]
 \end{proposition}
 \begin{proof}
  We have shown that \(\Phi(s,\xi)\) has a lower bound. If \(\Phi(s,\xi)\) is unbounded on 
  \(I_{1}\times I_{2}\), then there exists \(\widetilde{s}\in I_{1}\) and \(\widetilde{\xi}\in I_{2}\) such that 
  \[
  \Phi(\widetilde{s},\widetilde{\xi})=+\infty.  
  \]
  Then, according to the definition of \(\Phi(s,\xi)\), one has 
  \[
  \frac{E(W)}{E_{1}(W)}=+\infty,~\text{for~every~}W\in H^2\cap H_{0}^1,  
  \]
  which is impossible. Hence, the conclusion of the proposition is valid. 
 \end{proof}    

\begin{proposition}\label{junbulibuxizhi0410}\(\Phi(s,\xi)\in C_{\text{loc}}^{0,1}((0,+\infty)
  \times (0,+\infty))\) and is increasing with \(s\) on \((0,+\infty)\) 
  when \(\mu-\alpha R_{1}\geq 0\).
\end{proposition}
\begin{proof}
  Let \(I_{1}=[a,b]\subset (0,+\infty)\) and \(I_{2}=[c,d]\subset(0,+\infty)\) be two bounded sets. We will 
  prove the Lipschitz continuity of \(\Phi(s,\xi)\) with respect to \(s\) and \(\xi\), respectively. 
  The condition \(\mu-\alpha R_{1}\geq 0\) implies that 
  \(E_{2}(\xi,W)>0\) for any \(W\in H^{2}(I)\cap H_{0}^{1}(I)\). 
  
  \textbf{Step 1}: \(\Phi(s,\xi)\) is Lipschitz continuous with \(s\) on \(I_{1}\) and 
  is increasing with \(s\) on \((0,+\infty)\). 

  \(\forall a\leq s_{1}\leq s_{2}\leq b\) and \(\xi\in I_{2}\) is fixed, from the definition of 
  \(\Phi(s,\xi)\), one can show that there exist \(W_{s_{1}}\) and \(W_{s_{2}}\)\(\in H^{2}(I)\cap H_{0}^{1}(I)\)
  such that 
  \[
  \Phi(s_{1},\xi)=\frac{E(\xi,W_{s_{1}})}{E_{1}(\xi,W_{s_{1}})},~  
  \Phi(s_{2},\xi)=\frac{E(\xi,W_{s_{2}})}{E_{1}(\xi,W_{s_{2}})}.
  \]
Then a direct computation gives that 
\begin{align}\label{anxin0410}
  \begin{aligned}
    \Phi(s_{2},\xi)-\Phi(s_{1},\xi)&=\frac{\mu s_{2}E_{2}(\xi,W_{s_{2}})+E_{3}(\xi,W_{s_{2}})}{E_{1}(\xi,W_{s_{2}})}
    -\frac{\mu s_{1}E_{2}(\xi,W_{s_{1}})+E_{3}(\xi,W_{s_{1}})}{E_{1}(\xi,W_{s_{1}})}
    \\
    &\leq \frac{\mu s_{2}E_{2}(\xi,W_{s_{1}})+E_{3}(\xi,W_{s_{1}})}{E_{1}(\xi,W_{s_{1}})}
    -\frac{\mu s_{1}E_{2}(\xi,W_{s_{1}})+E_{3}(\xi,W_{s_{1}})}{E_{1}(\xi,W_{s_{1}})}
    \\
    &\leq \frac{\mu (s_{2}-s_{1})E_{2}(\xi,W_{s_{1}})}{E_{1}(\xi,W_{s_{1}})}.
  \end{aligned}
\end{align}
Similarly, one obtains 
\begin{align}\label{dandian10410}
  \begin{aligned}
    \Phi(s_{1},\xi)-\Phi(s_{2},\xi)\leq \frac{\mu(s_{1}-s_{2})E_{2}(\xi,W_{s_{2}})}{E_{1}(\xi,W_{s_{2}})}\leq 0, 
  \end{aligned}
\end{align}
which indicates the monotonicity of \(s\) on \((0,+\infty)\). 

With the local bound of \(\left|\Phi(s,\xi)\right|\leq M_{I_{1}I_{2}}\) (see Proposition \ref{junbuyoujie0405}), 
it follows that 
\begin{align}\label{youjie0410}
 \mu s \left|
 \frac{E_{2}(\xi,W)}{E_{1}(\xi,W)}
 \right| \leq M_{I_{1}I_{2}}+\left|
 \frac{E_{3}(\xi,W)}{E_{1}(\xi,W)}
 \right|\leq M_{I_{1}I_{2}}+\frac{R_{2}\|\frac{D\overline{\rho}}{\overline{\rho}}\|_{L^\infty(I)}g}{R_{1}}
 :=\widetilde{M}_{I_{1}I_{2}}.
\end{align}
Combining with \eqref{anxin0410}, \eqref{dandian10410} and \eqref{youjie0410}, we have 
\begin{align*}
  \left|
  \Phi(s_{1},\xi)-\Phi(s_{2},\xi)
  \right|\leq \widetilde{M}_{I_{1}I_{2}}\left|
  s_{1}-s_{2}
  \right|.
\end{align*}

\textbf{Step 2}: \(\Phi(s,\xi)\) is Lipschitz continuous with \(\xi^{2}\) on \(I_{2}\). 

\(\forall \xi_{1}\leq \xi_{2}\) and \(\xi_{1},~\xi_{2}\in I_{2}\), then there exist 
\(W_{\xi_{1}}\) and \(W_{\xi_{2}}\) \(\in H^{2}(I)\cap H_{0}^{1}(I)\) such that 
\[
\Phi(s,\xi_{1})=\frac{E(\xi_{1},W_{\xi_{1}})}{E_{1}(\xi_{1},W_{\xi_{1}})},~
\Phi(s,\xi_{2})=\frac{E(\xi_{2},W_{\xi_{2}})}{E_{1}(\xi_{2},W_{\xi_{2}})}. 
\]
Similarly, with the bound of \(\Phi(s,\xi)\) on \(I_{1}\times I_{2}\) and the definition of minimum, one can 
verify that there is also a positive constant \(\widehat{M}_{I_{1}I_{2}}\) such that
\begin{align*}
  \begin{aligned}
    \left|\Phi(s,\xi_{1})-\Phi(s,\xi_{2})\right|\leq \widehat{M}_{I_{1}I_{2}}\left|
    \xi_{2}^2-\xi_{1}^2\right|.
  \end{aligned}
\end{align*}
Therefore, \(\forall (s_{1},\xi_{1})\) and \((s_{2},\xi_{2})\) \(\in I_{1}\times I_{2}\), we have 
\begin{align*}
  \begin{aligned}
    \left|\Phi(s_{1},\xi_{1})-\Phi(s_{2},\xi_{2})\right|
    &=\left|\Phi(s_{1},\xi_{1})-\Phi(s_{1},\xi_{2})+\Phi(s_{1},\xi_{2})-\Phi(s_{2},\xi_{2})\right|
    \\
    &\leq \max\{\widetilde{M}_{I_{1}I_{2}},\widehat{M}_{I_{1}I_{2}}\}\left[\left|
    \xi_{1}^2-\xi_{2}^2\right|
    +\left|s_{1}-s_{2}\right|\right],
  \end{aligned}
\end{align*}
which completes the proof of the proposition. 
\end{proof}
In order to show that \(-s^2=\Phi(s)\) has a solution in \((0,+\infty)\), we have to verify that 
\(\Phi(s)<0\) for some \(s>0\). Let \(E(s,W)=\mu s E_{2}(W)+E_{3}(W)<0\). Thus, we consider the following 
extreme value problem 
\begin{align*}
  s<\sup\limits_{W\in H^2(I)\cap H_{0}^{1}(I)}\frac{-E_{3}(W)}{\mu E_{2}(W)}
  =\sup\limits_{W\in \mathcal{B}_{\xi}}-E_{3}(W),
\end{align*}
where \(\mathcal{B}_{\xi}=\left\{W\in H^2(I)\cap H_{0}^{1}(I)\big{|}\mu E_{2}(W)=1~\text{and~}
DW\left(R_{2}\right)=0\right\}\).
\begin{proposition}\label{jidazhi0410}
 Under the conditions of 
\eqref{zizhu0410}, \(\mu-\alpha R_{1}\geq 0\)
 and \(\xi\neq 0\), \(-E_{3}(W)\) achieves its maximum \(\lambda_{c}(\xi)\) 
on \(\mathcal{B}_{\xi}\). Furthermore, \(\lambda_{c}(\xi)\)\(\in\)
\(\left(0,\frac{g\xi^2R_{2}^2\|D\overline{\rho}\|_{L^{\infty}(I)}}
{\mu\left[R_{1}R_{2}\left(R_{2}-R_{1}\right)^2\left(2\xi^2+1\right)
+\left(\xi^2-1\right)^2\right]}\right)\). 
\end{proposition}
\begin{proof}
   Under the condition of 
   \(\mu-\alpha R_{1}\geq 0\), \(E_{2}(W)>0\). 
In addition, due to the condition \eqref{zizhu0410}, there exists \(\sigma>0\) such that
  \(D\overline{\rho}(r)>0,~\) 
  \(\forall r\in B\left(r_{s},\sigma\right)\subset I\). We choose \(\widetilde{W}\in C_{0}^{\infty}
  \left(B\left(r_{s},\sigma\right)\right)\). Thus, \(-E_{3}\left(\widetilde{W}\right)>0\) indicating 
  \(\sup\limits_{W\in \mathcal{B}_{\xi}}-E_{3}(W)>0\).

  Since 
  \begin{align}\label{yongheng0410}
    \left|E_{3}(W)\right|<
    \frac{g\xi^2R_{2}^2\|D\overline{\rho}\|_{L^{\infty}(I)}}
{\mu\left[R_{1}R_{2}\left(R_{2}-R_{1}\right)^2\left(2\xi^2+1\right)
+\left(\xi^2-1\right)^2\right]}<+\infty
  \end{align}
  for any \(W\in\mathcal{B}_{\xi}\), then there exists a sequence \(\left\{\widehat{W}^{n}\right\}\subset 
  \mathcal{B}_{\xi}\) such that 
  \[
  \lim\limits_{n\rightarrow+\infty}-E_{3}\left(\widehat{W}^{n}\right)=\sup\limits_{W\in\mathcal{B}_{\xi}}-E_{3}(W),  
  \]
  and there exists a \(\widehat{W}^{0}\in H^{2}(I)\cap H_{0}^{1}(I)\) such that 
  \(\widehat{W}^{n}\rightarrow \widehat{W}^{0}\) weakly in \(H^{2}(I)\cap H_{0}^{1}(I)\) and strongly 
  in \(H_{0}^{1}(I)\). Thus, by the strong convergence, one has 
  \(-E_{3}(\widehat{W}^{0})=\lim\limits_{n\rightarrow+\infty}-E_{3}(\widehat{W}^{n})
  =\sup\limits_{W\in\mathcal{B}_{\xi}}-E_{3}(W)\).

  Now, we verify that \(\widehat{W}^{0}\in \mathcal{B}_{\xi}\). Obviously, \(E_{2}(W)\) is weakly lower 
  semi-continuous, thus, \(E_{2}\left(\widehat{W}^{0}\right)\leq 1\). Assume that 
  \(E_{2}\left(\widehat{W}^{0}\right)=c^2<1\). Then \(E_{2}\left(\frac{\widehat{W}^{0}}{c}\right)=1\) which means 
  \(\frac{\widehat{W}^{0}}{c}\in \mathcal{B}_{\xi}\). However, 
  \(-E_{3}\left(\frac{\widehat{W}^{0}}{c}\right)=\frac{-E_{3}\left(\widehat{W}^{0}\right)}{c^2}> 
  \sup\limits_{W\in\mathcal{B}_{\xi}}-E_{3}(W)\) which contradicts the maximum on \(\mathcal{B}_{\xi}\). Thus, 
  \(c=1\) implying that \(\widehat{W}^{0}\in\mathcal{B}_{\xi}\). We show that 
  \(-E_{3}(W)\) achieves its maximum \(\lambda_{c}(\xi)\) on \(\mathcal{B}_{\xi}\) and \(\lambda_{c}(\xi)\in 
  \left(0,\frac{g\xi^2R_{2}^2\|D\overline{\rho}\|_{L^{\infty}(I)}}
  {\mu\left[R_{1}R_{2}\left(R_{2}-R_{1}\right)^2\left(2\xi^2+1\right)
  +\left(\xi^2-1\right)^2\right]}\right)\).
   \end{proof}
   \begin{remark}\label{taochulai0410}
From \eqref{yongheng0410}, one can conclude that \(\lambda_{c}(\xi)\rightarrow 0\) as \(\xi\rightarrow \infty\) or \(\xi\rightarrow 0\) 
or \(\mu\rightarrow +\infty\).
   \end{remark}
   \begin{proposition}\label{wupingshangheng0410}Under the conditions of Proposition \ref{jidazhi0410}, 
    \(\Phi(s)<0\) as \(0<s<\lambda_{c}(\xi)\); \(\Phi(s)\geq 0\) as \(s\geq \lambda_{c}(\xi)\).
   \end{proposition}
   \begin{proof}
    When \(0<s<\lambda_{c}(\xi)\), there exists a \(W\in H^{2}(I)\cap H_{0}^{1}(I)\) such that 
    \[
    \mu s E_{2}(W)+E_{3}(W)<0,  
    \]
    which means that \(\Phi(s)=\inf\limits_{W\in H^{2}(I)\cap H_{0}^{1}(I)}\)
    \(\frac{\mu sE_{2}\left(W\right)+E_{3}\left(W\right)}{E_{1}(W)}<0.\)

    When \(s\geq \lambda_{c}(\xi)\), then for any \(W\in H^{2}(I)\cap H_{0}^{1}(I)\) such that 
    \[
    \mu s E_{2}\left(W\right)+E_{3}\left(W\right)\geq 0,  
    \]
    which indicates that \(\Phi(s)=\inf\limits_{W\in H^{2}(I)\cap H_{0}^{1}(I)}\)
    \(\frac{\mu sE_{2}\left(W\right)+E_{3}\left(W\right)}{E_{1}(W)}>0.\)

    Furthermore, one can conclude that \(\Phi(\lambda_{c}(\xi))=0\) since \(\Phi(s)\) is 
    continuous on \(\left(0,+\infty\right)\), see Proposition \ref{junbulibuxizhi0410}.
   \end{proof}
   \begin{proposition}\label{shaonian0410}
    Under the conditions of Proposition \ref{jidazhi0410}, there exists a unique 
    \(\lambda_{0}\left(\xi\right)\in \left(0,\lambda_{c}\left(\xi\right)\right)\) satisfying 
    \(-\lambda_{0}^{2}\left(\xi\right)=\Phi\left(\lambda_{0}\left(\xi\right),\xi\right)\).
   \end{proposition}
   \begin{proof}
    Let \(f(s)=s^2+\Phi(s)\). From the monotoncity of \(\Phi(s)\), one can conclude that 
    \(f(s)\) is also increasing with \(s\) on \((0,+\infty)\). 

    From Proposition \ref{wupingshangheng0410}, we have 
    \[
    f\left(0\right)=\Phi(0)<0,~f\left(\lambda_{c}\left(\xi\right)\right)=
    \Phi(\lambda_{c}(\xi))+\lambda_{c}^2\left(\xi\right)=
    \lambda_{c}^2\left(\xi\right)>0.  
    \]
    Then from the intermediate value theorem of continuous function, there exists a unique 
    \(\lambda_{0}\left(\xi\right)\in\left(0,\lambda_{c}\left(\xi\right)\right)\) such that 
    \(f\left(\lambda_{0}\left(\xi\right)\right)=\Phi\left(\lambda_{0}\left(\xi\right)\right)
    +\lambda_{0}^{2}\left(\xi\right)=0,\)
    which yields that \(-\lambda_{0}^2\left(\xi\right)=\Phi\left(\lambda_{0}\left(\xi\right),\xi\right)\).
   \end{proof}
   Therefore, in view of Proposition \ref{kequdao0405}-\ref{shaonian0410}, for each 
   \(k\neq 0\), we can obtain a solution to problem \eqref{kanbian0405} with \(\lambda_{0}(k)>0\). 
   \begin{theorem}\label{jiecunzai0415}For each \(|k|\neq 0\), there exist \(W=W\left(k,r\right)\) and 
    \(\lambda_{0}\left(k\right)>0\) satisfying \eqref{kanbian0405}. Moreover, 
    \(W\in H^{q}\left(I\right)\) for any positive integer \(q\).
   \end{theorem}
\begin{proposition}\label{chuntiandeshandian0410}
    \(\lambda_{0}\left(\xi\right)\) is continuous on \((0,+\infty)\) and bounded, that is, 
    \begin{align}\label{zuidazhi0410}
      \Lambda:=\sup\limits_{\xi\in (0,+\infty)}\lambda_{0}\left(\xi\right)>0.
    \end{align}
    \begin{proof}
      Let \(f\left(s,\xi\right)=\Phi\left(s,\xi\right)+s^2\). Then 
      \(f\) is increasing with respect to \(s\) on \((0,+\infty)\) and continuous on 
      \(\left(0,+\infty\right)\times [1,+\infty)\). 
\(\forall\) \(\xi_{0}\in \left[1,+\infty\right)\) and \(\forall\) \(\varepsilon>0\),
      since \(f\left(\lambda_{0}\left(\xi_{0}\right),\xi_{0}\right)=0\) and \(f\) is increasing with 
      respect to s, then 
      \[
      f\left(\lambda_{0}\left(\xi_{0}\right)-\varepsilon,\xi_{0}\right)<0,~
      f\left(\lambda_{0}\left(\xi_{0}\right)+\varepsilon,\xi_{0}\right)>0. 
      \]
      Subsequently, according to the continuity of \(f\), there exists a \(\widetilde{\delta}>0\) such that 
      for any \(\xi\in O\left(\xi_{0},\widetilde{\delta}\right)\), we have 
      \[
        f\left(\lambda_{0}\left(\xi_{0}\right)-\varepsilon,\xi\right)<0,~
        f\left(\lambda_{0}\left(\xi_{0}\right)+\varepsilon,\xi\right)>0.
      \]
      Again, by virtue of monotoncity of \(f\) and \(f\left(\lambda_{0}\left(\xi\right),\xi\right)\)=0, 
      we have \(\lambda_{0}\left(\xi_{0}\right)-\varepsilon<\lambda_{0}\left(\xi\right)<\lambda_{0}\left(\xi_{0}\right)
      +\varepsilon\) which implies that \(\lambda_{0}\left(\xi\right)\) is continuous on \(\xi\).

      Since \(\lim\limits_{\xi\rightarrow+\infty}\lambda_{0}\left(\xi\right)=0\) and \(\lim\limits_{\xi\rightarrow 0}\lambda_{0}\left(\xi\right)=0\)(see Remark \ref{taochulai0410} and \(\lambda_{0}\left(\xi\right)\leq \lambda_{c}\left(\xi\right)\)), then one can conclude that 
      \(\lambda_{0}\left(\xi\right)\) is bounded on \(\left(0,+\infty\right)\). We denote 
      \begin{align*}
      \Lambda_{c}=\sup\limits_{\xi\in (0,+\infty)}\lambda_{0}\left(\xi\right)>0.  
      \end{align*}
    \end{proof}
   \end{proposition}
   \begin{remark}\label{lisanzuidazhi1101}
   From Remark \ref{taochulai0410} and the continuity of \(\lambda_{0}\left(\xi\right)\) on \(\xi\), one can conclude that 
   the \(\max\limits_{k\in\mathbf{N}}\left\{\lambda_{0}\left(k\right)\right\}\)
   can be achieved for some finite \(k\in\mathbf{N}\) Since \(\lim\limits_{\xi\rightarrow+\infty}\lambda_{0}\left(\xi\right)=0\) and \(\lim\limits_{\xi\rightarrow 0}\lambda_{0}\left(\xi\right)=0\). Thus, we define
   \begin{align*}
\widetilde{\Lambda}=\max\limits_{k\in\mathbf{N}}\left\{\lambda_{0}\left(k\right)\right\}.
   \end{align*}
   \end{remark}
   \subsection{Construction of a solution to system \eqref{bianlianfenli0403}}\label{gouzao1211}
   We have shown the existence of solution to the problem \eqref{kanbian0405}, see Theorem \ref{jiecunzai0415}. 
   Consequently, for the system \eqref{bianlianfenli0403}, we have the following conclusion. 
   \begin{theorem}\label{kuoda04015}Let \(k\neq 0\) be fixed. Then the equation \eqref{bianlianfenli0403} has 
    a solution \(\left(w_{1},w_{2},h_{1},h_{2}\right):=\)\((w_{1}(k,\)\(r),\)
    \(w_{2}\left(k,r\right), h_{1}\left(k,r\right),h_{2}\left(k,r\right))\),
    which belongs to \(H^{q}(I)\) for any positive integer \(q\) with \(\lambda=\lambda_{0}\left(k\right)>0\).
   \end{theorem}
  \begin{proof}
    With the help of Theorem \ref{jiecunzai0415}, we obtain \(w_{1}:=W\left(k,r\right)\in H^{q}\left(I\right)\) and 
    \(\lambda_{0}\left(k\right)\)
    solving problem \eqref{kanbian0405}. Then, we can take 
    \begin{align}\label{dongpo0415}
    w_{2}=-\frac{D\left(rW\right)}{k},~h_{1}=\frac{r}{k}\left[\frac{2\mu k W}{r^2}
    -\mu\left(D^2+\frac{D}{r}-\frac{k^2+1}{r^2}-\frac{\lambda_{0}\overline{\rho}}{\mu}\right)w_{2}\right],
    ~h_{2}=-\frac{WD\overline{\rho}}{\lambda_{0}}. 
    \end{align}
    It is easy to verify that \(\left(w_{1},w_{2},h_{1},h_{2}\right)\) satisfies the system 
    \eqref{bianlianfenli0403}.

    Besides, since \(w_{1}\) is smooth with respect to \(r\), then 
    \(w_{2}\), \(h_{1}\) and \(h_{2}\) are also smooth about \(r\).
  \end{proof}
  \begin{remark}\label{qiouxing0415} 
  At present, we can conclude the linear instability of \(\left(\mathbf{0},\overline{\rho}\right)\) to this system \eqref{xiaxing0403}-\eqref{wuming1211}. 
  Furthermore,
  from problem \eqref{kanbian0405}, one can verify 
    \(W\left(k,r\right)=W\left(-k,r\right)\). Then
  by virtue of \eqref{dongpo0415}, one has the following conclusions: when \(r\) is fixed,
  \begin{enumerate}
  \item\(w_{1}\), \(h_{1}\) and \(h_{2}\) are even on \(k\); 
  \item\(w_{2}\) is odd on \(k\).
  \end{enumerate}
  \end{remark}
  Next, we consider the boundedness of the solution constructed in Theorem \ref{kuoda04015}. 
  \begin{lemma}\label{youjiexing0415}
    Let \(|k|\in \left[N_{1},N_{2}\right]\), where \(N_{1}\) and \(N_{2}\) are two fixed positive integers, and 
    \(w_{1},w_{2},h_{1},h_{2}\) be constructed in Theorem \ref{kuoda04015}. Then there exist positive 
    constants \(C_{q}\) depending on \(q,~N_{1},~N_{2},~R_{1},\) \(R_{2},~\overline{\rho},~\mu\) such that 
    \begin{align}\label{buyaoquxiang04015}
    \left\|\left(w_{1},w_{2},h_{1},h_{2}\right)\right\|_{H^{q}(I)}\leq C_{q},~
    \text{for~any~integer~}q\geq 0.  
    \end{align}
    In addition, 
    \begin{align}\label{dayu004015}
      \|w_{1}\left(r,k\right)\|_{L^2\left(I\right)}>0.
    \end{align}
  \end{lemma} 
  \begin{proof}
    Since \(w_{1}=W\in \mathcal{A}_{k}=\left\{W\in H^{2}\left(I\right)\cap H_{0}^{1}\left(I\right)
    |E_{1}\left(W\right)=1,~D\left(W\right)|_{r=R_{2}}=0 \right\}\), 
    then \eqref{dayu004015} is valid. 
    Furthermore, from \eqref{kanbian0405}, one has 
    \begin{align*}
      \begin{aligned}
        D^{4}\left(rW\right)
        &=\frac{1}{r}\bigg{[}\left(D^2+\frac{D}{r}-\frac{k^2+1}{r^2}\right)k^2W+\frac{4k^2 W}{r}
        -2D^3\left(rW\right)
        \\
        &-\left(k^2+1\right)D\left(\frac{D\left(rW\right)}{r}\right)\bigg{]}+\frac{k^2 g W D\overline{\rho}}{\mu r\lambda_{0}\left(k\right)}
        -\frac{\lambda_{0}\left(k\right)}{\mu r}\left[\overline{\rho}k^2 W-
        D\left(r\overline{\rho}D\left(rW\right)\right)\right],
      \end{aligned}
    \end{align*}
    which together with interpolation inequality (see Theorem 1.11 in \cite{ma_2011}), one can prove that 
    \begin{align*}
      \left\|W\right\|_{H^4\left(I\right)}\leq C_{4}. 
    \end{align*}    
    Therefore, by virtue of \eqref{kanbian0405}, one can conclude that 
    \[
    \left\|W\right\|_{H^{q}\left(I\right)}\leq C_{q},~\text{for~any~positive~integer}~q.   
    \]
    Finally, by \eqref{dongpo0415}, one can prove \eqref{buyaoquxiang04015}.
  \end{proof}
   \subsection{Exponential growth rate}\label{butongzengzhang1211}
   For investigating the nonlinear instability in Lipschitz sense, we  construct the solution to \eqref{xiaxing0403} 
   by linear combination of solutions of equation \eqref{bianlianfenli0403}. 
   \begin{theorem}\label{endian04015} Let \(\mathbf{a}^{N}=\left(a_{j},a_{j+1},\cdots,a_{j+N-1}\right)\in 
    \mathbf{R}^{N}\), where \(j\) and \(N\) are two fixed positive integers. Define the functions 
    \begin{align}\label{tingjiang04015}
      \begin{aligned}
        &v_{r}^{N}\left(r,\theta,t\right)=\sum\limits_{k=j}^{j+N-1}a_{k} w_{1}(r,k)e^{\lambda_{0}\left(k\right)t}\cos{k\theta} ,
        v_{\theta}^{N}\left(r,\theta,t\right)=\sum\limits_{k=j}^{j+N-1}a_{k} w_{2}(r,k)e^{\lambda_{0}\left(k\right)t}\sin{k\theta}, 
        \\
        &p^{N}\left(r,\theta,t\right)=\sum\limits_{k=j}^{j+N-1}a_{k} h_{1}(r,k)e^{\lambda_{0}\left(k\right)t}\cos{k\theta} ,
        \rho^{N}\left(r,\theta,t\right)=\sum\limits_{k=j}^{j+N-1}a_{k} h_{2}(r,k)e^{\lambda_{0}\left(k\right)t}\cos{k\theta},
      \end{aligned}
    \end{align}
    then we have the following conclusions, 
    \begin{enumerate}
      \item \(\left(v_{r}^{N},v_{\theta}^{N},p^{N},\rho^{N}\right)\) is a solution to linearized problem 
      \eqref{xiaxing0403} with boundary conditions \eqref{huajuan0403}.
      \item Due to the boundedness of \(\left(w_{1},w_{2},h_{1},h_{2}\right)\), there exist positive constants 
      \(C_{q}\) depending on \(\mathbf{a}^{N}\), \(j\), \(N\), \(q\), \(R_{1}\), \(R_{2}\), \(\overline{\rho}\) 
      and \(\mu\) such that 
      \[
      \left\|\left(v_{r}^{N},v_{\theta}^{N},p^{N},\rho^{N}\right)\left(0\right)\right\|_{H^{q}\left(I\times\left[0,2\pi\right]\right)}
      \leq C_{q},~\text{for~any~positive~integer}~q.   
      \]
      \item For every \(t>0\), the boundedness of \(\lambda_{0}(k)\) over \([1,+\infty)\) implies that the solution 
      \eqref{tingjiang04015} is smooth and satisfies 
      \begin{align}\label{wunian04015}
        \begin{aligned}
      e^{\lambda_{0,j,N}t}\left\|\left(v_{r}^{N},v_{\theta}^{N},p^{N},\rho^{N}\right)\left(0\right)
      \right\|&_{H^{q}\left(I\times [0,2\pi]\right)} 
      \leq \left\|\left(v_{r}^{N},v_{\theta}^{N},p^{N},\rho^{N}\right)\right\|_{H^{q}\left(I\times [0,2\pi]\right)} 
      \\
      &\leq  e^{\widetilde{\Lambda} t}\left\|\left(v_{r}^{N},v_{\theta}^{N},p^{N},\rho^{N}\right)\left(0\right)
      \right\|_{H^{q}\left(I\times [0,2\pi]\right)},
        \end{aligned} 
      \end{align}
      where \(\lambda_{0,j,N}:=\inf\limits_{k\in\{j,j+1,\cdots,j+N-1\}}\left\{\lambda_{0}(k)\right\}\) and 
      \(\widetilde{\Lambda}>0\) is defined in Remark \ref{lisanzuidazhi1101}.
      \item If \(\mathbf{a}^{N}\neq \mathbf{0}\), then one can choose \(N\) such that
      \begin{align}\label{dayulingba04015}
        \left\|v_{r}^{N}\left(0\right)\right\|_{H^{q}\left(I\times \left[0,2\pi\right]\right)}>0. 
      \end{align}
      \item We can choose \(j\) and \(N\) such that 
      \begin{align}\label{suiyuan04015} 
        \lambda_{0,j,N}\in \left[\frac{
        \widetilde{\Lambda}}{2},\frac{3\widetilde{\Lambda}}{4}\right].
      \end{align} 
    \end{enumerate}
   \end{theorem}
\begin{proof}
  From Remark \ref{qiouxing0415}, one can conclude that for any \(k\neq 0\), 
  \begin{align*}
    \begin{aligned}
      &\widetilde{v}_{r}^{k}=w_{1}e^{ik\theta}e^{\lambda_{0}\left(k\right)t}
      +w_{1}e^{-ik\theta}e^{\lambda_{0}\left(k\right)t}=2w_{1}e^{\lambda_{0}\left(k\right)t}\cos{k\theta},
      \\
      &\widetilde{v}_{\theta}^{k}=-iw_{2}e^{ik\theta}e^{\lambda_{0}\left(k\right)t}
      +iw_{2}e^{-ik\theta}e^{\lambda_{0}\left(k\right)t}=2w_{2}e^{\lambda_{0}\left(k\right)t}\sin{k\theta},
      \\
      &\widetilde{p}^{k}=h_{1}e^{ik\theta}e^{\lambda_{0}\left(k\right)t}
      +h_{1}e^{-ik\theta}e^{\lambda_{0}\left(k\right)t}=2h_{1}e^{\lambda_{0}\left(k\right)t}\cos{k\theta},
      \\
      &\widetilde{\rho}^{k}=h_{2}e^{ik\theta}e^{\lambda_{0}\left(k\right)t}
      +h_{2}e^{-ik\theta}e^{\lambda_{0}\left(k\right)t}=2h_{2}e^{\lambda_{0}\left(k\right)t}\cos{k\theta},
    \end{aligned}
  \end{align*}
  is a solution to linearized problem 
  \eqref{xiaxing0403} with boundary conditions \eqref{bianjie0329}. Furthermore, one can 
  show that \eqref{tingjiang04015} is also a solution to  \eqref{xiaxing0403} 
  with boundary conditions \eqref{bianjie0329} by principle of superposition. Then (1) is valid. 

 From the boundedness of \(w_{1},w_{2},h_{1},h_{2}\), see Lemma \ref{youjiexing0415}, one can verify 
 (2).

 By virtue of the boundedness of \(\lambda_{0}\left(k\right)\) (see Remark \ref{lisanzuidazhi1101}), 
 the (3) is valid. 

 (4) follows from \(\left\|w_{1}\right\|_{L^{2}\left(I\right)}>0\). 

 According to the Remark \ref{lisanzuidazhi1101}, one can prove (5).
\end{proof}

\section{Nonlinear instability in Lipschitz sense}\label{lipschitz1211}
This section introduces the nonlinear instability in Lipschitz sense. 
Assume that \(\left(\widetilde{\mathbf{V}}_{0},\widetilde{\rho}_{0}\right)\) is an initial value to nonlinear problem \eqref{raodongfeixianxing0403}, where \(\left(\widetilde{\mathbf{V}}_{0},\widetilde{\rho}_{0}\right)\in\)
\(\left(H^{2}\left(\widetilde{\Omega}\right)\right)^{2}\times \left[H^{1}\left(\widetilde{\Omega}\right)\cap L^{\infty}\left(\widetilde{\Omega}\right)\right]\). In addition, \(\inf\limits_{\left(r,\theta\right)\in\widetilde{\Omega}}
  \{\widetilde{\rho}_{0}+\overline{\rho}\}
  >\delta>0\), where \(\delta\) is a constant fixed.

\subsection{Uniqueness of solution to linearized equation}
This subsection gives the uniqueness of weak solution to the linearized problem \eqref{jubucunzai0605}-\eqref{weiyixing0615}, which is crucial in the proof of Theorem \ref{lipuxizhi1104}.
\begin{lemma}\label{manmanlai0627}
Assume that \(\left(\widetilde{\mathbf{V}},\widetilde{\nabla}\widetilde{p},\widetilde{\rho}\right)\) 
is the solution to the 
linearized equation \eqref{xiaxing0403} with the initial value \(\left(\widetilde{\mathbf{V}}, 
\widetilde{\rho}\right)\left(0\right)=\left(\mathbf{0},0\right)\). Then 
\(\left(\widetilde{\mathbf{V}},\widetilde{\nabla}\widetilde{p},\widetilde{\rho}\right)=
\left(\mathbf{0},\mathbf{0},0\right)\) for any \(t>0\).
\end{lemma}
\begin{proof}
  Multiplying \(\eqref{xiaxing0403}_{1}\) and \(\eqref{xiaxing0403}_{2}\) by \(r\widetilde{v}_{r}\) 
  and \(r\widetilde{v}_{\theta}\), respectively, then adding up the results and integrating over 
  \(\widetilde{\Omega}\) give that 
  \begin{align}\label{jiandanyidian0627}
    \begin{aligned}
    \frac{1}{2}&\frac{d}{dt}\int_{\widetilde{\Omega}}r\overline{\rho}\left|\widetilde{\mathbf{V}}\right|^2drd\theta 
    +\mu\int_{\widetilde{\Omega}}\left[
      r\left|\partial_{r}\widetilde{\mathbf{V}}\right|^2
      +\frac{1}{r}\left(\left|\partial_{\theta}\widetilde{v}_{\theta}+\widetilde{v}_{r}\right|^2
      +\left|\partial_{\theta}\widetilde{v}_{r}-\widetilde{v}_{\theta}\right|^2\right)
    \right]drd\theta
    \\
    &+\int_{0}^{2\pi}\left(\mu-\alpha r\right)\left|\widetilde{v}_{\theta}\right|^2|_{r=R_{1}}d\theta 
    =-g\int_{\widetilde{\Omega}}r\widetilde{\rho}\widetilde{v}_{r}drd\theta
    \leq \varepsilon\left\|\widetilde{\mathbf{V}}\right\|_{L^2\left(\widetilde{\Omega}\right)}^2
    +C\left\|\sqrt{r}\widetilde{\rho}\right\|_{L^2\left(\widetilde{\Omega}\right)}^2,
    \end{aligned}
  \end{align}
  where the Cauchy inequality is used. 

  Via multiplying \(\eqref{xiaxing0403}_{3}\) by \(r\widetilde{\rho}\), integrating over 
  \(\widetilde{\Omega}\) and utilizing the Cauchy inequality, we obtain 
  \begin{align}\label{bucai0627}
    \begin{aligned}
      \frac{1}{2}\frac{d}{dt}\left\|\sqrt{r}\widetilde{\rho}\right\|_{L^2\left(\widetilde{\Omega}\right)}^2
      =-\int_{\widetilde{\Omega}}r\widetilde{v}_{r}D\overline{\rho}\widetilde{\rho}drd\theta
      \leq \varepsilon\left\|\widetilde{\mathbf{V}}\right\|_{L^2\left(\widetilde{\Omega}\right)}^2 
      +C\left\|D\overline{\rho}\right\|_{L^{\infty}\left(\widetilde{\Omega}\right)}^2
      \left\|\sqrt{r}\widetilde{\rho}\right\|_{L^2\left(\widetilde{\Omega}\right)}^2.
    \end{aligned}
  \end{align}
 Add up \eqref{jiandanyidian0627} and \eqref{bucai0627}, then from Lemmas \ref{poincarebudengshi0517} 
 and \ref{dengjia0523} with \(\varepsilon\) small enough and \(\mu-\alpha R_{1}\geq 0\), we have 
 \begin{align*}
  \begin{aligned}
    \frac{1}{2}\frac{d}{dt}\left[ 
      \int_{\widetilde{\Omega}}r\overline{\rho}\left|\widetilde{\mathbf{V}}\right|^2drd\theta 
      +\left\|\sqrt{r}\widetilde{\rho}\right\|_{L^2\left(\widetilde{\Omega}\right)}^2
    \right]\leq C\left(\left\|\sqrt{r}\widetilde{\rho}\right\|_{L^2\left(\widetilde{\Omega}\right)}^2
    +\int_{\widetilde{\Omega}}r\overline{\rho}\left|\widetilde{\mathbf{V}}\right|^2drd\theta\right),
  \end{aligned}
 \end{align*} 
 which together with Gronwall's inequality and zero initial value yields that 
\[
\int_{\widetilde{\Omega}}r\overline{\rho}\left|\widetilde{\mathbf{V}}\right|^2drd\theta
+\left\|\sqrt{r}\widetilde{\rho}\right\|_{L^2\left(\widetilde{\Omega}\right)}^2=0.  
\]
Therefore, we complete the proof.
\end{proof}

\subsection{Nonlinear estimates}
This subsection investigates some energy estimates for the perturbed problem, which
ensures the existence of the limit of a family of solution to the nonlinear problem \eqref{raodongfeixianxing0403}.
\begin{proposition}\label{mingti0627} 
  There exists a \(\widetilde{\delta}>0\) small enough such that when 
   \(\left\|\left(\widetilde{\mathbf{V}}_{0},\widetilde{\rho}_{0}\right)\right\|_{H^2\left(\widetilde{\Omega}\right)}^2=\delta_{0}^2<\widetilde{\delta}^2<1\), 
  \(\widetilde{\rho}_{0}\in L^{\infty}\left(\widetilde{\Omega}\right)\) and 
  \(\inf\limits_{\left(r,\theta\right)\in\widehat{\Omega}}
  \{\widetilde{\rho}_{0}+\overline{\rho}\}
  >\delta>0\), where
  \(\delta_{0}\) is a positive
   constant, then 
  the strong solution \(\left(\widetilde{\mathbf{V}},\widetilde{p},\widetilde{\rho}\right)\) to 
  the nonlinear equation \eqref{raodongfeixianxing0403} with the initial value 
  \(\left(\widetilde{\mathbf{V}}_{0},\widetilde{\rho}_{0}\right)\) satisfies the following 
  estimate
  \begin{align*}
    \sup\limits_{t\in\left[0,T\right]}
    &\left[ 
      \left\|\widetilde{\mathbf{V}}\right\|_{H^2\left(\widetilde{\Omega}\right)}^2
      +\left\|\widetilde{\rho}\right\|_{H^1\left(\widetilde{\Omega}\right)}^2 
      +\left\|\widetilde{\nabla}\widetilde{p}\right\|_{L^2\left(\widetilde{\Omega}\right)}^2
      +\left\|\partial_{t}\widetilde{\mathbf{V}}\right\|_{L^2\left(\widetilde{\Omega}\right)}^2 
      +\left\|\partial_{t}\widetilde{\rho}\right\|_{L^2\left(\widetilde{\Omega}\right)}^2
    \right]
    \\
    &+\int_{0}^{t}\left[
      \left\|\partial_{s}\widetilde{\mathbf{V}}\right\|_{H^{1}\left(\widetilde{\Omega}\right)}^2 
      +\left\|\widetilde{\nabla}\widetilde{\mathbf{V}}\right\|_{H^{1}\left(\widetilde{\Omega}\right)}^2\right]ds 
      \leq C\delta_{0}^2,
  \end{align*}
  where \(C=C\left(T,R_{1},R_{2},\mu,\overline{\rho},\left\|\widetilde{\rho}_{0}
  \right\|_{L^{\infty}\left(\widetilde{\Omega}\right)},\alpha\right)>0\) and \(T<T_{max}\) (\(T_{max}\) is 
  the maximum existence time of solution).  
\end{proposition}
\begin{proof}
  The proof is divided into several parts. 

  \textbf{1}, estimate of \(\left\|\widetilde{\rho}\right\|_{L^{q}\left(\widetilde{\Omega}\right)}\) for 
  any \(q>1\) and 
  \(\left\|\widetilde{\nabla}\widetilde{\mathbf{V}}\right\|_{L^2\left(\widetilde{\Omega}\right)}\).

  Multiplying \(\eqref{raodongfeixianxing0403}_{3}\) by 
  \(r\left(\widetilde{\rho}+\overline{\rho}\right)^{q-1}\) 
  and integrating over \(\widetilde{\Omega}\) with \(\partial_{r}\left(r\widetilde{v}_{r}\right)
  +\partial_{\theta}\widetilde{v}_{\theta}=0\) give that 
  \begin{align*}
    \frac{d}{dt}\int_{\widetilde{\Omega}}r\left(\widetilde{\rho}+\overline{\rho}\right)^{q}drd\theta 
    =0,
  \end{align*}
  which indicates that 
  \[
  \left\|r^{\frac{1}{q}}\left(\widetilde{\rho}+\overline{\rho}\right)
  \right\|_{L^{q}\left(\widetilde{\Omega}\right)}
  =  \left\|r^{\frac{1}{q}}\left(\widetilde{\rho}_{0}+\overline{\rho}\right)
  \right\|_{L^{q}\left(\widetilde{\Omega}\right)}.
  \]
  Thus, we have 
  \begin{align*}
    \left\|\widetilde{\rho}\right\|_{L^{q}\left(\widetilde{\Omega}\right)}
    \leq \left(\frac{R_{2}}{R_{1}}\right)^{\frac{1}{q}}
    \left(\left\|\widetilde{\rho}_{0}\right\|_{L^{q}\left(\widetilde{\Omega}\right)}
    +\left\|\overline{\rho}\right\|_{L^{q}\left(\widetilde{\Omega}\right)}\right)
    +\left\|\overline{\rho}\right\|_{L^{q}\left(\widetilde{\Omega}\right)}.
  \end{align*}
  In particular, let \(q\rightarrow+\infty\), we obtain 
  \begin{align}\label{guji20627}
    \left\|\widetilde{\rho}\right\|_{L^{\infty}\left(\widetilde{\Omega}\right)}
    \leq \left\|\widetilde{\rho}_{0}\right\|_{L^{\infty}\left(\widetilde{\Omega}\right)}
    +2\left\|\overline{\rho}\right\|_{L^{\infty}\left(\widetilde{\Omega}\right)},~
    \delta<\left\|\widetilde{\rho}+\overline{\rho}\right\|_{L^{\infty}\left(\widetilde{\Omega}\right)}
    \leq \left\|\widetilde{\rho}_{0}\right\|_{L^{\infty}\left(\widetilde{\Omega}\right)}
    +3\left\|\overline{\rho}\right\|_{L^{\infty}\left(\widetilde{\Omega}\right)}.
  \end{align}

  Multiplying \(\eqref{raodongfeixianxing0403}_{1}\) and \(\eqref{raodongfeixianxing0403}_{2}\) 
  by \(r\partial_{t}\widetilde{v}_{r}\) and \(r\partial_{t}\widetilde{v}_{\theta}\), respectively, 
  then integrating over \(\widetilde{\Omega}\) give that 
  \begin{align}\label{ziranshengfa0627}
    \begin{aligned}
      \int_{\widetilde{\Omega}}r\left(\widetilde{\rho}+\overline{\rho}\right)
      \left|\partial_{t}\widetilde{\mathbf{V}}\right|^2drd\theta+ 
      &\frac{\mu}{2}\frac{d}{dt}\int_{\widetilde{\Omega}}\left[
        r\left|\partial_{r}\widetilde{\mathbf{V}}\right|^2
        +\frac{1}{r}\left(\left|\partial_{\theta}\widetilde{v}_{\theta}
        +\widetilde{v}_{r}\right|^2
        +\left|\partial_{\theta}\widetilde{v}_{r}-\widetilde{v}_{\theta}\right|^2\right)
      \right]drd\theta
      \\
      &+\frac{1}{2}\frac{d}{dt}\int_{0}^{2\pi}\left(\mu-\alpha r\right)
      \widetilde{v}_{\theta}^2|_{r=R_{1}}d\theta :=K_{1},
    \end{aligned}
  \end{align}
  where  
 \begin{align}\label{K1suojian1103}
    \begin{aligned}
     K_{1}&=-\int_{\widetilde{\Omega}}r\left(\widetilde{\rho}+\overline{\rho}\right)
     \left(\widetilde{v}_{r}\partial_{r}\widetilde{\mathbf{V}}+\frac{\widetilde{v}_{\theta}}{r}
     \partial_{\theta}\widetilde{\mathbf{V}}\right)\cdot\partial_{t}\widetilde{\mathbf{V}}drd\theta 
     \\
     &+\int_{\widetilde{\Omega}}\left(\widetilde{\rho}+\overline{\rho}\right)
     \left(\widetilde{v}_{\theta}^2\partial_{t}\widetilde{v}_{r}-\widetilde{v}_{r}\widetilde{v}_{\theta} 
     \partial_{t}\widetilde{v}_{\theta}\right)drd\theta 
     -\int_{\widetilde{\Omega}}rg\widetilde{\rho}\partial_{t}\widetilde{v}_{r}drd\theta,
    \end{aligned}
  \end{align}
  which from Cauchy inequality and Lemmas \ref{poincarebudengshi0517} and \ref{zuidazhi0517} deduces that 
  \begin{align}\label{zaiyici0627}
    \begin{aligned}
      K_{1}\leq 
      \frac{3}{4}\int_{\widetilde{\Omega}}r\left(\widetilde{\rho}+\overline{\rho}\right)
      \left|\partial_{t}\widetilde{\mathbf{V}}\right|^2drd\theta
      +C\left\|\widetilde{\nabla}\widetilde{\mathbf{V}}\right\|
      _{L^2\left(\widetilde{\Omega}\right)}^{6}+\varepsilon
      \left\|\widetilde{\nabla}^2\widetilde{\mathbf{V}}\right\|_{L^2\left(\widetilde{\Omega}\right)}^2+C,
    \end{aligned}
  \end{align}
  where \(C=C\left(R_{1},R_{2},\delta,\left\|\widetilde{\rho}_{0}\right\|
  _{L^{\infty}\left(\widetilde{\Omega}\right)},\overline{\rho}\right)\).

  In addition, we have 
  \begin{align}\label{xuyaoduociyongdao0628}
    \begin{aligned}
  -\mu
  \begin{pmatrix}
  \Delta_{r}\widetilde{v}_{r}-\frac{\widetilde{v}_{r}}{r^2}-\frac{2}{r^2}\partial_{\theta}\widetilde{v}_{\theta}
    \\
  \Delta_{r}\widetilde{v}_{\theta}-\frac{\widetilde{v}_{\theta}}{r^2}
  +\frac{2}{r^2}\partial_{\theta}\widetilde{v}_{r}
  \end{pmatrix}
  +\begin{pmatrix}
    \partial_{r}\widetilde{p}
    \\
    \frac{1}{r}\partial_{\theta}\widetilde{p}
  \end{pmatrix}
  &=-\left(\widetilde{\rho}+\overline{\rho}\right)\partial_{t}\widetilde{\mathbf{V}}
  -\left(\widetilde{\rho}+\overline{\rho}\right)
  \left(\widetilde{v}_{r}\partial_{r}\widetilde{\mathbf{V}}+
  \frac{\widetilde{v}_{\theta}}{r}\partial_{\theta}\widetilde{\mathbf{V}}\right)
  \\
  &+\frac{\widetilde{\rho}+\overline{\rho}}{r}
  \begin{pmatrix}
    \widetilde{v}_{\theta}^2 
    \\
    -\widetilde{v}_{r}\widetilde{v}_{\theta}
  \end{pmatrix}-
  g\begin{pmatrix}
      \widetilde{\rho}+\overline{\rho}
      \\
      0
  \end{pmatrix}:=\Upsilon .
\end{aligned}  
\end{align}
Thus, from Lemma \ref{nanrenyaoyoubaqi0528}, we can conclude that 
\begin{align}\label{gaicuo0627}
  \begin{aligned}
\left\|\widetilde{\nabla}^2\widetilde{\mathbf{V}}\right\|_{L^2\left(\widetilde{\Omega}\right)}^2
+\left\|\widetilde{\nabla}\widetilde{p}\right\|_{L^2\left(\widetilde{\Omega}\right)}^2 
\leq \left\|\Upsilon \right\|_{L^2\left(\widetilde{\Omega}\right)}^2.
  \end{aligned}
\end{align}
From Lemmas \ref{poincarebudengshi0517}, \ref{yuanlaihaishao0619}, \ref{tiduL4guji0619}, \ref{zuidazhi0517}, 
\eqref{guji20627}
and Cauchy inequality, we obtain
\[
  \left\|\widetilde{\nabla}^2\widetilde{\mathbf{V}}\right\|_{L^2\left(\widetilde{\Omega}\right)}^2
  \leq C\int_{\widetilde{\Omega}}r\left(\widetilde{\rho}+\overline{\rho}\right)\left|
  \partial_{t}\widetilde{\mathbf{V}}\right|^2drd\theta
  +\frac{3}{4}\left\|\widetilde{\nabla}^2\widetilde{\mathbf{V}}\right\|_{L^2\left(\widetilde{\Omega}\right)}^2
  +C\left\|\widetilde{\nabla}\widetilde{\mathbf{V}}\right\|_{L^2\left(\widetilde{\Omega}\right)}^{6}+C,
\]
where \(C=C\left(R_{1},R_{2},\left\|\widetilde{\rho}_{0}\right\|_{L^\infty\left(\widetilde{\Omega}\right)}
,\overline{\rho}\right)\). 
Thus, 
\begin{align}\label{xuyao0627}
  \left\|\widetilde{\nabla}^2\widetilde{\mathbf{V}}\right\|_{L^2\left(\widetilde{\Omega}\right)}^2
+\left\|\widetilde{\nabla}\widetilde{p}\right\|_{L^2\left(\widetilde{\Omega}\right)}^2 
\leq C\int_{\widetilde{\Omega}}r\left(\widetilde{\rho}+\overline{\rho}\right)\left|
\partial_{t}\widetilde{\mathbf{V}}\right|^2drd\theta
+C\left\|\widetilde{\nabla}\widetilde{\mathbf{V}}\right\|_{L^2\left(\widetilde{\Omega}\right)}^{6}+C.
\end{align}
Make use of \eqref{ziranshengfa0627}, \eqref{zaiyici0627} and \eqref{xuyao0627}, we obtain 
\begin{align*}
  \frac{d}{dt}&\int_{\widetilde{\Omega}}\left[
        r\left|\partial_{r}\widetilde{\mathbf{V}}\right|^2
        +\frac{1}{r}\left(\left|\partial_{\theta}\widetilde{v}_{\theta}
        +\widetilde{v}_{r}\right|^2
        +\left|\partial_{\theta}\widetilde{v}_{r}-\widetilde{v}_{\theta}\right|^2\right)
      \right]drd\theta
\\
&+\frac{d}{dt}\int_{0}^{2\pi}\left(\mu-\alpha r\right)
      \widetilde{v}_{\theta}^2|_{r=R_{1}}d\theta
      \leq C\left\|\widetilde{\nabla}\widetilde{\mathbf{V}}\right\|_{L^2\left(\widetilde{\Omega}\right)}^{6}+C.
\end{align*}
Then integrate the above inequality from \(0\) to \(t\) and from the trace theorem, we have 
\[
\left\|\widetilde{\nabla}\widetilde{\mathbf{V}}\right\|_{L^2\left(\widetilde{\Omega}\right)}^2
\leq C\int_{0}^{t}\left\|
\widetilde{\nabla}\widetilde{\mathbf{V}}\right\|_{L^2\left(\widetilde{\Omega}\right)}^{6}drd\theta
+Ct+C_{1},  
\] 
where \(C=C\left(R_{1},R_{2},\left\|\widetilde{\rho}_{0}\right\|_{L^\infty\left(\widetilde{\Omega}\right)}
,\overline{\rho},\alpha,\mu\right)\) and \(C_{1}=C\left(R_{1},R_{2},\alpha,\mu\right)\delta_{0}\). Here, we require that \(\delta_{0}\) is small enough 
to assure the above inequality can be solved. 
From the same steps in Lemma \ref{budengshidezhengming0530}, we obtain 
\begin{align}\label{xuyaotayoujie0627}
  \left\|\widetilde{\nabla}\widetilde{\mathbf{V}}\right\|_{L^2\left(\widetilde{\Omega}\right)}^2\leq C,
\end{align}
where \(C=C\left(T,R_{1},R_{2},\left\|\widetilde{\rho}_{0}\right\|_{L^\infty\left(\widetilde{\Omega}\right)}
,\overline{\rho},\alpha,\mu\right)\).

\textbf{2}, estimate of \(\left\|\widetilde{\rho}\right\|_{L^2\left(\widetilde{\Omega}\right)}\), 
\(\left\|\widetilde{\mathbf{V}}\right\|_{L^2\left(\widetilde{\Omega}\right)}\),
\(\left\|\widetilde{\nabla}\widetilde{\mathbf{V}}\right\|_{L^2\left(\widetilde{\Omega}\right)}\), 
\(\int_{0}^{t}\left\|\widetilde{\nabla}\widetilde{\mathbf{V}}\right\|_{H^1\left(\widetilde{\Omega}\right)}^2ds\)
and \(\int_{0}^{t}\left\|\partial_{s}\widetilde{\mathbf{V}}\right\|_{L^2\left(\widetilde{\Omega}\right)}^2ds\). 

Multiplying \(\eqref{raodongfeixianxing0403}_{3}\) by \(r\widetilde{\rho}\) and integrating over 
\(\widetilde{\Omega}\) yield that
\begin{align}\label{chongchongluguo0627}
  \begin{aligned}
    \frac{1}{2}\frac{d}{dt}\int_{\widetilde{\Omega}}r\widetilde{\rho}^2drd\theta 
    =-\int_{\widetilde{\Omega}}r\widetilde{v}_{r}D\overline{\rho}\widetilde{\rho}drd\theta 
    \leq C\left[\left\|\widetilde{\mathbf{V}}\right\|_{L^2\left(\widetilde{\Omega}\right)}^2
    +\left\|\widetilde{\rho}\right\|_{L^2\left(\widetilde{\Omega}\right)}^2\right],
  \end{aligned}
\end{align}  
where \(C=C\left(R_{2},\overline{\rho}\right)\) and H\"{o}lder inequality is used.

Multiplying \(\eqref{raodongfeixianxing0403}_{1}\) and \(\eqref{raodongfeixianxing0403}_{2}\) by 
\(r\widetilde{v}_{r}\) and \(r\widetilde{v}_{\theta}\), respectively, then
adding up the results and integrating over \(\widetilde{\Omega}\)
give that 
\begin{align}\label{bujingyijian0627}
  \begin{aligned}
    \frac{1}{2}\frac{d}{dt}\int_{\widetilde{\Omega}}
    r\left(\widetilde{\rho}+\overline{\rho}\right)\left|\widetilde{\mathbf{V}}\right|^2drd\theta 
    &+\mu\int_{\widetilde{\Omega}}\left[
      r\left|\partial_{r}\widetilde{\mathbf{V}}\right|^2+\frac{1}{r}
      \left(\left|\partial_{\theta}\widetilde{v}_{\theta}+\widetilde{v}_{r}\right|^2
      +\left|\partial_{\theta}\widetilde{v}_{r}-\widetilde{v}_{\theta}\right|^2\right)
    \right]drd\theta
    \\
    &+\int_{0}^{2\pi}\left(\mu-r\alpha\right)\widetilde{v}_{\theta}^2|_{r=R_{1}}d\theta 
    :=K_{2},
  \end{aligned}
\end{align}
where 
\begin{align*}
  \begin{aligned}
    K_{2}=&
    -\int_{\widetilde{\Omega}}r\widetilde{\rho}g\widetilde{v}_{r}drd\theta.
  \end{aligned}
\end{align*}
Then from Lemmas \ref{poincarebudengshi0517}, \ref{yuanlaihaishao0619}, \ref{tiduL4guji0619}, \ref{zuidazhi0517}, 
\eqref{guji20627}, \eqref{xuyaotayoujie0627} and Cauchy inequality, we have 
\begin{align}\label{xuyaodexuyao0627}
  \begin{aligned}
    K_{2}\leq \int_{\widetilde{\Omega}}r\left(\widetilde{\rho}+\overline{\rho}\right)
    \left|\widetilde{\mathbf{V}}\right|^2drd\theta
    +C\int_{\widetilde{\Omega}}r\widetilde{\rho}^2drd\theta, 
  \end{aligned}
\end{align}
where \(C=C\left(T,R_{1},R_{2},\left\|\widetilde{\rho}_{0}\right\|_{L^\infty\left(\widetilde{\Omega}\right)}
,\overline{\rho},\alpha,\mu\right)\).

Utilize Lemmas \ref{poincarebudengshi0517}, \ref{yuanlaihaishao0619}, \ref{zuidazhi0517}, 
\eqref{xuyaotayoujie0627} and Cauchy inequality, we have 
\begin{align}\label{k1zaiguji0627}
  \begin{aligned}
  &K_{1}\leq \frac{3}{4}\int_{\widetilde{\Omega}}r\left(\widetilde{\rho}+\overline{\rho}\right)
  \left|\partial_{t}\widetilde{\mathbf{V}}\right|^2drd\theta 
  +C\left(\int_{\widetilde{\Omega}}r\widetilde{\rho}^2drd\theta+
  \left\|\widetilde{\nabla}\widetilde{\mathbf{V}}\right\|_{L^2\left(\widetilde{\Omega}\right)}^2\right)
  +\varepsilon\left\|\widetilde{\nabla}^2\widetilde{\mathbf{V}}\right\|_{L^2\left(\widetilde{\Omega}\right)}^2,
  \\
  &\left\|\Upsilon\right\|_{L^2\left(\widetilde{\Omega}\right)}^2 
  \leq C\left(\int_{\widetilde{\Omega}}\left(\widetilde{\rho}+\overline{\rho}\right)
  \left|\partial_{t}\widetilde{\mathbf{V}}\right|^2drd\theta
  +\left\|\widetilde{\nabla}\widetilde{\mathbf{V}}\right\|_{L^2\left(\widetilde{\Omega}\right)}^2\right)
  +\frac{1}{2}\left\|\widetilde{\nabla}^2\widetilde{\mathbf{V}}\right\|_{L^2\left(\widetilde{\Omega}\right)}^2,
  \end{aligned}
\end{align}
where \(C=C\left(T,R_{1},R_{2},\left\|\widetilde{\rho}_{0}\right\|_{L^\infty\left(\widetilde{\Omega}\right)}
,\overline{\rho},\alpha,\mu\right)\), \(K_{1}\) and \(\Upsilon\) are defined as \eqref{K1suojian1103} and \eqref{xuyaoduociyongdao0628}, respectively.
Thus, from \eqref{gaicuo0627}, one can obtain
\begin{align}\label{zaiyicixuyao0627}
\left\|\widetilde{\nabla}^2\widetilde{\mathbf{V}}\right\|_{L^2\left(\widetilde{\Omega}\right)}^2 
  +\left\|\widetilde{\nabla}\widetilde{p}\right\|_{L^2\left(\widetilde{\Omega}\right)}^2 
  \leq C\left(\int_{\widetilde{\Omega}}r\left(\widetilde{\rho}+\overline{\rho}\right)
  \left|\partial_{t}\widetilde{\mathbf{V}}\right|^2drd\theta+
  \left\|\widetilde{\nabla}\widetilde{\mathbf{V}}\right\|_{L^2\left(\widetilde{\Omega}\right)}^2\right),
\end{align} 
where \(C=C\left(T,R_{1},R_{2},\left\|\widetilde{\rho}_{0}\right\|_{L^\infty\left(\widetilde{\Omega}\right)}
,\overline{\rho},\alpha,\mu\right)\).

  Therefore, from \eqref{ziranshengfa0627}, \eqref{chongchongluguo0627}, \eqref{bujingyijian0627}, 
  \eqref{xuyaodexuyao0627}, \(\eqref{k1zaiguji0627}_{1}\) and \eqref{zaiyicixuyao0627}, we have 
\begin{align*}
  \frac{dz\left(t\right)}{dt}\leq C z\left(t\right),
\end{align*}
where \(C=C\left(T,R_{1},R_{2},\left\|\widetilde{\rho}_{0}\right\|_{L^\infty\left(\widetilde{\Omega}\right)}
,\overline{\rho},\alpha,\mu\right)\) and 
\(z\left(t\right)=\int_{\widetilde{\Omega}}\left[r\widetilde{\rho}^2+
r\left(\widetilde{\rho}+\overline{\rho}\right)\left|\widetilde{\mathbf{V}}\right|^2\right]drd\theta
+\int_{\widetilde{\Omega}}\bigg{[}
  r\left|\partial_{r}\widetilde{\mathbf{V}}\right|^2
  +\frac{1}{r}\left(\left|\partial_{\theta}\widetilde{v}_{\theta}
  +\widetilde{v}_{r}\right|^2
  +\left|\partial_{\theta}\widetilde{v}_{r}-\widetilde{v}_{\theta}\right|^2\right)
\bigg{]}drd\theta
+\int_{0}^{2\pi}\left(\mu-\alpha r\right)
\widetilde{v}_{\theta}^2|_{r=R_{1}}d\theta\). Then from Gronwall's inequality, we have 
\begin{align*}
  \begin{aligned}
  \int_{\widetilde{\Omega}}&\left[r\widetilde{\rho}^2+
r\left(\widetilde{\rho}+\overline{\rho}\right)\left|\widetilde{\mathbf{V}}\right|^2\right]drd\theta
+\int_{\widetilde{\Omega}}\left[
  r\left|\partial_{r}\widetilde{\mathbf{V}}\right|^2
  +\frac{1}{r}\left(\left|\partial_{\theta}\widetilde{v}_{\theta}
  +\widetilde{v}_{r}\right|^2
  +\left|\partial_{\theta}\widetilde{v}_{r}-\widetilde{v}_{\theta}\right|^2\right)
\right]drd\theta
\\
&+\int_{0}^{2\pi}\left(\mu-\alpha r\right)
\widetilde{v}_{\theta}^2|_{r=R_{1}}d\theta\leq C\delta_{0}^2,
  \end{aligned}
\end{align*}
where \(C=C\left(T,R_{1},R_{2},\left\|\widetilde{\rho}_{0}\right\|_{L^\infty\left(\widetilde{\Omega}\right)}
,\overline{\rho},\alpha,\mu\right)\).
From Lemma \ref{dengjia0523}, we have 
\begin{align}\label{xuyaodede0627}
  \left\|\widetilde{\rho}\right\|_{L^2\left(\widetilde{\Omega}\right)}^2
  +\left\|\widetilde{\mathbf{V}}\right\|_{H^1\left(\widetilde{\Omega}\right)}^2\leq C\delta_{0}^2,
\end{align}
where \(C=C\left(T,R_{1},R_{2},\left\|\widetilde{\rho}_{0}\right\|_{L^\infty\left(\widetilde{\Omega}\right)}
,\overline{\rho},\alpha,\mu\right)\). 

Furthermore, from \eqref{ziranshengfa0627}, \eqref{chongchongluguo0627}, \eqref{bujingyijian0627}, 
\eqref{xuyaodexuyao0627}, \(\eqref{k1zaiguji0627}_{1}\) and \eqref{zaiyicixuyao0627}  we also have 
\begin{align}\label{zheyeshixuyaode0627} 
  \begin{aligned}
    \int_{0}^{t}\left\|\widetilde{\nabla}\widetilde{\mathbf{V}}\left(s\right)
    \right\|_{H^1\left(\widetilde{\Omega}\right)}^2ds+ 
    \int_{0}^{t}\left\|\partial_{s}\widetilde{\mathbf{V}}\right\|_{L^2\left(\widetilde{\Omega}\right)}^2ds
    \leq C\delta_{0}^2,
  \end{aligned}
\end{align}
where \(C=C\left(T,R_{1},R_{2},\left\|\widetilde{\rho}_{0}\right\|_{L^\infty\left(\widetilde{\Omega}\right)}
,\overline{\rho},\alpha,\mu\right)\).

\textbf{3}, estimate of \(\left\|\partial_{t}\widetilde{\mathbf{V}}\right\|_{L^2\left(\widetilde{\Omega}\right)}\)
and \(\int_{0}^{t}\left\|\partial_{s}\widetilde{\nabla}\widetilde{\mathbf{V}}\right\|
_{L^2\left(\widetilde{\Omega}\right)}^2ds\).

Multiplying \(\eqref{raodongfeixianxing0403}_{1}\) and \(\eqref{raodongfeixianxing0403}_{2}\) by 
\(r\partial_{t}\widetilde{v}_{r}\) and \(r\partial_{t}\widetilde{v}_{\theta}\), respectively, then 
integrating over \(\widetilde{\Omega}\) give that 
\begin{align*}
  \begin{aligned}
    \int_{\widetilde{\Omega}}r\left(\widetilde{\rho}+\overline{\rho}\right)
    \left|\partial_{t}\widetilde{\mathbf{V}}\right|^2drd\theta 
    =&-\int_{\widetilde{\Omega}}r\left(\widetilde{\rho}+\overline{\rho}\right)
    \left(\widetilde{v}_{r}\partial_{r}\widetilde{\mathbf{V}}+\frac{\widetilde{v}_{\theta}}{r}
    \partial_{\theta}\widetilde{\mathbf{V}}\right)\cdot\partial_{t}\widetilde{\mathbf{V}}drd\theta
    \\
    &+\int_{\widetilde{\Omega}}\left(\widetilde{\rho}+\overline{\rho}\right)
    \left(\widetilde{v}_{\theta}^2,-\widetilde{v}_{r}\widetilde{v}_{\theta}\right)\cdot\partial_{t}
    \widetilde{\mathbf{V}}drd\theta-g\int_{\widetilde{\Omega}}r\widetilde{\rho}\partial_{t}
    \widetilde{v}_{r}drd\theta 
    \\
    &+\mu\int_{\widetilde{\Omega}}
    \begin{pmatrix}
    \Delta_{r}\widetilde{v}_{r}-\frac{\widetilde{v}_{r}}{r^2}-\frac{2}{r^2}\partial_{\theta}\widetilde{v}_{\theta}
      \\
    \Delta_{r}\widetilde{v}_{\theta}-\frac{\widetilde{v}_{\theta}}{r^2}
    +\frac{2}{r^2}\partial_{\theta}\widetilde{v}_{r}
    \end{pmatrix}\cdot \partial_{t}\widetilde{\mathbf{V}}drd\theta:=K_{3}.
  \end{aligned}
\end{align*}
From Cauchy inequality and \eqref{poincarebudengshi0517}, it is easy to obtain that 
\begin{align*}
  K_{3}\leq \varepsilon\int_{\widetilde{\Omega}}r\left(\widetilde{\rho}+\overline{\rho}\right)
  \left|\partial_{t}\widetilde{\mathbf{V}}\right|^2drd\theta+C\left( 
    \left\|\widetilde{\mathbf{V}}\right\|_{L^2\left(\widetilde{\Omega}\right)}^2
    \left\|\widetilde{\nabla}^2\widetilde{\mathbf{V}}\right\|_{L^2\left(\widetilde{\Omega}\right)}^2
    +\int_{\widetilde{\Omega}}r\widetilde{\rho}^2drd\theta
  \right),
\end{align*}
where \(C=C\left(R_{1},R_{2},
\left\|\widetilde{\rho}\right\|_{L^\infty\left(\widetilde{\Omega}\right)},\overline{\rho}\right)>0\).
Then, from \eqref{jiushizheyang0628}, we obtain 
\[
\int_{\widetilde{\Omega}}r\left(\widetilde{\rho}+\overline{\rho}\right)
\left|\partial_{t}\widetilde{\mathbf{V}}\right|^2drd\theta 
\leq C\left( 
  \left\|\widetilde{\mathbf{V}}\right\|_{L^2\left(\widetilde{\Omega}\right)}^2
  \left\|\widetilde{\nabla}^2\widetilde{\mathbf{V}}\right\|_{L^2\left(\widetilde{\Omega}\right)}^2
  +\int_{\widetilde{\Omega}}r\widetilde{\rho}^2drd\theta\right).
\]
Let \(t\rightarrow 0^{+}\). Then, we have 
\begin{align}\label{guanyishijian0628}
  \begin{aligned}
    \int_{\widetilde{\Omega}}r\left(\widetilde{\rho}+\overline{\rho}\right)
    \left|\partial_{t}\widetilde{\mathbf{V}}\right|^2|_{t=0}drd\theta \leq 
    C\delta_{0}^2,
  \end{aligned}
\end{align}
where \(C=C\left(R_{1},R_{2},
\left\|\widetilde{\rho}\right\|_{L^\infty\left(\widetilde{\Omega}\right)},\overline{\rho},
\left\|\widetilde{\mathbf{V}}_{0}\right\|_{H^2\left(\widetilde{\Omega}\right)}, 
\left\|\widetilde{\rho}_{0}\right\|_{L^2\left(\widetilde{\Omega}\right)}\right)>0\).

Subsequently, we differentiate \(\eqref{raodongfeixianxing0403}_{1}\) and 
\(\eqref{raodongfeixianxing0403}_{2}\) with respect to \(t\). Multiplying the results by 
\(r\partial_{t}\widetilde{v}_{r}\) and \(r\partial_{t}\widetilde{v}_{\theta}\), then 
adding up the results and integrating over \(\widetilde{\Omega}\) yield that 
\begin{align}\label{pingbi0628}
  \begin{aligned}
    \frac{1}{2}&\frac{d}{dt}
    \int_{\widetilde{\Omega}}r\left(\widetilde{\rho}+\overline{\rho}\right) 
    \left|\partial_{t}\widetilde{\mathbf{V}}\right|^2drd\theta 
    +\mu\int_{\widetilde{\Omega}}\bigg{[}
      r\left|\partial_{tr}\widetilde{\mathbf{V}}\right|^2
      +\frac{1}{r}\bigg{(}\left|\partial_{t\theta}\widetilde{v}_{\theta}
      +\partial_{t}\widetilde{v}_{r}\right|^2
      \\
      &+\left|\partial_{t\theta}\widetilde{v}_{r}-\partial_{t}\widetilde{v}_{\theta}\right|^2\bigg{)}
    \bigg{]}drd\theta
    +\int_{0}^{2\pi}\left(\mu-\alpha r\right)
    \partial_{t}\widetilde{v}_{\theta}^2|_{r=R_{1}}d\theta=\sum\limits_{i=1}^{6}
    S_{i},
  \end{aligned}
\end{align}
where
\begin{align*}
  &\begin{aligned}
    S_{1}=-\int_{\widetilde{\Omega}}\left(\widetilde{\rho}+\overline{\rho}\right)
    \bigg{\{}
     &r\widetilde{v}_{r}\bigg{[}
     \left(\partial_{r}\widetilde{v}_{r}\partial_{r}\widetilde{\mathbf{V}}
     +\widetilde{v}_{r}\partial_{rr}\widetilde{\mathbf{V}}+ 
     \partial_{r}\left(\frac{\widetilde{v}_{\theta}}{r}\right)\partial_{\theta}\widetilde{\mathbf{V}}
     +\frac{\widetilde{v}_{\theta}}{r}\partial_{r\theta}\widetilde{\mathbf{V}}\right)\cdot\partial_{t}
     \widetilde{\mathbf{V}}
     \\
     &+\left(\widetilde{v}_{r}\partial_{r}\widetilde{\mathbf{V}}+\frac{\widetilde{v}_{\theta}}{r}
     \partial_{\theta}\widetilde{\mathbf{V}}\right)\cdot\partial_{tr}\widetilde{\mathbf{V}}
     \bigg{]}
     \\
     &+\widetilde{v}_{\theta}\bigg{[}
     \left(\partial_{\theta}\widetilde{v}_{r}\partial_{r}\widetilde{\mathbf{V}}
     +\widetilde{v}_{r}\partial_{r\theta}\widetilde{\mathbf{V}}+ 
     \partial_{\theta}\left(\frac{\widetilde{v}_{\theta}}{r}\right)\partial_{\theta}\widetilde{\mathbf{V}} 
     +\frac{\widetilde{v}_{\theta}}{r}\partial_{\theta\theta}\widetilde{\mathbf{V}}\right)\cdot\partial_{t}
     \widetilde{\mathbf{V}}
     \\
     &+\left(\widetilde{v}_{r}\partial_{r}\widetilde{\mathbf{V}}+\frac{\widetilde{v}_{\theta}}{r}
     \partial_{\theta}\widetilde{\mathbf{V}}\right)\cdot\partial_{t\theta}\widetilde{\mathbf{V}}
     \bigg{]}
    \bigg{\}}drd\theta,
  \end{aligned}
  \\
  &\begin{aligned}
    S_{2}=-\int_{\widetilde{\Omega}}\left(\widetilde{\rho}+\overline{\rho}\right)
    \left(r\widetilde{v}_{r}\partial_{tr}\widetilde{\mathbf{V}}\cdot\partial_{t}\widetilde{\mathbf{V}}
    +\widetilde{v}_{\theta}\partial_{t\theta}\widetilde{\mathbf{V}}\cdot\partial_{t}\widetilde{\mathbf{V}}\right)
    drd\theta,
  \end{aligned}
  \\
  &\begin{aligned}
    S_{3}=-\int_{\widetilde{\Omega}}
    r\left(\widetilde{\rho}+\overline{\rho}\right)
    \left(\partial_{t}\widetilde{v}_{r}\partial_{r}\widetilde{\mathbf{V}}+ 
    \widetilde{v}_{r}\partial_{tr}\widetilde{\mathbf{V}}+\frac{\partial_{t}\widetilde{v}_{\theta}}{r}
    \partial_{\theta}\widetilde{\mathbf{V}}+\frac{\widetilde{v}_{\theta}}{r}
    \partial_{t\theta}\widetilde{\mathbf{V}}\right)\cdot\partial_{t}\widetilde{\mathbf{V}}drd\theta,
  \end{aligned}
  \\
  &\begin{aligned}
    S_{4}=\int_{\widetilde{\Omega}}\left(\widetilde{\rho}+\overline{\rho}\right)
    &\bigg{\{}
    r\widetilde{v}_{r}\bigg{[}
    \partial_{r}\left(\frac{\widetilde{v}_{\theta}^2}{r}\right)\partial_{t}\widetilde{v}_{r} 
    +\frac{\widetilde{v}_{\theta}^2}{r}\partial_{tr}\widetilde{v}_{r}
    -\partial_{r}\left(\frac{\widetilde{v}_{r}\widetilde{v}_{\theta}}{r}\right)\partial_{t}\widetilde{v}_{\theta}
    -\frac{\widetilde{v}_{r}\widetilde{v}_{\theta}}{r}\partial_{tr}\widetilde{v}_{\theta}
    \bigg{]}
    \\
    & +\widetilde{v}_{\theta} 
    \bigg{[}
    \partial_{\theta}\left(\frac{\widetilde{v}_{\theta}^2}{r}\right)\partial_{t}\widetilde{v}_{r} 
    +\frac{\widetilde{v}_{\theta}^2}{r}\partial_{t\theta}\widetilde{v}_{r}
    -\partial_{\theta}\left(\frac{\widetilde{v}_{r}\widetilde{v}_{\theta}}{r}\right)\partial_{t}
    \widetilde{v}_{\theta}
    -\frac{\widetilde{v}_{r}\widetilde{v}_{\theta}}{r}\partial_{t\theta}\widetilde{v}_{\theta}
    \bigg{]}
    \bigg{\}}drd\theta,
  \end{aligned}
  \\
  &\begin{aligned}
    S_{5}=\int_{\widetilde{\Omega}}\left(\widetilde{\rho}+\overline{\rho}\right)
    \left(\widetilde{v}_{\theta}\partial_{t}\widetilde{v}_{\theta}\partial_{t}\widetilde{v}_{r} 
    -\widetilde{v}_{r}\partial_{t}\widetilde{v}_{\theta}\partial_{t}\widetilde{v}_{\theta} \right)
    drd\theta,
  \end{aligned}
  \\
  &\begin{aligned}
    S_{6}=-g\int_{\widetilde{\Omega}}
    \left[
      \widetilde{\rho}\left(r\widetilde{v}_{r}\partial_{tr}\widetilde{v}_{r}
      +\widetilde{v}_{\theta}\partial_{t\theta}\widetilde{v}_{r}\right)
      -rD\overline{\rho}\widetilde{v}_{r}\partial_{t}\widetilde{v}_{r}
    \right]
    drd\theta.
  \end{aligned}
\end{align*}
Then, 
from Lemmas \ref{poincarebudengshi0517}, \ref{yuanlaihaishao0619}, \ref{tiduL4guji0619}, \ref{zuidazhi0517}, 
\eqref{xuyaodede0627}, \eqref{zheyeshixuyaode0627}
and Cauchy inequality, we have 
\begin{align*}
  &\begin{aligned}
    S_{1},S_{3},S_{4},S_{6}\leq C\left(t\right)\int_{\widetilde{\Omega}}r\left(\widetilde{\rho}+\overline{\rho}\right)
    \left|\partial_{t}\widetilde{\mathbf{V}}\right|^2drd\theta +
    C\left(t\right)\delta_{0}^2+\varepsilon\left\|
   \partial_{t} \widetilde{\nabla}\widetilde{\mathbf{V}}\right\|_{L^2\left(\widetilde{\Omega}\right)}^2,
  \end{aligned}
  \\
  &\begin{aligned}
    S_{2}\leq \varepsilon\left\|\partial_{t}\widetilde{\nabla}
    \widetilde{\mathbf{V}}\right\|_{L^2\left(\widetilde{\Omega}\right)}^2+
    C\left(t\right)\int_{\widetilde{\Omega}}r\left(\widetilde{\rho}+\overline{\rho}\right)
    \left|\partial_{t}\widetilde{\mathbf{V}}\right|^2drd\theta,
  \end{aligned}
  \\
  &\begin{aligned}
   S_{5}\leq C\left(t\right)\int_{\widetilde{\Omega}}r\left(\widetilde{\rho}+\overline{\rho}\right)
   \left|\partial_{t}\widetilde{\mathbf{V}}\right|^2drd\theta,
  \end{aligned}
\end{align*}
where \(C\left(t\right)\) is an integrable function. Thus, applying Gronwall's inequality to 
\eqref{pingbi0628}, we can conclude that
\begin{align}\label{renshengfenjing0628}
  \left\|\partial_{t}\widetilde{\mathbf{V}}\right\|_{L^2\left(\widetilde{\Omega}\right)}^2
  +\int_{0}^{t}\left\|\partial_{s}\widetilde{\nabla}\widetilde{\mathbf{V}}
  \right\|_{L^2\left(\widetilde{\Omega}\right)}^2ds\leq C\delta_{0}^2,
\end{align} 
where \(C=C\left(T,R_{1},R_{2},\left\|\widetilde{\rho}_{0}\right\|_{L^\infty\left(\widetilde{\Omega}\right)}
,\overline{\rho},\alpha,\mu\right)\).

\textbf{4}, estimate of \(\left\|\widetilde{\nabla}^2
\widetilde{\mathbf{V}}\right\|_{L^2\left(\widetilde{\Omega}\right)}\) and 
\(\left\|\widetilde{\nabla}\widetilde{p}\right\|_{L^2\left(\widetilde{\Omega}\right)}\).

From \eqref{zaiyicixuyao0627} and \eqref{xuyaodede0627}, it is easy to obtain that 
\begin{align}\label{duoduochifan0628}
  \left\|\widetilde{\nabla}^2\widetilde{\mathbf{V}}\right\|_{L^2\left(\widetilde{\Omega}\right)}^2
  +\left\|\widetilde{\nabla}\widetilde{p}\right\|_{L^2\left(\widetilde{\Omega}\right)}^2
  \leq C\delta_{0}^2,
\end{align}
where \(C=C\left(T,R_{1},R_{2},\left\|\widetilde{\rho}_{0}\right\|_{L^\infty\left(\widetilde{\Omega}\right)}
,\overline{\rho},\alpha,\mu\right)\).

\textbf{5}, estimate of \(\left\|\widetilde{\nabla}\widetilde{\rho}\right\|_{L^2\left(\widetilde{\Omega}\right)}\)
and \(\left\|\partial_{t}\widetilde{\rho}\right\|_{L^2\left(\widetilde{\Omega}\right)}\).

First, from \eqref{renshengfenjing0628}, we can conclude that 
\(\Upsilon \in L^2\left(0,T;\left[L^4\left(\widetilde{\Omega}\right)\right]^2\right)\) in 
\eqref{xuyaoduociyongdao0628}. Then from Lemma \ref{stokesguji0521}, it is easy to obtain 
that \(\widetilde{\mathbf{V}}\in L^2\left(0,T;\left[W^{2,4}\left(\widetilde{\Omega}\right)\right]^2\right)\). 
Thus, \(\left\|\widetilde{\nabla}\widetilde{\mathbf{V}}\right\|_{L^{\infty}\left(\widetilde{\Omega}\right)}\) 
is integrable. 

From \(\eqref{raodongfeixianxing0403}_{3}\), we can obtain 
\begin{align}\label{wodouzhidao0628}
  \begin{aligned}
    &\partial_{tr}\widetilde{\rho}+\partial_{r}\widetilde{v}_{r}\partial_{r}\widetilde{\rho} 
    +\widetilde{v}_{r}\partial_{rr}\widetilde{\rho}-\frac{\widetilde{v}_{\theta}}{r^2}
    \partial_{\theta}\widetilde{\rho}
    +\frac{\partial_{r}\widetilde{v}_{\theta}}{r}\partial_{\theta}\widetilde{\rho}
    +\frac{\widetilde{v}_{\theta}}{r}\partial_{r\theta}\widetilde{\rho}+\partial_{r}
    \widetilde{v}_{r}D\overline{\rho}+\widetilde{v}_{r}D^2\overline{\rho}=0,
    \\
    &\partial_{t\theta}\widetilde{\rho}+\partial_{\theta}\widetilde{v}_{r}\partial_{r}\widetilde{\rho} 
    +\widetilde{v}_{r}\partial_{r\theta}\widetilde{\rho}+\frac{\partial_{\theta}\widetilde{v}_{\theta}}{r}
    \partial_{\theta}\widetilde{\rho}+\frac{\widetilde{v}_{\theta}}{r}\partial_{\theta\theta}\widetilde{\rho} 
    +\partial_{\theta}\widetilde{v}_{r}D\overline{\rho}=0.
  \end{aligned}
\end{align}
Multiply \(\eqref{wodouzhidao0628}_{1}\) and \(\eqref{wodouzhidao0628}_{2}\) by 
\(r\partial_{r}\widetilde{\rho}\) and \(r\partial_{\theta}\widetilde{\rho}\), respectively, then we add up the 
results, integrate over \(\widetilde{\Omega}\) and obtain 
\begin{align*}
  \frac{d}{dt}\int_{\widetilde{\Omega}}r\left|\widetilde{\nabla}\widetilde{\rho}\right|^2drd\theta 
  \leq C\int_{\widetilde{\Omega}}\left|\widetilde{\nabla}\widetilde{\mathbf{V}}\right|
  \left|\widetilde{\nabla}\widetilde{\rho}\right|^2drd\theta, 
\end{align*}
where \(C=C\left(R_{1},R_{2},\overline{\rho}\right)\). Thus, from Gronwall's inequality, we conclude that 
\begin{align}\label{chongxin0628}
  \left\|\widetilde{\nabla}\widetilde{\rho}\right\|_{L^2\left(\widetilde{\Omega}\right)}^2
  \leq C\delta_{0}^2,
\end{align}
where \(C=C\left(T,R_{1},R_{2},\left\|\widetilde{\rho}_{0}\right\|_{L^\infty\left(\widetilde{\Omega}\right)}
,\overline{\rho},\alpha,\mu\right)\).

Finally, Since \(\partial_{t}\widetilde{\rho}=-\widetilde{v}_{r}\partial_{r}\widetilde{\rho}- 
\frac{\widetilde{v}_{\theta}}{r}\partial_{\theta}\widetilde{\rho}-\widetilde{v}_{r}D\overline{\rho}\), then 
\begin{align}\label{jiushizheyang0628}
  \left\|\partial_{t}\widetilde{\rho}\right\|_{L^2\left(\widetilde{\Omega}\right)}^2
  \leq C\delta_{0}^2,
\end{align}
where \(C=C\left(T,R_{1},R_{2},\left\|\widetilde{\rho}_{0}\right\|_{L^\infty\left(\widetilde{\Omega}\right)}
,\overline{\rho},\alpha,\mu\right)\).

Therefore, from \eqref{xuyaodede0627}, \eqref{zheyeshixuyaode0627}, 
\eqref{renshengfenjing0628}, \eqref{duoduochifan0628}, \eqref{chongxin0628}, 
\eqref{jiushizheyang0628}, we can obtain 
\begin{align*}
  \sup\limits_{t\in\left[0,T\right]}
  &\left[ 
    \left\|\widetilde{\mathbf{V}}\right\|_{H^2\left(\widetilde{\Omega}\right)}^2
    +\left\|\widetilde{\rho}\right\|_{H^1\left(\widetilde{\Omega}\right)}^2 
    +\left\|\widetilde{\nabla}\widetilde{p}\right\|_{L^2\left(\widetilde{\Omega}\right)}^2
    +\left\|\partial_{t}\widetilde{\mathbf{V}}\right\|_{L^2\left(\widetilde{\Omega}\right)}^2 
    +\left\|\partial_{t}\widetilde{\rho}\right\|_{L^2\left(\widetilde{\Omega}\right)}^2
  \right]
  \\
  &+\int_{0}^{t}\left[
    \left\|\partial_{s}\widetilde{\mathbf{V}}\right\|_{H^{1}\left(\widetilde{\Omega}\right)}^2 
    +\left\|\widetilde{\nabla}\widetilde{\mathbf{V}}\right\|_{H^{1}\left(\widetilde{\Omega}\right)}^2\right]ds 
    \leq C\delta_{0}^2,
\end{align*}
where \(C=C\left(T,R_{1},R_{2},\mu,\overline{\rho},\left\|\widetilde{\rho}_{0}
\right\|_{L^{\infty}\left(\widetilde{\Omega}\right)},\alpha\right)>0\).
\end{proof}

\subsection{Construction of the solution to linearized problem}
From Theorem \ref{endian04015}, for some initial values 
\(\left(\widetilde{\mathbf{V}},\widetilde{\rho}\right)\left(0\right)=
\left(\widetilde{\mathbf{V}}_{0},\widetilde{\rho}_{0}\right)\), 
we can construct a solution \(\left(\widetilde{\mathbf{V}},\widetilde{p},\widetilde{\rho}\right)\) 
to the linearized problem \eqref{xiaxing0403} with boundary conditions \eqref{bianjie0329}
satisfying properties \eqref{wunian04015}-\eqref{suiyuan04015}.
Here, note that \(\left\|\widetilde{v}_{r}\left(0\right)\right\|_{L^{2}\left(\widetilde{\Omega}\right)}>0\). And without loss of generality, we can assume that 
\begin{align}\label{yizhongjiashe0701}
  \left\|\left(\widetilde{\mathbf{V}},\widetilde{\rho}\right)\left(0\right)\right\|_{H^{2}\left(\widetilde{\Omega}\right)}
  =\delta_{0},
\end{align}
where \(\delta_{0}\) is a constant defined in Proposition \ref{mingti0627}. Indeed, if 
\[
  0<\left\|\left(\widetilde{\mathbf{V}},\widetilde{\rho}\right)\left(0\right)\right\|_{H^{2}\left(\widetilde{\Omega}\right)}
\neq \delta_{0},
  \]
  then we define 
  \[
  \left(\widetilde{\widetilde{\mathbf{V}}},\widetilde{\widetilde{p}},\widetilde{\widetilde{\rho}}\right)
  =\frac{\delta_{0}\left(\widetilde{\mathbf{V}},\widetilde{p},\widetilde{\rho}\right)}
  {\left\|\left(\widetilde{\mathbf{V}},\widetilde{\rho}\right)\left(0\right)\right\|_{H^{2}\left(\widetilde{\Omega}\right)}},  
  \]
which obviously is a solution to \eqref{xiaxing0403} with boundary conditions \eqref{bianjie0329} and 
satisfies 
\[
  \left\|\left(\widetilde{\widetilde{\mathbf{V}}},
  \widetilde{\widetilde{\rho}}\right)\left(0\right)\right\|_{H^{2}\left(\widetilde{\Omega}\right)}=\delta_{0}.
\]
Thus, we ensure the validness of \eqref{yizhongjiashe0701}.
In addition, from the property \eqref{suiyuan04015}, we can 
require 
\begin{align}\label{zheshiyaoqiu0701}
  \left\|\left(\widetilde{\mathbf{V}},\widetilde{\rho}\right)\right\|_{L^{2}\left(\widetilde{\Omega}\right)}
  \geq e^{\frac{\widetilde{\Lambda} t}{2}}
  \left\|\left(\widetilde{\mathbf{V}},\widetilde{\rho}\right)
  \left(0\right)\right\|_{L^{2}\left(\widetilde{\Omega}\right)}=e^{\widetilde{\Lambda}t}\left[\left\|\widetilde{\mathbf{V}}_{0}\right\|_{L^{2}\left(\widetilde{\Omega}\right)}^2+\left\|\widetilde{\rho}_{0}\right\|_{L^{2}\left(\widetilde{\Omega}\right)}^2\right]^{\frac{1}{2}},
\end{align}
where \(\widetilde{\Lambda}\) is defined in Remark \ref{lisanzuidazhi1101}. Let 
\(t_{K}=\frac{2}{\widetilde{\Lambda}}\ln{\left(\frac{2K}{a}\right)}\), where 
\(a:=\frac{\left\|\widetilde{v}_{r}\left(0\right)\right\|_{L^2\left(\widetilde{\Omega}\right)}}{\delta_{0}}\). Then from \eqref{zheshiyaoqiu0701}, we can choose the initial value \(\left(\widetilde{\mathbf{V}}_{0},\widetilde{\rho}_{0}\right)\) such that  
\begin{align}\label{yexujiushizheyang0701}
  \left\|\widetilde{v}_{r}\left(t_{k}\right)\right\|_{L^{2}\left(\widetilde{\Omega}\right)}
  \geq 2K\delta_{0}.
\end{align} 

\subsection{Construction of the solution to nonlinear problem}
Firstly, from Theorems \ref{jubucunzai0605}-\ref{weiyixing0615}, 
for any smooth initial value,  
we can obtain a strong solution to nonlinear problem \eqref{raodongfeixianxing0403} subject to the boundary conditions 
\eqref{bianjie0329} and the strong solution exists in a finite time. Subsequently, 
we define a family of initial values as follows 
\begin{align*}
  \left(\widetilde{\mathbf{V}}_{0}^{\varepsilon},\widetilde{\rho}_{0}^{\varepsilon}\right)
  =\varepsilon\left(\widetilde{\mathbf{V}}_{0},\widetilde{\rho}_{0}\right),
\end{align*}   
where \(\varepsilon>0\) is small enough. Here, from Lemma \ref{budengshidezhengming0530} and Remark \ref{yuexiaoyueda1021} , we find that 
the smaller the initial value, the longer time the strong solution exists. Thus, \(\varepsilon\) 
is so small that the time for strong solution existing is greater than \(t_{K}\). 

Note that 
\[
\left(\widetilde{\mathbf{V}}_{0}^{\varepsilon},\widetilde{\rho}_{0}^{\varepsilon}\right)  
\in \left(H^{2}\left(\widetilde{\Omega}\right)\right)^{2}\times \left[H^{1}\left(\widetilde{\Omega}\right)\cap L^{\infty}\left(\widetilde{\Omega}\right)\right],~
\text{and}~\left\|\left(\widetilde{\mathbf{V}}_{0}^{\varepsilon},
\widetilde{\rho}_{0}^{\varepsilon}\right)\right\|_{H^{2}\left(\widetilde{\Omega}\right)}=\delta_{0}\varepsilon 
<\delta_{0}.
\]
By virtue of Proposition \ref{mingti0627}, there exists a strong solution 
\(\left(\widetilde{\mathbf{V}}^{\varepsilon},\widetilde{p}^{\varepsilon},\widetilde{\rho}^{\varepsilon}\right)\) 
with initial value 
\(
  \left(\widetilde{\mathbf{V}}_{0}^{\varepsilon},\widetilde{\rho}_{0}^{\varepsilon}\right)
  \)
  on \([0,T]\times\widetilde{\Omega}\) with \(t_{K}<T<T_{max}\), which also satisfies  
  \begin{align*}
    \sup\limits_{t\in\left[0,T\right]}
    &\left[ 
      \left\|\widetilde{\mathbf{V}}^{\varepsilon}\right\|_{H^2\left(\widetilde{\Omega}\right)}^2
      +\left\|\widetilde{\rho}^{\varepsilon}\right\|_{H^1\left(\widetilde{\Omega}\right)}^2 
      +\left\|\widetilde{\nabla}\widetilde{p}^{\varepsilon}\right\|_{L^2\left(\widetilde{\Omega}\right)}^2
      +\left\|\partial_{t}\widetilde{\mathbf{V}}^{\varepsilon}\right\|_{L^2\left(\widetilde{\Omega}\right)}^2 
      +\left\|\partial_{t}\widetilde{\rho}^{\varepsilon}\right\|_{L^2\left(\widetilde{\Omega}\right)}^2
    \right]
    \\
    &+\int_{0}^{t}\left[
      \left\|\partial_{s}\widetilde{\mathbf{V}}^{\varepsilon}\right\|_{H^{1}\left(\widetilde{\Omega}\right)}^2 
      +\left\|\widetilde{\nabla}\widetilde{\mathbf{V}}^{\varepsilon}\right\|_{H^{1}\left(\widetilde{\Omega}\right)}^2\right]ds 
      \leq C\varepsilon^2\delta_{0}^2,
  \end{align*}
where \(C\) is independent of \(\varepsilon\). 
\begin{lemma}\label{zuixuyaodeyinli0701}
  For any initial value  
  \(\left(\widetilde{\mathbf{V}}_{0}^{\varepsilon},\widetilde{\rho}_{0}^{\varepsilon}\right)\), 
  the corresponding strong solution \(\left(\widetilde{\mathbf{V}}^{\varepsilon}, 
  \widetilde{p}^{\varepsilon},\widetilde{\rho}^{\varepsilon}\right)\) to nonlinear problem 
  \eqref{raodongfeixianxing0403} subject to the boundary conditions \eqref{bianjie0329} satisfies the 
  following estimates 
  \[
  \left\|\widetilde{v}_{r}^{\varepsilon}\left(t_{K}\right)\right\|_{L^2{\left(\widetilde{\Omega}\right)}}
  >\overline{F}\left(\left\|\left(\widetilde{\mathbf{V}}_{0}^{\varepsilon},
  \widetilde{\rho}_{0}^{\varepsilon}\right)\right\|_{H^{2}\left(\widetilde{\Omega}\right)}\right), 
  ~t_{K}=\frac{2}{\widetilde{\Lambda}}\ln{\left(\frac{2K}{a}\right)}<T_{\text{max}}, 
  \]
  where \(a:=\frac{\left\|\widetilde{v}_{r}\left(0\right)\right\|_{L^2\left(\widetilde{\Omega}\right)}}{\delta_{0}}\) and \(T_{max}\) is the maximum existence time.
\end{lemma}
\begin{proof}
  In order to obtain the conclusion desired, we employ the contradiction method. First, we assume that 
  for some initial value \(\left(\widetilde{\mathbf{V}}^{\varepsilon_{0}},
  \widetilde{\rho}^{\varepsilon_{0}}\right)\), where \(\varepsilon_{0}>0\) is small enough, the corresponding 
  strong solution \(\left(\widetilde{\mathbf{V}}^{\varepsilon_{0}},\widetilde{p}^{\varepsilon_{0}}, 
  \widetilde{\rho}^{\varepsilon_{0}}\right)\) satisfies 
  \begin{align*}
    \left\|\widetilde{v}_{r}^{\varepsilon_{0}}\left(t\right)\right\|_{L^{2}\left(\widetilde{\Omega}\right)}\leq
    \overline{F}\left(\left\|\left(\widetilde{\mathbf{V}}_{0}^{\varepsilon_{0}},
  \widetilde{\rho}_{0}^{\varepsilon_{0}}\right)\right\|_{H^{2}\left(\widetilde{\Omega}\right)}\right),
  ~\text{for~any}~t\in\left(0,t_{K}\right].
  \end{align*} 
  Then, according to the property of \(\overline{F}\), we obtain 
  \begin{align*}
    \left\|\widetilde{v}_{r}^{\varepsilon_{0}}\left(t\right)\right\|_{L^{2}\left(\widetilde{\Omega}\right)}\leq
    K\left\|\left(\widetilde{\mathbf{V}}^{\varepsilon_{0}}_{0}, 
    \widetilde{\rho}^{\varepsilon_{0}}_{0}\right)\right\|_{H^{2}\left(\widetilde{\Omega}\right)}
    = K\delta_{0}\varepsilon_{0}.
  \end{align*}
  Subsequently, define \(\left(\widetilde{\widetilde{\mathbf{V}}}^{\varepsilon_{0}}, 
  \widetilde{\widetilde{p}}^{\varepsilon_{0}},\widetilde{\widetilde{\rho}}^{\varepsilon_{0}}\right):= 
  \frac{\left(\widetilde{\mathbf{V}}^{\varepsilon_{0}},
  \widetilde{p}^{\varepsilon_{0}},\widetilde{\rho}^{\varepsilon_{0}}\right)}{\varepsilon_{0}}\). Then, 
  we have 
  \begin{align}\label{yijiedesijiaoqukan0701}
  \begin{aligned}
    &\begin{aligned}
      \left(\varepsilon_{0}\widetilde{\widetilde{\rho}}^{\varepsilon_{0}}+\overline{\rho}\right)
      &\partial_{t}\widetilde{\widetilde{\mathbf{V}}}^{\varepsilon_{0}}+
      \left(\varepsilon_{0}\widetilde{\widetilde{\rho}}^{\varepsilon_{0}}+\overline{\rho}\right)
      \varepsilon_{0}\left(
        \widetilde{\widetilde{v}}_{r}^{\varepsilon_{0}}
        \partial_{r}\widetilde{\widetilde{\mathbf{V}}}^{\varepsilon_{0}}
        +\frac{\widetilde{\widetilde{V}}_{\theta}^{\varepsilon_{0}}}{r}
        \partial_{\theta}\widetilde{\widetilde{\mathbf{V}}}^{\varepsilon_{0}}\right)
        +\begin{pmatrix}
          \widetilde{\widetilde{\rho}}g
          \\
          0
         \end{pmatrix}
         +\begin{pmatrix}
          \partial_{r}\widetilde{\widetilde{p}}^{\varepsilon_{0}}
          \\
          \frac{\partial_{\theta}\widetilde{\widetilde{p}}^{\varepsilon_{0}}}{r}
         \end{pmatrix}
       \\
       &=\mu
       \begin{pmatrix}
        \Delta_{r}\widetilde{\widetilde{v}}_{r}^{\varepsilon_{0}}-\frac{\widetilde{\widetilde{v}}_{r}^{\varepsilon_{0}}}{r^2}
        -\frac{2}{r^2}\partial_{\theta}\widetilde{\widetilde{v}}_{\theta}^{\varepsilon_{0}}
        \\
        \Delta_{r}\widetilde{\widetilde{v}}_{\theta}^{\varepsilon_{0}}-\frac{\widetilde{\widetilde{v}}_{\theta}^{\varepsilon_{0}}}{r^2}
        -\frac{2}{r^2}\partial_{\theta}\widetilde{\widetilde{v}}_{r}^{\varepsilon_{0}}
       \end{pmatrix}
       +\frac{\varepsilon_{0}\left(\varepsilon_{0}\widetilde{\widetilde{\rho}}^{\varepsilon_{0}}
       +\overline{\rho}\right)}{r}\left(\left(\widetilde{\widetilde{v}}_{\theta}^{\varepsilon_{0}}\right)^2,
       -\widetilde{\widetilde{v}}_{r}^{\varepsilon_{0}}
       \widetilde{\widetilde{v}}_{\theta}^{\varepsilon_{0}}\right),
    \end{aligned}
    \\
    &\partial_{t}\widetilde{\widetilde{\rho}}^{\varepsilon_{0}}
    +\varepsilon_{0}\widetilde{\widetilde{v}}_{r}^{\varepsilon_{0}}
    \partial_{r}\widetilde{\widetilde{\rho}}^{\varepsilon_{0}}
    +\varepsilon_{0}\frac{\widetilde{\widetilde{v}}_{\theta}^{\varepsilon_{0}}}{r}
    \partial_{\theta}\widetilde{\widetilde{\rho}}^{\varepsilon_{0}}
    +\widetilde{\widetilde{v}}_{r}^{\varepsilon_{0}}D\overline{\rho}=0.
    \end{aligned}
  \end{align} 
In addition, we can obtain 
\begin{align*}
  \sup\limits_{t\in\left[0,t_{K}\right]}
  &\left[ 
    \left\|\widetilde{\widetilde{\mathbf{V}}}^{\varepsilon_{0}}\right\|_{H^2\left(\widetilde{\Omega}\right)}^2
    +\left\|\widetilde{\widetilde{\rho}}^{\varepsilon_{0}}\right\|_{H^1\left(\widetilde{\Omega}\right)}^2 
    +\left\|\widetilde{\nabla}\widetilde{\widetilde{p}}^{\varepsilon_{0}}\right\|_{L^2\left(\widetilde{\Omega}\right)}^2
    +\left\|\partial_{t}\widetilde{\widetilde{\mathbf{V}}}^{\varepsilon_{0}}\right\|_{L^2\left(\widetilde{\Omega}\right)}^2 
    +\left\|\partial_{t}\widetilde{\widetilde{\rho}}^{\varepsilon_{0}}\right\|_{L^2\left(\widetilde{\Omega}\right)}^2
  \right]
  \\
  &+\int_{0}^{t}\left[
    \left\|\partial_{s}\widetilde{\widetilde{\mathbf{V}}}^{\varepsilon_{0}}\right\|_{H^{1}\left(\widetilde{\Omega}\right)}^2 
    +\left\|\widetilde{\nabla}\widetilde{\widetilde{\mathbf{V}}}^{\varepsilon_{0}}\right\|_{H^{1}\left(\widetilde{\Omega}\right)}^2\right]ds 
    \leq C\delta_{0}^2,
\end{align*}
which allows us to get the following convergent subsequences (we still denote 
\(\left(\widetilde{\widetilde{\mathbf{V}}}^{\varepsilon_{0}},\widetilde{\widetilde{p}}^{\varepsilon_{0}}, 
\widetilde{\widetilde{\rho}}^{\varepsilon_{0}}\right)\)) when \(\varepsilon_{0}\rightarrow 0^{+}\) as follows 
\begin{align*}
  \begin{aligned}
&\left(\partial_{t}\widetilde{\widetilde{\mathbf{V}}}^{\varepsilon_{0}},\widetilde{\nabla}
\widetilde{\widetilde{p}}^{\varepsilon_{0}},\partial_{t}\widetilde{\widetilde{\rho}}^{\varepsilon_{0}}\right)
\rightarrow\left(\partial_{t}\widetilde{\widetilde{\mathbf{V}}},\widetilde{\nabla}\widetilde{\widetilde{p}},
\partial_{t}\widetilde{\widetilde{\rho}}\right)~\text{weakly-star}~\text{in~}L^{\infty}\left(0,t_{K};
\left(L^{2}\left(\widetilde{\Omega}\right)\right)^{5}\right),
\\
&\left(\widetilde{\widetilde{\mathbf{V}}}^{\varepsilon_{0}},\widetilde{\widetilde{\rho}}^{\varepsilon_{0}}\right)
\rightarrow\left(\widetilde{\widetilde{\mathbf{V}}},\widetilde{\widetilde{\rho}}\right)~
\text{weakly~star~}\text{in~}L^{\infty}\left(0,t_{K};\left(H^{2}\left(\widetilde{\Omega}\right)^2 
\times H^{1}\left(\widetilde{\Omega}\right)\right)\right),
\\
&\left(\widetilde{\widetilde{\mathbf{V}}}^{\varepsilon_{0}},\widetilde{\widetilde{\rho}}^{\varepsilon_{0}}\right)
\rightarrow\left(\widetilde{\widetilde{\mathbf{V}}},\widetilde{\widetilde{\rho}}\right)~
\text{strongly~in~}C\left([0,t_{K}];\left[L^{2}\left(\widetilde{\Omega}\right)\right]^{3}\right),
  \end{aligned}
\end{align*}
and 
\begin{align}\label{zhegebijiaozhongyao0701}
  \sup\limits_{t\in\left[0,t_{K}\right]}\left\|\widetilde{\widetilde{\mathbf{V}}}
  \left(t\right)\right\|_{L^2\left(\widetilde{\Omega}\right)}\leq K\delta_{0}.
\end{align}
In addition, as \(\varepsilon_{0}\rightarrow 0^{+}\), from \eqref{yijiedesijiaoqukan0701}, we obtain 
\begin{align*}
  &\begin{aligned}
    \overline{\rho}
    &\partial_{t}\widetilde{\widetilde{\mathbf{V}}}
     =\mu
     \begin{pmatrix}
      \Delta_{r}\widetilde{\widetilde{v}}_{r}-\frac{\widetilde{\widetilde{v}}_{r}}{r^2}
      -\frac{2}{r^2}\partial_{\theta}\widetilde{\widetilde{v}}_{\theta}
      \\
      \Delta_{r}\widetilde{\widetilde{v}}_{\theta}-\frac{\widetilde{\widetilde{v}}_{\theta}}{r^2}
      -\frac{2}{r^2}\partial_{\theta}\widetilde{\widetilde{v}}_{r}
     \end{pmatrix}-\begin{pmatrix}
      \widetilde{\widetilde{\rho}}g
      \\
      0
     \end{pmatrix}-\begin{pmatrix}
      \partial_{r}\widetilde{\widetilde{p}}
      \\
      \frac{\partial_{\theta}\widetilde{\widetilde{p}}}{r}
     \end{pmatrix},
  \end{aligned}
  \\
  &\partial_{t}\widetilde{\widetilde{\rho}}
  +\widetilde{\widetilde{v}}_{r}D\overline{\rho}=0,
\end{align*}
which indicates that \(\left(\widetilde{\widetilde{\mathbf{V}}},\widetilde{\widetilde{p}}, 
\widetilde{\widetilde{\rho}}\right)\) is the solution to linearized problem \eqref{xiaxing0403} with the same initial value \(\left(\widetilde{\mathbf{V}}_{0},\widetilde{\rho}_{0}\right)\). Then, 
from the Lemma \ref{manmanlai0627}, we can conclude that 
\[
\widetilde{\widetilde{\mathbf{V}}}=\widetilde{\mathbf{V}},~\text{for~}t\in \left[0,t_{K}\right].  
\]
However, from \eqref{yexujiushizheyang0701} and \eqref{zhegebijiaozhongyao0701}, we deduce that  
\[
2K\delta_{0}\leq \left\|\widetilde{v}_{r}\left(t_{K}\right)\right\|_{L^2\left(\widetilde{\Omega}\right)}
=\left\|\widetilde{\widetilde{v}}_{r}\left(t_{K}\right)\right\|_{L^2\left(\widetilde{\Omega}\right)} 
\leq K\delta_{0},
\]
which is a contradiction. Thus, we finish the proof.
\end{proof}
From the above analysis, we complete the proof of Theorem \ref{lipuxizhi1104}.

\section{Nonlinear instability in Hadamard sense}\label{Hadamard1211}
This section examines the nonlinear instability of problem \eqref{raodongfeixianxing0403} in the Hadamard sense. To achieve this, we first determine the maximum linear growth rate. Next, we derive a nonlinear estimate that is crucial for proving Theorem \ref{hadamardyiyixia1217}. Finally, we present the proof of Theorem \ref{hadamardyiyixia1217}.
\subsection{The maximum linear growth rate}\label{zuidazhi1212}
A different perspective is employed to investigate the growth rate of linearized equation \eqref{xiaxing0403}. 
We still consider the linearized problem \eqref{yalouwuyan0403}- \eqref{huajuan0403}. But a direct variational method is used instead of 
Fourier series. Multiply \(\eqref{yalouwuyan0403}_{1}\), \(\eqref{yalouwuyan0403}_{2}\) and \(\eqref{yalouwuyan0403}_{3}\) by 
\(rv_{1}\), \(rv_{2}\) and \(rh\), respectively, then, add up the results and integrate over \(\widetilde{\Omega}\),
we obtain the following extreme value problem,
\begin{align}\label{jizhi0124}
-\widetilde{\widetilde{\Lambda}}=\inf\limits_{\left(v_{1},v_{2},h\right)\in \mathbf{G}^{1,2}\times L^{2}\left(\widetilde{\Omega}\right)}
\frac{\widetilde{E}\left(v_{1},v_{2},h\right)}{\widetilde{E}_{1}\left(v_{1},v_{2},h\right)}=\inf\limits_{\left(v_{1},v_{2},h\right)\in \mathcal{D}}\widetilde{E}\left(v_{1},v_{2},h\right),
\end{align}
where 
\begin{align*}
\begin{aligned}
&\widetilde{E}:=\widetilde{E}\left(v_{1},v_{2},h\right)
=\widetilde{E}_{2}\left(v_{1},v_{2}\right)+\widetilde{E}_{3}\left(v_{2}\right)
+\widetilde{E}_{4}\left(v_{1},h\right)
:=\widetilde{E}_{2}+\widetilde{E}_{3}+\widetilde{E}_{4},
\\
&\widetilde{E}_{2}=\mu\int_{\widetilde{\Omega}}\left\{r\left(\left|\partial_{r}v_{1}\right|^2+\left|\partial_{r}v_{2}\right|^2\right)+\frac{1}{r}\left[\left|\partial_{\theta}v_{1}-v_{2}\right|^2+\left|\partial_{\theta}v_{2}+v_{1}\right|^2\right]\right\}drd\theta,
\\
&\widetilde{E}_{3}=\int_{0}^{2\pi}\left(\mu-r\alpha\right)v_{2}^2|_{r=R_{1}}d\theta,~\widetilde{E}_{4}=\int_{\widetilde{\Omega}}\left(g+D\overline{\rho}\right)rhv_{1}drd\theta,
\\
&\widetilde{E}_{1}=\int_{\widetilde{\Omega}}r\left(\overline{\rho}v_{1}^2+\overline{\rho}v_{2}^2+h^2\right)drd\theta,~\mathcal{D}=\left\{\left(v_{1},v_{2},h\right)\in \mathbf{G}^{1,2}\times L^{2}\left(\widetilde{\Omega}\right)|
\widetilde{E}_{1}=1\right\}.
\end{aligned}
\end{align*}
The reason why we construct the functional \(\widetilde{E}\) is that we want to obtain the maximum linear growth rate. It is worth mention that the linearized problem \eqref{yalouwuyan0403}- \eqref{huajuan0403} does not have a variational structure. That is, the extreme value point of \(\widetilde{E}\) must not be the solution to the linearized problem \eqref{yalouwuyan0403}- \eqref{huajuan0403}. Fortunately, this does not affect the existence of maximum linear growth rate. 
From the previous linear instability analysis, one can easily conclude that if \(\widetilde{\widetilde{\Lambda}}\) exists, then 
\(\widetilde{\widetilde{\Lambda}}\geq\widetilde{\Lambda}>0\), where \(\widetilde{\Lambda}\) can be found in Remark \ref{lisanzuidazhi1101}. In fact, there exists a \(\left(v_{1},v_{2},h,\psi\right)=\left(u_{10},u_{20},\pi_{0},\widetilde{\psi}_{0}\right)\) that satisfies the equation \eqref{yalouwuyan0403}- \eqref{huajuan0403} as \(\lambda=\widetilde{\Lambda}\).
Thus, we have 
\begin{align*}
\begin{aligned}
\widetilde{\widetilde{\Lambda}}\geq -\frac{\widetilde{E}\left(u_{10},u_{20},\pi_{0}\right)}{\widetilde{E}_{4}\left(u_{10},u_{20},\pi_{0}\right)}=\widetilde{\Lambda}>0.
\end{aligned}
\end{align*}
About the extreme value point problem \eqref{jizhi0124},
we have the following conclusion.
\begin{proposition}\label{yusiyi1213}
     \(\widetilde{E}\) achieves its minimum on \(\mathcal{D}\).
\end{proposition}
\begin{proof}

Since \(\left(v_{1},v_{2},h\right)\in\mathcal{D}\), we obtain
\[
\widetilde{E}\left(v_{1},v_{2},h\right)
\geq \int_{\widetilde{\Omega}}\left(g+D\overline{\rho}\right)rhv_{1}drd\theta\geq-\left\|\frac{g+D\overline{\rho}}{\sqrt{\overline{\rho}}}\right\|_{L^{\infty}\left(\widetilde{\Omega}\right)},
\]
which indicates the existence of the lower bound of \(\widetilde{E}\) on \(\mathcal{D}\). Thus, \(\inf\limits_{\left(v_{1},v_{2},h\right)\in\mathcal{D}}\widetilde{E}\left(v_{1},v_{2},h\right)\) exists.

We choose the sequence \(\left\{\left(v_{1n},v_{2n},h_{n}\right)\right\}\subset\mathcal{D}\) satisfying
\[
\lim\limits_{n\rightarrow+\infty}\widetilde{E}\left(v_{1n},v_{2n},h_{n}\right)
=\inf\limits_{\left(v_{1},v_{2},h\right)\in\mathcal{D}}\widetilde{E}\left(v_{1},v_{2},h\right).
\]
Thus, \(\left\{v_{1n},v_{2n}\right\}\) is bounded in \(\left[H^{1}\left(\widetilde{\Omega}\right)\right]^2\) and \(\left\{h_{n}\right\}\) is bounded in \(L^{2}\left(\widetilde{\Omega}\right)\). There exist \(\left(v_{10},v_{20}\right)\in\left[H^{1}\left(\widetilde{\Omega}\right)\right]^2\) and \(h_{0}\in L^{2}\left(\widetilde{\Omega}\right)\) such that \(\left(v_{1n},v_{2n}\right)\rightarrow \left(v_{10},v_{20}\right)\) weakly in \(\left[H^{1}\left(\widetilde{\Omega}\right)\right]^2\) and strongly in \(\left[L^{2}\left(\widetilde{\Omega}\right)\right]^2\), \(h_{n}\rightarrow h_{0}\) weakly in \(L^{2}\left(\widetilde{\Omega}\right)\). By virtue of convexity and strong convergence, one can verify that \(\widetilde{E}\) is weak lower semicontinuous. Thus, we have 
\begin{align}\label{xuyao1219}
\widetilde{E}\left(v_{10},v_{20},h_{0}\right)\leq \lim\limits_{n\rightarrow+\infty}\widetilde{E}\left(v_{1n},v_{2n},h_{n}\right)
=\inf\limits_{\left(v_{1},v_{2},h\right)\in\mathcal{D}}\widetilde{E}\left(v_{1},v_{2},h\right)<0.
\end{align}
Finally, we show that \(\left(v_{10},v_{20},h_{0}\right)\in\mathcal{D}\) by the method of contradiction. Since \(\widetilde{E}_{1}\) is weak lower semicontinuous,  we can assume \(\widetilde{E}_{1}\left(v_{10},v_{20},h_{0}\right)=\alpha^2\in [0,1)\). 
If \(\alpha=0\), then \(v_{10}=v_{20}=h_{0}=0\). At present 
\(\widetilde{E}\left(v_{10},v_{20},h_{0}\right)=0\) which contradicts \eqref{xuyao1219}. Thus, \(\alpha\in (0,1)\). Then \(\left(\frac{v_{10}}{\alpha},\frac{v_{20}}{\alpha},\frac{h_{0}}{\alpha}\right)\in \mathcal{D}\). Therefore, we have from \eqref{xuyao1219} 
\begin{align*}
\frac{1}{\alpha^2}\inf\limits_{\left(v_{1},v_{2},h\right)\in\mathcal{D}}\widetilde{E}\geq \frac{1}{\alpha^2}\widetilde{E}\left(v_{10},v_{20},h_{0}\right)=\widetilde{E}\left(\frac{v_{10}}{\alpha},\frac{v_{20}}{\alpha},\frac{h_{0}}{\alpha}\right)\geq \inf\limits_{\left(v_{1},v_{2},h\right)\in\mathcal{D}}\widetilde{E}\left(v_{1},v_{2},h\right),
\end{align*}
which is impossible since \(\inf\limits_{\left(v_{1},v_{2},h\right)\in\mathcal{D}}\widetilde{E}<0\). Thus, \(\alpha=1\), that is, \(\left(v_{10},v_{20},h_{0}\right)\in \mathcal{D}\). 
\end{proof}

\begin{remark}\label{suiyuan1219}
From the Propositions \ref{yusiyi1213}, one can conclude that 
\begin{align*}
\begin{aligned}
\mu\int_{\widetilde{\Omega}}&\left\{r\left(\left|\partial_{r}v_{1}\right|^2+\left|\partial_{r}v_{2}\right|^2\right)+\frac{1}{r}\left[\left|\partial_{\theta}v_{1}-v_{2}\right|^2+\left|\partial_{\theta}v_{2}+v_{1}\right|^2\right]\right\}drd\theta
\\
&+\int_{0}^{2\pi}\left(\mu-r\alpha\right)v_{2}^2|_{r=R_{1}}d\theta+\int_{\widetilde{\Omega}}\left(g+D\overline{\rho}\right)rhv_{1}drd\theta
\geq -\widetilde{\widetilde{\Lambda}}
\int_{\widetilde{\Omega}}r\left(\overline{\rho}v_{1}^2+\overline{\rho}v_{2}^2+h^2\right)drd\theta,
\end{aligned}
\end{align*}
 for any \(\left(v_{1},v_{2},h\right)\in \mathbf{G}^{1,2}\times L^{2}\left(\widetilde{\Omega}\right)\). We call the \(\widetilde{\widetilde{\Lambda}}\) as the maximum linear growth rate. 
\end{remark}



\subsection{Nonlinear estimate for the strong solution}

The following is the nonlinear estimate for the strong solution.
\begin{proposition}\label{gujishizi0828} Assume that the initial value \(\left(\widetilde{\mathbf{V}}_{0},\widetilde{\rho}_{0}\right)\in \mathbf{V}^{1,2}\times 
\left(H^{1}\left(\widetilde{\Omega}\right)\cap L^{\infty}\left(\widetilde{\Omega}\right)\right)\) and \(\left(\widetilde{\mathbf{V}}\left(t\right),\widetilde{\rho}\left(t\right)\right)\) is the solution to the equation \eqref{raodongfeixianxing0403} with the initial value \(\left(\widetilde{\mathbf{V}}_{0},\widetilde{\rho}_{0}\right)\). Then,
there exists a positive constant \(\delta_{0}<<1\), such that when 
\[
\mathcal{E}\left(t\right):=
\sqrt{\left\|\widetilde{\mathbf{V}}\left(t\right)\right\|_{H^1\left(\widetilde{\Omega}\right)}^2
+\left\|\widetilde{\rho}\left(t\right)\right\|^2_{L^2\left(\widetilde{\Omega}\right)}}\leq \delta_{0},~
\text{as~}t\in[0,T]\subset[0,T_{\text{max}}),
\]
the following estimate holds, 
\begin{align}\label{houmianyongdao0828}
\begin{aligned}
\widetilde{\mathcal{E}}^2\left(t\right)
+\left\|\left(\partial_{t}\widetilde{\mathbf{V}},\widetilde{\nabla}\widetilde{p}\right)\right\|_{L^2\left(\widetilde{\Omega}\right)}^2&+
\int_{0}^{t}\left\|\left(\widetilde{\nabla}\widetilde{\mathbf{V}},\partial_{s}\widetilde{\mathbf{V}},\widetilde{\nabla}\partial_{s}\widetilde{\mathbf{V}}\right)\right\|_{L^2\left(\widetilde{\Omega}\right)}^2ds
\\
&\leq 
C\left(\widetilde{\mathcal{E}}_{0}^{2}+
\int_{0}^{t}\left\|\left(\widetilde{\mathbf{V}},\widetilde{\rho}\right)\right\|_{L^2\left(\widetilde{\Omega}\right)}^2ds\right)
\end{aligned}
\end{align}
where \(\widetilde{\mathcal{E}}\left(t\right):=\sqrt{\left\|\widetilde{\mathbf{V}}\left(t\right)\right\|_{H^2\left(\widetilde{\Omega}\right)}^2
+\left\|\widetilde{\rho}\right\|^2_{L^2\left(\widetilde{\Omega}\right)}}\), 
\(\widetilde{\mathcal{E}}_{0}=
\widetilde{\mathcal{E}}\left(0\right)\),
\(T_{max}\) is the maximum existence time and \(C=C\left(R_{1},R_{2},\mu,g,\overline{\rho},\alpha,\widetilde{\rho}_{0}\right)>0\). 
\end{proposition}
\begin{proof}

From the continuity equation and the incompressibility conditions it follows that 
\begin{align}\label{diyige0828}
\frac{d}{dt}\int_{\widetilde{\Omega}}r\widetilde{\rho}^2drd\theta=-2\int_{\widetilde{\Omega}}r\widetilde{v}_{r}\widetilde{\rho}D\overline{\rho}drd\theta \leq 
C\left\|\left(\widetilde{\mathbf{V}},\sqrt{r}\widetilde{\rho}\right)\right\|_{L^2\left(\widetilde{\Omega}\right)}^2,
\end{align}
where \(C=C\left(R_{1},R_{2},\overline{\rho}\right)>0\) and the H\"{o}lder inequality is used.

Besides, the continuity equation and the incompressibility condition also bring us
\begin{align*}
\left\|r^{\frac{1}{q}}\left(\widetilde{\rho}+\overline{\rho}\right)\right\|_{L^{q}\left(\widetilde{\Omega}\right)}
=\left\|r^{\frac{1}{q}}\left(\widetilde{\rho}_{0}+\overline{\rho}\right)\right\|_{L^{q}\left(\widetilde{\Omega}\right)},
\end{align*}
where \(q>1\). Thus, we can obtain 
\begin{align*}
\left\|\widetilde{\rho}_{0}+\overline{\rho}\right\|_{L^{\infty}\left(\widetilde{\Omega}\right)}=
\left\|\widetilde{\rho}+\overline{\rho}\right\|_{L^{\infty}\left(\widetilde{\Omega}\right)},
\end{align*}
and 
from Remark \ref{buweiling04016}, we have 
\begin{align*}
\begin{aligned}
\inf\limits_{\left(r,\theta\right)\in\widetilde{\Omega}}\left(\widetilde{\rho}+\overline{\rho}\right)=\inf\limits_{\left(r,\theta\right)\in\widetilde{\Omega}}\left(\widetilde{\rho}_{0}+\overline{\rho}\right)>0.
\end{aligned}
\end{align*}

Multiply \(\eqref{raodongfeixianxing0403}_{1}\) and \(\eqref{raodongfeixianxing0403}_{2}\) by 
\(r\widetilde{v}_{r}\) and \(r\widetilde{v}_{\theta}\), respectively, then add up the results, integrate it over \(\widetilde{\Omega}\), we obtain 
\begin{align}\label{dalei0828}
\begin{aligned}
\frac{d}{dt}\int_{\widetilde{\Omega}}r\left(\widetilde{\rho}+\overline{\rho}\right)\left|\widetilde{\mathbf{V}}\right|^2 drd\theta+\left\|\widetilde{\nabla}\widetilde{\mathbf{V}}\right\|_{L^{2}\left(\widetilde{\Omega}\right)}^2\leq C\left\|\left(\widetilde{\mathbf{V}},\sqrt{r}\widetilde{\rho}\right)\right\|_{L^2\left(\widetilde{\Omega}\right)}^2,
\end{aligned}
\end{align}
where \(C=C\left(R_{1},R_{2},g,\alpha,\mu\right)>0\), \(\eqref{raodongfeixianxing0403}_{3}\), the condition \(\mu>\alpha R_{1}\), the Lemma \ref{dengjia0523} and the H\"{o}lder inequality are used.
Thus, from \eqref{diyige0828} and \eqref{dalei0828}, one can easily obtain 
\begin{align}\label{daleixiayu0828}
\begin{aligned}
\left\|\left(\widetilde{\mathbf{V}},\widetilde{\rho}\right)\right\|_{L^2\left(\widetilde{\Omega}\right)}^2+\int_{0}^{t}\left\|\widetilde{\nabla}\widetilde{\mathbf{V}}\right\|_{L^2\left(\widetilde{\Omega}\right)}^2ds
\leq C\left(\widetilde{\mathcal{E}}_{0}^2+\int_{0}^{t}\left\|\left(\widetilde{\mathbf{V}},\widetilde{\rho}\right)\right\|_{L^2\left(\widetilde{\Omega}\right)}^2ds\right), 
\end{aligned}
\end{align}
where \(C=C\left(R_{1},R_{2},\mu,g,\overline{\rho},\alpha,\sigma_{0},\sigma_{1}\right)>0\).

Multiplying \(\eqref{raodongfeixianxing0403}_{1}\) and 
\(\eqref{raodongfeixianxing0403}_{2}\) by 
\(r\partial_{t}\widetilde{v}_{r}\) and 
\(r\partial_{t}\widetilde{v}_{\theta}\), respectively, and adding them up, then integrating the result over \(\widetilde{\Omega}\) give that
\begin{align}\label{xiangpingfujiaoshou0828}
\int_{\widetilde{\Omega}}r\left(\widetilde{\rho}+\overline{\rho}\right)\left|\partial_{t}\widetilde{\mathbf{V}}\right|^2drd\theta=\sum\limits_{j=1}^{4}K_{j},
\end{align}
where 
\begin{align*}
\begin{aligned}
&K_{1}=-\int_{\widetilde{\Omega}}r
\left(\widetilde{\rho}+\overline{\rho}\right)
\left(\widetilde{v}_{r}\partial_{r}\widetilde{\mathbf{V}}+\frac{\widetilde{v}_{\theta}}{r}\partial_{\theta}\widetilde{\mathbf{V}}\right)
\cdot\partial_{t}\widetilde{\mathbf{V}}drd\theta,
\\
&K_{2}=\int_{\widetilde{\Omega}}
\left(\widetilde{\rho}+\overline{\rho}\right)
\left(\widetilde{v}_{\theta}^2\partial_{t}\widetilde{v}_{r}-\widetilde{v}_{r}\widetilde{v}_{\theta}\partial_{t}\widetilde{v}_{\theta}\right)
drd\theta,~
K_{3}=-\int_{\widetilde{\Omega}}rg\widetilde{\rho}\partial_{t}\widetilde{v}_{r}drd\theta,
\\
&K_{4}=\mu\int_{\widetilde{\Omega}}\left[
\left(\Delta_{r}\widetilde{v}_{r}-\frac{\widetilde{v}_{r}}{r^2}-\frac{2}{r^2}\partial_{\theta}\widetilde{v}_{\theta}\right)\partial_{t}\widetilde{v}_{r} 
+\left(\Delta_{r}\widetilde{v}_{\theta}-\frac{\widetilde{v}_{\theta}}{r^2}+\frac{2}{r^2}\partial_{\theta}\widetilde{v}_{r}\right)\partial_{t}\widetilde{v}_{\theta}
\right]drd\theta.
\end{aligned}
\end{align*}
From the H\"{o}lder inequality and \(\delta_{0}<<1\), we obtain 
\begin{align*}
\lim\limits_{t\rightarrow 0^{+}}
\int_{\widetilde{\Omega}}r\left(\widetilde{\rho}+\overline{\rho}\right)\left|\partial_{t}\mathbf{\widetilde{V}}\right|^2drd\theta
\leq C\left(R_{1},R_{2},\widetilde{\rho}_{0},\overline{\rho},g,\mu\right)\left(\left\|\widetilde{\mathbf{V}}_{0}\right\|_{H^2\left(\widetilde{\Omega}\right)}^2+\left\|\widetilde{\rho}_{0}\right\|_{L^2\left(\widetilde{\Omega}\right)}^2\right).
\end{align*}
From \eqref{xiangpingfujiaoshou0828}, we have 
\begin{align*}
\begin{aligned}
\int_{\widetilde{\Omega}}r\left(\widetilde{\rho}+\overline{\rho}\right)\left|\partial_{t}\widetilde{\mathbf{V}}\right|^2drd\theta&+\frac{\mu}{2}
\frac{d}{dt}\int_{\widetilde{\Omega}}\left[
r\left|\partial_{r}\widetilde{\mathbf{V}}\right|^2+\frac{1}{r}\left(\left|\partial_{\theta}\widetilde{v}_{\theta}+\widetilde{v}_{r}\right|^2+\left|\partial_{\theta}\widetilde{v}_{r}-\widetilde{v}_{\theta}\right|^2\right)\right]drd\theta
\\
&+\frac{1}{2}\frac{d}{dt}\int_{0}^{2\pi}
\left(\mu-\alpha r\right)\widetilde{v}_{\theta}^2|_{r=R_{1}}d\theta=\sum\limits_{j=1}^{3}K_{j}.
\end{aligned}
\end{align*}
From the H\"{o}lder inequality and the Cauchy inequality, we have 
\begin{align*}
\begin{aligned}
&K_{1}\leq C\left\|\widetilde{\mathbf{V}}\right\|_{L^2\left(\widetilde{\Omega}\right)}^2
+C\left\|\widetilde{\nabla}\widetilde{\mathbf{V}}\right\|_{L^2\left(\widetilde{\Omega}\right)}^2\left\|\widetilde{\nabla}\partial_{t}\widetilde{\mathbf{V}}\right\|_{L^2\left(\widetilde{\Omega}\right)}^2,
~
K_{2}\leq \frac{1}{4}\int_{\widetilde{\Omega}}r\left(\widetilde{\rho}+\overline{\rho}\right)\left|\partial_{t}\widetilde{\mathbf{V}}\right|^2drd\theta+ 
C\left\|\widetilde{\nabla}\widetilde{\mathbf{V}}\right\|_{L^2\left(\widetilde{\Omega}\right)}^2,
\\
&K_{3}\leq \frac{1}{4}\int_{\widetilde{\Omega}}r\left(\widetilde{\rho}+\overline{\rho}\right)\left|\partial_{t}\widetilde{\mathbf{V}}\right|^2drd\theta+C
\left\|\widetilde{\rho}\right\|_{L^2\left(\widetilde{\Omega}\right)}^2,
\end{aligned}
\end{align*}
where \(C\left(R_{1},R_{2},\widetilde{\rho}_{0},\overline{\rho},g,\mu\right)>0\) and the Lemma \ref{yuanlaihaishao0619} is used. Thus, we have 
\begin{align*}
\begin{aligned}
\int_{\widetilde{\Omega}}&r\left(\widetilde{\rho}+\overline{\rho}\right)\left|\partial_{t}\widetilde{\mathbf{V}}\right|^2drd\theta+
\frac{d}{dt}\int_{\widetilde{\Omega}}\left[
r\left|\partial_{r}\widetilde{\mathbf{V}}\right|^2+\frac{1}{r}\left(\left|\partial_{\theta}\widetilde{v}_{\theta}+\widetilde{v}_{r}\right|^2+\left|\partial_{\theta}\widetilde{v}_{r}-\widetilde{v}_{\theta}\right|^2\right)\right]drd\theta
\\
&+\frac{d}{dt}\int_{0}^{2\pi}
\left(\mu-\alpha r\right)\widetilde{v}_{\theta}^2|_{r=R_{1}}d\theta
\leq C\left\|\left(\widetilde{\nabla}\widetilde{\mathbf{V}},\widetilde{\rho}\right)\right\|_{L^2\left(\widetilde{\Omega}\right)}^2+C\left\|\widetilde{\nabla}\widetilde{\mathbf{V}}\right\|_{L^2\left(\widetilde{\Omega}\right)}^2\left\|\widetilde{\nabla}\partial_{t}\widetilde{\mathbf{V}}\right\|_{L^2\left(\widetilde{\Omega}\right)}^2.
\end{aligned}
\end{align*}

Subsequently, we differentiate \(\eqref{raodongfeixianxing0403}_{1}\) and 
\(\eqref{raodongfeixianxing0403}_{2}\) with respect to \(t\). Multiplying the results by 
\(r\partial_{r}\widetilde{v}_{r}\) and \(r\partial_{r}\widetilde{v}_{\theta}\), then 
adding up the results and integrating over \(\widetilde{\Omega}\) yield that 
\begin{align}\label{qiansan0828}
  \begin{aligned}
    \frac{1}{2}&\frac{d}{dt}
    \int_{\widetilde{\Omega}}r\left(\widetilde{\rho}+\overline{\rho}\right) 
    \left|\partial_{t}\widetilde{\mathbf{V}}\right|^2drd\theta 
    +\mu\int_{\widetilde{\Omega}}\bigg{[}
      r\left|\partial_{tr}\widetilde{\mathbf{V}}\right|^2
      +\frac{1}{r}\bigg{(}\left|\partial_{t\theta}\widetilde{v}_{\theta}
      +\partial_{t}\widetilde{v}_{r}\right|^2
      \\
      &+\left|\partial_{t\theta}\widetilde{v}_{r}-\partial_{t}\widetilde{v}_{\theta}\right|^2\bigg{)}
    \bigg{]}drd\theta
    +\int_{0}^{2\pi}\left(\mu-\alpha r\right)
    \partial_{t}\widetilde{v}_{\theta}^2|_{r=R_{1}}d\theta=\sum\limits_{i=1}^{6}
    S_{i},
  \end{aligned}
\end{align}
where
\begin{align*}
  &\begin{aligned}
    S_{1}=-\int_{\widetilde{\Omega}}\left(\widetilde{\rho}+\overline{\rho}\right)
    \bigg{\{}
     &r\widetilde{v}_{r}\bigg{[}
     \left(\partial_{r}\widetilde{v}_{r}\partial_{r}\widetilde{\mathbf{V}}
     +\widetilde{v}_{r}\partial_{rr}\widetilde{\mathbf{V}}+ 
     \partial_{r}\left(\frac{\widetilde{v}_{\theta}}{r}\right)\partial_{\theta}\widetilde{\mathbf{V}} 
     +\frac{\widetilde{v}_{\theta}}{r}\partial_{r\theta}\widetilde{\mathbf{V}}\right)\cdot\partial_{t}
     \widetilde{\mathbf{V}}
     \\
     &+\left(\widetilde{v}_{r}\partial_{r}\widetilde{\mathbf{V}}+\frac{\widetilde{v}_{\theta}}{r}
     \partial_{\theta}\widetilde{\mathbf{V}}\right)\cdot\partial_{tr}\widetilde{\mathbf{V}}
     \bigg{]}
     \\
     &+\widetilde{v}_{\theta}\bigg{[}
     \left(\partial_{\theta}\widetilde{v}_{r}\partial_{r}\widetilde{\mathbf{V}}
     +\widetilde{v}_{r}\partial_{r\theta}\widetilde{\mathbf{V}}+ 
     \partial_{\theta}\left(\frac{\widetilde{v}_{\theta}}{r}\right)\partial_{\theta}\widetilde{\mathbf{V}} 
     +\frac{\widetilde{v}_{\theta}}{r}\partial_{\theta\theta}\widetilde{\mathbf{V}}\right)\cdot\partial_{t}
     \widetilde{\mathbf{V}}
     \\
     &+\left(\widetilde{v}_{r}\partial_{r}\widetilde{\mathbf{V}}+\frac{\widetilde{v}_{\theta}}{r}
     \partial_{\theta}\widetilde{\mathbf{V}}\right)\cdot\partial_{t\theta}\widetilde{\mathbf{V}}
     \bigg{]}
    \bigg{\}}drd\theta,
  \end{aligned}
  \\
  &\begin{aligned}
    S_{2}=-\int_{\widetilde{\Omega}}\left(\widetilde{\rho}+\overline{\rho}\right)
    \left(r\widetilde{v}_{r}\partial_{tr}\widetilde{\mathbf{V}}\cdot\partial_{t}\widetilde{\mathbf{V}}
    +\widetilde{v}_{\theta}\partial_{t\theta}\widetilde{\mathbf{V}}\cdot\partial_{t}\widetilde{\mathbf{V}}\right)
    drd\theta,
  \end{aligned}
  \\
  &\begin{aligned}
    S_{3}=-\int_{\widetilde{\Omega}}
    \left(\widetilde{\rho}+\overline{\rho}\right)
    \left(\partial_{t}\widetilde{v}_{r}\partial_{r}\widetilde{\mathbf{V}}+ 
    \widetilde{v}_{r}\partial_{tr}\widetilde{\mathbf{V}}+\frac{\partial_{t}\widetilde{v}_{\theta}}{r}
    \partial_{\theta}\widetilde{\mathbf{V}}+\frac{\widetilde{v}_{\theta}}{r}
    \partial_{t\theta}\widetilde{\mathbf{V}}\right)\cdot\partial_{t}\widetilde{\mathbf{V}}drd\theta,
  \end{aligned}
  \\
  &\begin{aligned}
    S_{4}=\int_{\widetilde{\Omega}}\left(\widetilde{\rho}+\overline{\rho}\right)
    &\bigg{\{}
    r\widetilde{v}_{r}\bigg{[}
    \partial_{r}\left(\frac{\widetilde{v}_{\theta}^2}{r}\right)\partial_{t}\widetilde{v}_{r} 
    +\frac{\widetilde{v}_{\theta}^2}{r}\partial_{tr}\widetilde{v}_{r}
    -\partial_{r}\left(\frac{\widetilde{v}_{r}\widetilde{v}_{\theta}}{r}\right)\partial_{t}\widetilde{v}_{\theta}
    -\frac{\widetilde{v}_{r}\widetilde{v}_{\theta}}{r}\partial_{tr}\widetilde{v}_{\theta}
    \bigg{]}
    \\
    & +\widetilde{v}_{\theta} 
    \bigg{[}
    \partial_{\theta}\left(\frac{\widetilde{v}_{\theta}^2}{r}\right)\partial_{t}\widetilde{v}_{r} 
    +\frac{\widetilde{v}_{\theta}^2}{r}\partial_{t\theta}\widetilde{v}_{r}
    -\partial_{\theta}\left(\frac{\widetilde{v}_{r}\widetilde{v}_{\theta}}{r}\right)\partial_{t}
    \widetilde{v}_{\theta}
    -\frac{\widetilde{v}_{r}\widetilde{v}_{\theta}}{r}\partial_{t\theta}\widetilde{v}_{\theta}
    \bigg{]}
    \bigg{\}}drd\theta,
  \end{aligned}
  \\
  &\begin{aligned}
    S_{5}=\int_{\widetilde{\Omega}}\left(\widetilde{\rho}+\overline{\rho}\right)
    \left(\widetilde{v}_{\theta}\partial_{t}\widetilde{v}_{\theta}\partial_{t}\widetilde{v}_{r} 
    -\widetilde{v}_{r}\partial_{t}\widetilde{v}_{\theta}\partial_{t}\widetilde{v}_{\theta} \right)
    drd\theta,
  \end{aligned}
  \\
  &\begin{aligned}
    S_{6}=-g\int_{\widetilde{\Omega}}
    \left[
      \overline{\rho}\left(r\widetilde{v}_{r}\partial_{tr}\widetilde{v}_{r}
      +\widetilde{v}_{\theta}\partial_{t\theta}\widetilde{v}_{r}\right)
      -rD\overline{\rho}\widetilde{v}_{r}\partial_{t}\widetilde{v}_{r}
    \right]
    drd\theta.
  \end{aligned}
\end{align*}
From the H\"{o}lder inequality, Cauchy inequality, the Lemmas \ref{yuanlaihaishao0619} and \ref{tiduL4guji0619}, we can obtain 
\begin{align}\label{moyige0828}
\begin{aligned}
&S_{1}\leq \varepsilon\left\|\widetilde{\nabla}\partial_{t}\widetilde{\mathbf{V}}\right\|_{L^2\left(\widetilde{\Omega}\right)}^2+\left(\varepsilon+\left\|\widetilde{\nabla}\widetilde{\mathbf{V}}\right\|_{L^2\left(\widetilde{\Omega}\right)}^4\right)\left\|\widetilde{\nabla}^2\widetilde{\mathbf{V}}\right\|_{L^2\left(\widetilde{\Omega}\right)}^2+C\left\|\widetilde{\mathbf{V}}\right\|_{L^2\left(\widetilde{\Omega}\right)}^4\left\|\widetilde{\nabla}\widetilde{\mathbf{V}}\right\|_{L^2\left(\widetilde{\Omega}\right)}^6,
\\
&S_{2},S_{3},S_{5}\leq C\left\|\widetilde{\nabla}\widetilde{\mathbf{V}}\right\|_{L^2\left(\widetilde{\Omega}\right)}\left\|\widetilde{\nabla}\partial_{t}\widetilde{\mathbf{V}}\right\|_{L^2\left(\widetilde{\Omega}\right)}^2,
~S_{4}\leq \varepsilon\left\|\widetilde{\nabla}\partial_{t}\widetilde{\mathbf{V}}\right\|_{L^2\left(\widetilde{\Omega}\right)}^2+C\left\|\widetilde{\nabla}\widetilde{\mathbf{V}}\right\|_{L^2\left(\widetilde{\Omega}\right)}^6,
\\
&S_{6}\leq \varepsilon\left\|\widetilde{\nabla}\partial_{t}\widetilde{\mathbf{V}}\right\|_{L^2\left(\widetilde{\Omega}\right)}^2+C\left\|\widetilde{\mathbf{V}}\right\|_{L^2\left(\widetilde{\Omega}\right)}^2,
\end{aligned}
\end{align}
where \(C\left(R_{1},R_{2},\widetilde{\rho}_{0},\overline{\rho},g,\mu\right)>0\).

Thus, from \eqref{qiansan0828}, \eqref{moyige0828} and 
\(\varepsilon,~\delta_{0}\) small enough, one can obtain 
\begin{align}\label{haixuyaoguji0828}
\begin{aligned}
\frac{d}{dt}
    \int_{\widetilde{\Omega}}r\left(\widetilde{\rho}+\overline{\rho}\right) 
    \left|\partial_{t}\widetilde{\mathbf{V}}\right|^2drd\theta +\left\|\widetilde{\nabla}\partial_{t}\widetilde{\mathbf{V}}\right\|_{L^2\left(\widetilde{\Omega}\right)}^2
    &\leq C\left(\varepsilon+\left\|\widetilde{\nabla}\widetilde{\mathbf{V}}\right\|_{L^2\left(\widetilde{\Omega}\right)}^4\right)\left\|\widetilde{\nabla}^2\widetilde{\mathbf{V}}\right\|_{L^2\left(\widetilde{\Omega}\right)}^2
    \\
&+C\left(\left\|\widetilde{\mathbf{V}}\right\|_{L^2\left(\widetilde{\Omega}\right)}^2+\left\|\widetilde{\nabla}\widetilde{\mathbf{V}}\right\|_{L^2\left(\widetilde{\Omega}\right)}^2\right),
\end{aligned}
\end{align}
where \(C=C\left(R_{1},R_{2},\widetilde{\rho}_{0},\mu,g,\overline{\rho},\alpha\right)>0\).

From the Stokes' estimate, we can obtain 
\begin{align*}
\begin{aligned}
\left\|\widetilde{\nabla}^2\widetilde{\mathbf{V}}\right\|_{L^2\left(\widetilde{\Omega}\right)}^2&+
\left\|\widetilde{\nabla}\widetilde{p}\right\|_{L^2\left(\widetilde{\Omega}\right)}^2
\leq C\left(\left\|\partial_{t}\widetilde{\mathbf{V}}\right\|_{L^2\left(\widetilde{\Omega}\right)}^2+\left\|\widetilde{\mathbf{V}}\cdot\widetilde{\nabla}\widetilde{\mathbf{V}}\right\|_{L^2\left(\widetilde{\Omega}\right)}^2+\left\|\widetilde{\rho}\right\|_{L^2\left(\widetilde{\Omega}\right)}^2\right)
\\
&\leq C\left(\left\|\partial_{t}\widetilde{\mathbf{V}}\right\|_{L^2\left(\widetilde{\Omega}\right)}^2+\left\|\widetilde{\mathbf{V}}\right\|_{L^2\left(\widetilde{\Omega}\right)}\left\|\widetilde{\nabla}\widetilde{\mathbf{V}}\right\|_{L^2\left(\widetilde{\Omega}\right)}\left\|\widetilde{\nabla}\widetilde{\mathbf{V}}\right\|_{L^2\left(\widetilde{\Omega}\right)}^2+\left\|\widetilde{\rho}\right\|_{L^2\left(\widetilde{\Omega}\right)}^2\right),
\end{aligned}
\end{align*}
where \(C\left(R_{1},R_{2},\widetilde{\rho}_{0},\overline{\rho},g,\mu,\alpha\right)>0\). Since \(\delta_{0}<<1\), then from the Cauchy inequality, we have 
\begin{align}\label{jiushizheyang0830}
\begin{aligned}
\left\|\widetilde{\nabla}^2\widetilde{\mathbf{V}}\right\|_{L^2\left(\widetilde{\Omega}\right)}^2+
\left\|\widetilde{\nabla}\widetilde{p}\right\|_{L^2\left(\widetilde{\Omega}\right)}^2
\leq 
C\left\|\left(\widetilde{\rho},\widetilde{\mathbf{V}},\partial_{t}\widetilde{\mathbf{V}}\right)\right\|_{L^2\left(\widetilde{\Omega}\right)}^2.
\end{aligned}
\end{align}

By \eqref{daleixiayu0828}, \eqref{haixuyaoguji0828} and \eqref{jiushizheyang0830}, it is not difficult to obtain 
\begin{align*}
\begin{aligned}
\int_{\widetilde{\Omega}}r\left(\widetilde{\rho}+\overline{\rho}\right) 
    \left|\partial_{t}\widetilde{\mathbf{V}}\right|^2drd\theta +\int_{0}^{t}\left\|\left(\widetilde{\nabla}\partial_{s}\widetilde{\mathbf{V}},\widetilde{\nabla}^2\widetilde{\mathbf{V}}\right)\right\|_{L^2\left(\widetilde{\Omega}\right)}^2ds\leq C\left(\widetilde{\mathcal{E}}_{0}^2+\int_{0}^{t}\left\|\left(\widetilde{\mathbf{V}},\widetilde{\rho}\right)\right\|_{L^2\left(\widetilde{\Omega}\right)}^2ds\right)
\end{aligned}
\end{align*}
Then from \eqref{daleixiayu0828} and \eqref{jiushizheyang0830}, one can verify 
\eqref{houmianyongdao0828}.
\end{proof}

\subsection{The proof of Theorem \ref{hadamard1105}}\label{zhengming1227}
Firstly, we choose the smooth initial value\(\left(\widetilde{\mathbf{V}}_{0},\widetilde{\rho}_{0}\right)=\left(\widetilde{v}_{0r},\widetilde{v}_{0\theta},\widetilde{\rho}_{0}\right)\) such that \(e^{\widetilde{\Lambda}t}\left(\widetilde{v}_{0r},\widetilde{v}_{0\theta},\widetilde{\rho}_{0}\right)\) is the solution to linear system \eqref{xiaxing0403}-\eqref{wuming1211}, where \(\widetilde{\Lambda}\) can be found in Remark \ref{lisanzuidazhi1101}. From the relationship \eqref{dongpo0415}, it is not difficult to verify that 
\begin{align}\label{kankan1106}
\begin{aligned}
m_{0}=\min\left\{
\left\|\widetilde{v}_{0r}\right\|_{L^1\left(\widetilde{\Omega}\right)},\left\|\widetilde{v}_{0\theta}\right\|_{L^1\left(\widetilde{\Omega}\right)},\left\|\widetilde{\rho}_{0}\right\|_{L^1\left(\widetilde{\Omega}\right)}
\right\}>0.
\end{aligned}
\end{align}
\begin{proof}
The constants \(\delta_{0}\) and \(m_{0}\) can be found in Proposition \ref{gujishizi0828} and \eqref{kankan1106}, respectively. 
Without loss generality, we assume that 
\begin{align*}
\sqrt{\left\|\widetilde{\rho}_{0}\right\|_{L^2\left(\widetilde{\Omega}\right)}^2+\left\|\widetilde{\mathbf{V}}_{0}\right\|_{H^2\left(\widetilde{\Omega}\right)}^2}=1.
\end{align*}
In addition, define 
\begin{align*}
\begin{aligned}
\left(\widetilde{\mathbf{V}}_{0}^{\delta^{*}},\widetilde{\rho}_{0}^{\delta^{*}}\right):=\delta^{*}\left(\widetilde{\mathbf{V}}_{0},\widetilde{\rho}_{0}\right),~C_{1}=\left\|\left(\widetilde{\rho}_{0},\widetilde{\mathbf{V}}_{0}\right)\right\|_{L^2\left(\widetilde{\Omega}\right)},~C_{2}=\left\|\left(\widetilde{\rho}_{0},\widetilde{\mathbf{V}}_{0}\right)\right\|_{L^1\left(\widetilde{\Omega}\right)},
\end{aligned}
\end{align*}
where \(\left\|\left(\widetilde{\rho}_{0},\widetilde{\mathbf{V}}_{0}\right)\right\|_{L^2\left(\widetilde{\Omega}\right)}=\left\|\widetilde{\rho}_{0}\right\|_{L^2\left(\widetilde{\Omega}\right)}+\left\|\widetilde{\mathbf{V}}_{0}\right\|_{L^2\left(\widetilde{\Omega}\right)}\), \(
\left\|\left(\widetilde{\rho}_{0},\widetilde{\mathbf{V}}_{0}\right)\right\|_{L^1\left(\widetilde{\Omega}\right)}=\left\|\widetilde{\mathbf{V}}_{0}\right\|_{L^1\left(\widetilde{\Omega}\right)}+\left\|\widetilde{\rho}_{0}\right\|_{L^1\left(\widetilde{\Omega}\right)}
\) and
\(\delta^{*}<\delta_{0}\).

Let \(\left(\widetilde{\mathbf{V}}^{\delta^{*}},\widetilde{\rho}^{\delta^{*}}\right)\) be the solution to the system \eqref{raodongfeixianxing0403} when the initial value 
is \(\left(\widetilde{\mathbf{V}}_{0}^{\delta^{*}},\widetilde{\rho}_{0}^{\delta^{*}}\right)\). Then 
\[\left(\widetilde{\mathbf{V}}^{\delta^{*}},\widetilde{\rho}^{\delta^{*}}\right)\in C\left(0,T;\mathbf{G}^{1,2}\times L^2\left(\widetilde{\Omega}\right)\right).
\]

Define 
\begin{align*}
\begin{aligned}
&T^{*}=\sup\left\{t\in(0,+\infty)|\sqrt{\left\|\widetilde{\rho}^{\delta^{*}}\left(s\right)\right\|_{L^2\left(\widetilde{\Omega}\right)}^2+\left\|\widetilde{\mathbf{V}}^{\delta^{*}}\left(s\right)\right\|_{H^1\left(\widetilde{\Omega}\right)}^2}\leq \delta_{0},~\forall~0\leq s\leq t\right\},
\\
&T^{**}=\sup\left\{
t\in(0,+\infty)|\left\|\left(\widetilde{\rho}^{\delta^{*}},\widetilde{\mathbf{V}}^{\delta^{*}}\right)\left(s\right)\right\|_{L^2\left(\widetilde{\Omega}\right)}\leq 2\delta^{*} C_{1}e^{\widetilde{\Lambda}s},~\forall 0\leq s\leq t
\right\}.
\end{aligned}
\end{align*}
Since \(\delta^{*}<\delta_{0}\) and \(\left(\widetilde{\mathbf{V}}^{\delta^{*}},\widetilde{\rho}^{\delta^{*}}\right)\in C\left(0,T;\mathbf{G}^{1,2}\times L^2\left(\widetilde{\Omega}\right)\right)\), then 
\(T^* T^{**}>0\). In addition, one can easily find that 
\begin{align*}
\begin{aligned}
&\text{if~}T^{*}<+\infty,~\text{then}~\sqrt{\left\|\widetilde{\rho}^{\delta^{*}}\left(T^*\right)\right\|_{L^2\left(\widetilde{\Omega}\right)}^2+\left\|\widetilde{\mathbf{V}}^{\delta^{*}}\left(T^{*}\right)\right\|_{H^1\left(\widetilde{\Omega}\right)}^2}= \delta_{0},
\\
&\text{if~}T^{**}<T_{\text{max}},~\text{then}~
\left\|\left(\widetilde{\rho}^{\delta^{*}},\widetilde{\mathbf{V}}^{\delta^{*}}\right)\left(T^{**}\right)\right\|_{L^2\left(\widetilde{\Omega}\right)}= 2\delta^{*} C_{1}e^{\widetilde{\Lambda}T^{**}}.
\end{aligned}
\end{align*}
Then, for any \(t\leq \min\left\{T^{*},T^{**}\right\}\), 
from the Proposition \ref{gujishizi0828}, we have 
\begin{align}\label{guji0831}
\begin{aligned}
\left\|\widetilde{\mathbf{V}}^{\delta^{*}}\right\|_{H^2\left(\widetilde{\Omega}\right)}^2+\left\|\widetilde{\rho}^{\delta^{*}}\right\|_{L^2\left(\widetilde{\Omega}\right)}^2
+\left\|\partial_{t}\widetilde{\mathbf{V}}^{\delta^{*}}\right\|_{L^2\left(\widetilde{\Omega}\right)}^2&\leq 
C\left(\delta^{*}\right)^2\left(\left\|\widetilde{\rho}_{0}\right\|_{L^2\left(\widetilde{\Omega}\right)}^2+\left\|\widetilde{\mathbf{V}}_{0}\right\|_{H^2\left(\widetilde{\Omega}\right)}^2\right)
\\
&+C\left(\int_{0}^{t}\left\|\left(\widetilde{\rho}^{\delta^{*}},\widetilde{\mathbf{V}}^{\delta^{*}}\right)\left(s\right)\right\|_{L^2\left(\widetilde{\Omega}\right)}^2ds\right)
\\
&\leq C_{3}\left(\delta^{*}\right)^2e^{2\widetilde{\Lambda}t},
\end{aligned}
\end{align}
where \(C_{3}>0\) is independent of \(\delta^{*}\).

Let \(\left(\widetilde{\mathbf{V}}^{d},\widetilde{\rho}^{d}\right)=\left(\widetilde{\mathbf{V}}^{\delta^{*}},\widetilde{\rho}^{\delta^{*}}\right)-\delta^{*}\left(\widetilde{\mathbf{V}}^{l},\widetilde{\rho}^{l}\right)\), where 
\begin{align*}
\begin{aligned}
\begin{cases}
    \overline{\rho}\frac{\partial \widetilde{v}_{r}^{l}}{\partial t}=
    \mu\left(\Delta_{r}\widetilde{v}_{r}^{l}-\frac{\widetilde{v}_{r}^{l}}{r^2}
      -\frac{2}{r^2}
      \frac{\partial\widetilde{v}_{\theta}^{l}}{\partial\theta}\right)
      -\frac{\partial\widetilde{p}^{l}}{\partial r}-\widetilde{\rho}^{l}g,
      \\
      \overline{\rho}\frac{\partial \widetilde{v}_{\theta}^{l}}{\partial t}=
      \mu\left(\Delta_{r}\widetilde{v}_{\theta}^{l}-\frac{\widetilde{v}_{\theta}^{l}}{r^2}+
      \frac{2}{r^2}\frac{\partial\widetilde{v}_{r}^{l}}{\partial\theta}\right)-
      \frac{1}{r}\frac{\partial\widetilde{p}^{l}}{\partial\theta},
      \\
      \frac{\partial\widetilde{\rho}^{l}}{\partial t}=-\widetilde{v}_{r}^{l}D\overline{\rho},~
      \frac{\partial \left(r\widetilde{v}_{r}^{l}\right)}{\partial r}
    +\frac{\partial \widetilde{v}_{\theta}^{l}}{\partial \theta}=0,
    \\
\left(\widetilde{\mathbf{V}}^{l},\widetilde{\rho}^{l}\right)=\left(\widetilde{\mathbf{V}}_{0},\widetilde{\rho}_{0}\right).
  \end{cases}
\end{aligned}
\end{align*}
From the previous disscussion, we can conclude that 
\begin{align*}
\left(\widetilde{\mathbf{V}}^{l},\widetilde{\rho}^{l}\right)=e^{\widetilde{\Lambda}t}\left(\widetilde{\mathbf{V}}_{0},\widetilde{\rho}_{0}\right).
\end{align*}
Then, 
\begin{align}\label{meiyuan0901}
\begin{aligned}
&\partial_{t}\widetilde{\rho}^{d}+\widetilde{v}_{r}^{d}D\overline{\rho}=-\widetilde{v}_{r}^{\delta^{*}}\partial_{r}\widetilde{\rho}^{\delta^{*}}-\frac{\widetilde{v}_{\theta}^{\delta^{*}}}{r}\partial_{\theta}\widetilde{\rho}^{\delta^{*}},
\\
&\overline{\rho}\partial_{t}\widetilde{\mathbf{V}}^{d}-\mu\begin{pmatrix}
    \Delta_{r}\widetilde{v}_{r}^{d}-\frac{\widetilde{v}_{r}^{d}}{r^2}-\frac{2}{r^2}\partial_{\theta}\widetilde{v}_{\theta}^{d}
    \\
    \Delta_{r}\widetilde{v}_{\theta}^{d}-\frac{\widetilde{v}_{\theta}^{d}}{r^2}+\frac{2}{r^2}\partial_{\theta}\widetilde{v}_{r}^{d}
\end{pmatrix}
+\begin{pmatrix}
    \partial_{r}\widetilde{p}^{d}
    \\
    \frac{1}{r}\partial_{\theta}\widetilde{\rho}^{d}
\end{pmatrix}
+g\begin{pmatrix}
    \widetilde{\rho}^{d}
    \\
    0
\end{pmatrix}
\\
&~~=\frac{\widetilde{\rho}^{\delta^{*}}+\overline{\rho}}{r}\begin{pmatrix}
    \left(\widetilde{v}_{\theta}^{\delta^{*}}\right)^2
    \\
    -\widetilde{v}_{r}^{\delta^{*}}\widetilde{v}_{\theta}^{\delta^{*}}
\end{pmatrix}
-\widetilde{\rho}^{\delta^{*}}\partial_{t}\widetilde{\mathbf{V}}^{\delta^{*}}-\left(\widetilde{\rho}^{\delta^{*}}+\overline{\rho}\right)\left(\widetilde{v}_{r}^{\delta^{*}}\partial_{r}\widetilde{\mathbf{V}}^{\delta^{*}}+\frac{\widetilde{v}_{\theta}^{\delta^{*}}}{r}\partial_{\theta}\widetilde{\mathbf{V}}^{\delta^{*}}\right),
\\
&\partial_{r}\left(r\widetilde{v}^{d}_{r}\right)+\partial_{\theta}\widetilde{v}_{\theta}^{d}=0,~
\left(\widetilde{\mathbf{V}}^{d},\widetilde{\rho}^{d}\right)\left(0\right)=\left(\mathbf{0},0\right).
\end{aligned}
\end{align}
Multiplying \(\eqref{meiyuan0901}_{1}\) and \(\eqref{meiyuan0901}_{2}\) by \(r\widetilde{\rho}^{d}\), \(r\widetilde{\mathbf{V}}^{d}\), respectively, and a direct calculation give that 
\begin{align}\label{manmanda0901}
\begin{aligned}
&\frac{d}{dt}
\int_{\widetilde{\Omega}}\left[
r\overline{\rho}\left(\widetilde{\mathbf{V}}^{d}\right)^2
+r\left(\widetilde{\rho}^{d}\right)^2
\right]drd\theta
+2\int_{0}^{2\pi}\left(\mu-\alpha r\right)\left(\widetilde{v}_{\theta}^{d}\right)^2|_{r=R_{1}}d\theta 
\\
&~~+2\mu\int_{\widetilde{\Omega}}
r\left[\left|\partial_{r}\widetilde{\mathbf{V}}^{d}_{r}\right|^2+\frac{1}{r}\left(\left|\partial_{\theta}\widetilde{v}_{\theta}^{d}+\widetilde{v}_{r}^{d}\right|^2+\left|\partial_{\theta}\widetilde{v}_{r}^{d}-\widetilde{v}_{\theta}^{d}\right|^2\right)\right]dr d\theta
\\
&~~+2\int_{\widetilde{\Omega}}
r\left(g+D\overline{\rho}\right)\widetilde{v}_{r}^{d}\widetilde{\rho}^{d}drd\theta=\sum\limits_{k=1}^{3}\Sigma_{k},
\end{aligned}
\end{align}
where
\begin{align*}
\begin{aligned}
&\Sigma_{1}=-2\delta^{*}\int_{\widetilde{\Omega}}
\left(r\widetilde{v}_{r}^{\delta^{*}}\partial_{r}\widetilde{\rho}^{l}\widetilde{\rho}^{\delta^{*}}+\widetilde{v}_{\theta}^{\delta^{*}}\partial_{\theta}\widetilde{\rho}^{l}\widetilde{\rho}^{\delta^{*}}\right)drd\theta,
\\
&\Sigma_{2}=-2\int_{\widetilde{\Omega}}
\left[
\widetilde{\rho}^{\delta^{*}}\partial_{t}\widetilde{\mathbf{V}}^{\delta^{*}}\cdot r\widetilde{\mathbf{V}}^{d}
+\left(\widetilde{\rho}^{\delta^{*}}+\overline{\rho}\right)\left(\widetilde{v}_{r}^{\delta^{*}}\partial_{r}\widetilde{\mathbf{V}}^{\delta^{*}}+\frac{\widetilde{v}_{\theta}^{\delta^{*}}}{r}\partial_{\theta}\widetilde{v}_{\theta}^{\delta^{*}}\right)\cdot r\widetilde{\mathbf{V}}^{d}
\right]drd\theta,
\\
&\Sigma_{3}=2
\int_{\widetilde{\Omega}}\left(\widetilde{\rho}^{\delta^{*}}+\overline{\rho}\right)\left[
\left(\widetilde{v}_{\theta}^{\delta^{*}}\right)^2\widetilde{v}_{r}^{d}-\widetilde{v}_{r}^{\delta^{*}}\widetilde{v}_{\theta}^{\delta^{*}}\widetilde{v}_{\theta}^{d}
\right]dr d\theta.
\end{aligned}
\end{align*}
A direct computation with \eqref{guji0831} gives that 
\begin{align}\label{diyige0901}
\Sigma_{1}\leq C_{4}\left(\delta^{*}\right)^{3}e^{3\widetilde{\Lambda}t},
\end{align}
In addition, from the H\"{o}lder inequality and 
\eqref{guji0831}, one can also obtain 
\begin{align}\label{wangyong0901}
\begin{aligned}
\Sigma_{2}\text{,}~\Sigma_{3}\leq C_{5}\left(\delta^{*}\right)^3e^{3\widetilde{\Lambda}t},
\end{aligned}
\end{align}
where \(C_{4}\) and \(C_{5}\) are independent of 
\(\delta^{*}\).
Besides, we also have from Remark \ref{suiyuan1219} 
\begin{align}\label{changmingbaisui0901}
\begin{aligned}
&2\int_{0}^{2\pi}\left(\mu-\alpha r\right)\left(\widetilde{v}_{\theta}^{d}\right)^2|_{r=R_{1}}d\theta 
\\
&~~+2\mu\int_{\widetilde{\Omega}}
\left\{r\left|\partial_{r}\widetilde{\mathbf{V}}^{d}_{r}\right|^2+\frac{1}{r}\left(\left|\partial_{\theta}\widetilde{v}_{\theta}^{d}+\widetilde{v}_{r}^{d}\right|^2+\left|\partial_{\theta}\widetilde{v}_{r}^{d}-\widetilde{v}_{\theta}^{d}\right|^2\right)\right\}dr d\theta
\\
&~~+2\int_{\widetilde{\Omega}}
r\left(g+D\overline{\rho}\right)\widetilde{v}_{r}^{d}\widetilde{\rho}^{d}drd\theta\geq -2\int_{\widetilde{\Omega}}\widetilde{\widetilde{\Lambda}}\left[
r\overline{\rho}\left(\widetilde{\mathbf{V}}^{d}\right)^2
+r\left(\widetilde{\rho}^{d}\right)^2
\right]drd\theta.
\end{aligned}
\end{align}
Thus, from \eqref{manmanda0901}-\eqref{changmingbaisui0901}, the condition \(3\widetilde{\Lambda}>2\widetilde{\widetilde{\Lambda}}\) and the Gronwall's inequality, we have 
\begin{align}\label{changguzuofa}
\left\|\left(\widetilde{\mathbf{V}}^{d},\widetilde{\rho}^{d}\right)\right\|_{L^2\left(\widetilde{\Omega}\right)}^2\leq C_{6}\left(\delta^{*}\right)^{3}e^{3\widetilde{\Lambda}t},
\end{align}
where \(C_{6}>0\) is independent of \(\delta^{*}\).

Let \(\epsilon_{0}=\min\left\{
\frac{\delta_{0}}{4\sqrt{C_{3}}},\frac{C_{1}^2}{32C_{6}},\frac{m_{0}^2}{2\pi C_{6}\left(R_{2}-R_{1}\right)}
\right\}\) and 
\(T^{\delta^{*}}=\frac{1}{\widetilde{\Lambda}}\ln{\frac{2\epsilon_{0}}{\delta^{*}}}\). Obviously, \(\epsilon_{0}\) is independent of \(\delta^{*}\) and \(T^{\delta^{*}}>0\) since \(\delta^{*}\) is small enough. 

Then, we verify that 
\begin{align}\label{baicunyi0901}
T^{\delta^{*}}=\min\left\{T^{\delta^{*}},T^{*},T^{**}\right\}.
\end{align}
In fact, on one hand,
if \(T^{*}=\min\left\{T^{\delta^{*}},T^{*},T^{**}\right\}\), then \(T^{*}<+\infty\) and
\[
\sqrt{\left\|\widetilde{\rho}^{\delta^{*}}\left(T^{*}\right)\right\|_{L^2\left(\widetilde{\Omega}\right)}^2+\left\|\widetilde{\mathbf{V}}^{\delta^{*}}\left(T^{*}\right)\right\|_{H^1\left(\widetilde{\Omega}\right)}^2}\leq \sqrt{C_{3}}\delta^{*} e^{\widetilde{\Lambda}T^{*}}\leq \sqrt{C_{3}}\delta^{*} e^{\widetilde{\Lambda}T^{\delta^{*}}}=2\sqrt{C_{3}}\epsilon_{0}<\delta_{0},
\]
which is contradiction; on the other hand, 
if \(T^{**}=\min\left\{T^{\delta^{*}},T^{*},T^{**}\right\}\), then \(T^{**}<T_{\text{max}}\) and from \eqref{changguzuofa}, we have 
\begin{align*}
\begin{aligned}
\left\|\left(\widetilde{\mathbf{V}}^{\delta^{*}},\widetilde{\rho}^{\delta^{*}}\right)\left(T^{**}\right)\right\|_{L^2\left(\widetilde{\Omega}\right)}
&\leq \delta^{*}\left\|\left(\widetilde{\mathbf{V}}^{l},\widetilde{\rho}^{l}\right)\left(T^{**}\right)\right\|_{L^2\left(\widetilde{\Omega}\right)}+
\left\|\left(\widetilde{\mathbf{V}}^{d},\widetilde{\rho}^{d}\right)\left(T^{**}\right)\right\|_{L^2\left(\widetilde{\Omega}\right)}
\\
&\leq \delta^{*} C_{1}e^{\widetilde{\Lambda}T^{**}}+\sqrt{C_{6}}\left(\delta^{*}\right)^{\frac{3}{2}}e^{\frac{3\widetilde{\Lambda}T^{**}}{2}}
\\
&\leq \delta^{*} e^{\widetilde{\Lambda}T^{**}}\left(C_{1}+\sqrt{2C_{6}\epsilon_{0}}\right)< 2\delta^{*} C_{1}e^{\widetilde{\Lambda}T^{**}},
\end{aligned}
\end{align*}
which is a contradiction. 
Thus, \eqref{baicunyi0901} is valid.
Therefore, we have the following inequalities, 
\begin{align*}
\begin{aligned}
\left\|\widetilde{\rho}^{\delta^{*}}\left(T^{\delta^{*}}\right)\right\|_{L^1\left(\widetilde{\Omega}\right)}
&\geq \delta^{*}\left\|\widetilde{\rho}^{l}\left(T^{\delta^{*}}\right)\right\|_{L^1\left(\widetilde{\Omega}\right)}-\left\|\widetilde{\rho}^{d}\left(T^{\delta^{*}}\right)\right\|_{L^1\left(\widetilde{\Omega}\right)}
\\
&\geq \delta^{*} e^{\widetilde{\Lambda}T^{\delta^{*}}}\left\|\widetilde{\rho}_{0}\right\|_{L^1\left(\widetilde{\Omega}\right)}
-\left[2\pi\left(R_{2}-R_{1}\right)\right]^{\frac{1}{2}}\left\|\widetilde{\rho}^{d}\left(T^{\delta^{*}}\right)\right\|_{L^2\left(\widetilde{\Omega}\right)}
\\
&\geq 2\epsilon_{0}\left\|\widetilde{\rho}_{0}\right\|_{L^1\left(\widetilde{\Omega}\right)}-\sqrt{C_{6}}\epsilon_{0}^{\frac{3}{2}}\left[2\pi\left(R_{2}-R_{1}\right)\right]^{\frac{1}{2}}
\\
&\geq m_{0}\epsilon_{0}.
\end{aligned}
\end{align*}
Follow the similar steps, one can also obtain 
\[
\left\|\widetilde{\mathbf{V}}^{\delta^{*}}\left(T^{\delta^{*}}\right)\right\|_{L^1\left(\widetilde{\Omega}\right)}\geq m_{0}\epsilon_{0}.
\]
Finally, let \(\epsilon=m_{0}\epsilon_{0}\). Thus, the proof is completed.
\end{proof}

\section{Appendix}\label{zijixiugai0113}
This section presents the proofs for the lemmas introduced in Section \ref{chubuzhishi1212}. For Lemmas \ref{yuanlaihaishao0619} and \ref{tiduL4guji0619}, please refer to the proof of Lemma 3.5 in \cite{li_global_2021}. For Lemma \ref{zuidazhi0517}, one can refer to the proof of Lemma 3.11 in \cite{li_global_2021}.
For Lemma \ref{tezhengzhi04018}, please consult Chapter 12 in \cite{zhanggongqing2011}.
\begin{proof}\textbf{The proof of Lemma \ref{poincarebudengshi0517}}. 
  Due to the density, we assume that \(\mathbf{V}=\left(v_{r},v_\theta\right)
  \in\mathbf{G}^{1,2}\) is smooth. 
  From the Newton-Leibniz formula, we have 
  \[
  v_{r}\left(r,\theta\right)= v_{r}\left(r,\theta\right)-v_{r}\left(R_{1},\theta\right)=\int_{R_{1}}^{r}
  \partial_{s} v_{r}\left(s,\theta\right)ds,  
  \]
which together with H\"{o}lder inequality yields that 
\begin{align*}
\left|v_{r}\left(r,\theta\right)\right|^2\leq  \left(R_{2}-R_{1}\right) \int_{R_{1}}^{R_{2}}
\left|\partial_{r} v_{r}\left(r,\theta\right)\right|^2dr.
\end{align*}
Then, integrating the above inequality over \(\widetilde{\Omega}\) gives that 
\begin{align}\label{chunzhen0518}
  \begin{aligned}
    \left\|v_{r}\right\|_{L^{2}\left(\widetilde{\Omega}\right)}\leq \left(R_{2}-R_{1}\right)
    \left\|\partial_{r} v_{r}\right\|_{L^{2}\left(\widetilde{\Omega}\right)}
  \end{aligned}
\end{align}
In addition, from \(\partial_{r} \left(rv_{r}\right)+\partial_{\theta} v_{\theta}=0\), 
we can deduce that 
\[
\partial_{r}\int_{0}^{2\pi}rv_{r}d\theta=-\int_{0}^{2\pi}\partial_{\theta} v_{\theta}d\theta=0,  
\]
which implies that 
\begin{align}\label{xuyao0728}
\int_{0}^{2\pi}rv_{r}d\theta=\text{const}.
\end{align} 
Now, we verify that 
\(\int_{0}^{2\pi}rv_{r}d\theta=0\). In fact, from the boundary conditions and divergence theorem, we have 
\begin{align}\label{jiuzhemebeixuyao0521}
0=\int_{\partial\widetilde{\Omega}}r\left(rv_{r},v_{\theta}\right)\cdot\mathbf{n}ds 
=\int_{\widetilde{\Omega}}\left[r\left(\partial_{r} \left(rv_{r}\right)+\partial_{\theta} v_{\theta}
\right)+rv_{r}\right]drd\theta=\int_{\widetilde{\Omega}}
rv_{r}drd\theta.   
\end{align}
And from \eqref{xuyao0728}, it follows that for any \(r\in [R_{1},R_{2}]\), there exists a \(\theta_{r}\in [0,2\pi]\)
such that 
\(v_{r}\left(r,\theta_{r}\right)=0\). 
Then, from the Newton-Leibniz formula again and by the same procedure as in obtaining \eqref{chunzhen0518}, 
we can verify 
\begin{align}\label{doujimo0518}
\left\|v_{r}\right\|_{L^{2}\left(\widetilde{\Omega}\right)}\leq 2\pi
\left\|\partial_{\theta} v_{r}\right\|_{L^2\left(\widetilde{\Omega}\right)}.
\end{align}
As a result of \eqref{chunzhen0518} and \eqref{doujimo0518}, we get 
\begin{align}\label{tixingwo0518}
\|v_{r}\|_{L^2\left(\widetilde{\Omega}\right)}\leq \left(2\pi\right)^{\frac{1}{2}}
\left(R_{2}-R_{1}\right)^{\frac{1}{2}} \left\|\partial_{r} v_{r}\right\|_{L^2\left(\widetilde{\Omega}\right)}^{\frac{1}{2}}
\left\|\partial_{\theta} v_{r}\right\|_{L^2\left(\widetilde{\Omega}\right)}^{\frac{1}{2}}  
\end{align}
Subsequently, we prove \eqref{mingtiannihao0518} for \(v_{r}\) by application of mathematical induction. 
We have verified the case when \(k=1\), see \eqref{tixingwo0518}. Now, assume that 
\eqref{mingtiannihao0518} for \(v_{r}\) is valid for \(k=n\), where \(n\) is a fixed and positive integer. 
From the Newton-Leibniz formula and H\"{o}lder inequality, we can obtain that 
\begin{align*}
  \begin{aligned}
  &\left|v_{r}\right|^{n+1}\leq \left(n+1\right)\int_{R_{1}}^{R_{2}}\left|v_{r}\right|^{n}
  \left|\partial_{r} v_{r}\right|dr\leq 
  \left(n+1\right)\left[\int_{R_{1}}^{R_{2}}\left|v_{r}\right|^{2n}dr\right]^{\frac{1}{2}}
  \left[\int_{R_{1}}^{R_2}\left|\partial_{r} v_{r}\right|^2dr\right]^{\frac{1}{2}},
  \\
  &\left|v_{r}\right|^{n+1}\leq \left(n+1\right)\int_{0}^{2\pi}\left|v_{r}\right|^{n}
  \left|\partial_{\theta} v_{r}\right|d\theta\leq 
  \left(n+1\right)\left[\int_{0}^{2\pi}\left|v_{r}\right|^{2n}d\theta\right]^{\frac{1}{2}}
  \left[\int_{0}^{2\pi}\left|\partial_{\theta} v_{r}\right|^2d\theta\right]^{\frac{1}{2}}.
  \end{aligned}
\end{align*}
Multiply the above two inequalities, integrate the result over \(\widetilde{\Omega}\), then from the 
H\"{o}lder inequality and \eqref{mingtiannihao0518} for \(v_{r}\) as \(k=n\), we have 
\begin{align}\label{buyongzhongchui0619}
\int_{\widetilde{\Omega}}\left|v_{r}\right|^{2\left(n+1\right)}drd\theta
\leq C\left(n,R_{1},R_{2}\right)\left[\int_{\widetilde{\Omega}}
\left|\partial_{r} v_{r}\right|^{2}drd\theta\right]^{\frac{n+1}{2}}   
\left[\int_{\widetilde{\Omega}}
\left|\partial_{\theta} v_{r}\right|^{2}drd\theta\right]^{\frac{n+1}{2}},
\end{align}
which indicates the validness of \eqref{mingtiannihao0518} for \(v_{r}\).

Since \(v_{\theta}\left(R_{2},\theta\right)=0\), then from the Newton-Leibniz formula and H\"{o}lder inequality, 
we obtain 
\begin{align}\label{guanyuanjian0619}
  \left\|v_{\theta}\right\|_{L^2\left(\widetilde{\Omega}\right)}
    \leq \left(R_{2}-R_{1}\right)\left\|
    \partial_{r} v_{\theta}\right\|_{L^2\left(\widetilde{\Omega}\right)}.
\end{align}

Then, we consider the function \(\xi v_{\theta}\left(r,\theta\right)\), where 
\(\xi=\xi\left(\theta\right)\in C^{\infty}\left[0,4\pi\right]\) satisfies: 1) 
\(\xi\equiv 1\) as \(\theta\in \left[0,2\pi\right]\); 2) \(\xi\in\left[0,1\right]\); 
3)\(\xi\equiv 0\) as \(\theta\in \left[3\pi,4\pi\right]\). Since \(v_{\theta}\) is periodic in \(\theta\) 
and \(\xi v_{\theta}\left(R_{2},\theta\right)=
\xi\left(3\pi\right)v_{\theta}\left(r,3\pi\right)=0\), then we can employ the similar steps to obtain 
\begin{align}\label{onemoretime0619}
  \left\|\xi v_{\theta}\right\|_{L^{2k}\left(\widetilde{\widetilde{\Omega}}\right)}\leq C
    \left\|\partial_{r} \left(\xi v_{\theta}\right)\right\|_{L^2\left(\widetilde{\widetilde{\Omega}}\right)}^{\frac{1}{2}}
\left\|\partial_{\theta} \left(\xi v_{\theta}\right)\right\|_{L^2\left(\widetilde{\widetilde{\Omega}}\right)}^{\frac{1}{2}},
\end{align}  
where \(\widetilde{\widetilde{\Omega}}=\left[R_{1},R_{2}\right]\times \left[0,4\pi\right]\). A direct computation 
with \eqref{guanyuanjian0619} yields that 
\begin{align*}
  \begin{aligned}
    &\left\|\partial_{r} \left(\xi v_{\theta}\right)\right\|_{L^2\left(\widetilde{\widetilde{\Omega}}\right)}
    \leq C\left\|\partial_{r}v_{\theta}\right\|_{L^2\left(\widetilde{\Omega}\right)},~
    \left\|v_{\theta}\right\|_{L^{2k}\left(\widetilde{\Omega}\right)}\leq 
    \left\|\xi v_{\theta}\right\|_{L^{2k}\left(\widetilde{\widetilde{\Omega}}\right)},
    \\
    &\left\|\partial_{\theta} \left(\xi v_{\theta}\right)\right\|_{L^2\left(\widetilde{\widetilde{\Omega}}\right)}
    \leq C\left[\left\|v_{\theta}\right\|_{L^2\left(\widetilde{\Omega}\right)}+ 
    \left\|\partial_{\theta}v_{\theta}\right\|_{L^2\left(\widetilde{\Omega}\right)}\right] 
    \leq C\left\|\widetilde{\nabla}v_{\theta}\right\|_{L^2\left(\widetilde{\Omega}\right)},
  \end{aligned}
\end{align*}
which together with \eqref{buyongzhongchui0619} and \eqref{onemoretime0619} yields \eqref{mingtiannihao0518} for \(v_{\theta}\).
\end{proof}

\begin{proof}\textbf{The proof of Lemma \ref{tiduL2guji0518}.}

We just verify that 
\[
\left\|\partial_{r}v_{\theta}\right\|_{L^2\left(\widetilde{\Omega}\right)}^{2}\leq
      \left\|v_{\theta}\right\|_{L^2\left(\widetilde{\Omega}\right)}
      \left\|\partial_{rr}v_{\theta}\right\|_{L^2\left(\widetilde{\Omega}\right)}.
\]
The rest three inequality can be obtained by the same steps.  Since \(\partial_{r} v_{\theta}=\left(\frac{1}{r}-\frac{\alpha}{\mu}\right)v_{\theta}, 
  r=R_{1}\), then from the conditions that
  \(1-\frac{\alpha R_{1}}{\mu}\geq 0\) and \(v_{\theta}\left(R_{2},\theta\right)=0\), we can obtain 
  \[
  \int_{0}^{2\pi}\left(\partial_{r}v_{\theta}v_{\theta}\right)|_{r=R_{1},R_{2}}  d\theta
  =-\int_{0}^{2\pi}\left(\frac{1}{r}-\frac{\alpha}{\mu}\right)v_{\theta}^2|_{r=R_{1}}d\theta\leq 0.
  \]
Therefore, we have 
\begin{align*}
  \begin{aligned}
    \left\|\partial_{r}v_{\theta}\right\|_{L^2\left(\Omega\right)}^2
    &=\int_{\widetilde{\Omega}}\left(\partial_{r}v_{\theta}\right)^2drd\theta
    =\int_{0}^{2\pi}\left(\partial_{r}v_{\theta}v_{\theta}\right)|_{r=R_{1},R_{2}}  d\theta
    -\int_{\widetilde{\Omega}}\partial_{rr}v_{\theta}v_{\theta}drd\theta
    \\
    &\leq 
    -\int_{\widetilde{\Omega}}\partial_{rr}v_{\theta}v_{\theta}drd\theta
    \leq \left\|v_{\theta}\right\|_{L^2\left(\widetilde{\Omega}\right)}
    \left\|\partial_{rr}v_{\theta}\right\|_{L^2\left(\widetilde{\Omega}\right)},
  \end{aligned}
\end{align*}
where the H\"{o}lder inequality is used.
\end{proof}

\begin{proof}\textbf{The proof of Lemma \ref{dengjia0523}.}

  Under the help of Lemma \ref{poincarebudengshi0517} and H\"{o}lder inequality, we can obtain 
  \begin{align*}
    \begin{aligned}
    F_{1}\left(\mathbf{V}\right)&\leq \max\left\{R_{2},\frac{1}{R_{1}}\right\}
     \int_{\widetilde{\Omega}}\left\{\left(\left|\partial_{r} v_{r}\right|^2+\left|
     \partial_{r} v_{\theta}\right|^2\right)
     +\left[\left(\partial_{\theta} v_{r}-v_{\theta}\right)^2 +
     \left(\partial_{\theta} v_{\theta}+v_{r}\right)^2\right]
     \right\}drd\theta
     \\
     &\leq C_{2}\|\widetilde{\nabla}\mathbf{V}\|_{L^2\left(\widetilde{\Omega}\right)}^2.
    \end{aligned}
  \end{align*}
  From \(\left\|\partial_{r} v_{r}\right\|_{L^2\left(\widetilde{\Omega}\right)}^2
  \leq C\left(R_{1},R_{2}\right)F_{1}\left(\mathbf{V}\right)\), \(\partial_{\theta} v_{\theta}=-v_{r}
  -r\partial_{r} v_{r}\) and Lemma \ref{poincarebudengshi0517}, we can get 
  \begin{align}\label{guji10523}
    \left\|\partial_{\theta} v_{\theta}\right\|_{L^2\left(\widetilde{\Omega}\right)}^2
    \leq C\left(R_{1},R_{2}\right)F_{1}\left(\mathbf{V}\right).
  \end{align}
  From \(\left\|\partial_{\theta} v_{r}-v_{\theta}\right\|_{L^2\left(\widetilde{\Omega}\right)}^2
  \leq C\left(R_{1},R_{2}\right)F_{1}\left(\mathbf{V}\right)\), 
  \(\left\|\partial_{r} v_{\theta}\right\|_{L^2\left(\widetilde{\Omega}\right)}^2
  \leq C\left(R_{1},R_{2}\right)F_{1}\left(\mathbf{V}\right)\), and Lemma \ref{poincarebudengshi0517}, we have 
  \begin{align}\label{guji20523}
    \left\|\partial_{\theta} v_{r}\right\|_{L^2\left(\widetilde{\Omega}\right)}^2
    \leq C\left(R_{1},R_{2}\right)F_{1}\left(\mathbf{V}\right).
  \end{align}
  Then, from \(\left\|\partial_{r} v_{r}\right\|_{L^2\left(\widetilde{\Omega}\right)}^2
    +\left\|\partial_{r} v_{\theta}\right\|_{L^2\left(\widetilde{\Omega}\right)}^2
  \leq C\left(R_{1},R_{2}\right)F_{1}\left(\mathbf{V}\right)\), \eqref{guji10523} and \eqref{guji20523}, we can conclude that 
  \[
  F_{1}\left(\mathbf{V}\right)\geq C_{1}\left\|\widetilde{\nabla}\mathbf{V}\right\|_{L^2\left(\widetilde{\Omega}\right)}^2.  
  \]
  This completes this proof.
\end{proof}

\begin{proof}\textbf{The proof of Lemma \ref{zhengqi0523}.}

  Let 
  \[
  F_{2}\left(\mathbf{V}\right)=  \left\|
  \begin{pmatrix}
    \Delta_{r}v_{r}-\frac{v_{r}}{r^2}-\frac{2}{r^2}\partial_{\theta} v_{\theta}
    \\
    \Delta_{r}v_{\theta}-\frac{v_{\theta}}{r^2}+\frac{2}{r^2}\partial_{\theta} v_{r}
  \end{pmatrix}
  \right\|_{L^2\left(\widetilde{\Omega}\right)}^2. 
  \]
  Then, from Lemmas \ref{poincarebudengshi0517}, \ref{tiduL2guji0518} and H\"{o}lder inequality, we can 
  conclude that there exists a constant \(C_{2}\) depending on \(R_{1}\) and \(R_{2}\) such that 
  \begin{align}\label{xinshan0523}
    F_{2}\left(\mathbf{V}\right)\leq C_{2}\left\|\widetilde{\Delta}\mathbf{V}\right\|_{L^2\left(\widetilde{\Omega}\right)}^2. 
  \end{align}
  Define 
  \begin{align}\label{dingying0523}
    \begin{aligned}
      &\Delta_{r}v_{r}-\frac{v_{r}}{r^2}-\frac{2}{r^2}\partial_{\theta} v_{\theta}=f_{1},
      \\
      &\Delta_{r}v_{\theta}-\frac{v_{\theta}}{r^2}+\frac{2}{r^2}\partial_{\theta} v_{r}=f_{2}.
    \end{aligned}
  \end{align}
  Multiply \(\eqref{dingying0523}_{1}\) by \(rv_{r}\) and multiply \(\eqref{dingying0523}_{2}\) by 
  \(rv_{\theta}\), then integrate the results over \(\Omega\) and add them up gives that 
  \[
    \int_{\widetilde{\Omega}}\left\{r\left(\left|\partial_{r} v_{r}\right|^2+\left|
    \partial_{r} v_{\theta}\right|^2\right)
    +\frac{1}{r}\left[\left(\partial_{\theta} v_{r}-v_{\theta}\right)^2 +
    \left(\partial_{\theta} v_{\theta}+v_{r}\right)^2\right]
    \right\}drd\theta =-\int_{\widetilde{\Omega}}\left[
rf_{1}v_{r}+rf_{2}v_{\theta}
    \right]drd\theta, 
  \]
  which together with Lemmas \ref{poincarebudengshi0517}, \ref{dengjia0523} and H\"{o}lder inequality yields that 
  \begin{align}\label{guanjian0523}
    \left\|\left(f_{1},f_{2}\right)\right\|_{L^2\left(\widetilde{\Omega}\right)}
    \geq C\left(R_{1},R_{2}\right)\left\|\widetilde{\nabla}\mathbf{V}\right\|_{L^2\left(\widetilde{\Omega}\right)}.
  \end{align}
  In addition, 
  \begin{align*}
    \begin{aligned}
      \left\|\left(f_{1},f_{2}\right)\right\|_{L^2\left(\widetilde{\Omega}\right)}&\geq 
      \left\|
      \left(\partial_{rr} v_{r}+\frac{1}{r^2}\partial_{\theta\theta} v_{r}
      ,\partial_{rr} v_{\theta}+\frac{1}{r^2}\partial_{\theta\theta} v_{\theta}\right)
      \right\|_{L^2\left(\widetilde{\Omega}\right)}
      -\left\|\left(\frac{1}{r}\partial_{r} v_{r}, 
      \frac{1}{r}\partial_{r} v_{\theta}\right)\right\|_{L^2\left(\widetilde{\Omega}\right)}
      \\
      &~~-\left\|\left(\frac{v_{r}}{r^2}+\frac{2}{r^2}\partial_{\theta} v_{\theta}, 
      \frac{v_{\theta}}{r^2}-\frac{2}{r^2}\partial_{\theta} v_{r}\right)\right\|_{L^2\left(\widetilde{\Omega}\right)}.
    \end{aligned}
  \end{align*}
  Then from \eqref{guanjian0523}, we have 
  \begin{align}\label{enen0523}
    \begin{aligned}
    \sqrt{F_{2}\left(\mathbf{V}\right)}=
    \left\|\left(f_{1},f_{2}\right)\right\|_{L^2\left(\widetilde{\Omega}\right)}&\geq 
      C\left(R_{1},R_{2}\right)\left\|
      \left(\partial_{rr} v_{r}+\frac{1}{r^2}\partial_{\theta\theta} v_{r}
      ,\partial_{rr} v_{\theta}+\frac{1}{r^2}\partial_{\theta\theta} v_{\theta}\right)
      \right\|_{L^2\left(\widetilde{\Omega}\right)}
      \\
      &\geq C\left(R_{1},R_{2}\right)\left\|
      \left(\partial_{rr} v_{r}+\partial_{\theta\theta} v_{r}
      ,\partial_{rr} v_{\theta}+\partial_{\theta\theta} v_{\theta}\right)
      \right\|_{L^2\left(\widetilde{\Omega}\right)}.
  \end{aligned}
\end{align}
Therefore, from \eqref{xinshan0523} and \eqref{enen0523}, we finish this proof. 
\end{proof}

\begin{proof}\textbf{The proof of Lemma \ref{stokesguji}.}

The proof of this Lemma is lengthy, we have to divide it into two main parts: 1) improve the regularity of \(\mathbf{u}\) with the outer boundary; 
2)improve the regularity of \(\mathbf{u}\) with the inner boundary. 

\textbf{1}, since \(\mathbf{u}|_{r=R_{2}}=\mathbf{0}\), we can extend \(\mathbf{u}\) from \(\Omega\) to the region \(\left\{\left(x_{1},x_{2}\right)|x_{1}^2+x_{2}^2\geq R_{1}^2\right\}\) by the way 
that \(\mathbf{u}\equiv \mathbf{0}\) as \(x_{1}^2+x_{2}^2\geq R_{2}^2\) and 
\(\mathbf{u}\) remains unchanged when \(\left(x_{1},x_{2}\right)\in \Omega\).
We also treat \(p\) and \(\mathbf{f}\) by the same way. Thus, we have 
\begin{align*}
\mu\int_{r\geq R_{1}}\nabla\mathbf{u}\cdot\nabla\phi dx-\mu\left\langle \frac{\partial \mathbf{u}}{\partial\mathbf{n}},\phi\right\rangle
    -\int_{r\geq R_{1}}p\nabla\cdot\phi dx=
    \int_{r\geq R_{1}}\mathbf{f}\cdot\phi dx,
\end{align*}
for any \(\mathbf{\phi}\in \widetilde{\mathbf{E}}^{1,q'}\left(\Omega\right)\). Let \(\psi=\psi\left(x_{1},x_{2}\right)\) be smooth in the region \(\left\{\left(x_{1},x_{2}\right)|x_{1}^{2}+x_{2}^2\geq R_{1}^{2}\right\}\) and 
satisfy: 1) when \(\left(x_{1},x_{2}\right)\in\Omega_{\epsilon}=\left\{\left(x_{1},x_{2}\right)|x_{1}^{2}+x_{2}^2\geq R_{1}^2+\epsilon^2\right\}\), 
\(\psi\equiv 1\); 2) when \(\left(x_{1},x_{2}\right)\in\Omega^{\frac{\epsilon}{2}}=\left\{
\left(x_{1},x_{2}\right)|x_{1}^2+x_{2}^2\leq R_{1}^{2}+\frac{\epsilon^2}{4}
\right\}\), \(\psi\equiv 0\); 
3)\(\psi\in\left[0,1\right]\). Then, from the definition of \(q-\)generalized solution, one can verify that for any \(\phi\in\left[C_{0}^{\infty}\left(R^2\right)\right]^2\),
\begin{align*}
\mu\int_{R^2}\nabla\mathbf{u}\cdot 
\nabla\left(\psi\phi\right)dx-
\int_{R^2}p\nabla\cdot\left(\psi\phi\right)dx=\int_{R^2}\mathbf{f}\cdot \psi\phi dx.
\end{align*}
If \(\widetilde{\mathbf{u}}:=\psi\mathbf{u}\) and \(\widetilde{p}:=\psi p\), then 
we have 
\begin{align}\label{jinxin1020}
\mu \int_{R^2}\nabla\widetilde{\mathbf{u}}\cdot\nabla\phi dx-\int_{R^2}\widetilde{p}\nabla\cdot\phi dx=\int_{R^2}\mathbf{F}\cdot\phi dx, 
\end{align}
where 
\(\mathbf{F}:=\psi\mathbf{f}-\mu\nabla\cdot\left(\nabla\psi \mathbf{u}\right)-\mu\nabla\mathbf{u}\cdot\nabla\psi+p\nabla\psi\in\left[L^{q}\left(\Omega\right)\right]^2\). 

Furthermore, we define two difference operators as follows, 
\begin{align*}
\Delta_{h}\mathbf{u}=
\frac{\mathbf{u}\left(x_{1}+h,x_{2}\right)-\mathbf{u}\left(x_{1},x_{2}\right)}{h},~\Delta^{h}\mathbf{u}=\frac{\mathbf{u}\left(x_{1},x_{2}+h\right)-\mathbf{u}\left(x_{1},x_{2}\right)}{h}.
\end{align*}
From \eqref{jinxin1020}, we have 
\begin{align*}
\mu \int_{R^2}\nabla\widetilde{\mathbf{u}}\cdot\nabla\Delta_{-h}\phi dx-\int_{R^2}\widetilde{p}\nabla\cdot\Delta_{-h}\phi dx=\int_{R^2}\mathbf{F}\cdot\Delta_{-h}\phi dx, 
\end{align*}
which implies that 
\begin{align*}
\mu \int_{R^2}\nabla\Delta_{h}\widetilde{\mathbf{u}}\cdot\nabla\phi dx-\int_{R^2}\Delta_{h}\widetilde{p}\nabla\cdot\phi dx=\int_{R^2}\Delta_{h}\mathbf{F}\cdot\phi dx, 
\end{align*}. Thus, \(\left(\Delta_{h}\mathbf{u},\Delta_{h}\widetilde{p}\right)\) is the \(q-\)generalized solution to the following problem
\begin{align*}
\begin{cases}
-\mu\Delta\Delta_{h}\widetilde{\mathbf{u}}+\nabla\Delta_{h}\widetilde{p}=\Delta_{h}\mathbf{F},
\\
\nabla\cdot \Delta_{h}\widetilde{\mathbf{u}}=\nabla\psi\cdot \mathbf{u}.
\end{cases}
\end{align*}
Since \(\Delta_{h}\mathbf{F}\in \left[D_{0}^{-1,q}\left(R^2\right)\right]^2\), then from the Theorem IV.2.2 in \cite{Galdi2011}, we conclude that there exists a constant \(C\) independent of \(h\) such that 
\[
\left\|\Delta_{h}\nabla \widetilde{\mathbf{u}}\right\|_{L^{q}\left(R^2\right)}
+\left\|\Delta_{h}\widetilde{p}\right\|_{L^{q}\left(R^2\right)}\leq C
\left(\left\|\Delta_{h}\mathbf{F}\right\|_{-1,q}+
\left\|\nabla\psi\cdot \mathbf{u}\right\|_{L^q\left(R^2\right)}\right),
\]
indicating that 
\begin{align*}
\left\|\partial_{1}\nabla \widetilde{\mathbf{u}}\right\|_{L^{q}\left(\sqrt{R_{1}^2+\epsilon^2}\leq r\leq R_{2}\right)}
+\left\|\partial_{1}\widetilde{p}\right\|_{L^{q}\left(\sqrt{R_{1}^2+\epsilon^2}\leq r\leq R_{2}\right)}\leq C
\left(\left\|\mathbf{f}\right\|_{L^{q}\left(\Omega\right)}+
\left\| \mathbf{u}\right\|_{W^{1,q}\left(\Omega\right)}\right),
\end{align*}
where \eqref{youshihou1020} is used.

Similarly, we can obtain that 
\begin{align*}
\left\|\partial_{2}\nabla \widetilde{\mathbf{u}}\right\|_{L^{q}\left(\sqrt{R_{1}^2+\epsilon^2}\leq r\leq R_{2}\right)}
+\left\|\partial_{2}\widetilde{p}\right\|_{L^{q}\left(\sqrt{R_{1}^2+\epsilon^2}\leq r\leq R_{2}\right)}\leq C
\left(\left\|\mathbf{f}\right\|_{L^{q}\left(\Omega\right)}+
\left\| \mathbf{u}\right\|_{W^{1,q}\left(\Omega\right)}\right).
\end{align*}
Thus, for any \(0<\epsilon<R_{2}-R_{1}\), we have 
\begin{align*}
\left\|\nabla^2 \widetilde{\mathbf{u}}\right\|_{L^{q}\left(\sqrt{R_{1}^2+\epsilon^2}\leq r\leq R_{2}\right)}
+\left\|\nabla\widetilde{p}\right\|_{L^{q}\left(\sqrt{R_{1}^2+\epsilon^2}\leq r\leq R_{2}\right)}\leq C
\left(\left\|\mathbf{f}\right\|_{L^{q}\left(\Omega\right)}+
\left\| \mathbf{u}\right\|_{W^{1,q}\left(\Omega\right)}\right),
\end{align*}
where \(C>0\) is independent of \(\mathbf{u}\).

\textbf{2}, from the first step, we can define the value of \(\partial_{1}u_{2}-\partial_{2}u_{1}\) on the outer boundary since \(\mathbf{u}\in \left[W^{2,q}\left(\sqrt{R_{1}^2+\epsilon^2}\leq r\leq R_{2}^2\right)\right]^2\). Follow the steps in \cite{Clopeau1998,li_global_2021}, we will construct a solution to the following equation 
\begin{align}\label{wufadingyi1020}
\begin{cases}
-\mu\Delta \mathbf{w}+\nabla\Theta=\mathbf{f},
\\
\nabla\cdot \mathbf{w}=0,
\\
\mathbf{w}\cdot\mathbf{n}|_{r=R_{1},R_{2}}=0,~
\left(\partial_{1}w_{2}-\partial_{2}w_{1}\right)|_{r=R_{1}}
=\left(\frac{2}{r}-\frac{\alpha}{\mu}\right)v_{\theta}|_{r=R_{1}},
\\
\left(\partial_{1}w_{2}-\partial_{2}w_{1}\right)|_{r=R_{2}}
=\left(\partial_{1}u_{2}-\partial_{2}u_{1}\right)|_{r=R_{2}},
\end{cases}
\end{align}
where \eqref{huajian0126} is used, the relationships between \(\left(u_{1},u_{2}\right)\) and \(\left(v_{r},v_{\theta}\right)\) are \eqref{bianhuan0603} and \eqref{bianhuan0913}. 
Obviously, \(\left(\mathbf{u},p\right)\) is a \(q-\)generalized solution of \eqref{wufadingyi1020}.  

First, we consider the following auxiliary problem 
\begin{align}\label{fuzhu11020}
\begin{cases}
-\mu\Delta G=\partial_{1}f_{2}-\partial_{2}f_{1},
\\
G|_{\partial\Omega}=
\frac{R_{2}-r}{R_{2}-R_{1}}\left(\frac{2}{r}-\frac{\alpha}{\mu}\right)v_{\theta}
+\frac{r-R_{1}}{R_{2}-R_{1}}\left(\partial_{1}u_{2}-\partial_{2}u_{1}\right):=U\in W^{1,q}\left(\Omega\right).
\end{cases}
\end{align}
From the elliptic theory, one can conclude that there exists a weak solution 
\(G\) satisfying 
\[
\left\|\nabla G\right\|_{L^{q}\left(\Omega\right)}
\leq C\left(\left\|\mathbf{f}\right\|_{L^{q}\left(\Omega\right)}+\left\|\mathbf{u}\right\|_{W^{1,q}\left(\Omega\right)}\right).
\]
Furthermore, we consider another auxiliary problem 
\begin{align}\label{fuzhu21020}
\begin{cases}
\Delta\Phi=G,
\\
\Phi|_{\partial\Omega}=0.
\end{cases}
\end{align}
Similarly, from the elliptic theory, there exists a \(\Phi\) satisfying the above equation and 
\[
\left\|\nabla^3\Phi\right\|_{L^{q}\left(\Omega\right)}\leq C\left\|\nabla G\right\|_{L^{q}\left(\Omega\right)}
\leq C\left(\left\|\mathbf{f}\right\|_{L^q\left(\Omega\right)}+\left\|\mathbf{u}\right\|_{W^{1,q}\left(\Omega\right)}\right).
\]
Let 
\[
w_{1}=-\partial_{2}\Phi,~w_{2}=\partial_{1}\Phi,
\]
then 
\(\nabla\cdot\mathbf{w}=0\) and \(\mathbf{w}\cdot\mathbf{n}|_{r=R_{1},R_{2}}=0\). Furthermore, 
\(\Delta\Phi=\partial_{1}w_{2}-\partial_{2}w_{1}=G\). Thus, 
\(-\mu\Delta G=-\mu\Delta\left(\partial_{1}w_{2}-\partial_{2}w_{1}\right)=\partial_{1}f_{2}-\partial_{2}f_{1}\). Then, from the field theory(an irrotational field must be  potential), there exists a \(\nabla\Theta\) such that
\[
-\mu\Delta \mathbf{w}+\nabla\Theta=\mathbf{f}.
\]
From the above construction, one can easily verify that \(\left(\mathbf{w},\Theta\right)\) satisfies the equation \eqref{wufadingyi1020}. And 
\begin{align*}
\left\|\nabla^2\mathbf{w}\right\|_{L^{q}\left(\Omega\right)}+\left\|\nabla\Theta\right\|_{L^{q}\left(\Omega\right)}\leq
C\left(\left\|\mathbf{f}\right\|_{L^q\left(\Omega\right)}+\left\|\mathbf{u}\right\|_{W^{1,q}\left(\Omega\right)}\right).
\end{align*}
If we can show that \(\mathbf{u}=\mathbf{w}\) almost everywhere, then the proof is finished. Thus, we will show the uniqueness of solution to the following equation 
\begin{align}\label{weiyixing1020}
\begin{cases}
-\mu\Delta \mathbf{H}+\nabla \Theta_{1}=\mathbf{0},
\\
\nabla\cdot \mathbf{H}=0,
\\
\mathbf{H}\cdot\mathbf{n}|_{r=R_{1},R_{2}}=0,~\left(\partial_{1}H_{2}-\partial_{2}H_{1}\right)|_{r=R_{1},R_{2}}=0.
\end{cases}
\end{align}
To this end, we consider the following 
auxiliary problem 
\begin{align}\label{renyi1020}
\begin{cases}
-\mu\Delta \mathbf{Q}+\nabla \widetilde{\Theta}=\widetilde{\mathbf{F}},
\\
\nabla\cdot \mathbf{Q}=0,~
\mathbf{Q}\cdot\mathbf{n}|_{r=R_{1},R_{2}}=0,~\left(\partial_{1}Q_{2}-\partial_{2}Q_{1}\right)|_{r=R_{1},R_{2}}=0,
\end{cases}
\end{align}
where \(\widetilde{\mathbf{F}}\in \left[L^{q'}\left(\Omega\right)\right]^2\) is arbitrary. Following the same steps for the auxiliary problems 
\eqref{fuzhu11020} and \eqref{fuzhu21020} where 
\(U\equiv 0\) and \(\mathbf{f}\) is replaced by \(\widetilde{\mathbf{F}}\), one can obtain the strong solution \(\left(\mathbf{Q},\widetilde{\Theta}\right)\in \left[W^{2,q'}\left(\Omega\right)\right]^2\times \left[L^{q'}\left(\Omega\right)\right]^2\) to the problem \eqref{renyi1020}. Multiplying \(\eqref{weiyixing1020}_{1}\) by \(\mathbf{Q}\) and 
\(\eqref{renyi1020}_{1}\) by \(\mathbf{H}\), then integrating the results over \(\Omega\) by parts, respectively, one can find that 
\[
0=\int_{\Omega}\widetilde{\mathbf{F}}\cdot \mathbf{H}dx,
\]
which implies \(\mathbf{H}=\mathbf{0}\) almost everywhere since 
\(\widetilde{\mathbf{F}}\) is arbitrary. Thus, the uniqueness for the problem \eqref{weiyixing1020} holds. Therefore, we can conclude that \(\mathbf{u}=\mathbf{w}\) almost everywhere. Thus, \(\left(\mathbf{u},p\right)\) satisfies the estimate \eqref{manzu1020}. Finally, from the definition of \(q-\)generalized solution, one can show that \(\left(\mathbf{u},p\right)\) satisfies the boundary condition \(\eqref{stokeguji1020}_{4}\). 
\end{proof}

\begin{proof}\textbf{The proof of Lemma \ref{leraytouying0514}.}

  Since \(\left[C^{\infty}\left(\widetilde{\Omega}\right)\right]^2\) is dense in 
  \(\left[L^2\left(\widetilde{\Omega}\right)\right]^2\), thus, we assume that 
  \(\mathbf{f}\in \left[C^{\infty}\left(\widetilde{\Omega}\right)\right]^2\) and consider the following equation 
  \begin{align}\label{jiushizheyang0514}
    \begin{cases}
      r\partial_{rr} p+\partial_{r} p
      +\frac{1}{r}\partial_{\theta\theta} p
      =\partial_{r}\left(r f_{1}\right)+\partial_{\theta} f_{2},
    \\
    \partial_{r} p\left(R_{1},\theta\right)= f_{1}\left(R_{1},\theta\right),~
    \partial_{r} p\left(R_{2},\theta\right)=f_{1}\left(R_{2},\theta\right).
    \end{cases}
  \end{align}
  The existence of weak solution \(p\in H^{1}\left(\widetilde{\Omega}\right)\) to \eqref{jiushizheyang0514} 
  can be obtained by Lax-Milgram theorem. \(p\) satisfies the following equality 
  \begin{align}\label{yinggaixuyao0514}
     \int_{\widetilde{\Omega}}\left(rf_{1}\partial_{r} \widetilde{p}+f_{2}
     \partial_{\theta}\widetilde{p}\right)drd\theta=\int_{\widetilde{\Omega}}
     \left(r\partial_{r} p \partial_{r}\widetilde{p}+
     \frac{1}{r}\partial_{\theta} p\partial_{\theta}\widetilde{p}\right)drd\theta,~\forall~\widetilde{p}\in 
     \overline{H}^{1}\left(\widetilde{\Omega}\right). 
   \end{align}
  
  Then, by the bootstrap method, we can conclude that \(p\in C^{\infty}\left(\widetilde{\Omega}\right)\). Thus, 
  \(p\) satisfies \(\eqref{jiushizheyang0514}_{1}\). 
  Subsequently, we show that \(p\in C^{\infty}\left(\widetilde{\Omega}\right)\) satisfies the 
  boundary conditions \(\eqref{jiushizheyang0514}_{2}\). Multiply \eqref{jiushizheyang0514} by 
  \(\widetilde{p}\in H^{1}\left(\widetilde{\Omega}\right)\) and integrate over \(\widetilde{\Omega}\),
  then from \eqref{yinggaixuyao0514}, we have 
  \begin{align*}
    \begin{aligned}
    \int_{0}^{2\pi}\left(r f_{1}|_{r=R_{1},R_{2}}\right)\widetilde{p}d\theta 
    =\int_{0}^{2\pi}\left(r \partial_{r} p|_{r=R_{1},R_{2}}\right)\widetilde{p} d\theta,
    \end{aligned} 
  \end{align*}
  Thus, we have \( \partial_{r} p\left(R_{1},\theta\right)= f_{1}\left(R_{1},\theta\right)\) and
  \(\partial_{r} p\left(R_{2},\theta\right)=f_{1}\left(R_{2},\theta\right).\)
  
  Let 
  \[
  v_{r}=f_{1}-\partial_{r} p, ~v_{\theta}=f_{2}-\frac{1}{r}\partial_{\theta} p.  
  \]
  We can check that \(\partial_{r}\left(rv_{r}\right)+\partial_{\theta} v_{\theta}=0\)
  and \(v_{r}|_{r=R_{1},R_{2}}=0\). 
  
  Finally, assume that  
  \[
  \mathbf{f}=\begin{pmatrix}
    v_{r}
    \\
    v_{\theta}
  \end{pmatrix}
  +  \begin{pmatrix}
    \partial_{r} p
    \\
    \frac{1}{r}\partial_{\theta} p
  \end{pmatrix}
  =\begin{pmatrix}
    \widetilde{v}_{r}
    \\
    \widetilde{v}_{\theta}
  \end{pmatrix}
  +  \begin{pmatrix}
    \partial_{r} \widetilde{p}
    \\
    \frac{1}{r}\partial_{\theta} \widetilde{p}
  \end{pmatrix},
  \]
  where \(\partial_{r}\left(r v_{r}\right)+\partial_{\theta} v_{\theta}
  =\partial_{r}\left(r \widetilde{v}_{r}\right)
  +\partial_{\theta} \widetilde{v}_{\theta}=0\) and \(v_{r}|_{r=R_{1},R_{2}}=
  \widetilde{v}_{r}|_{r=R_{1},R_{2}}=0\). Then we have 
  \[
    \begin{pmatrix}
      v_{r}-\widetilde{v}_{r}
      \\
      v_{\theta}-\widetilde{v}_{\theta}
    \end{pmatrix}
    =\begin{pmatrix}
      \partial_{r} \widetilde{p}-\partial_{r} p
      \\
      \frac{1}{r}\partial_{\theta} \widetilde{p}-\frac{1}{r}\partial_{\theta} p
    \end{pmatrix}.
  \]
  Multiplying the above equality by \(\left(r\left(v_{r}-\widetilde{v}_{r}\right),
  r\left(v_{\theta}-\widetilde{v}_{\theta}\right)\right)\) and integrating over \(\widetilde{\Omega}\)
  give  
  \[
  \int_{\widetilde{\Omega}}r\left(\left|v_{r}-\widetilde{v}_{r}\right|^2
  +\left|v_{\theta}-\widetilde{v}_{\theta}\right|^2\right)drd\theta 
  =0, 
  \]
  which implies the uniqueness of the decomposition. 
\end{proof}

  \begin{proof}\textbf{The proof of Lemma \ref{budengshidezhengming0530}.}
  
    Let \(H\left(t\right)=\int_{0}^{t}
    \left\|\widetilde{\nabla}\mathbf{V}\right\|_{L^2\left(\widetilde{\Omega}\right)}^6ds
    \). Then from Cauchy inequality, we have 
    \[
      H'\left(t\right)\leq C\left[H\left(t\right)\right]^3+\widetilde{C}_{1}(t+1)^3,
    \]
   where \(|\widetilde{C}_{1}-C_{1}|\) is small enough. Then, we have  
    \[
    \frac{d}{dt}\left[e^{-C\int_{0}^{t}\left[H\left(s\right)\right]^2ds}H\left(t\right)\right]\leq \widetilde{C}_{1}(t+1)^3.   
    \]
    Therefore, we have 
    \[
      e^{-C\int_{0}^{t}\left[H\left(s\right)\right]^2ds}H\left(t\right)\leq \widetilde{C}_{1}(t+1)^4,
    \]
    which indicates 
    \[
      e^{-C\int_{0}^{t}\left[H\left(s\right)\right]^2ds}H^2\left(t\right)\leq \widetilde{C}_{1}^2(t+1)^8. 
    \]
    Therefore, 
    \[
      -\frac{d}{dt}e^{-C\int_{0}^{t}\left[H\left(s\right)\right]^2ds}\leq \widetilde{C}_{1}^2 (t+1)^8,
    \]
    As a result, we obtain that when \(1-\widetilde{C}_{1}^2(t+1)^9>0\),
    \[
      e^{C\int_{0}^{t}\left[H\left(s\right)\right]^2ds}\leq \frac{1}{1-\widetilde{C}_{1}^2(t+1)^9},
    \]
    which implies the existence of \(T^{*}\) and \(C^{*}\). Furthermore, we 
    have \(H\left(t\right)\leq C^{*}\), indicating that 
    \(\left\|\widetilde{\nabla}\mathbf{V}\left(t\right)\right\|_{L^2\left(\widetilde{\Omega}\right)}^2
    \leq C^{*}\).
  \end{proof}
\begin{remark}\label{yuexiaoyueda1021}
    It is easy to find that 
    the smaller \(C_{1}\) is, the longer the existence time \(T^{*}\) and the smaller \(C^{*}\) will be.
  \end{remark}

    \section*{Acknowledgement}
 Q. Wang was supported by the
Natural Science Foundation of Sichuan Province (No.2025ZNSFSC0072).


\end{document}